\documentclass[12pt, a4paper, parskip=half, abstracton, bibliography=totoc]{scrartcl}
\pdfoutput=1

\usepackage{array}
\usepackage{marginnote}
\usepackage{xcolor}
\usepackage{amscd,amsfonts,amssymb,amsmath,latexsym,amsthm}
\usepackage{hyperref}
\usepackage[all,cmtip]{xy}
\usepackage{mathrsfs}
\usepackage{graphicx}
\usepackage{caption}
\usepackage{makeidx}
\makeindex

\usepackage{etoolbox}

\def\hB{\hspace*{\fill}$\qed$}

\let\counterwithout\relax

\usepackage{chngcntr} 

\usepackage{defs_pp1}
\usepackage{slashed}
\usepackage[utf8]{inputenc}
\usepackage{microtype}
\usepackage[english]{babel}

\usepackage{bm} 
\usepackage{mathtools}

\title{Homotopy theory with\\ bornological coarse spaces}
\author{
Ulrich Bunke\thanks{Fakult{\"a}t f{\"u}r Mathematik,
Universit{\"a}t Regensburg,
93040 Regensburg,
GERMANY\newline
ulrich.bunke@mathematik.uni-regensburg.de} 
\and
Alexander Engel\thanks{Fakult{\"a}t f{\"u}r Mathematik,
Universit{\"a}t Regensburg,
93040 Regensburg,
GERMANY\newline
alexander.engel@mathematik.uni-regensburg.de}
}

\date{}

\numberwithin{equation}{section}
\counterwithout{footnote}{section}

\newtheorem{theorem}{Theorem}[section] 
\newtheorem{prop}[theorem]{Proposition}
\newtheorem{lem}[theorem]{Lemma}
\newtheorem{ddd}[theorem]{Definition}
\newtheorem{kor}[theorem]{Corollary}

\theoremstyle{remark}
\theoremstyle{definition}

\newtheorem{ex}[theorem]{Example}
\newtheorem{rem}[theorem]{Remark}

\newcommand{\sep}{\mathrm{sep}}
\newcommand{\lc}{\mathrm{lc}}
\newcommand{\Loc}{\mathrm{Loc}}
\newcommand{\fin}{\mathrm{fin}}
\newcommand{\sph}{\mathrm{sph}}
\newcommand{\coarse}{\mathrm{coarse}}
\newcommand{\disc}{\mathrm{disc}}
\newcommand{\Kcat}{\mathrm{K}^{C^{*}Cat}}
\newcommand{\la}{\mathrm{la}}
\newcommand{\an}{\mathrm{an}}
\newcommand{\mot}{\mathrm{mot}}
\newcommand{\topp}{\mathrm{top}}
\newcommand{\alg}{\mathrm{alg}}
\newcommand{\Ass}{\mathrm{Ass}}
\newcommand{\ql}{\mathrm{ql}}
\newcommand{\gr}{\mathrm{gr}}
\newcommand{\cJ}{\mathcal{J}}
\newcommand{\Mor}{\mathrm{Mor}}
\newcommand{\Yo}{\mathrm{Yo}}

\newcommand{\yo}{\mathrm{yo}}

\newcommand{\Orb}{\mathbf{Orb}}

\newcommand{\cR}{\mathcal{R}}
\newcommand{\inter}{\mathrm{int}}

\newcommand{\BC}{\mathbf{BornCoarse}}
\newcommand{\TopBorn}{\mathbf{TopBorn}}

 \newcommand{\op}{\mathrm{op}}
 \newcommand{\fl}{\mathrm{fl}}

\newcommand{\Ob}{\mathrm{Ob}}

\newcommand{\Cofib}{\mathtt{Cofib}}
\newcommand{\bB}{{\mathbf{B}}}
\newcommand{\Fib}{{\mathtt{Fib}}}

\newcommand{\incl}{\mathrm{incl}}

\newcommand{\cP}{\mathcal{P}}

\newcommand{\bF}{{\mathbf{F}}}

\newcommand{\cZ}{{\mathcal{Z}}}
\newcommand{\bbA}{{\mathbb{A}}}

\newcommand{\cW}{{\mathcal{W}}}
\newcommand{\PSh}{{\mathbf{PSh}}}

\newcommand{\bA}{{\mathbf{A}}}

\newcommand{\bK}{{\mathbf{K}}}
\newcommand{\const}{{\mathtt{const}}}

\newcommand{\Alg}{{\mathbf{Alg}}}

\newcommand{\cO}{{\mathcal{O}}}
\newcommand{\cU}{{\mathcal{U}}}
\newcommand{\cY}{{\mathcal{Y}}}

\newcommand{\cD}{{\mathcal{D}}}

 \newcommand{\Cat}{{\mathbf{Cat}}}

\newcommand{\Fl}{\mathrm{Fl}}

\newcommand{\cE}{{\mathcal{E}}}

\newcommand{\KXql}{K\!\mathcal{X}_{\mathrm{ql}}}

\newcommand{\KX}{K\!\mathcal{X}}
\newcommand{\HX}{H\!\mathcal{X}}
\newcommand{\CX}{C\!\mathcal{X}}
\newcommand{\KK}{K\!K}
\newcommand{\Cech}{\v{C}ech }
\newcommand{\IK}{\mathbb{K}}
\newcommand{\lf}{\mathrm{lf}}
\newcommand{\free}{\mathrm{free}}
\newcommand{\mixed}{\mathrm{mixed}}

\newcommand{\Sigmalf}{\Sigma_{+}^{\infty,\lf}}
\renewcommand{\epsilon}{\varepsilon}

\renewcommand{\Dirac}{\slashed{D}}
\newcommand{\Born}{\mathbf{Born}}

\newcommand{\Spc}{\mathbf{Spc}}

\newcommand{\IN}{\mathbb{N}}
\newcommand{\IZ}{\mathbb{Z}}
\newcommand{\IR}{\mathbb{R}}
\newcommand{\IC}{\mathbb{C}}

\newcommand{\cCql}{\mathcal{C}_{\mathrm{ql}}}
\newcommand{\cClc}{\mathcal{C}_{\lc}}
\newcommand{\cClcql}{\mathcal{C}_{\lc,\ql}}
\newcommand{\Ccat}{{\mathbf{C}^{\ast}\mathbf{Cat}}}
\newcommand{\Calg}{{\mathbf{C}^{\ast}\mathbf{Alg}}}

\newcommand{\simpp}{\mathrm{simp}}

\newcommand{\Kast}{K^{C^{*}}}
\begin{document}

\maketitle
\vspace*{-3.5ex}
\begin{abstract}
We propose an axiomatic characterization of coarse homology theories defined
on the category of bornological coarse spaces $\BC$. {Using methods
of {abstract} homotopy theory we construct  the universal coarse homology theory $\Yo^{s}\colon\BC\to \Sp\cX$}
which sends a bornological coarse space to its coarse motivic spectrum.
 

The homological properties of $\Yo^{s}$ are used  {to}
 study the coarse motivic spectra of bornological coarse spaces  equipped with  N.~Wright's  hybrid coarse
structure which interpolates between coarse geometry and the local geometry of uniform spaces.

 In this book we construct various examples of coarse homology theories, e.g. coarse ordinary homology and coarse $K$-homology. We further discuss the coarsifications of locally finite homology theories. To this end we develop some background on locally finite homology theories in the context of topological bornological spaces $\TopBorn$ generalizing the theory of B.~Williams  {and} M.~Weiss. The example of locally finite analytic $K$-homology is of particular interest.
 
Besides the construction of coarse homology theories and the devlopment of the corresponding background the main result of the present paper concerns the classification of coarse homology theories by their values on discrete spaces.
These results will be used in subsequent work to prove isomorphism results for  the coarse assembly map
\end{abstract}

\tableofcontents

\section{Introduction}


Coarse geometry was invented by J.~Roe 
 (see e.g.~\cite{MR1147350,roe_lectures_coarse_geometry}) 
 in connection with applications to the index theory of Dirac-type operators on complete Riemannian manifolds. To capture 
 the index of these operators $K$-theoretically or numerically  J.~Roe furthermore constructed first examples  of coarse homology theories, namely coarse $K$-homology, coarse ordinary homology, and   coarsifications of locally finite homology theories.
 
The symbol class of a Dirac-type operator is a locally finite $K$-homology class. 
 The   transition from the symbol class to the coarse index of the  Dirac operator  proceeds in two steps.  The first step
 sends the symbol class to the coarse symbol  class in  the coarsification of the locally finite $K$-homology. This step removes the local information
 and builds a coarsely invariant object which is still of topological nature. The second step consists of an application of
 the  coarse assembly map   \cite[Sec.~5]{roe_lectures_coarse_geometry}. The resulting coarse index captures global analytic properties of the Dirac operator, in particular its invertibility. 
The $K$-theoretic coarse Baum--Connes conjecture   \cite{hr} asks under which conditions this assembly map is an isomorphism. 
We added the adjective ``$K$-theoretic'', {since} we now understand this assembly map as a specialization of a more general construction \cite{ass}.

Information about the $K$-theoretic coarse  assembly  map has many implications to index theory, geometry and topology. 
 In good situations surjectivity implies that 
 some coarse $K$-theory classes can be realized as indices of generalized Dirac operators. Because the analysis of Dirac operators is well-developed, this
  allows conclusions about properties of coarse $K$-theory classes, e.g. vanishing of delocalized traces. On the other hand, injectivity
  has applications to the positive scalar curvature question. Indeed, uniform positive scalar curvature implies that the {spin} Dirac operator 
  is invertible so that its index vanishes. One can then conclude that its symbol class vanishes at least coarsely, and this has consequences for the topology of the manifold. 
  Finally, if the  $K$-theoretic coarse assembly map  for the group $G$ with the word metric is an isomorphism, then this implies via the descent principle the Novikov conjecture
 which has  implications   to the topology of manifolds.


The initial motivation for this book was to uncover the basic structure of the proof of the $K$-theoretic coarse Baum--Connes conjecture  \cite{nw1}. 
Our main new insight is the following.
The proof of the coarse Baum--Connes conjecture  \cite{nw1} under the assumption of finite asymptotic dimension is not specific to $K$-theory. The argument actually  shows that a natural transformation between two coarse homology theories (if suitably axiomatized)  is an isomorphism on a space of finite asymptotic dimension  if it is an isomorphism on discrete spaces. This is analogous to the fact that a natural transformation between homology theories satisfying the Eilenberg--Steenrod axioms (except the dimension axiom) induces an isomorphism of their values on CW-complexes, if it induces an isomorphism on the one-point space. This idea can be applied to the $K$-theoretic Baum--Connes conjecture
by interpreting the $K$-theoretic coarse assembly map as a natural transformation between coarse homology theories in the sense above. This argument will only be completed in \cite{ass} while the present   book provides the construction of the homology theory appearing as  the target of the assembly map.

The first goal of the present book is to propose a set of axioms for coarse homology theories which make this idea work.
 The basic category containing the objects of interest  is the category of bornological coarse spaces $\BC$ (Definition \ref{etgkowergferfrwefrefw}). A bornological coarse space is a set equipped with a coarse structure and a compatible bornology.  A morphism between bornological coarse spaces  is a  map   which is  controlled and proper. The idea of this definition is to work with the minimal amount of structure.
Metric spaces present bornological coarse spaces, but the datum of a metric  is a much finer structure and not preserved by isomorphisms in $\BC$. Moreover there are important examples of bornological coarse spaces which  can not be presented by metric spaces, e.g. the spaces
with the hybrid structure    occurring in the argument of   \cite{nw1} or continuously controlled coarse structures.

{In order to} formulate the basic axioms for a coarse homology theory   technically, we will use the language of $\infty$-categories \cite{Cisinski:2017,htt}.
This theory is probably one of the most rapidly accepted new theories of mathematics and by now indispensible in modern homotopy theory.  We have tried to minimize the required technical knowledge about $\infty$-categories by using it in a  model independent way. 
A reader unfamiliar with this language  should consider this book  as an opportunity to  learn the language of $\infty$-categories  in action. But
 it  is not the place for explaining many internal  technical details of the language.

A coarse homology theory (Definition \ref{rgljogreggregrege}) is a functor 
$$E:\BC\to \bC$$
whose target is some stable $\infty$-category. 
This functor must be coarsely invariant, excisive and $u$-continuous, and it  must annihilate flasque bornological coarse spaces like the ray $[0,\infty)$.

On the one hand the choice of these axioms  is motivated by the application explaind above. On the other hand they  are natural from the point of view of a homotopy theory of bornological coarse spaces which we will explain in the following.

On the morphism sets of $\BC$ we have the equivalence relation called closeness (Definition \ref{ioewfjweoifoijwfw32424}).  Morphisms in $\BC$ which are invertible up to closeness are
called coarse equivalences (Definition \ref{ioioieorwerr3245345345432}).  In coarse homotopy theory 
  one studies coarse spaces up to coarse equivalence. Furthermore, one considers flasque bornological coarse spaces (Definition \ref{efijewifjewiofwifwfew322423424})  as trivial.   
   Coarse homology theories are defined such that they  provide homotopy theoretic invariants
  of bornological coarse spaces which in view of the excision axiom  {are in addition}   local  in a certain sense.  
  
For different   set-ups and axioms for coarse homology theories see \cite{MR1834777,Heiss:2019aa}.

As our goal is to show   results for arbitrary  coarse homology theories, it is natural to prove them for a universal coarse homology theory (Definition \ref{ewiogwergergeffwfrf})
\begin{equation}\label{rqefjknkjfcewfcewfqecqewx}
\Yo^{s}:\BC\to \Sp\cX\, ,
\end{equation}  where $\Sp\cX$ is called the category of coarse motivic spectra.
 For $X$ in $\BC$ we call $\Yo^{s}(X)$ in $ \Sp\cX$ the coarse motivic spectrum   of $X$.
Using modern homotopy theory the  construction of this universal  coarse homology theory is completely standard.
It starts with the Yoneda embedding $$\yo:\BC\to \PSh(\BC)\, .$$ In order to explain the right-hand side note that 
  for a $\infty$-category  $\cC$ we write  $ \PSh(\cC):=\Fun(\cC^{\op},\Spc)$
for the  $\infty$-category of $\Spc$-valued presheaves. 
 The $\infty$-category $\Spc$ is {the} $\infty$-category of spaces which e.g. contains the mapping spaces between objects in arbitrary $\infty$-categories.
In general we view (without explicitly mentioning) ordinary categories as
$\infty$-categories using the nerve functor.   

In order to construct \eqref{rqefjknkjfcewfcewfqecqewx}
one then applies to $\PSh(\BC)$ various localizations forcing the properties
which are dual to the axioms for a coarse homology theory. This yields the  unstable version (Definition \ref{qergkjefkoerfrefwefwerfev})
$$\Yo \colon \BC\to \Spc\cX\, ,$$ where $\Spc\cX$ is the category of coarse motivic spaces. The stable version \eqref{rqefjknkjfcewfcewfqecqewx}  is then obtained 
 by applying a stabilization functor.

We now formulate one of the main results of the present paper.  
We consider the minimal  cocomplete stable full    subcategory
$$\Sp\cX\langle \cA_{\disc}\rangle\subseteq \Sp\cX$$  of $\Sp\cX$ containing    coarse motivic spectra of the elements of the set $\cA_{\disc}$ of discrete bornological coarse spaces. A discrete bornological coarse space has the simplest possible
coarse structure on its underlying set, but its bornology might be interesting. 
 \begin{theorem}[Theorem \ref{jfweofjwoeifjewfoewfewfewfewf}] \label{eifwuewifeiue23423424324}
If $X$ is a  bornological coarse space of weakly finite  {asymptotic} dimension (Definition \ref{wegoijobgwtwferfrewfer}), then $\Yo^{s}(X)\in \Sp\cX\langle \cA_{\disc}\rangle$.
\end{theorem}
By the universal property of $\Yo^{s}$ a 
 coarse homology theory with values in $\bC$ is  the same as a colimit preserving functor  $\Sp\cX\to \bC$. 
 The following corollary is now immediate.
Let $E\to F$ be a transformation between $\bC$-valued coarse homology theories, and let 
 $X$ be in $\BC$.
 \begin{kor}[Corollary \ref{thm:sdf98245csv}] \label{iewfjweifiiweofu98uu98q}Assume:
 \begin{enumerate}
 \item 
The induced map $E(Y)\to F(Y)$ is an equivalence for all discrete  $Y$ in $\BC$. \item 
$X$ has weakly     finite asymptotic dimension.  \end{enumerate}
Then $E(X)\to F(X)$ is an equivalence.
\end{kor}
 
 Under more restrictive assumptions we also have the following result which is closer to the classical uniqueness results in ordinary   algebraic topology. It requires an additional additivity assumption on the coarse homology theory.
 
  \begin{theorem}[Theorem \ref{wekfhweuifhweiui23u4242342rf}] \label{iewfjweifiiweofu98uu98q1}
 Assume:
 \begin{enumerate}
 \item $E,F$ are additive.
 \item 
The induced map $E(*)\to F(*)$ is an equivalence. 
\item  $X$ has weakly finite asymptotic dimension.
\item $X$ has the minimal compatible bornology. \end{enumerate}
Then $E(X)\to F(X)$ is an equivalence.    \end{theorem}
 
The proof of Theorem \ref{eifwuewifeiue23423424324} follows ideas of \cite{nw1} and uses the coarsening space. This is a simplicial complex build from an anti-\v{C}ech system of coverings of the bornological coarse space in question  (see Subsection \ref{rtkohwrgwgvwerwevfds}). 
The condition of weakly finite asymptotic dimension translates to finite-dimensionality of the coarsening space. 
We  equip this coarsening space with the hybrid coarse structure (Definition \ref{defn:sdfuh8934t})  which interpolates between the coarse structure of the original bornological coarse space and the local metric geometry of the simplicial complex. 
We then show that the coarsening space with the hybrid structure is flasque (Theorem \ref{thm:sdfdvw4v}). As a consequence  its   motivic
coarse spectrum vanishes, and the motivic coarse spectrum of the original space is equivalent up to suspension to the 
germs-at-$\infty$ (see below)   of the coarsening space. 
Using the Decomposition Theorem and the Homotopy Theorem also described below we can then decompose these germs-at-$\infty$   into motivic coarse spectra of a collection of cells, and finally of discrete spaces (see Section \ref{triohjorgvwevfdvsfdv}).  In this argument we  argue by
  induction by the dimension and therefore the  finiteness of the latter is important.
Similar techniques (with different words) are also used elsewhere, e.g. in the proof of the Farell--Jones conjecture in 
\cite{blr}. 

The technical details involved in these arguments are tricky. One goal of  Section \ref{efwljkfowpfjowefwefewfwef}  is to develop the background in great generality and in a ready-to-use way.  After introducing hybrid structures we provide the following two main results which we can only formulate
  roughly at this place. We think of the hybrid structure as a  coarse structure constructed from a uniform structure.
    The first is the Homotopy Theorem
\begin{theorem}[Thm.~\ref{fewijwefio23ri3ohewkjfwefewf}]
Uniform homotopy equivalences  between uniform spaces induce equivalences of the coarse motivic spectra of the spaces equipped with the associated hybrid structures.
\end{theorem}

Note that the hybrid structure on  a uniform space $X$ also depends on an exhaustion $\cY$ of the uniform space. 
We consider  the coarse motivic spectrum $\Cofib(\Yo^{s}(\cY)\to \Yo^{s}(X))$ as the motive of the germ-at-$\infty$ of $X$ relative to the exhaustion $\cY$.
The second result is the Decomposition Theorem.
\begin{theorem}[Thm.~\ref{weiofjewi98u3298r32r32rr}]
The germ-at-$\infty$ is excisive for coarsely and uniformly excisive decompositions.
\end{theorem}
These results will    be used in subsequent papers \cite{equicoarse,ass,fj}.

The next goal of this book is to construct various examples of coarse homology theories satisfying our axioms.  
In particular we will discuss the classical examples of coarse ordinary homology, coarse $K$-homology, and coarsifications of locally finite homology theories.   In all cases we must extend the classical definitions to functors  defined on the whole category $\BC$
with values in an appropriate stable $\infty$-category (e.g.\ chain complexes $\Ch_{\infty}$ for ordinary coarse homology or  spectra $\Sp$ for coarse $K$-homology).

Extending and lifting the construction of coarse ordinary homology is essentially a straightforward matter. The results of Subsection \ref{ewifjewoifoi232344234} can be summarized as follows:
\begin{theorem}
There exists a coarse homology theory  $\HX:\BC\to \Ch_{\infty}$ such that $\pi_{*}\HX:\BC\to \Ab^{\Z\gr}$
is the classical ordinary coarse homology.
\end{theorem}

In  Section \ref{ihjiovqceqwecewcqwecq} we provide a systematic study of locally finite homology theories and their coarsifications.  The part on locally finite homology theories could be understood as a generalization of \cite{ww_pro}.
A locally finite homology theory (Definition \ref{rewfiewrjgoiegergreefwerfw}) is a 
  functor $$F\colon\TopBorn\to \bC\, ,$$ where $\TopBorn$ is the category of topological bornological spaces
(i.e., topological spaces with an additional compatible bornology, see Definition \ref{fljwlewkiou2ori23r2323}), and $\bC$ is a complete and cocomplete stable $\infty$-category. 
Besides the usual homotopy invariance and excision properties  we require the local finiteness condition:
$$\lim_{B} F(X\setminus B)\simeq 0\, ,$$
where the limit runs over the bounded subsets of $X$.

We show that every object $C$ of $\bC$ gives rise to a locally finite homology  {theory} $$(C\wedge \Sigma_{+}^{\infty,\topp})^{\lf}:\BC\to \bC \, .$$
Another example  of completely different origin is the $\Sp$-valued analytic locally finite $K$-homology $K^{\an,\lf}$ which is constructed using Kasparov-$K\! K$-theory. We have $K^{\an,\lf}(*)\simeq KU$ and it is an interesting question to compare  $K^{\an,\lf}(*)$ with its  homotopy theoretic version $(KU\wedge \Sigma_{+}^{\infty,\topp})^{\lf}$.

Motivated by this question we provide a general classification result for locally finite homology theories which we again only formulate in a rough way at this point.
\begin{prop}[Proposition \ref{ifjewifweofewifuoi23542335345}]
The values of  a locally finite homology theory on nice spaces are completely determined by {the} value of the theory  on the point.
\end{prop}
 Since locally finite simplicial complexes are nice
we can conclude:
\begin{kor}[Corollary \ref{wergkjeowrgerfwefref}]  $(KU\wedge \Sigma_{+}^{\infty,\topp})^{\lf}$ and $K^{\an,\lf}$ coincide 
on  topological bornological spaces which are homotopy equivalent to locally finite simplicial complexes.
\end{kor} 
We think that this result is folklore, but we could not locate a proof in the literature.

In the present book our  main usage of locally finite homology theories is as an input for the coarsification machine.
The main result of Subsection \ref{jhksdf2323r} can be formulated as follows.

\begin{theorem}
To  every locally finite homology theory $F$ we can associate functorially a  coarse homology theory $QF$ such that $QF(*)\simeq F(*)$.
\end{theorem}
 The coarse homology $QF$ is called the coarsification of $F$. 
 The coarseification of the analytic locally finite homology theory $QK^{\an,\lf}$ is of particular importance since it 
 is the domain of the $K$-theoretic coarse assembly map which we explain in Subsection \ref{sec:kjnbsfd981}.

The last example of a coarse homology theory considered in the present book is the coarse topological $K$-homology.
In contrast to the construction of the coarse ordinary homology theory, lifting the classical definition of coarse $K$-homology from a group-valued functor defined on proper metric spaces to a spectrum-valued functor defined on $\BC$ is highly non-trivial.  The whole Section \ref{sec:sdn934} is devoted to this problem.
 Its  short summary can be stated as follows (Theorems~\ref{irgfjiwoowierer23423424243} and~\ref{fwefiwjfeiooi234234324434}):
\begin{theorem}
There exists a coarse homology theory $\KX\colon\BC\to \Sp$ such that  
$\pi_{*}\KX(X)$ is the classical coarse $K$-homology for proper metric spaces $X$.
\end{theorem}
The construction of the functor $\KX$ proceeds in two steps. In the first one associates to a bornological coarse space $X$ a Roe category
$\bC^{*}(X)$ (Definition \ref{foiefoiefoewfewfewiofef}). This is a $C^{*}$-category which generalizes the classical construction of a Roe algebra. The functor $\KX$ is then obtained from $\bC^{*}$ by composing with a topological $K$-theory functor for $C^{*}$-categories (Definition \ref{wegjioegergewfewrfewef}).
In   Section \ref{sec:sdn934}  will also give a selfcontained presentation of the necessary background
on $C^{*}$-categories and their $K$-theory.

In the following we indicate why the construction of $\KX$ required some new ideas.
The classical  construction  presents  $\KX_{*}$ in terms of  the  $K$-theory groups of Roe algebras.
 Note that the construction of a Roe algebra (Definition \ref{qekdfjqodqqwdqwdqwdqdwq})
involves the choice of an $X$-controlled  Hilbert space   and is therefore not strictly functorial in the space $X$. 
The idea behind using the Roe category 
 is to ensure
functoriality  by working   with the $C^{*}$-category of all possible choices. Unfortunately, the details are more complicated.
For example, in order to get the correct $K$-theory the $X$-controlled Hilbert space used in the construction of the Roe algebra must be ample   (Definition \ref{oiwfiowioewioewfewfew}). But 
 ample Hilbert spaces only exist under additional assumptions on $X$ (Proposition \ref{fkwjlkwejlewfewfewfewf3242344}).

The morphism spaces  in $\bC^{*}(X)$ are defined as closures of spaces of operators of controlled propagation.
There is a slightly weaker condition called quasi-locality which leads 
to a slightly bigger $C^{*}$-category $\bC^{*}_{\ql}(X)$. By composing the functor $\bC^{*}_{\ql}$ with  the  topological $K$-theory functor for $C^{*}$-categories we get a  another coarse $K$-theory functor $\KX_{\ql}$.
The inclusions $\bC^{*}(X)\to \bC^{*}_{\ql}(X)$ induce a natural transformation of coarse homology theories
$\KX\to \KX_{\ql}$ which is obviously an equivalence on discrete bornological coarse spaces.
So Corollary \ref{iewfjweifiiweofu98uu98q} has the following application.
Let $X$ be in $\BC$.
\begin{kor}\label{eqriojgerwgwergregwf}
If $X$ has weakly finite asymptotic dimension, then $\KX(X)\to \KX_{\ql}(X)$ is an equivalence.
\end{kor}

After the preprint version of this book appeared,  \v{S}pakula--Tikuisis \cite{st} have shown that Corollary \ref{eqriojgerwgwergregwf} is true under the weaker condition that $X$ has finite straight {decomposition complexity.}    Later, \v{S}pakula--Zhang generalized the result further to exact spaces \cite{spakula_zhang}.

In order to apply Corollary \ref{iewfjweifiiweofu98uu98q} or Theorem \ref{iewfjweifiiweofu98uu98q1}  
to a coarse assembly map 
it is necessary to 
interpret this coarse assembly map   as a natural transformation between coarse homology theories in the sense above. In the present book
this goal is only reached partially and only in the case of coarse $K$-theory, see Subsection \ref{sec:kjnbsfd981}. 
 
 In the following we explain some details.
 A coarse homology theory $E:\BC\to \bC$ gives rise to the object $E(*)$ in $\bC$.
  As explained above the object $E(*)$  in $\bC$ gives rise to a locally finite 
homology theory
$$  (E(*)\wedge \Sigma^{\infty,\topp}_{+})^{\lf} :\TopBorn\to \bC\, .$$
We can now form its coarsification \begin{equation}\label{dcqwoijodqwcdadscdascadsc}
Q(E(*)\wedge \Sigma^{\infty,\topp}_{+})^{\lf}:\BC\to \bC\, .
\end{equation}
 Our original intention was to construct the coarse assembly map as a natural transformation 
 $$ Q(E(*)\wedge \Sigma^{\infty,\topp}_{+})^{\lf}\to E$$
 since this was suggested by the classical example of coarse $K$-homology which we discuss in  Subsection \ref{sec:kjnbsfd981}.
 The true nature of the coarse assembly map  
 as a natural transformation between coarse homology theories
 was only uncovered in the follow-up paper \cite{ass}. The  main new insight explained in that paper is that  \eqref{dcqwoijodqwcdadscdascadsc} is not 
 the correct domain of the coarse assembly map. 
 The correct domain is described in \cite{ass}.  
In {that} paper we will  study these domains called local homology theories and the  coarse assembly map in a motivic way.
 
   But fortunately, as also shown in \cite{ass},   the  correct domain of the coarse assembly map for $E$   and
  $Q(E(*)\wedge \Sigma^{\infty,\topp}_{+})^{\lf}$ are equivalent on nice spaces like finite-dimensional simplicial complexes with their natural path metric or complete Riemannian manifolds. 
  Using the framework developed in this book we  will show in    \cite{ass} a variety of isomorphism results for the   coarse assembly map in general, most of 
  which were previously only known for the $K$-theoretic version.

  After the first version of this book was written it turned out that the setup developed here can be applied to other interesting problems. 
  {As} in the case of the proof of  the $K$-theoretic coarse Baum--Connes conjecture   \cite{nw1} it turned out that the proof of
  the Novikov conjecture for algebraic $K$-theory 
    in \cite{Guentner:2010aa,Ramras:2011aa} for groups whose underlying metric space has   finite-decomposition complexity 
  is actually  not specific to algebraic $K$-theory. In the following we will give a very rough overview of this development. 
  For a group $G$ we  consider the orbit category $G\Orb$ of transitive $G$-sets and equivariant maps.  For a functor
  $M:G\Orb\to \bC$ and a family of subgroups $\cF$ we then consider the assembly map \begin{equation}\label{dasclkrnqfoefvdfvafvadsv}
\Ass_{M,\cF}:\colim_{G_{\cF}\Orb} M\to M(*)\, ,
\end{equation}
   where $G_{\cF}\Orb$ is the full subcategory of orbits with stabilizers in $\cF$, and $*$ is the one-point orbit.
  The basic question is now under which conditions on $G$, $\cF$ and $M$ this assembly map is an equivalence or at least split injective.
  
  If $M$ is the equivariant topological $K$-theory $K^{\topp,G}$,  and $\cF$ is the family of finite subgroups, then this assembly map is the homotopy theoretic version\footnote{{In  \cite{hamped}  it is stated  that  it coincides with the analytic definitions using $K\!K$-theory, but some details of the proof must still be worked out.}}
  of the Baum--Connes assembly map. 
  If $M$ is the algebraic $K$-theory $K^{\alg}\bA^{G}$ associated to an additive category  $\bA$ with $G$-action  and $\cF$ is the family of virtually cyclic subgroups, then this assembly map is the one appearing in the Farell--Jones conjecture.
  We refer to \cite{davis_lueck} for the construction of the  examples of the  functors $M$ mentioned above.
 The main new development based on this book is the observation, that the  proof  from \cite{Guentner:2010aa,Ramras:2011aa}
 for algebraic $K$-theory only depends on the fact that the functor $K^{\alg}\bA^{G}$ extends to an equivariant coarse homology theory
 with some very natural additional properties. The precise formulation of this condition is condensed in the notion of a CP-functor
 introduced in \cite{injectivity}.
 
The whole set-up introduce in the present book has a natural $G$-equivariant generalization which is developed in \cite{equicoarse}.
The proof of the split injectivity for CP-functors is based on the descent principle whose details are presented in  {\cite{injectivity}.} 
In this paper we also compare the assembly map \eqref{dasclkrnqfoefvdfvafvadsv} with an equivariant versions of the coarse assembly map. In order to apply the descent principle one  must know  that this equivariant generalization of the coarse assembly map 
is an equivalence. In  \cite{transb} we show that this is the case  even on  a motivic  level under the geometric assumption that the group
has finite decomposition complexity.

In \cite{equicoarse} we axiomatized equivariant coarse homology theories in complete analogy to the non-equivariant case considered in the present book. One of the additional structures on a coarse homology theory needed in order to verify the CP-condition are transfers.  These transfers are  discussed in \cite{coarsetrans}.

As explained above, in order to show injectivity of the assembly map for a given functor on the orbit category we must extend it to an equivariant coarse homology theory. This motivated the construction of further examples 
of equivariant coarse homology theories:
\begin{enumerate}
\item equivariant coarse ordinary homology \cite{equicoarse}.
\item equivariant coarse algebraic $K$-homology associated to an additive category with $G$-action  \cite{equicoarse}.
\item equivariant coarse Waldhausen $K$-homology of spaces \cite{Bunke:aa}.
\item equivariant coarse $K$-homology  associated to a left-exact $\infty$-category with $G$-action \cite{unik}.
\item topological equivariant coarse $K$-homology  associated to a  $C^{*}$-category with $G$-action \cite{Bunke:ad}.
\end{enumerate}
Eventually, for all of these cases we can show injectivity results for the corresponding assembly map.

The Farell--Jones conjecture asserts that the assembly map \eqref{dasclkrnqfoefvdfvafvadsv}
for the case $M= K^{\alg}\bA$ and the family  of virtually cyclic subgroups is an equivalence. Again it turned out that
the proofs (under additional geometric assumptions on the group) given e.g. in  \cite{blr,MR3598160}
are not specific to algebraic $K$-theory. In \cite{fj} we will show that the method extends to the $K$-theory functor
for left-exact $\infty$-categories with $G$-action. This case subsumes all previously known cases of the Farell--Jones conjecture
and provides a variety of new examples. The argument again depends on an extension of the functor on the orbit category to an equivariant coarse homology theory satisfying our axiomatics. This extension actually differs from the extension used for the injectivity results.

At the end of this overview we mention some further developments based on the foundations provided by this book. 
In   \cite{coho} we introduce axioms for coarse cohomology theories by dualizing the axioms for coarse homology theories. 
Examples of coarse cohomology theories can be constructed by dualizing coarse homology theories. 
As an application we settle a conjecture of Klein \cite[Conj.~on Page~455]{klein} that the dualizing spectrum of  a group is a quasi-isometry invariant.  We further discuss the pairing between coarse $K$-homology and coarse $K$-cohomology.

In \cite{caputi_coarseHH} equivariant coarse cyclic and Hochschild homology was defined.  These homology theories are related by a trace map with equivariant coarse algebraic $K$-homology. 
In \cite{Bunke:2017aa} a universal equivariant coarse algebraic $K$-theory functor  with values in the motives  of \cite{MR3070515}
was constructed.   

Coarse fixed point spaces and localization results for equivariant coarse homology theories were studied in 
\cite{Bunke:2019aa}. In  \cite{Bunke:2019ab} we showed that coarse algebraic $K$-theory  of additive categories is actually a lax symmetric monoidal functor. A similar fact for left-exact $\infty$-categories is a crucial ingredient of \cite{fj}.

Aspects of index theory using (equivariant) coarse homology theories were considered in 
\cite{indexclass,cw}.

The present book lays the foundations for all these development. 
It provides a reference for the technical material in Sections \ref{efwljkfowpfjowefwefewfwef}, \ref{feiojoiwfiowfu234234324}, \ref{ihjiovqceqwecewcqwecq} and \ref{sec:sdn934} which will be used in the subsequent papers. These papers   employ with references  the basic notions and constructions  provided in Sections  \ref{fwwefeopwfewioufo234234}, \ref{lfjwefiowefewi9f0234} and \ref{fwelfjweoflewfewf}.

This book   develops the basics of coarse geometry in the framework of the category $\BC$ from scratch.   This part  should be accessible to students on the Bachelor level and could serve as a condensed  introduction to the  basic notions and basic examples of the field.

As mentioned above the discussion of coarse homology theories and the motivic approach requires some basic experience with the language of $\infty$-categories, say on the level of  one-semester user-oriented introduction.   
The combination of Sections \ref{fwwefeopwfewioufo234234}, \ref{lfjwefiowefewi9f0234} and \ref{fwelfjweoflewfewf}
 is a comprehensive introduction to coarse homotopy theory using the modern language of $\infty$-categories.
This part could also be seen as a demonstration of the motivic approach which is completely analogous to
the one in algebraic geometry or differential cohomology theory. 
 The discussion of  locally finite homology theories and their coarsification using  the language of $\infty$-categories is new. Adopting the level reached in the preceeding sections  does not require any additional $\infty$-categorial prerequisites.

In order to understand the discussion of coarse
{$K$-homology} one needs background knowledge about $C^{*}$-algebras and their $K$-theory.  But we  will explain the necessary background about $C^{*}$-categories in detail.

The book served as  the basis for PhD-level courses on coarse homotopy theory, coarse  $K$-homology and assembly maps.

\paragraph{Acknowledgements}
\textit{The authors were supported by the SFB 1085 ``Higher Invariants'' funded by the Deutsche Forschungsgemeinschaft DFG. The second named author was also supported by the Research Fellowship EN 1163/1-1 ``Mapping Analysis to Homology'' of the Deutsche Forschungsgemeinschaft DFG.}

\textit{Parts of the material in the present paper were presented in the course ``Coarse Geometry'' held in the summer term 2016 at the University of Regensburg. We thank the participants for their valuable input.}

\textit{We thank Clara Löh  for her helpful comments  on a first draft of this paper. Furthermore we thank Daniel Kasprowski and Christoph Winges for suggesting  various improvements, and Rufus Willett for helpful discussions.}

\part{ {Motivic coarse spaces and spectra}}

\section{Bornological coarse spaces}\label{fwwefeopwfewioufo234234}

In this section we introduce from scratch the category $\BC$ of bornological {coarse} spaces.
We further provide basic examples for objects and morphisms {of} $\BC$ {and} finally we discuss some categorical properties of $\BC$.

\begin{rem}
 {In} this remark we fix 
 the set-theoretic size issues. We choose a sequence of three Grothendieck universes
whose elements  {are} called very small, small and large sets. 

The objects of the categories of geometric objects  like $\Set$, $\Top$, $\BC$ considered below belong to the very small universe, while these categories themselves are   small.

If not said differently, categories or $\infty$-categories will assumed to be small, and cocompleteness or completeness refers to the existence of all colimits or limits indexed by very small categories.

By this convention $\Spc$ and $\Sp$ are the small $\infty$-categories of very small spaces and spectra. 

But we will also consider the large $\infty$-categries $\Spc^{\la}$ and $\Sp^{\la}$ of small spaces or small spectra.\index{$\Sp^{\la}$} Furthermore, the categories of motivic coarse spaces $\Spc\cX$ and $\Sp\cX$ constructed below are large. 
\hB
\end{rem}

\subsection{Basic definitions}
 
In this section sets are very small sets.

Let $X$ be a set, and let $\cP(X)$ denote the set of all subsets of $X$.
\begin{ddd}
\index{bornology}
A bornology on $X$ is a subset $\cB$ of $ \cP(X)$\index{$\cB$|see{bornology}}  which is closed under taking finite unions {and} subsets{,} and such that $\bigcup_{B\in \cB}B=X$.
\end{ddd}
The elements of $\cB$ are called the bounded subsets\index{subset!bounded}\index{bounded} of $X$.

The main role of bounded subsets in the present paper  is to declare certain regions of $X$ to be small in order to define  further  finiteness conditions. 
Let $D$ be a subset of a space $X$ with a bornology $\cB$.
\begin{ddd}
\label{defn:asdfwewt}
\index{locally!finite!subset}\index{locally!countable!subset}$D$ is locally finite (resp., locally countable) if for every $B$ in $\cB$  the intersection $B\cap D$ is finite (resp., countable).
 \end{ddd}

For a subset  $U $ of $X\times X$ we define the inverse by $$U^{-1}:=\{(x,y)\in X\times X : (y,x)\in U\}\, .$$ If $U^{\prime}$ is a second subset of $X\times X$, then we  define their composition by
\[U\circ U^{\prime}:=\{(x,y)\in X\times X :  (\exists z\in X :  (x,z)\in U \text{ and } (z,y)\in U^{\prime})\}.\]

Let $X$ be a set.
\begin{ddd}\label{ewgiubhwiergfvsvdfvsfdvs}
\index{coarse structure}
A coarse structure on $X$ is a subset $\cC$ {of} $\cP(X\times X)$\index{$\cC$|see{coarse structure}} which is closed under taking finite unions, composition,  inverses and subsets, and which contains the diagonal   of~$X$.
\end{ddd}

\begin{rem}
{Coarse structures are not always required to contain the diagonal, especially in early publications on this topic. In these days coarse structures containing the diagonal were called unital. See Higson--Pedersen--Roe \cite[End of Sec.~2]{higson_pedersen_roe} for an example of a naturally occuring non-unital coarse structure. We are using here the nowadays common convention to requiere the coarse structures to contain the diagonal and correspondingly we drop the adjective `unital'.}
\hB
\end{rem}

The elements of $\cC$ will be called entourages\index{entourage} of $X$ or controlled subsets.\index{subset!controlled}\index{controlled!subset} 

For a subset $U$ of $X\times X$ and a subset $B$ of $X$ we define the $U$-thickening\index{$U$-!thickening}\index{thickening!$U$-}\index{$U[-]$|see{$U$-thickening}} of $B$ by
$$U[B]:=\{x\in X :  (\exists b\in B :  (x,b)\in U)\}\, .$$

If $X$ is equipped with a coarse structure $\cC$, then we call the subsets $U[B]$ for  $U$ in $\cC$ the controlled thickenings of $B$\index{controlled!thickening}\index{thickening!controlled}.
\begin{ddd}
A bornology and a coarse structure are said to be compatible\index{compatible!bornology and coarse structure} if every controlled thickening of a bounded subset is again bounded.
\end{ddd}

\begin{ddd}
\label{jewflewjjwoiuoifewfewe}
A bornological coarse space\index{bornological coarse space}\index{space!bornological coarse} is a triple $(X,\cC,\cB)$\index{$(X,\cC,\cB)$}  of a set $X$  with a coarse structure $\cC$  and a bornology $\cB$ such that the bornology and coarse structure are compatible.
\end{ddd}

We now consider a  map  $f:X\to X^{\prime}$ between sets equipped with bornological structures $\cB$ and $\cB^{\prime}${, respectively}.

\begin{ddd}\label{iefewiofeuwio2342342434} \mbox{}
\begin{enumerate}
\item We say that $f$ is proper\index{proper map}\index{map!proper} if for every $B^{\prime}$ in $\cB^{\prime}$  we have  $f^{-1}(B^{\prime})\in \cB$.
\item The map $f$ is bornological\index{bornological map}\index{map!bornological} if for every $B$ in $\cB$ we have $f(B)\in \cB^{\prime}$.
\end{enumerate}
\end{ddd}

Assume   that $f:X\to X^{\prime}$ is a map between  sets equipped with
  coarse structures $\cC$ and $ \cC^{\prime}$, respectively.
\begin{ddd}\label{efijwefoie98u32234234}
We say that $f$ is controlled\index{controlled!map}\index{map!controlled} if for every   $U$ in $\cC$   we have $(f\times f)(U)\in \cC^{\prime} $.
\end{ddd}

We now consider two bornological coarse spaces $X$ and $X^{\prime}$.
\begin{ddd}
A morphism\index{morphism!between bornological coarse spaces} between bornological coarse spaces $f:X\to X^{\prime}$  is a map  between the underlying sets  $f :  X\to X^{\prime}$  
which is controlled and proper.
\end{ddd}

\begin{ddd}\label{etgkowergferfrwefrefw}
\index{bornological coarse space}
We let $\BC$\index{$\BC$|see{bornological coarse space}} denote the small category of very small bornological coarse spaces and  morphisms.
\end{ddd}

\subsection{Examples}

\begin{ex}\label{sdf89232111332}
Let $X$ be a set and consider  a subset $A $ of  $\cP(X\times X)$. Then we can consider the minimal coarse structure
$\cC\langle A\rangle$ containing $A$. We say that $\cC\langle A\rangle$ is generated by $A$.\index{coarse structure!generated by $-$}

Similarly, let $W$ be a subset of   $\cP(X)$. Then we can consider  the minimal bornology $\cB\langle W\rangle$
containing $W$. Again we say that $ \cB\langle W\rangle$ is generated by $W$. \index{bornology!generated by $-$}

If for every generating entourage $U$ in $A$ and  every generating bounded subset $B$ in $W $ the subsets $U[B]$  {and $U^{-1}[B]$ are} contained in finite unions of members of $W$, 
then the bornology $ \cB\langle W\rangle$ and the coarse structure $\cC\langle A\rangle$ are compatible.\index{compatible!bornology and coarse structure}
\hB
\end{ex}

\begin{ex}\label{welkfjwekfo23u4234234}
 Every set $X$ has a minimal coarse structure $\cC_{min}:=\cC\langle{\emptyset}\rangle$ generated by the empty set.\index{coarse structure!minimal}\index{minimal!coarse structure}
 This coarse structure consists of all subsets of the diagonal $\diag_{X}\subseteq X\times X$ and is compatible with  every bornological structure on $X$.
 In particular, it is compatible with the minimal bornological structure $\cB_{min}$ which consists of the finite subsets of $X$.\index{bornology!minimal}\index{minimal!bornology}

 The maximal coarse structure on $X$ is $\cC_{max}:=\cP(X\times X)$.
 The only compatible bornological structure is then the maximal one $\cB_{max}:=\cP(X)$.\index{bornology!maximal}\index{maximal!bornology}\index{coarse structure!maximal}\index{maximal!coarse structure}
   \hB
   \end{ex}

\begin{ddd}
\label{ijwieorfjuwefuwe9few435345}
We call a bornological coarse space $(X,\cC,\cB)$ discrete if 
$\cC=\cC_{min} $.\index{discrete!bornological coarse space}\index{bornological coarse space!discrete}
\end{ddd}

Note that a discrete bornological coarse space has a maximal entourage $\diag_{X}$.

We consider a coarse space $(X,\cC)$, an entourage $U$ in $\cC$, and a subset $B$ of $X$.
\begin{ddd}\label{jfewiofjweofwef234}
 The subset  $B$ called  is $U$-bounded if  $B\times B\subseteq U$.\index{$U$-!bounded}\index{bounded!$U$-}\index{subset!$U$-bounded}
\end{ddd}

If $(X,\cC,\cB)$ is a bornological coarse space, then the bornology $\cB$ necessarily contains all  subsets which are $U$-bounded for some $U$ in $\cC$.
 
\begin{ex}\label{wfijiofweofewfwefewfewf}
Let $(X,\cC)$ be a coarse space. Then there exists a minimal compatible bornology $\cB$\index{minimal!compatible bornology}\index{compatible!minimal bornology}. 
It is the bornology  consisting of  {finite unions of} 
the 
 subsets of $X$ which are bounded for some entourage of $X$. We also say that $\cB$ is generated by $\cC$.
\hB
\end{ex}

\begin{ex}\label{wfijweifjewfoiuewoifoewfu98fewfwfw}
Let $(X^{\prime},\cC^{\prime},\cB^{\prime})$ be a bornological coarse space, and let $f:X\to X^{\prime}$ be a map of sets. We define the induced bornological coarse structure by\index{induced!bornological coarse structure}
$$f^{*}\cC:=\cC\langle \{(f\times f)^{-1}(U^{\prime}) :  U^{\prime}\in \cC^{\prime}\}\rangle$$ and
$$f^{*}\cB:=\cB\langle \{f^{-1}(B^{\prime})   :   B^{\prime}\in \cB^{\prime}   \}\rangle\, .$$
Then
$$f :  (X,f^{*}\cC^{\prime},f^{*}\cB^{\prime})\to (X^{\prime},\cC^{\prime},\cB^{\prime})\, .$$
is a morphism of bornological coarse spaces.
In particular, if $f$ is the inclusion of a subset, then we say that this morphism is the inclusion of a bornological coarse subspace.
\hB
\end{ex}
 
\begin{ex}\label{welifjwelife89u32or2}
Let $(X,d)$ be a metric space\index{metric space}\index{space!metric}. In the present book we will allow infinite distances. Such metrics are sometimes called quasi-metrics. For every positive real number $r$ we define the entourage\index{induced!coarse structure from a metric}\index{coarse structure!from a metric}
\begin{equation}\label{ffwefewfe24543234wfefewef}
U_{r}:=\{(x,y)\in X\times X :  d(x,y)< r\}.
\end{equation} The collection of these entourages generates
  the coarse structure
\begin{equation}\label{hvhdgqwuztd8t1z87231231233}
 \cC_{d}:= \cC\langle \{U_{r} :  r\in (0,\infty)\}\rangle 
\end{equation}
 which will be called be the coarse structure induced by the metric.

We further define the bornology of metrically bounded sets\index{bornology!of metrically bounded sets}
$$\cB_{d}:=\cB\langle\{B_{d}(r,x) :  r\in (0,\infty)\:\text{and}\:x\in X\} \rangle\, ,$$
where $B_{d}(r,x)$ denotes the metric ball of radius $r$ centered at $x$.  

We say that   $ (X,\cC_{d},\cB_{d})$ is the underlying bornological coarse space of the metric space $(X,d)$.\index{bornological coarse space!underlying a metric space}
We will denote this bornological coarse space often by $X_{d}$.
\hB
\end{ex}
 
Recall that a path metric space\index{path metric space}\index{metric space!path metric space}\index{space!path metric space} is a metric space where the distance between any two points is given by the infimum of the lengths of all curves connecting these two points. 
 
Let $(X,d)$ be a path metric space, and let $\cC_{d}$ be the associated coarse structure  \eqref{hvhdgqwuztd8t1z87231231233}.
\begin{lem}\label{lem:sdf7834}
There exists an entourage $U$  in $\cC_{d}$ such that   $\cC_{d}=\cC\langle U\rangle$.
\end{lem}
\begin{proof}
Recall  that the coarse structure $\cC_{d}$ is generated by the  set of  entourages $U_{r}${, see} \eqref{ffwefewfe24543234wfefewef}, for {all} $r$ in $(0,\infty)$. For a path metric space one can check that  
$\cC_{d}=\cC\langle \{U_{r}\}\rangle$ for any $r$  in $(0,\infty)$.
\end{proof}

\begin{ex} Let $X$ be a set,  and let $U$ be an entourage on $X$. Then the coarse structure $\cC\langle U\rangle$ 
  is 
  induced by a metric as in Example~\ref{welifjwelife89u32or2}. In fact, we have  $\cC\langle U\rangle=\cC_{d}$ for the quasi-metric
$$d(x,y):=\inf \{n\in \nat\:|\: (x,y)\in V^{n}\} \, ,$$
where $V:=(U\circ U^{-1})\cup \diag_{X}$, and where we set $V^{0}:=\diag_{X}$.

In general one can check that a coarse structure can be presented by a metric if and only if it is countably generated.\hB
 \end{ex}

 \begin{ex}\label{ewfijwfioiouiou23roi2r32rrew}
 Let $\Gamma$ be a group. We equip $\Gamma$ with the bornology $\cB_{min}$ consisting of all finite subsets.
 For every finite subset $B$ we can consider the $\Gamma$-invariant entourage $\Gamma(B\times B)$. The family of these entourages generates the canonical coarse structure \index{canonical!coarse structure} \index{entourage!$\Gamma$-invariant}
  $$\cC_{can}:=\cC\langle \{\Gamma(B\times B) :  B\in \cB_{min}\}\rangle\, .$$
 Then $(\Gamma,\cC_{can},\cB_{min})$ is a bornological coarse space. The group $\Gamma$ acts via left multiplication on this bornological coarse space by automorphisms. This example is actually an example of a $\Gamma$-bornological coarse space which we will define now.
   \hB
 \end{ex}
 
 \begin{ddd}
 A $\Gamma$-bornological coarse space\index{bornological coarse space!$\Gamma$-equivariant} is a bornological coarse space $(X,\cC,\cB)$ with an action of $\Gamma$  by automorphisms such that the set $\cC^{\Gamma}$ of $\Gamma$-invariant entourages  is cofinal in $\cC$.
\end{ddd}

The foundations of the theory  of $\Gamma$-bornological coarse spaces will be {thoroughly} developed in \cite{equicoarse}.

\subsection{Categorical properties of \texorpdfstring{$\BC$}{BornCoarse}}

The  families of objects considered in this subsection are always indexed by very small sets.

\begin{lem}\label{wleifjewfiewiofuew9fuewofewf}
The category $\BC$ has {all} non-empty products.\index{product!in $\BC$}
\end{lem}

\begin{proof}
Let $(X_i,\cC_i,\cB_i)_{i\in I}$  be a family of bornological coarse spaces. On the set  $X:=\prod_{i \in I} X_i$
we define the coarse structure $\cC
$ generated by the entourages
$\prod_{i\in I} U_{i}$ for all elements $(U_{i})_{i\in I}$ in $\prod_{i\in I}\cC_{i}$, and the bornology $\cB$ generated by the subsets  $B_j \times \prod_{i \in I\setminus\{j\}} \! \! X_i$ for all $j$ in $I$ and $B_{j}$ in $\cB_{j}$.
%
The projections  from $(X,\cC,\cB)$ to the factors are morphisms. This bornological coarse space  together with the projections represents the product of the family $(X_i,\cC_i,\cB_i)_{i\in I}$.   
\end{proof}
 
 Note that $\BC$ does not have a final object and therefore does not admit the empty product.
 
 \begin{ex}\label{hffweiufweif}
We will often consider the following variant $X\ltimes X^{\prime}$\index{$- \ltimes -$}\index{product!$- \ltimes -$} of a product of two bornological coarse spaces $(X,\cC,\cB)$ and $(X^{\prime},\cC^{\prime},\cB^{\prime})$.
 In detail, the underlying set of  this bornological coarse space is $X\times X^{\prime}$, its coarse structure is $\cC\times \cC^{\prime}$, but its bornology differs from the cartesian product and is generated by the subsets $X\times B^{\prime}$ for all bounded
subsets $B^{\prime}$ of $X^{\prime}$. The projection  to the second factor $X\ltimes X^{\prime}\to X^{\prime}$ is a morphism, but the projection to the first factor is in general not. On the other hand, 
for a point $x\in X$ the map $X^{\prime}\to X\ltimes X^{\prime}$ given by $x^{\prime}\mapsto (x,x^{\prime})$ is a morphism. In general, the cartesian product does not have this property.
\hB
\end{ex}

\pagebreak[2]
\begin{lem}\label{ewoifweofe90wf09wef2334}
The category $\BC$ has all coproducts.\index{coproducts!in $\BC$}
\end{lem}

\begin{proof}
Let $(X_{i},\cC_{i},\cB_{i})_{i\in I}$ be a family of bornological coarse spaces.
We define  
  the bornological coarse space
$(X,\cC,\cB)$ by $$X:= \bigsqcup_{i\in I}  X_{i} 
\, , \quad \cC:=\cC \big\langle \bigcup_{i\in I} \cC_{i} \big\rangle\, , \quad \cB:=\{B\subseteq X\:|\: (\forall i\in I: B\cap X_{i}\in \cB_{i})\}  \, .$$
Here we secretly identify subsets of $X_{i}$ or of $X_{i}\times X_{i}$ with the corresponding subset of $X$ or of $X\times X$, respectively.
The obvious embeddings $X_{i}\to X$ are morphisms. The bornological coarse space $(X,\cC,\cB)$ together with these embeddings represents the   coproduct ${\coprod_{i\in I}} (X_{i},\cC_{i},\cB_{i})$ of the family in $\BC$. 
\end{proof}

\begin{rem}
It turns out that the category $\BC$ has all non-empty very small limits and many more (but not all) very small colimits. We refer to  \cite[Sec. 2.2]{equicoarse} for more details. Even better, one can  modify the definition of a bornology slightly in order to obtain a complete and cocomplete category $\widetilde{\BC}$ of generalized bornological coarse spaces without changing the homotopy theory developed below \cite{Heiss:2019aa}.
\end{rem}

Let $(X_{i},\cC_{i},\cB_{i})_{i\in I}$ be a family of bornological coarse spaces.
\begin{ddd}\label{foifewieof89u8924r443535}
The free union\index{free union!in $\BC$}\index{union!free!in $\BC$} $\bigsqcup_{i\in I}^{\free} (X_{i},\cC_{i},\cB_{i})$ of the family is the following bornological coarse space:
\begin{enumerate}
\item The underlying set of  the free union is the disjoint union $\bigsqcup_{i\in I} X_{i}$.
\item The coarse structure of the free union is generated by the entourages
$\bigcup_{i\in I} U_{i}$ for all families $(U_{i})_{i\in I}$ with $U_{i}$ in $\cC_{i}$.
\item The bornology of the free union is given by $\cB\big\langle \bigcup_{i\in I} \cB_{i}\big\rangle$.
\end{enumerate}
\end{ddd}

\begin{rem}
The free union should not be confused with the coproduct. The coarse structure of the free union is bigger, while the bornology is smaller. The free union plays a role in the discussion of additivity of coarse homology theories.
\hB
\end{rem}

\begin{ddd}\label{eiofewiofo9u23948234325}
We define the mixed union\index{mixed union}\index{union!mixed} $\bigsqcup_{i\in I}^{\mixed} (X_{i},\cC_{i},\cB_{i})$ such that the underlying coarse space is that of the coproduct and the bornology is the one from the free union.
\end{ddd}

Note that the identity morphism on the underlying set gives morphisms in $\BC$
\begin{equation}
\label{jknsd8923hd}
\coprod_{i \in I} X_i \to \bigsqcup_{i\in I}^{\mixed} X_{i} \to \bigsqcup_{i\in I}^{\free} X_{i}\, .
\end{equation}

Let $(X,\cC)$ be  a set with a coarse structure.
Then $$\cR_{\cC}:=\bigcup_{U\in \cC} U\subseteq X\times X$$ is an equivalence relation on $X$.

\pagebreak[2]
\begin{ddd}\label{ekhqiuheiuqe123122342334}\mbox{}
\begin{enumerate}
\item We call the equivalences classes of $X$ with respect to {the equivalence relation} $\cR_\cC$ the coarse components of $X$.\index{coarse!components}
\item We will use the notation $\pi_{0}^{\coarse}(X,\cC)$ for the set of  coarse   components of $(X,\cC)$.
\item 
We call $(X,\cC)$ coarsely connected\index{coarsely!connected} if  $\pi_{0}^{\coarse}(X,\cC)$ has a single element.
  \end{enumerate}
\end{ddd}
 We can apply the construction to bornological coarse spaces. If $ (X_{i})_{i\in I}$ is a family of bornological coarse spaces, then we have a bijection
$$\pi_{0}^{\coarse} \big( \coprod_{i\in I} X_{i} \big)\cong \bigsqcup_{i\in I}\pi_{0}^{\coarse}( X_{i}) \, .$$

If a  bornological coarse space has finitely many coarse components then  it is the coproduct of its coarse components with the induced structures. But  in the case of infinitely many components this is not true    in general.

\begin{ex}\label{wfeoihewiufh9ewu98u2398234324234}
Let $(X,\cC,\cB)$   be a bornological coarse space. For an entourage $U$ of $X$ we consider the bornological coarse space\index{$-_U$}
$$X_{U}:=(X,\cC\langle \{U\}\rangle,\cB)\, ,$$ i.e., we keep the bornological
 structure but we consider the coarse structure generated by a single entourage.
 In $\BC$ we have an isomorphism
$$X\cong \colim_{U\in \cC} X_{U}$$
induced by the natural  morphisms $X_{U}\to X$.
This observation often allows us to reduce arguments to the case of bornological coarse spaces whose coarse structure is generated by a single entourage.
\hB
\end{ex}

\begin{ex}\label{eiofweoifwefuewfieuwf9wwfwef}
Let $(X,\cC,\cB)$ and $(X^{\prime},\cC^{\prime},\cB^{\prime})$ be bornological coarse spaces. We will often consider
the bornological coarse space\index{$-  \otimes  -$!bornological coarse space}\index{product!$-  \otimes  -$}
$$(X,\cC,\cB) \otimes (X^{\prime},\cC^{\prime},\cB^{\prime}):=(X\times X^{\prime},\cC\times \cC^{\prime},\cB\times \cB^{\prime})\, .$$
We use the notation $\otimes$ in order to distinguish this product from the cartesian product. In fact, the 
  bornology 
$$ \cB\times \cB^{\prime}=\cB\langle\{B\times B^{\prime} :  B\in \cB \text{ and } B^{\prime}\in \cB^{\prime} \}\rangle$$
is the one of the cartesian product in the category  of bornological spaces and bornological\footnote{instead of proper} maps.

 The product $$-\otimes-:\BC\times \BC\to \BC$$ is a symmetric monoidal structure on $\BC$ with tensor unit $*$. Note that it is not the cartesian product in $\BC$.  The latter has no tensor unit since $\BC$ does not have a final object. 
 \hB
\end{ex}

\section{Motivic coarse spaces}\label{lfjwefiowefewi9f0234}

In this section we introduce the   $\infty$-category of motivic coarse spaces $\Spc\cX$. Our goal is to do homotopy theory with bornological coarse spaces. To this end we first complete the category of bornological coarse spaces formally and then implement the desired geometric properties through localizations. 

In the following $\Spc^{\la}$ is the large presentable $\infty$-category of small spaces\index{$\Spc^{\la}$}. This category 
can be characterized as the universal large presentable $\infty$-category generated by a final  object.
We start with the large presentable $\infty$-category of space-valued presheaves
 $$\PSh (\BC):=\Fun(\BC^{\op},\Spc^{\la})\, .$$
 We  construct the $\infty$-category of motivic coarse spaces $\Spc\cX$  by a sequence of 
   localization{s} (descent, coarse equivalences, vanishing on flasque bornological coarse spaces, $u$-continuity) of $\PSh (\BC)$ at various sets of morphisms. Our goal is to encode well-known invariance and descent properties of coarse homology theories in a motivic way.
  The result of this construction is the functor  $$\Yo\colon\BC\to \Spc\cX$$ (Definition \ref{qergkjefkoerfrefwefwerfev})
  which sends a bornological coarse space to its motivic coarse space. 
 
   Technically one could work with model categories and
   Bousfield localizations or, as we prefer, with $\infty$-categories (Remark \ref{fwioejewif98u329r32r32r}). A reference for the general theory of localizations  is Lurie \cite[Sec.~5.5.4]{htt}.
   

   Our approach is parallel to constructions in $\mathbb{A}^{1}$-homotopy theory \cite{Hoyois:2015aa} or differential cohomology theory \cite{Bunke:2013aa}, \cite{MR3462099}.  
   
   In the final Subsection \ref{efijoefvvfvsdfvdfsdfvsfvsfdv} we show some further properties of $\Yo$ which could be considered as first steps into the field of coarse homotopy theory.

\subsection{Descent}

We  introduce a Grothendieck topology $\tau_{\chi}$ on $\BC$ and the associated $\infty$-category of sheafs. This Grothendieck topology will encode coarse excision.

\begin{rem}
The naive idea to define a Grothendieck topology on $\BC$ would be to use covering families given by coarsely excisive pairs \cite{higson_roe_yu}, see Definition \ref{efjwelkfwefoiu2or4234234}. The problem with this approach is that in general a pullback does not preserve coarsely excisive pairs, see Example \ref{ewf3424342323}. Our way out will be to replace excisive pairs by complementary pairs of a subset and a big family, see Example \ref{sdjfh32r23r}.
\hB
\end{rem}

Let $ X $ be a bornological coarse space.
\begin{ddd} \label{ifjweijwoiefewfewf}
A big family\index{big family!in $\BC$}   $\cY$  on $X$ is a filtered family of subsets $(Y_{i})_{i\in I}$ of $X$ such that for every
  $i$ in $I$ and entourage $U$ of $X$   there exists $j$ in $I$ such that $U[Y_{i}]\subseteq Y_{j}$.
\end{ddd}

In this definition the set $I$ is a filtered, partially ordered set and the map $I\to \cP(X)$, $i\mapsto Y_{i}$ is order-preserving where $\cP(X)$ is equipped with the partial order given by the subset relation.

\begin{ex}\label{fwijweiofjweoifewfu9fuwe98ff}
Let $\cC$ denote the coarse structure of $X$. If $A$ is a subset of $X$, then we can consider the big family in $X$\index{$\{-\}$|see{big family generated by $-$}}\index{big family!generated by $-$}
\begin{equation}\label{erkjnkjnvkfvfvfvdscac}
\{A\}:=(U[A])_{U\in \cC}
\end{equation} generated by $A$. Here we use that the set of entourages $\cC$ is  a filtered partially ordered set with respect to the inclusion relation. 
\hB
\end{ex}

\begin{ex}\label{4oj3iotiu43otu43t4390t3t430t}
If $(X,\cC,\cB)$ is a bornological coarse space, its family $\cB$ of bounded subsets is big. Indeed,
the condition stated in Definition \ref{ifjweijwoiefewfewf} is exactly the compatibility condition between the bornology and the coarse structure required in Definition \ref{jewflewjjwoiuoifewfewe}. In this example we consider $\cB$ as a filtered partially ordered set with the inclusion relation.
\hB
\end{ex}

Let $X$ be a bornological coarse space.
\begin{ddd}\label{jfwoifjweofeiwf234234324}
A complementary pair\index{complementary pair} $(Z,\cY)$ is a pair consisting of a subset $Z$ of $X$ and a big family $\cY=(Y_{i})_{i\in I}$
such that there exists an index $i$ in $I$ with $Z\cup Y_{i}=X$.
\end{ddd}

\begin{rem}
Implicitly the definition implies that if a big family $(Y_{i})_{i\in I}$ is a part of a complementary pair, then the index set $I$ is not empty. But if the other component $Z$ is the whole space $X$, then  it could be the family $(\emptyset)$.
\hB
\end{rem}

\begin{ex}\label{wokfewofwfkopewfefewf}
Let $(Z,\cY)$ be a complementary pair on a bornological coarse space $X$. Then $Z$ has an induced bornological coarse structure. For its construction we apply the construction given in Example \ref{wfijweifjewfoiuewoifoewfu98fewfwfw} to the embedding $Z\to X$. We define the family $$Z\cap \cY:=(Z\cap Y_{i})_{i\in I}$$ of subsets of $Z$.
This is then  a big family on $Z$.
\hB
\end{ex}

\begin{ex}\label{sdjfh32r23r}
Let $(Z,\cY)$ be a complementary pair on a bornological coarse space $X$ and $f: X^{\prime} \to X$ be a morphism. Then $f^{-1} \cY := (f^{-1}(Y_{i}))_{i \in I}$ is a big family on $X^{\prime}$ since $f$ is controlled. Furthermore, $(f^{-1}(Z), f^{-1} \cY)$ is a complementary pair on $X^{\prime}$.
\hB
\end{ex}


\begin{rem}\label{fwioejewif98u329r32r32r}
In the present paper we work with $\infty$-categories (or more precisely {with} $(\infty,1)$-categories) in a model-independent way. For concreteness one could   think of quasi-categories as introduced by Joyal \cite{Joyal} and worked out in detail by Lurie  \cite{htt,HA}.  

Implicitly, we will consider an ordinary category as an $\infty$-category using nerves.
  In order to construct a model for $\Spc^{\la}$\index{large $\infty$-category of spaces}\index{$\infty$-category of spaces!large}  one could start with the large ordinary category $\sSet^{\la}$ of small simplicial sets. We let $W$ denote the set of weak homotopy equivalences in $\sSet^{\la}$. Then  
\begin{equation}\label{fwefwe234234234}
\Spc^{\la}\simeq  \sSet^{\la}[W^{-1}]\, .
\end{equation}
 The $\infty$-category of spaces is also characterized by a universal property: 
it is the universal large presentable $\infty$-category generated by a final object (usually denoted by $*$).

If $\bC$ is some small $\infty$-category, then we can consider the large category of space-valued presheaves\index{presheaves!space valued}\index{$\PSh(-)$}
$$\PSh(\bC):=\Fun(\bC^{\op},\Spc^{\la})\, .$$
 
We have  the Yoneda embedding\index{Yoneda embedding}\index{$\yo$}   $$\yo:\bC\to \PSh(\bC)\, .$$ It presents   
the category of presheaves on $\bC$  as the free  cocompletion of $\bC$.
We will often use the identification (following from Lurie \cite[Thm.~5.1.5.6]{htt})\footnote{Lurie's formula is the marked equivalence in the chain \[\mathclap{\PSh(\bC)= \Fun(\bC^{\op},\Spc^{\la})\simeq \Fun(\bC,\Spc^{la,\op})^{\op}\stackrel{!}{\simeq} \Fun^{\colim}(\PSh(\bC),\Spc^{la,\op})^{\op}\simeq \Fun^{\lim}(\PSh(\bC)^{\op},\Spc^{\la}).}\]}
$$\PSh(\bC)\simeq \Fun^{\lim}(\PSh (\bC)^{\op},\Spc^{\la} )\, ,$$
where the superscript $\lim$ stands for small limit-preserving functors.
In particular, if $E$ and $F$ are presheaves, then the evaluation $E(F)$ in $\Spc^{\la}$ is defined  and satisfies
\begin{equation}\label{ejfiowejfoifjoifewfefef}E(F)\simeq \lim_{(\yo(X)\to F)\in \bC/F} E(X)\, ,\end{equation}
where $\bC/F$ denotes the slice category.
 Note that for $X$ in  $\bC$ we have the equivalence $E(\yo(X))\simeq E(X)$.
\hB
\end{rem}

 For a big family $\cY=(Y_{i})_{i\in I}$ we define   $$\yo(\cY):=\colim_{i\in I} \yo(Y_{i})\in  \PSh(\BC) \, .$$
  Note that we take the colimit after applying the Yoneda embedding. If we would apply these operations in the opposite order, then in general we would get a different result.
%
In order to simplify the notation, for a presheaf $E$  we will often abbreviate $E(\yo(\cY))$ by
 $E(\cY)$. By definition we have the equivalence
 $E(\cY)  \simeq \lim_{i\in I} E(Y_{i})$.
 
\bigskip 

Let  $E$ be in  $ \PSh(\BC)$.
  \begin{ddd}\label{ewlfjwefewoi2r543555}
We say that
$E $ satisfies descent\index{descent} for
complementary pairs if $E(\emptyset)\simeq *$ and if for every complementary pair\index{complementary pair} $(Z,\cY)$ on a bornological coarse space $X$ the square
 \begin{equation}\label{wefkewjfio32ru23r23r23r32r}
\xymatrix{E(X)\ar[r]\ar[d]&E(Z)\ar[d]\\E(\cY)\ar[r]&E(Z\cap \cY)} 
\end{equation}
is cartesian in $\Spc$.
\end{ddd}

\begin{lem}
There is a Grothendieck topology\index{Grothendieck topology} $\tau_{\chi}$ on $\BC$ such that
the $\tau_{\chi}$-sheaves are exactly the presheaves which satisfy descent  for
complementary pairs.
\end{lem}

\begin{proof}
We consider the largest Grothendieck topology  $\tau_{\chi}$ on $\BC$ for which all the presheaves satisfying descent  for
complementary pairs  are sheaves.  The main non-trivial point is now to observe that the condition of descent
for a  complementary pair $(Z,\cY)$ on a bornological coarse space $X$ is equivalent to  the condition of descent for the sieve generated by the covering family of $X$ consisting of $Z$ and the subsets $Y_{i}$ for all $i$ in $I$. By the definition of the Grothendieck topology $\tau_{\chi}$ this sieve belongs to $\tau_{\chi}$. Therefore a $\tau_{\chi}$-sheaf satisfies descent for   complementary pairs. For more details we refer to \cite{Heiss:2019aa}.
\end{proof}

We let $$\Sh (\BC)\subseteq \PSh(\BC)$$ denote the full subcategory of $\tau_{\chi}$-sheaves. We have the usual sheafification\index{sheafification} adjunction   $$L :\PSh (\BC)\leftrightarrows \Sh (\BC):\incl\, .$$

\begin{lem}\label{wfeoijweiofoewfwe234234}
The Grothendieck topology $\tau_{\chi}$ on $\BC$ is subcanonical.\index{subcanonical}
\end{lem}

\begin{proof}
We must show that for every $X^{\prime}$ in $\BC$ the representable
presheaf $\yo(X^{\prime})$ is a sheaf. Note that $\yo(X^{\prime})$ comes from the  set-valued presheaf $y(X^{\prime})$
by postcomposition with the functor $$\iota:\Set^{\la}\to \sSet^{\la}\to  \Spc^{\la}\, ,$$
where $$y:\BC\to \PSh_{\Set^{\la}}(\BC)$$ is the usual one-categorical Yoneda embedding from bornological coarse spaces to set-valued presheaves on $\BC$.
Since $\iota$ preserves limits
it  suffices to show descent for the set-valued sheaf $y(X^{\prime})$ in $\PSh_{\Set^{\la}}(\BC)$. Let $(Z,\cY)$ be a complementary pair on a bornological coarse space $X$. Then we must show that the natural map
$$y(X^{\prime})(X)\to y(X^{\prime})(Z)\times_{y(X^{\prime})(Z\cap \cY)}y(X^{\prime})(\cY)$$
is a bijection. The condition that there exists an $i_{0}$ in $I$ such that $Z\cup Y_{i_{0}}=X$ implies injectivity. 
We consider now an element on the right-hand side. It consists of a morphism $f:Z\to X^{\prime}$ and a compatible family of morphisms $g_{i}:Y_{i}\to X^{\prime}$ for all $i$ in $I$ whose restrictions to $Z\cap Y_{i}$ coincide with the restriction of $f$.
We therefore get a well-defined map of sets $h:X\to X^{\prime}$ restricting to $f$ and the $g_{i}$, respectively. In order to finish the argument it suffices to show that $h$ is a morphism in $\BC$.
If $B^{\prime}$ is a bounded subset of $X^{\prime}$, then the equality  $$h^{-1}(B^{\prime})=\big(g^{-1}(B^{\prime})\cap Z\big)\cup \big(f_{i_{0}}^{-1}(B^{\prime})\cap Y_{i_{0}}\big)$$ shows that $f^{-1}(B)$ is bounded. Therefore $h$ is proper.
We now show that $h$ is controlled. Let $U$ be an entourage of $X$. Then we have a non-disjoint decomposition
$$U= \big(U\cap (Z\times Z)\big) \cup  \big(U\cap (Y_{i_{0}}\times Y_{i_{0}})\big)\cup  \big(U\cap (Z\times Y_{i_{0}})\big)\cup  \big(U\cap (Y_{i_{0}}\times Z)\big)\, .$$ We see that
$$(h\times h)(U\cap (Z\times Z))=(f\times f)(U\cap (Z\times Z))$$ and   $$ (h\times h)(U\cap (Y_{i_{0}}\times Y_{i_{0}}))= (g_{i_{0}}\times g_{i_{0}})(U\cap (Y_{i_{0}}\times Y_{i_{0}})) $$
are controlled. For the remaining two pieces we argue as follows. Since the family $\cY$ is big there exists $i$ in 
$I$ such that $U\cap (Y_{i_{0}}\times Z)\subseteq U\cap (Y_{i}\times Y_{i})$.
Then
$$(h\times h)(U\cap (Y_{i_{0}}\times Z))\subseteq (g_{i}\times g_{i})( U\cap (Y_{i}\times Y_{i}))$$
is controlled. For the remaining last piece we argue similarly. 
We finally conclude that $(h\times h)(U)$ is controlled. 
\end{proof}

\subsection{Coarse equivalences}

Coarse geometry is the study of  bornological coarse spaces up to equivalence.  {In order to define} the notion of equivalence we first introduce the relation of closeness between morphisms.

Let $X, X^{\prime}$ be bornological coarse spaces, and  let $f_{0},f_{1}\colon X\to  X^{\prime}$ be a pair of morphisms.
\begin{ddd} \label{ioewfjweoifoijwfw32424} We say that the 
  morphisms $f_{0}$ and $f_{1} $  are  close\index{close morphisms}  {to each other} if \-
$(f_{0}\times f_{1})(\diag_{X})$ is an entourage of $X^{\prime}$.
\end{ddd}
Since the coarse structure on $X^{\prime}$ contains the diagonal and is closed under symmetry and composition the relation of
being close to each other is an equivalence relation on the set of morphisms from $X$ to $X^{\prime}$.

Let $f:X\to X^{\prime}$ be a morphism in $\BC$.
 
\begin{ddd} \label{ioioieorwerr3245345345432}
 We say  $ f$ is an equivalence\index{equivalence!in $\BC$}\index{coarse!equivalence}   if there exists   a morphism $g:X^{\prime}\to X$ in $\BC$ such that $f\circ g$ and $g\circ f$ are close to the respective identities.
\end{ddd}

\begin{ex}\label{fijweiofjwefoejufoeifueofeffff1}
Let $X$ be a bornological coarse space and $Y$  be a subset of $X$. If $U$ is an entourage of $X$ containing the diagonal, then the inclusion $Y\to U[Y]$ is an equivalence in $\BC$ if we equip the subsets with the induced structures.
In order to define an inverse $g: U[Y]\to Y$ we simply choose for every point  $x$ in $U[Y]$ a point $g(y)$ in $Y$ such that $(x,g(x))$ in $U$. 
\hB
\end{ex}

\begin{ex}
Let $(X,d)$ and $(X^{\prime},d^{\prime})$ be metric spaces and $f\colon X\to X^{\prime}$ be a map between the underlying sets. Then $f$ is a quasi-isometry if there exists $C,D,E$ in $(0,\infty)$ such that for all $x,y$ in $X$ we have
$$C^{-1}d^{\prime}(f(x),f(y)-D\le d (x,y)\le Cd^{\prime}(f(x),f(y))+D$$ and for every $x^{\prime}$ in $X^{\prime}$ there exists $x$ in $X$ such that $d^{\prime}(f(x),x^{\prime})\le E$.

If $f$ is a quasi-isometry, then the map
$f:X_{d}\to X_{d}^{\prime}$ (see Example \ref{welifjwelife89u32or2}) is an equivalence.\hB
\end{ex}

\begin{ex}
Invariance under equivalences is a basic requirement for useful invariants of bornological coarse spaces.
Typical examples of such invariants with values in $\{\text{false}, \text{true}\}$ are   local and global finiteness conditions or conditions on the generation of the coarse or bornological structure. Here is a list of examples of such invariants for a bornological coarse space $X$.
\begin{enumerate}
\item $X$ is locally countable (or locally finite), {see Definition \ref{defn:okmnsf}.} 
\item $X$ has the minimal compatible bornology, {see Example \ref{wfijiofweofewfwefewfewf}.}
\item $X$ has locally bounded geometry, {see Definition \ref{fwejfoiwejfewojfoefjewfewfewf}.} 
\item $X$ has bounded geometry, {see Definition \ref{dsf892323234}.}
\item The bornology of $X$ is countably generated, {see Example \ref{sdf89232111332}.}
\item The coarse structure of $X$ has one (or finitely many, or countably many) generators, {see Example \ref{sdf89232111332}.}
\end{enumerate}
The set $\pi_{0}^{\coarse}(X)$  of coarse components of $X$ is an example of a $\Set$-valued invariant, see Definition \ref{ekhqiuheiuqe123122342334}.\hB
\end{ex}

Our next localization implements the idea that a sheaf $E$ in $\Sh (\BC)$ should send equivalences in $\BC$ to equivalences in $\Spc^{\la}$.  To this end
we consider the coarse bornological space $\{0,1\}$ with the maximal bornological and coarse structures.
We now observe 
that two morphisms $f_{0},f_{1}:X\to X^{\prime}$ are close to each other if and only if we can combine them to a single morphism
$\{0,1\} \otimes X \to X^{\prime}$, see Example \ref{eiofweoifwefuewfieuwf9wwfwef} for $\otimes$.
 For every bornological coarse space $X$ we have a projection $\{0,1\}\otimes X\to X$.

We now implement the invariance under coarse equivalences into   $\Sh(\BC)$ using the interval object $\{0,1\}$ similarly as in the approach to $\bbA^{1}$-homotopy theory by Morel--Voevodsky \cite[Sec. 2.3]{mv99}.

Let $E$ be in $\Sh (\BC)$.
\begin{ddd}
We call  $E$ coarsely invariant\index{coarsely!invariant} if  for every $X$ in $\BC$ the projection
induces an equivalence
$$E(X)\to E(\{0,1\} \otimes X)\, .$$
\end{ddd}

\begin{lem}\label{ifjewoifjufioewfewfewfewf} The 
coarsely invariant sheaves form a full localizing subcategory of \linebreak$\Sh(\BC)$. 
\end{lem}

\begin{proof}
The subcategory of coarsely invariant sheaves is the full subcategory of objects which are local \cite[Def. 5.5.4.1]{htt} with respect to the small set of projections $\yo(\{0,1\} \otimes X)\to \yo(X)$ for all $X$ in $\BC$. The latter  are morphisms of sheaves by   Lemma \ref{wfeoijweiofoewfwe234234}. We then apply  \cite[Prop. 5.5.4.15]{htt} to finish the proof.
\end{proof}

We denote the    subcategory   of coarsely invariant sheaves by  $\Sh^{ \{0,1\}}(\BC)$\index{$\Sh^{ \{0,1\}}(-)$}. We have a corresponding adjunction
\[H^{\{0,1\}}  :  \Sh(\BC)\leftrightarrows  \Sh^{ \{0,1\}}(\BC):\incl\, .\]

Note that in order to show that a sheaf in  $\Sh (\BC)$ is coarsely invariant it suffices to show that it sends pairs of morphisms in $\BC$ which are close to each other to  pairs of equivalent morphisms in $\Spc^{\la}$. 

\begin{rem}
We use the notation $H^{\{0,1\}}$ since this functor resembles the homotopification functor $\cH$ in differential cohomology theory \cite{MR3462099}. The analogous  functor   in the construction of the motivic homotopy category in algebraic geometry is discussed, e.g., in  Hoyois \cite[Sec.~3.2]{Hoyois:2015aa} and denoted there by $L_{\cA}$.
\hB
\end{rem}

\subsection{Flasque spaces}

The following notion is modeled after  Higson--Pedersen--Roe \cite[Sec.~10]{higson_pedersen_roe}.

Let $X$ be in $\BC$.

\begin{ddd} 
 \label{efijewifjewiofwifwfew322423424}
 $X$ is  flasque\index{bornological coarse space!flasque} \index{flasque!bornological coarse space} if it admits an  endomorphism $f:X\to X$ with the following properties:
 \begin{enumerate}
\item \label{asdfgjtzkj561} The morphisms $f$ and $\id_{X}$ are close to each other.  
\item \label{sdfn3443243t61} For every entourage $U$ of $X$ the union
$\bigcup_{k\in \nat} (f^{k}\times f^{k})(U)$
is again an entourage of~$X$.
\item\label{fwoiejweoifeowi231} For every bounded subset $B$ of $X$ there exists $k  $ in $\IN$ such that  $f^{k}(X)\cap B=\emptyset$.
\end{enumerate}
\end{ddd}

In this case we  will say that flasqueness of $X$ is implemented by $f$.\index{flasqueness!implemented by $-$}

\begin{ex} 
\label{feijweifewfoiuwe9fuewfewfwf}
Consider $[0,\infty) $ as a bornological coarse space  with the structure induced by the metric; see~Example~\ref{welifjwelife89u32or2}.
Let $X $ be a bornological coarse space. Then the bornological coarse space
  $[0,\infty)\otimes X$ (see Example \ref{eiofweoifwefuewfieuwf9wwfwef}) 
 is flasque. For $f$ we can take the map
$f(t,x):=(t+1,x)$.   
Let $\cB$ denote the bornology of $X$ and 
note that the bornology  $ \cB_{d}\times \cB$ of $[0,\infty)\otimes X$ is smaller than the bornology of the cartesian product $[0,\infty)\times X$; see Lemma~\ref{wleifjewfiewiofuew9fuewofewf}.
Indeed, the bornological coarse space $[0,\infty)\times X$
 is not  flasque because Condition \ref{fwoiejweoifeowi231} of Definition~\ref{efijewifjewiofwifwfew322423424} is violated. But note that $X\ltimes [0,\infty)$ (see Example~\ref{hffweiufweif})
is flasque, too.

Similarly, if $\nat_{d}$ denotes the $\nat$ with structures induced form the embedding into $[0,\infty)$, then
$\nat_{d}\otimes X$ is flasque for every 
 bornological coarse space  $X$. Flasqueness is again witnessed by the map
 $f:\nat_{d}\otimes X\to \nat_{d}\otimes X$ given by $f(n,x):=f(n+1,x)$.
\hB
\end{ex}

Let $E$  be in  $\Sh^{ \{0,1\}}(\BC)$.
\begin{ddd}
We say that  $E$ vanishes on flasque bornological coarse spaces if  
  $E(X)$ is a final object in $\Spc$ for every flasque bornological coarse space $X$.
\end{ddd}

We will say shortly that $E$ vanishes on flasque spaces.

\begin{rem}\label{rgoirehgoir34rt43t3}
If $ \emptyset_{\Sh}$ denotes an initial object in $\Sh(\BC)$, then for  every $E$ in $\Sh(\BC)$ we have the equivalence  of spaces $$ \Map_{\Sh(\BC)}(\emptyset_{\Sh} ,E)\simeq *\, .$$
 On the other hand, by the sheaf condition (see Definition \ref{ewlfjwefewoi2r543555})
we have the equivalence of spaces $$\Map_{\Sh(\BC)}(\yo(\emptyset),E)\simeq E(\emptyset) \simeq*   \, .$$
We conclude that 
$$\yo(\emptyset)\simeq  \emptyset_{\Sh}\, ,$$ i.e., that $\yo(\emptyset)$ is an initial object of $\Sh({\BC})$.
Hence a sheaf $E$ vanishes on the  flasque bornological coarse space  $X$  if and only if it is local with respect to the morphism
$\yo(\emptyset)\to \yo(X)$.
\hB
\end{rem}


\begin{lem}\label{ifjweiofwoeifuew9u923445}
The coarsely invariant sheaves which vanish on flasque spaces form a 
full localizing subcategory of $ \Sh^{ \{0,1\}}(\BC)$.
\end{lem}

\begin{proof}
By Remark \ref{rgoirehgoir34rt43t3} the  category of coarsely invariant sheaves which vanish on flasque  spaces  is
  the full subcategory of sheaves which are local with respect to the union of the small sets
 of morphisms considered in  the proof of Lemma \ref{ifjewoifjufioewfewfewfewf} and the morphisms 
$ \yo(\emptyset)\to \yo(X)$ for all flasque $X$ in  $\BC$.
\end{proof}

We denote the subcategory  of coarsely invariant sheaves which vanish on flasque spaces  by
$\Sh^{ \{0,1\},\fl}(\BC)$. We have an adjunction
$$\Fl :\Sh^{ \{0,1\} }(\BC)\leftrightarrows \Sh^{ \{0,1\},\fl}(\BC):\incl\, .$$

\begin{rem}\label{fiwefhj43898345895}
The empty bornological coarse space $\emptyset$ is flasque. Therefore we have the equivalence $E(\emptyset)\simeq *$ for any $E$ in  $\Sh^{ \{0,1\},\fl}(\BC)$. This is compatible with the requirements for $E$ being a sheaf.
 
If $X$ in $\BC$ is flasque, then the canonical map $\emptyset \to X$ induces an equivalence $\Fl(\yo(\emptyset))\to \Fl(\yo(X))$.
\hB
\end{rem}

In some applications 
one considers an apparently more general definition of flasqueness. Let $X$ be in $\BC$.
\begin{ddd}[{\cite[Def.~3.10]{nw1}}]\label{ddd:rf435f}
$X$  is flasque in the generalized sense\index{flasque!generalized sense} if it admits a sequence
 $(f_{k})_{k\in \nat}$ of endomorphisms $f_{k}  :  X\to X$ with the following properties:
\begin{enumerate}
\item \label{sdfn34r} $f_{0}=\id_{X}$.
\item \label{asdfgjtzkj56}
The union $\bigcup_{k\in \nat} (f_{k}\times f_{k+1})(\diag_{X})$ is an entourage of $X$.
\item \label{sdfn3443243t6} For every entourage $U$ of $X$ the union
$\bigcup_{k\in \nat} (f_{k}\times f_{k})(U)$
is again an entourage of $X$.
\item\label{fwoiejweoifeowi23} For every bounded subset $B$ of $X$ there exists a $k_{0}$  in $\IN$ such that for all $k$ in $\IN$ with $k\ge k_{0}$ we have $f_{k}(X)\cap B=\emptyset$.
\end{enumerate}
\end{ddd}
If $X$ is flasque in the sense of Definition \ref{feijweifewfoiuwe9fuewfewfwf} with flasqueness implemented by $f$, then
it is flasque in the generalized sense. We just take  $f_{k}:=f^{k}$ for all $k $ in $\nat$.

Let $X$ be a bornological coarse space. In the following lemma $\nat_{d}$ denotes the bornological coarse space with the structures induced by the standard metric. Further, $\iota \colon X\to \nat_{d}\otimes X$ is the embedding determined by the point $0$ in $\nat$.
\begin{lem}\label{weiogowergfrfwrffwref}
The following assertions are equivalent:
\begin{enumerate}
\item $X$ is flasque in the generalized sense.
\item\label{weroigjorewfferwf} There exists a morphism $F:\nat_{d}\otimes X\to X$ such that $F\circ \iota=\id_{X}$.
\end{enumerate}
\end{lem}
\begin{proof}
 {Let $X$ be} flasque in the generalized sense. 
Let $(f_{k})_{k\in \nat}$ be as in Definition~\ref{ddd:rf435f}.
Then we define
$$F :  \IN_{d} \otimes X \to X \, \quad F(k,x) := f_k(x)\, .$$ 
We argue that $F$ is a morphism.  
We first show that $F$ is proper. Let $B$ be a bounded subset of $X$. 
Let $k_{0}$ in $\nat$ be as in Condition~\ref{ddd:rf435f}.\ref{fwoiejweoifeowi23}.
Then $F^{-1}(B)\subseteq  \bigcup_{k=0}^{k_{0}-1} \{k\}\times  f_{k}^{-1}(B)$. This subset is bounded in $\nat_{d}\otimes X$. We now show that $F$ is controlled. The coarse structure of $\nat_{d}\otimes X$ is generated by the entourages $V:=\{(n,n+1):n\in \nat\}\times \diag_{X}$ and
$\diag_{\nat}\times U$ for all entourages $U$ of $X$. We have
$F(V)=\bigcup_{k\in \nat} (f_{k}\times f_{k+1})(\diag_{X})$ which is an entourage of $X$ by Condition~\ref{ddd:rf435f}.\ref{asdfgjtzkj56}.
Furthermore, $F(U)=\bigcup_{k\in \nat} (f_{k}\times f_{k})(U)$ is an entourage of $X$ by Condition~\ref{ddd:rf435f}.\ref{sdfn3443243t6}. Finally, by Condition~\ref{ddd:rf435f}.\ref{sdfn34r} we have $F\circ \iota=\id_{X}$. 

Vice versa, assume that a morphism $F:\nat_{d}\otimes X\to X$ is given such that $F\circ \iota=X$. Then we define family of morphisms $(f_{k})_{k\in \nat}$ by
$f_{k}(x):=F(k,x)$. Similar considerations as above show that this family satisfies the conditions listed in Definition~\ref{ddd:rf435f}. 
\end{proof}

\begin{rem}
It is not clear that the condition of being flasque is coarsely invariant. But it is almost obvious that
flasqueness in the generalized sense is invariant under coarse equivalences.
Let $X$ be flasque in the generalized sense witnessed by the family $(f_{k})_{k\in \nat}$. Let furthermore
$g:X\to Y$ be a coarse equivalence. Then we can choose an inverse $h:Y\to X$ up to closeness.
We then define a family of endomorphisms $(f_{k}')_{k\in \nat}$ of $Y$ by $f'_{0}:=\id_{Y}$ and
$f_{k+1}':=g\circ f_{k}\circ h$. It is now easy to check that the family $(f_{k}')_{k\in \nat}$ satisfies the conditions listed 
in Definition \ref{ddd:rf435f}.  Hence the family  $(f_{k}')_{k\in \nat}$ witnesses that $Y$ is flasque in the generalized sense.
\hB
\end{rem}

Let $E$ be in $\Sh^{ \{0,1\},\fl}(\BC)$ and $X$ be in $\BC$.
\begin{lem}\label{rkgjergergergergerg}
If $E\in \Sh^{ \{0,1\},\fl}(\BC)$ and $X$ is flasque in the generalized sense, then $E(X)\simeq *$.
\end{lem}

\begin{proof}
By Lemma \ref{weiogowergfrfwrffwref} we  have the factorization 
  \begin{equation}\label{rwrgtret5345454z}
X \stackrel{ \iota}{\to}  \IN_{d} \otimes X \stackrel{F}{\to} X
\end{equation}
of the identity of $X$.
We apply $E$ and get the factorization
$$E(X) \stackrel{E(F)}{\to} E( \IN_{d} \otimes X)  \stackrel{E( \iota)}{\to} E(X)$$
of the identity of $E(X)$.
Since $E$  vanishes on flasques and $\nat_{d}\otimes X$ is flasque by Example~\ref{feijweifewfoiuwe9fuewfewfwf} we conclude that $\id_{E(X)}\simeq 0$. This implies $E(X)\simeq 0$.
%
%
\end{proof}

\subsection{\texorpdfstring{$u$}{u}-continuity and motivic coarse spaces}

Finally we will enforce that the Yoneda embedding preserves the colimits considered in Example~\ref{wfeoihewiufh9ewu98u2398234324234}.

Let $E$ be in $\Sh^{ \{0,1\},\fl}(\BC)$.
\begin{ddd}
$E$  is $u$-continuous\index{$u$-continuous}\index{continuous!$u$-} if
for every bornological coarse space $(X,\cC,\cB) $  the natural morphism
$$E(X)\to \lim_{U\in \cC} E(X_{U})$$
is an equivalence.
\end{ddd}

\begin{lem}\label{fwlkfjweklflwefewfewfwef}
 The full subcategory of  $\Sh^{ \{0,1\},\fl}(\BC)$ of $u$-continuous sheaves  is localizing. 
\end{lem}

\begin{proof}  
The proof  is similar to the proof of Lemma \ref{ifjweiofwoeifuew9u923445}. We just add the small set of morphisms
$$\colim_{U\in \cC} \yo(X_{U})\to \yo(X)$$ 
for all $X$ in $\BC$ to the list of morphisms for which our sheaves must be local.
\end{proof}

\begin{ddd} The subcategory  $\Spc\cX$\index{$\Spc\cX$} described in Lemma \ref{fwlkfjweklflwefewfewfwef} is  the subcategory   of motivic coarse spaces.\index{motivic!coarse space}\index{space!motivic coarse}
\end{ddd} It fits into  
the localization adjunction
$$U:\Sh^{ \{0,1\},\fl}(\BC)\leftrightarrows \Spc\cX:\incl\ .$$

\begin{ddd}\label{qergkjefkoerfrefwefwerfev}
We define the functor
$$\Yo:=U\circ \Fl\circ H^{\{0,1\}}\circ L\circ \yo:\BC\to \Spc\cX\ .$$
\end{ddd}
Because of Lemma  \ref{wfeoijweiofoewfwe234234} we could drop the functor $L$ in this composition.

Every bornological coarse space $X$ represents a motivic coarse space 
$  \Yo(X)$  in $\Spc\cX$,
which is called the coarse motivic space associated to $X$.
 
By construction we have the following rules:
\begin{kor}\label{kor:sdfgert}\mbox{}
\begin{enumerate}
\item \label{wefoihoi2rnjbkjfiufew} The $\infty$-category    motivic coarse spaces  $\Spc\cX$ is presentable and fits into a localization
$$U\circ \Fl\circ H^{\{0,1\}}\circ L:\PSh(\BC)\leftrightarrows \Spc\cX:\incl\, .$$
\item\label{ifjweifjewiojwefw23} If $(Z,\cY)$ is a complementary pair on a bornological coarse space $X$, then the square  $$\xymatrix{\Yo(Z\cap \cY)\ar[r]\ar[d]&\Yo(\cY)\ar[d]\\\Yo(Z)\ar[r]&\Yo(X)}$$ is cocartesian in $\Spc\cX$.
\item\label{wefiewhfeiwfewiofuewufewfewf} If $X\to X^{\prime}$ is an equivalence  of bornological coarse spaces, then $\Yo(X)\to \Yo(X^{\prime})$ is an equivalence in $\Spc\cX$.
\item\label{feijfewiofeoifuefuewofe234234324} If $X$  is a flasque  bornological coarse space, then $\Yo(X)$ is an initial  object of $\Spc\cX$.
\item \label{kor:sdfgert111} For every  bornological coarse space $X$ with coarse structure $\cC$ we have
$$\Yo(X)\simeq \colim_{U\in \cC} \Yo(X_{U})\, .$$
 \end{enumerate}
\end{kor}

Using Lurie \cite[Prop.~5.5.4.20]{htt} the $\infty$-category of motivic coarse spaces has by construction the following universal property:
\begin{kor}
For an $\infty$-category $\bC$ we have an equivalence between     $\Fun^{\colim}(\Spc\cX,\bC)$ and the full subcategory of 
  $\Fun^{\colim}( \PSh(\BC),\bC)$ of functors which preserve the equivalences or colimits listed in Corollary \ref{kor:sdfgert}.
\end{kor}
The superscript $\colim$ stands for small colimit-preserving\index{$\Fun^{\colim}(-,-)$} functors.
{By the} universal property of presheaves \cite[Thm.~5.1.5.6]{htt} {together with the above corollary} we get:
\begin{kor}\label{efiuweif8u23942342344}
 For every large, small cocomplete $\infty$-category $\bC$   we have an equivalence between $ \Fun^{\colim}(\Spc\cX ,\bC)$ and the full subcategory of
 $\Fun( \BC,\bC)$ of functors  which satisfy excision, preserve equivalences, annihilate flasque spaces, and which  are $u$-continuous.
\end{kor}


 \subsection{Coarse excision and further properties}\label{efijoefvvfvsdfvdfsdfvsfvsfdv}

The following  generalizes Point \ref{feijfewiofeoifuefuewofe234234324} of Corollary \ref{kor:sdfgert}.
Let $X$ be in $\BC$.
\begin{lem}\label{oigjwroifjweofewfew534}
If $X$ is flasque in the generalized sense\index{flasque!generalized sense}, then $\Yo(X)$ is an initial object of $\Spc\cX$.
\end{lem}

\begin{proof}
This follows from the proof of Lemma  \ref{rkgjergergergergerg}.
\end{proof}

Let $(X,\cC,\cB)$ be a bornological coarse space, and let $A$ be a subset of $X$. We consider the big family $\{A\}$  on $X$ generated by $A$. By Example~\ref{fijweiofjwefoejufoeifueofeffff1} the inclusion $A\to U[A]$ is an equivalence in $\BC$ for every $U$ in $\cC$ containing the diagonal of $X$. By Point~\ref{wefiewhfeiwfewiofuewufewfewf} of Corollary~\ref{kor:sdfgert} the induced map $\Yo(A)\to \Yo(U[A])$ is an equivalence. 
Note that $$\Yo(\{A\})\simeq \colim_{U\in \cC}  \Yo(U[A])\, .$$ Since the entourages containing the diagonal are cofinal in all entourages we 
 get the following corollary:

Let $X$ be in $\BC$,  and let $A$  be a subset of $X$.
\begin{kor}\label{fijweiofjwefoejufoeifueofeffff}
We have an equivalence $\Yo(A)\simeq \Yo(\{A\})$.
\end{kor}

We now 
  relate the Grothendieck topology\index{Grothendieck topology} $\tau_{\chi}$ with the usual notion of coarse excision.\index{coarse!excision}
Let $X$ be a bornological coarse space, and let $(Y,Z)$ be a pair of  subsets of $X$.

\begin{ddd}\label{efjwelkfwefoiu2or4234234}
 $(Y,Z)$ is a coarsely excisive pair\index{coarsely!excisive pair} if $X=Y\cup Z$ and for every entourage $U$ of $X$ there exists an entourage $V$ of $X$ such that
$$U[Y]\cap U[Z]\subseteq V[Y\cap Z]\, .$$
\end{ddd}
  
  Let $X$ be a bornological coarse space, and let $Y,Z$ be two subsets of $X$.
\begin{lem}\label{lefjweifjewoifeowufiewfewfwe234}
If $(Y,Z)$ is a coarsely excisive pair, then the  square\index{excision!coarse}
$$\xymatrix{\Yo(Y\cap Z)\ar[r]\ar[d]&\Yo(Y)\ar[d]\\\Yo(Z)\ar[r]&\Yo(X)}$$  is cocartesian in $\Spc\cX$.
\end{lem}

\begin{proof}
The pair $(Z,\{Y\})$ is complementary.
By Point~\ref{ifjweifjewiojwefw23} of Corollary~\ref{kor:sdfgert} the square
$$\xymatrix{\Yo(\{Y\}\cap Z)\ar[r]\ar[d]&\Yo(\{Y\})\ar[d]\\\Yo(Z)\ar[r]&\Yo(X)}$$
is cocartesian. We finish the proof by identifying the respective corners. 
By Corollary~\ref{fijweiofjwefoejufoeifueofeffff} we have $\Yo(Y)\simeq \Yo(\{Y\})$.
It remains to show that
$\Yo(Y\cap Z)\to \Yo(\{Y\}\cap Z)$ is an equivalence. 
Since $$U[Y\cap Z]\subseteq U[ {V}[Y]\cap Z]$$ for every two entourages $U$ and $V$ {(such that $V$ contains the diagonal)} of $X$ we have a map
$$\Yo(\{Y\cap Z\})\to  \Yo(\{{V}[Y]\cap Z\})\simeq \Yo({V}[Y]\cap Z)\to \Yo(\{Y\}\cap Z)\, .$$ Since $(Y,Z)$ is coarsely excisive, 
for every entourage $U$ of $X$  we can find an entourage $W$ of $X$ such that
$U[Y]\cap Z \subseteq W[Y\cap Z]$.
Hence we get an inverse map $\Yo(\{Y\}\cap Z)\to \Yo(\{Y\cap Z\})$. 
We conclude that both maps in the  chain $$\Yo(Y\cap Z)\to \Yo(\{Y\}\cap Z)\to  \Yo(\{Y\cap Z\}) $$ are equivalences.
\end{proof}

\begin{ex}\label{ewf3424342323}
Here we construct an example showing that a pullback in general does not preserve coarsely excisive pairs.

Let $X_{{max}}$ be the set $\{a,b,w\}$ with the maximal bornological coarse structure. We further consider the coproduct $X := \{a,b\}_{{max}} \sqcup \{w\}$ in $\BC$.

We consider the subsets $Y := \{a,w\}$ and $Z := \{b,w\}$ of $X$. Then $(Y,Z)$ is a coarsely excisive pair on $X_{max}$. The identity of the underlying set $X$ is a morphism $f: X \to X_{max}$. Then $(f^{-1}(Y), f^{-1}(Z)) = (Y,Z)$ is not a coarsely excisive pair on $X$.
\hB
\end{ex}

\section{Motivic coarse spectra }\label{fwelfjweoflewfewf}

We will define the category $\Sp\cX$\index{$\Sp\cX$|see{motivic coarse spectra}} of motivic coarse spectra\index{motivic!coarse spectra}  as the stable version of the category of motivic coarse spaces. We will then obtain a stable version $$\Yo^{s}:\BC\to \Sp\cX$$ of the Yoneda functor which turns out to be the universal coarse homology theory. In  Section \ref{fijwefowefew} we   introduce this stabilization and discuss the universal property of $\Yo^{s}$.  
Section \ref{erogerwrefreferfwef} contains the definition of the notion of a coarse homology theory and argue that $\Yo^{s}$ is the universal example.
In Section \ref{fkjflkwjfwofjweofwefw} we list some additional properties of $\Yo^{s}$. Finally, in Section \ref{kjfwoeff2243324324324}
 we  consider the  coarse homotopy invariance of $\Yo^{s}$. This is a very useful strengthening
of Property~\ref{wefiewhfeiwfewiofuewufewfewf} of Corollary~\ref{kor:sdfgert}. It is bound to the stable context because of the usage of fibre sequences in the proof of Proposition~\ref{kfwejfefjewfejwofewfw3343}.

\subsection{Stabilization}\label{fijwefowefew}

Before giving the construction of the category of motivic coarse spectra $\Sp\cX$ as a stabilization of the category of motivic coarse spaces $\Spc\cX$ we demonstrate the method in the case  of the construction of the  category of spectra  $\Sp^{\la}$\index{$\Sp^{\la}$} from the category of spaces $\Spc^{\la}$.

 \begin{rem} \label{ffjewfjewfoewfw234234234}
The large $\infty$-category of  small spectra is a presentable stable $\infty$-category. A reference for it is  Lurie \cite[Sec.~1]{HA}. The $\infty$-category of spectra can be characterized as the universal presentable stable $\infty$-category generated by one object.
It can be  constructed as the stabilization of the category of spaces $\Spc^{\la}$ (see   \eqref{fwefwe234234234} for a model) . To this end
one forms the pointed $\infty$-category $\Spc^{\la}_{*/} $ of pointed spaces. It has a suspension endofunctor
$$\Sigma:\Spc^{\la}_{*/}\to \Spc^{\la}_{*/}\, , \quad X\mapsto \colim ((*\to *)\leftarrow (*\to X)\rightarrow (*\to *))\, .$$
The category of spectra is then obtained by inverting this functor in the realm of presentable $\infty$-categories. According to this prescription   we set
$$\Sp^{\la}:=\Spc^{\la}_{/*}[\Sigma^{-1}]:=\colim\{\Spc^{\la}_{*/}\stackrel{\Sigma}{\to} \Spc^{\la}_{*/} \stackrel{\Sigma}{\to}\Spc^{\la}_{*/}\stackrel{\Sigma}{\to}\dots \}\, ,$$
where the colimit is taken in the  $\infty$-category $\mathbf{Pr}^{L}$ of  large presentable $\infty$-categories and left adjoint functors.
The category $\Sp^{\la}$ fits into an adjunction \begin{equation}\label{cnjkdsncjkkddscasdcadsc}
\Sigma_{+}^{\infty}:\Spc^{\la}\leftrightarrows \Sp^{\la}:\Omega^{\infty}\, ,
\end{equation}
where
$\Sigma_{+}^{\infty}$\index{$\Sigma_{+}^{\infty}$} is the composition 
$$\Spc^{\la}\to \Spc^{\la}_{*/}\to \Sp^{\la}\, .$$
Here   the first functor adds a disjoint base point, and the second functor {is} the canonical one.
The category $\Sp^{\la}$ is characterized by the obvious universal property that   precomposition by $\Sigma^{\infty}_{+}$ induces an equivalence
$$\Fun^{L}(\Sp^{\la},\bC)\simeq \Fun^{L}(\Spc^{\la},\bC)\simeq \bC $$
 for every large presentable stable $\infty$-category $\bC$.
 A reference for the  first equivalence is   Lurie \cite[1.4.4.5]{HA}, while the second follows from the universal property of $\Spc^{\la}$.
\hB
\end{rem}


In order to define the $\infty$-category of motivic coarse spectra we proceed in a similar manner, starting with the presentable category of $\Spc\cX$ of motivic coarse spaces. We then let $\Spc\cX_{*/}$ be the pointed version of coarse motivic spaces and again consider the suspension endofunctor
$$\Sigma:\Spc\cX_{*/}\to \Spc\cX_{*/}\, , \quad X\mapsto \colim((*\to *)\leftarrow (*\to X)\rightarrow (*\to *))\, .$$
\begin{ddd}\label{4thiojtgregegegg}
We define the category of motivic coarse spectra
by
$$\Sp\cX:= \Spc\cX_{*/}[\Sigma^{-1}] :=\colim\{\Spc\cX_{*/}\stackrel{\Sigma}{\to} \Spc\cX_{*/} \stackrel{\Sigma}{\to}\Spc\cX_{*/}\stackrel{\Sigma}{\to}\dots \}\, ,$$ 
where the colimit is taken in  $\mathbf{Pr}^{L}$.
\end{ddd}
\index{$\Sp\cX$}
By construction,
$\Sp\cX$ is a large presentable stable $\infty$-category.
We have a functor
$$\Sigma_{+}^{\mot}:\Spc\cX\to \Spc\cX_{*/}\to \Sp\cX$$
\index{$\Sigma_{+}^{\mot}$}which fits into an adjunction
$$\Sigma_{+}^{\mot}:\Spc\leftrightarrows \Sp\cX:\Omega^{\mot}\, .$$
\begin{ddd}\label{ewiogwergergeffwfrf}
We define the functor
$$\Yo^{s}:=\Sigma_{+}^{\mot}\circ \Yo:\BC\to  \Sp\cX\, .$$
\end{ddd}

In particular, every bornological coarse space   represents a coarse spectrum  
$$  \Yo^{s}(X):=\Sigma^{\mot}_{+}(  \Yo(X))\, .$$
\index{$\Yo^{s}$}\index{Yoneda embedding}{The functor $\Sigma^{\mot}_{+}:\Spc\cX\to \Sp\cX$ has the following} universal property  \cite[Cor.~1.4.4.5]{HA}: For every cocomplete stable {$\infty$-category} $\bC$ precomposition by $\Sigma_{+}^{\mot}$ provides an equivalence 
\begin{equation}\label{geruhgergiu34iut34t3t34t3t}
\Fun^{\colim} (\Spc\cX,\bC)\simeq \Fun^{\colim}(\Sp\cX,\bC)\, ,
\end{equation}
where the super-script $\colim$ stands for small colimit-preserving functors. 
To be precise, in order to apply \cite[Cor.~1.4.4.5]{HA} we must assume that  $\bC$ is presentable and stable. For an extension to all cocomplete stable $\infty$-categories see the next Lemma~\ref{lem5465utrter354}.\footnote{We thank Denis-Charles Cisinski and Thomas Nikolaus for providing this argument.}

If   $\bD$ is a   pointed $\infty$-category admitting finite colimits, then we can define the colimit
\begin{equation}\label{ejioewfjoiwefjweof}
\Sp(\bD):= \colim (\bD\stackrel{\Sigma}{\to} \bD \stackrel{\Sigma}{\to} \bD\stackrel{\Sigma}{\to} \dots) \, .
\end{equation}
In the lemma below $\bD$ is $\kappa$-presentable for some regular cardinal $\kappa$ (in the small universe),  and we interpret the colimit in the $\infty$-category $\mathbf{Pr}_{*}^{L,\kappa}$ of pointed $\kappa$-presentable $\infty$-categories.
 
Let $\bC$ be a  stable $\infty$-category, and let $\bD$ be a pointed $\infty$-category.
\begin{lem}\label{lem5465utrter354}
If $\bD$ is presentable and $\bC$ is small cocomplete, then
the natural functor $\bD\to \Sp(\bD)$ induces an equivalence
$$\Fun^{\colim}(\Sp(\bD),\bC)\simeq \Fun^{\colim}(\bD,\bC)\, .$$
\end{lem}
If $\bC$ is presentable, then this lemma is  shown in Lurie \cite[Cor.~1.4.4.5]{HA}.

\begin{proof}
We can assume that $\bD$ is $\kappa$-presentable for some regular cardinal $\kappa$. Then we have
 $\bD\simeq \Ind_{\kappa}(\bD^{\kappa})$, where
$\bD^{\kappa}$ denotes the small category of $\kappa$-compact objects in $\bD$ and $\Ind_{\kappa}(\bD^{\kappa})$ is the free completion of
$\bD^{\kappa}$ by $\kappa$-filtered colimits.  
We let  $\Cat^{L,\kappa}_{\infty,*}$   be the $\infty$-category of
$\kappa$-cocomplete $\infty$-categories and  functors which preserve $\kappa$-small colimits   (notation  $\Fun^{\colim,\kappa}(-,-)$).  
Then
$\Ind_{\kappa} :\Cat^{L,\kappa}_{\infty,*}\to \mathbf{Pr}^{L,\kappa}_{*}$
satisfies
$$\Fun^{\colim}(\Ind_{\kappa}(\bE),\bC)\simeq \Fun^{\colim,\kappa}(\bE,\bC)$$
for any cocomplete pointed $\infty$-category $\bC$.

 We let
$\Cat_{\infty}^{ex,\kappa}$ be the full subcategory of $\Cat_{\infty,*}^{L,\kappa}$ of stable $\infty$-categories.
 We define the functor
$$\Sp_{\kappa}:\Cat_{\infty,*}^{L,\kappa}\to \Cat_{\infty}^{ex,\kappa}$$ by Formula \eqref{ejioewfjoiwefjweof}, where the colimit is  now interpreted in $\Cat_{\infty,*}^{L,\kappa}$. If $\bC$ belongs to  $\Cat_{\infty}^{ex,\kappa}$, then for every $\bE$ in $ \Cat_{\infty}^{ex,\kappa}$ the natural functor $\bE\to \Sp(\bE)$ induces an equivalence
$$\Fun^{\colim,\kappa}(\Sp_{\kappa}(\bE),\bC)\simeq \Fun^{\colim,\kappa}(\bE,\bC)\, .$$
The functor   $\Ind_{\kappa}$ induces a fully faithful functor (which is even an equivalence if $\kappa>\omega$)
$\Cat_{\infty,*}^{L,\kappa}\to \mathbf{Pr}^{L,\kappa}$. Furthermore, $$\Ind_{\kappa}\circ \Sp_{\kappa}\simeq \Sp\circ \Ind_{\kappa}$$
since $\Ind_{\kappa}$ is a left-adjoint and therefore commutes with colimits.
Consequently, we get a chain of natural equivalences
\begin{align*}
\Fun^{\colim}(\Sp(\bD),\bC) & \simeq \Fun^{\colim}(\Sp(\Ind_{\kappa}(\bD^{\kappa})),\bC)\\
& \simeq \Fun^{\colim}(\Ind_{\kappa}(\Sp_{\kappa}(\bD^{\kappa})),\bC)\\
& \simeq \Fun^{\colim,\kappa}(\Sp_{\kappa}(\bD^{\kappa}),\bC)\\
& \simeq \Fun^{\colim,\kappa}(\bD^{\kappa},\bC)\\
& \simeq \Fun^{\colim}(\Ind_{\kappa}(\bD^{\kappa}),\bC)\\
& \simeq \Fun^{\colim}(\bD,\bC)\qedhere
\end{align*}
\end{proof}



Our main usage of the stability of an $\infty$-category $\bS$ is the existence of a functor
$$\Fun(\Delta^{1},\bS)\to \Fun(\Z,\bS)$$ (where $\Z$ is considered as a poset) which functorially extends  a morphism
$f:E\to F$ in $\bS$ to a cofibre sequence 
$$\dots\to\Sigma^{-1}E\to \Sigma^{-1}F\to  \Sigma^{-1} \Cofib(f)\to E\to F\to \Cofib(f)\to \Sigma E\to \Sigma F\to \cdots$$
We will also use the notation $\Fib(f):=\Sigma^{-1}\Cofib(f)$.

 If $\cY=(Y_{i})_{i\in I}$ is a 
big family in a bornological coarse space $X$, then we define the motivic coarse spectrum
\begin{equation}\label{wefweew254}
\Yo^{s}(\cY):=\colim_{i\in I} \Yo^{s}(Y_{i})\, .
\end{equation}
Since $\Sigma_{+}^{\mot}$ preserves colimits we have the equivalence
$\Yo^{s}(\cY)\simeq \Sigma^{\mot}_{+} (\Yo(\cY))$.
The family of inclusions $Y_{i}\to X$ induces via the universal property of the colimit 
  a canonical morphism $$\Yo^{s}(\cY)\to \Yo^{s}(X)\, .$$ We  will use the notation
 \begin{equation}\label{fkhwiufuiz823zr824234242424234234234234}
(X,\cY):=\Cofib(\Yo^{s}(\cY)\to \Yo^{s}(X))\, .
\end{equation}

Let $X$ be a bornolological coarse space.
The stabilized Yoneda functor $\Yo^s$ has the following properties:
\begin{kor}\label{kjeflwfjewofewuf98ewuf98u798798234234324324343}\mbox{}
\begin{enumerate}
\item\label{ifjweifjewiojwefw231}  If $\cY$ is a big family in $X$, then we have a fibre sequence
$$\dots\to \Yo^{s}(\cY)\to \Yo^{s}(X)\to (X,\cY)\to \Sigma \Yo^{s}(\cY)\to \dots $$
\item\label{fwejiofjweiofuewofewf234} For a complementary pair $(Z,\cY)$ on $X$  the natural morphism
$$(Z,Z\cap \cY)\to (X,\cY)$$ is an equivalence.
\item \label{iweufhf89wfu89ewfew245} If $X\to X^{\prime}$ is an equivalence  of bornological coarse spaces, then $\Yo^{s}(X)\to \Yo^{s}(X^{\prime})$ is an equivalence in $\Sp\cX$.
\item\label{ofjewofefoewiufewfiewf09i23423434234} If $X$ is flasque, then $ \Yo^{s}(X)$ is a zero object in the stable $\infty$-category $\Sp\cX$. In particular, $\Yo^{s}(\emptyset)$ is a zero object.
\item \label{efwijfiewfoiefe3u40934332r} For every  bornological coarse space $X$ with coarse structure $\cC$ we have the equivalence $\Yo^{s}(X)\simeq \colim_{U\in \cC} \Yo^{s}(X_{U})$.
\end{enumerate}
\end{kor}

\begin{proof}
The Property \ref{ifjweifjewiojwefw231} is clear from the definition \eqref{fkhwiufuiz823zr824234242424234234234234} of the relative motive and the stability of $\Sp\cX$.

The remaining assertions follow from the corresponding assertions of Corollary \ref{kor:sdfgert}
and the fact that $\Sigma_{+}^{\mot}$ is a left-adjoint. Hence it preserves colimits  and initial objects. In the following we give some details.

Property \ref{fwejiofjweiofuewofewf234} is an immediate consequence of
the fact that the square
\begin{equation}\label{hoijiojdio32ddd23ddd}
\xymatrix{\Yo^{s}(Z\cap \cY)\ar[r]\ar[d]&\Yo^{s}(\cY)\ar[d]\\\Yo^{s}(Z)\ar[r]&\Yo^{s}(X)}
\end{equation}
is cocartesian. This   follows from Point~\ref{ifjweifjewiojwefw23} in Corollary \ref{kor:sdfgert}.  

Property  \ref{iweufhf89wfu89ewfew245} immediately follows from Point \ref{wefiewhfeiwfewiofuewufewfewf} of  Corollary \ref{kor:sdfgert}.

We now show Property  \ref{ofjewofefoewiufewfiewf09i23423434234}.
  If $X$ in $\BC$ is flasque, then
$\Yo(X)$ is an initial object of $\Spc\cX$ by Point \ref{feijfewiofeoifuefuewofe234234324} of Corollary \ref{kor:sdfgert}. Hence
$\Yo^{s}(X)\simeq \Sigma^{\mot}_{+}(\Yo(X))$ is an  initial object of $\Sp\cX$.
Since $\Sp\cX$ is stable an initial object in $\Sp\cX$ is a zero object.

Property \ref{efwijfiewfoiefe3u40934332r}   follows   from the $u$-continuity of $\Yo$ stated in Point  \ref{kor:sdfgert111} of Corollary \ref{kor:sdfgert}.
\end{proof}



\subsection{Further properties of \texorpdfstring{$\Yo^{s}$}{Yo-s}}\label{fkjflkwjfwofjweofwefw}

Let $X$ be a bornological coarse space.
\begin{lem}\label{ewfifjewiofoiewfewfewfef}
If $X$ is flasque in the generalized sense, then $\Yo^{s}(X)\simeq 0$.
\end{lem}
\begin{proof}
This follows from Lemma \ref{oigjwroifjweofewfew534} and the fact that $\Sigma_{+}^{\mot}$ sends initial objects to zero objects.
\end{proof}

Let $X$ be in $\BC$, and let $A$ be a subset of $X$.  
\begin{lem}\label{leqwjfefeowewfew23r}
We have an equivalence $\Yo^{s}(A)\simeq \Yo^{s}(\{A\})$.
\end{lem}
\begin{proof}
This follows from Corollary \ref{fijweiofjwefoejufoeifueofeffff} and the fact that
$\Sigma_{+}^{\mot}$ preserves colimits.
\end{proof}

 Let $X$ be a bornological coarse space, and let $Y,Z$ be two subsets of $X$.
\begin{lem}\label{lgjrgrogijreogregregregegegrg}
If $(Y,Z)$ is a coarsely excisive pair on $X$, then the  square 
$$\xymatrix{\Yo^{s}(Y\cap Z)\ar[r]\ar[d]&\Yo^{s}(Y)\ar[d]\\\Yo^{s}(Z)\ar[r]&\Yo^{s}(X)}$$  is cocartesian.
\end{lem}
\begin{proof} This follows from Lemma \ref{lefjweifjewoifeowufiewfewfwe234} and the fact that
$\Sigma_{+}^{\mot}$ preserves colimits. \end{proof}

\begin{ex}
We consider a bornological coarse space $X$. As we have explained in the Example~\ref{feijweifewfoiuwe9fuewfewfwf} the bornological coarse space $[0,\infty) \otimes X$ is flasque. The pair of subsets
$$((-\infty,0] \otimes X,[0,\infty) \otimes X)$$  of $\R \otimes  X$  is coarsely excisive. 
By Lemma \ref{lgjrgrogijreogregregregegegrg} we get the push-out square $$\xymatrix{\Yo^{s}(X)\ar[r]\ar[d]&\Yo^{s}([0,\infty) \otimes  X)\ar[d]\\\Yo^{s}((-\infty,0] \otimes  X)\ar[r]&\Yo^{s}(\R \otimes X)}$$ 

The lower left and the upper right corners  are motives of flasque bornological coarse spaces and hence vanish. Excision therefore provides an equivalence 
\begin{equation}\label{dqwdqwdqwdwqwqdqwd}
  \Yo^{s}(\R \otimes  X)\simeq \Sigma\Yo^{s}(X)\, .
\end{equation}
Iterating this we get
\begin{equation}\label{dqwdqwdqwdwqwqdqwd1}
 \Yo^{s}(\R^{k} \otimes  X)\simeq \Sigma^{k} \Yo^{s}(X)
\end{equation}
for all $k\ge 0$.
\hB
\end{ex}

\begin{ex}\label{jwefjweofewfewfewfewf}
\index{union!mixed}\index{mixed union}The following is  an   application of the stability of the category of coarse motivic spectra.
We consider a family $(X_{i})_{i\in I}$ bornological coarse spaces and form  the {free union (Definition \ref{foifewieof89u8924r443535}) $X:={\bigsqcup^{\free}_{i\in I}} X_{i}$.}

We fix   an element $i\in I$ and set     $I^{\prime}:=I\setminus \{i\}$. Then
$(X_{i} ,{\bigsqcup^{\free}_{j\in I^{\prime}}} X_{j})$ is a coarsely excisive pair on {$X$}. 
Since the intersection of the two entries is disjoint and $\Yo^{s}(\emptyset)\simeq 0$, excision (Lemma \ref{lgjrgrogijreogregregregegegrg})  gives an equivalence
$$\Yo^{s}(X)\simeq \Yo^{s}(X_{i})\oplus \Yo^{s}\big({\bigsqcup^{\free}_{j\in I^{\prime}}X_{j}}\big)\, .$$
In particular we obtain a projection
\begin{equation}\label{uewhwiuhiewufheihfi2342341}
p_i : \Yo^{s}(X) \to \Yo^{s}(X_{i})\, .
\end{equation} 
Note that this projection does not come from a morphism in $\BC$. We can combine the projections $p_{i}$ for all $i$ in $ I$ to a morphism \begin{equation}\label{uewhwiuhiewufheihfi234234}
p:\Yo^{s}(X)\to  \prod_{j\in I} \Yo^{s}(X_{j})\, .
\end{equation} 
This map will play a role in the definition of additivity of a coarse homology theory later.
If $I$ is finite, then it is an equivalence by excision.
\hB
\end{ex}
 
Excision can deal with finite decompositions. The following lemma investigates what happens in the case of inifinite coproducts.

Let $\cA$ be a  set  of objects in   $\BC$.
 \begin{ddd}\label{defn:sdf09232}
We let $\Sp\cX\langle \cA\rangle$ denote the minimal cocomplete stable full subcategory  of  $\Sp\cX$\index{$\Sp\cX\langle -\rangle$} containing $\Yo^{s}(\cA)$.
\end{ddd}

 {We let} $\cA_{\disc}$ be the set   of discrete
bornological coarse spaces (see Definition \ref{ijwieorfjuwefuwe9few435345}).

Let $(X_{i})_{i\in I}$ be a family in $\BC$. 
Then we have a canonical maps \begin{equation}\label{adscasdcdqwefeweddecca}
\bigoplus_{i\in I}\Yo^{s}(X_{i})\to \Yo^{s}(\bigsqcup_{i\in I}X_{i})\ , \quad  \bigoplus_{i\in I}\Yo^{s}(X_{i})\to \Yo^{s}(\bigsqcup^{\mixed}_{i\in I}X_{i})
\end{equation}
induced by the inclusions of the components $X_{i}$ into the coproduct or the mixed union.
 \begin{lem}\label{efiwofoewf234234244}
The fibres of the maps in \eqref{adscasdcdqwefeweddecca} belong to $\Sp\cX\langle \cA_{\disc}\rangle$.
\end{lem}
\begin{proof}
We consider the case of a mixed union.
Let $X := \bigsqcup^{\mixed}_{i \in I}X_{i}$. Let $\cC$ denote  the coarse structure of $X$ and recall that we have the equivalence
$$\Yo^{s}(X)\simeq \colim_{U\in \cC} \Yo^{s}(X_{U})\, .$$
Let now $U$ in $\cC$ be given such that it contains the diagonal. Then there exists a finite subset $J$ of $I$   and entourages $U_{j}$ in $\cC_{j}$ for all $j$ in $J$ such that $$U=  \bigcup_{j\in J} U_{j}\cup \bigcup_{i\in I\setminus J} \diag(X_{i})\, .$$
 We conclude that
$$X_{U}\cong \coprod_{j\in J} X_{j,U_{j}}\sqcup \bigsqcup^{\mixed}_{i\in I\setminus J} X_{i,\disc}\, ,$$
where $X_{i,\disc}:=(X_{i},\cC\langle\emptyset\rangle,\cB_{i})$.
We conclude that
$$\Yo^{s}(X_{U})\simeq \bigoplus_{j\in J} \Yo^{s}(X_{j,U_{j}})\oplus \Yo^{s}\big({\bigsqcup^{\mixed}_{i\in I\setminus J}} X_{i,\disc}\big)\, .$$ 
Let not $J'$ be a second finite subset of $I$ such that 
 $J\subseteq J^{\prime}$. Then by excision  we have a decomposition
$$ \Yo^{s} \big( \bigsqcup^{\mixed}_{i\in I\setminus J} X_{i,\disc}\big)\simeq \Yo^{s}\big(\bigsqcup^{\mixed}_{i\in I\setminus J^{\prime}} X_{i,\disc}\big)\oplus \bigoplus_{i\in J^{\prime}\setminus J} \Yo^{s}(X_{i,\disc})\, .$$
In particular we have a projection
$$ \Yo^{s}\big( \bigsqcup^{\mixed}_{i\in I\setminus J} X_{i,\disc}\big)\to  \Yo^{s}\big( \bigsqcup^{\mixed}_{i\in I\setminus J^{\prime}} X_{i,\disc}\big)\, .$$

 We now take the colimit over $\cC$ in two stages such that the outer colimit increases $J$ and the inner colimits runs over the entourages with a fixed $J$. Then we get a fibre sequence $$\Sigma^{-1}R\to \bigoplus_{i\in I}\Yo^{s}(X_{i})\to \Yo^{s}(X) \to R\, ,$$ where \begin{equation}\label{vefjkbnerkjvbrjebvjkjvnrekvvrrev}
R:= \colim_{J \subseteq I\textit{ finite}}   \Yo^{s}\big( \bigsqcup^{\mixed}_{i\in I\setminus J} X_{i,\disc}\big)
\end{equation}
is a remainder term as claimed by the lemma.

The above argument works word for word also in the case of the coproduct $X := \coprod_{i \in I} X_i$. In this case the remainder term will be \begin{equation}\label{wefuhiuhj4fhjiuwefqef}
\colim_{J}   \Yo^{s}( \coprod_{i\in I\setminus J} X_{i,\disc} )\ .\qedhere
\end{equation}
\end{proof}

 \subsection{Homotopy invariance}\label{kjfwoeff2243324324324}

In order to define the notion of  homotopy in the context of bornological coarse spaces we introduce appropriate cylinder objects (see Mitchener \cite[Sec.~3]{mit} for a similar construction).
A coarse cylinder on $X$ depends on the choice $p=(p_{-},p_{+})$ of two bornological (Definition~\ref{iefewiofeuwio2342342434}) 
maps
$p_+  :  X\to [0,\infty)$ and $p_{-}  :  X\to (-\infty,0]$, where the rays have the usual bornology of bounded subsets.

Let $(X,\cC,\cB )$ be a bornological coarse space. Furthermore, we equip $\R$ with the  metric bornological coarse structure   $\cC_{d}$ and $\cB_{d}$. We consider the bornological coarse space
$\R \otimes X$ introduced in Example \ref{eiofweoifwefuewfieuwf9wwfwef}.


\begin{rem}
 Using $\R \otimes X$ instead of the cartesian product $\R\times X$ 
 has the effect that the projections $\R \otimes X\to \R$ (for unbounded $X$) and $\R \otimes X\to X$  are not    morphisms of bornological coarse spaces since they are not proper.   On the other hand, our definition guarantees that the  inclusion
$X\to \R \otimes X$ given by  $\{0\}$ in $\R$  is an inclusion of a bornological coarse subspace. 
A further reason for using $\otimes$ instead of $\times$ is that we want the half cylinders in $\R \otimes X$ to be flasque, see Example~\ref{feijweifewfoiuwe9fuewfewfwf}. This property will be used in an essential manner in the proof of Proposition~\ref{kfwejfefjewfejwofewfw3343} below.
\hB
\end{rem}

Let $X$ be a bornological coarse space,  and let  $p_{+}:X\to [0,\infty)$ and $p_{-}:X\to (-\infty,0]$ be bornological maps.
\begin{ddd}
The coarse cylinder\index{coarse!cylinder} $I_{p}X$\index{$I_{p}-$|see{coarse cylinder}} is defined as the subset
$$I_{p}X:=\{(t,x)\in \R\times X :  p_{-}(x)\le t\le p_{+}(x)\}\subseteq \R \otimes X$$
equipped with the induced bornological coarse structure.
\end{ddd}
\begin{figure}[ht]
\centering
\includegraphics[scale=0.4]{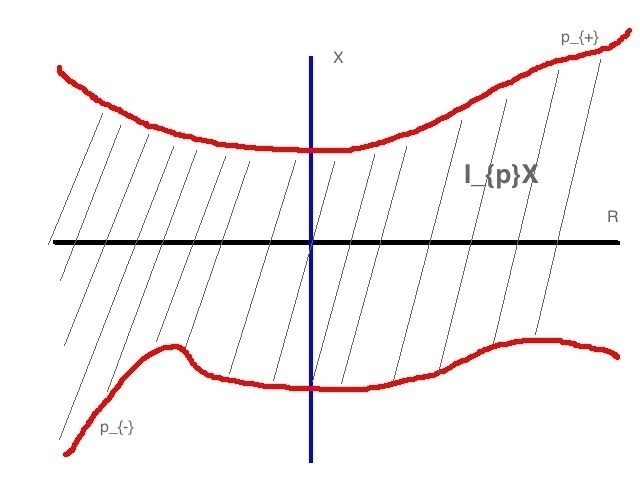}\caption*{Coarse cylinder}
\end{figure}

\begin{lem}
The projection $\pi:I_{p}X\to X$ is a morphism of bornological coarse spaces.
\end{lem}

\begin{proof}
The projection is clearly controlled. We must show that it is proper. Let $B$ be a bounded subset of $X$.
Since the maps $p_{\pm}$ are bornological there exists a positive real number $C$ such that
$|p_{\pm}(x)|\le C$ for all $x$ in $B$. We then have
$\pi^{-1}(B)\subseteq [-C,C]\times B\cap I_{p}X$, and this subset is bounded.
\end{proof}

\begin{prop}\label{kfwejfefjewfejwofewfw3343}
For a coarse cylinder $I_{p}X$ on $X$ the projection $I_{p}X\to X$ induces an equivalence
$$\Yo^{s}(I_{p}X)\to \Yo^{s}(X)\, .$$
\end{prop}

\begin{proof}
We consider the bornological coarse    subspaces of $\R \otimes X$
$$W:=(-\infty,0] \times X \cup    I_{p}(X) \, , \quad Z:=  [0,\infty)\times X\cap I_{p}(X)$$ and the big family
$\{Y\}$ generated by the subset  $Y:=(-\infty,0]\times X$, see Example~\ref{fwijweiofjweoifewfu9fuwe98ff}.
Then $(Z,\{Y\})$ is a complementary pair on $W$. By Point \ref{ifjweifjewiojwefw231} of Corollary~\ref{kjeflwfjewofewuf98ewuf98u798798234234324324343} we
 have   fibre sequences
 $$\dots\to \Yo^{s}(\{Y\}) \to \Yo^{s}(W) \to(W,\{Y\})\to \dots$$
and
$$\dots\to \Yo^{s}(Z\cap \{Y\}) \to \Yo^{s}(Z) \to(Z,Z\cap \{Y\})\to \dots\, ,$$
where we use the construction in Example~\ref{wokfewofwfkopewfefewf} in the second case.
By Point~\ref{fwejiofjweiofuewofewf234} of Corollary~\ref{kjeflwfjewofewuf98ewuf98u798798234234324324343} we have the excision equivalence \begin{equation}\label{wejkhwejkhwekjfw243}
 (Z,Z\cap \{Y\})\simeq  (W,\{Y\}) \, .
\end{equation}

We now observe that $Y$ and $W$ are flasque. In both cases flasqueness is implemented by the self maps  given by the restriction of the map 
$$f :\R \otimes X\to \R \otimes X\, , \quad  (t,x)\mapsto (t-1,x)\, .$$
The three conditions listed in Definition \ref{efijewifjewiofwifwfew322423424} are easy to verify.
It follows from Point~\ref{ofjewofefoewiufewfiewf09i23423434234} of Corollary~\ref{kjeflwfjewofewuf98ewuf98u798798234234324324343}
that $\Yo^{s}(W)\simeq 0$, and from Point~\ref{ofjewofefoewiufewfiewf09i23423434234} of Corollary~\ref{kjeflwfjewofewuf98ewuf98u798798234234324324343} together with  Lemma~\ref{fijweiofjwefoejufoeifueofeffff} that  
$\Yo^{s}(\{Y\})\simeq 0$. The first fibre sequence now implies that $(W,\{Y\})\simeq 0$. Using the second and \eqref{wejkhwejkhwekjfw243} we get the equivalence
$$\Yo^{s}(Z\cap \{Y\})\simeq  \Yo^{s}(Z)\, .$$
Note that
$Z\cap \{Y\}\simeq \{X\}$, where we identify $X$ with the subset $\{0\}\times X$ of $Z$.
We thus get the equivalence \begin{equation}\label{r32r23r23r}
\Yo^{s}(X)\simeq  \Yo^{s}(Z)\, .
\end{equation}

We now consider the subsets
$$V:=I_{p}X\cup [0,\infty)\times X \, , \quad  U:=[0,\infty)\times X $$ of $\R \otimes X$. We again have a complementary pair
$(I_{p}X, \{U\})$ in $V$. 
Argueing as above and using that
$V$ and $U$ are flasque we get an equivalence
$$\Yo^{s}(I_{p}X\cap \{U\})\simeq \Yo^{s}(I_{p}X)\, .$$
We now observe that
$I_{p}X\cap U=Z$. This implies 
by Point \ref{ofjewofefoewiufewfiewf09i23423434234} of Corollary~\ref{kjeflwfjewofewuf98ewuf98u798798234234324324343}   that
$$\Yo^{s}(I_{p}X\cap \{U\})\simeq \Yo^{s}(Z)\, .$$
Combining this with the  equivalence  \eqref{r32r23r23r} we get the equivalence
$$\Yo^{s}(X)\simeq \Yo^{s}(I_{p}X)$$
induced by the inclusion as the zero section. This is clearly an inverse of the projection.
\end{proof}
 
We consider a coarse cylinder $I_{p}X$ over $X$. We can consider the maps of sets
$$i_{\pm}:X\to I_{p}X\, , \quad i_{\pm}(x)=(p_{\pm}(x),x)\, .$$
These maps are morphisms if the bornological maps $p_{+}:X\to [0,\infty)$ and $p_{-}:X\to (-\infty,0]$ are in addition controlled.

Let $X$ and $X^{\prime}$ be bornological coarse  spaces and $f_{+}, f_{-} :  X\to X^{\prime}$ be two morphisms.
\begin{ddd}\label{jgerlgjreolgregregregreg}
We say that $f_{+}$ and $f_{-}$ are homotopic\index{homotopic morphisms}\index{morphisms!homotopic} to each other if there exist a pair of controlled and bornological maps
$p=(p_{+},p_{-})$ and a morphism $h :  I_{p}X\to X^{\prime}$ such that
  {$f_{\pm}=h \circ i_{\pm}$.}
\end{ddd}

Observe that $\pi\circ i_{\pm}=\id_{X}$. Since $\Yo^{s}(\pi)$ is an equivalence, the functor
$\Yo^{s}$ sends $i_{+}$ and $i_{-}$ to equivalent maps. We have the following consequence.

Let $X$ and $X^{\prime}$ be bornological coarse  spaces and $f_{+}, f_{-} :  X\to X^{\prime}$ be two morphisms.

\begin{kor}\label{liwejfweioiewfwefewfewf}
If $f_{+}$ and $f_{-}$ are homotopic, then $\Yo^{s}(f_{+})$ and $\Yo^{s}(f_{-})$ are equivalent.
\end{kor}

Let $f:X\to X^{\prime}$ be a morphism between bornological coarse spaces.
\begin{ddd}
We say that $f$ is a  homotopy equivalence\index{homotopy equivalence}\index{equivalence!homotopy} if there exist a morphism $g:X^{\prime}\to X$ such that the compositions $f\circ g$ and $g\circ f$ are   homotopic to the respective identities.
\end{ddd}

\begin{kor}\label{liwejfweioiewfwefewfewf1}
The functor $\Yo^{s}$ sends homotopy equivalences to equivalences.
\end{kor}
It is easy to see that if $f_{+}$ and $f_{-}$ are close to each other, then they are homotopic. By Point~\ref{iweufhf89wfu89ewfew245} of Corollary~\ref{kjeflwfjewofewuf98ewuf98u798798234234324324343} the functor $\Yo^{s}$ sends close morphisms to equivalent morphisms. 
But homotopy is a much weaker relation.  
The Corollary~\ref{liwejfweioiewfwefewfewf1} can be considered as a vast improvement of Point~\ref{iweufhf89wfu89ewfew245} of Corollary~\ref{kjeflwfjewofewuf98ewuf98u798798234234324324343}.

\begin{ex}\label{efwiojweiofeofewfefewfewf}
Let $n$ be in $\nat$ such that $n\ge 1$.
 For $r$  in $(0,\infty) $ we let $ B^{n}(0,r)$   denote the closed {E}uclidean  ball in $\R^{n}$ at the origin of radius $r$. 
Let $I$ be some set and $r:I\to (0,\infty)$ be some function.
We  consider the set $$X:=\bigsqcup_{i\in I}  B^{n}(0,r(i))\, .$$ We define  a coarse structure on $X$ to be generated by the family 
 of  entourages $(\tilde U_{r})_{r\ge 0}$ on $X$, where  $\tilde U_{r}:=\bigcup_{i\in I} U_{r,i}$,  and  where $U_{r,i}$ is the entourage $U_{r}$ (see Example \ref{ffwefewfe24543234wfefewef}) considered as an entourage of the component of $X$ with index $i$ in $I$. In other words, we consider $X$ as a quasi-metric space
where all components have the induced {E}uclidean metric and their mutual distance is infinite. We let the bornology on $X$  be generated by the  subsets $B^{n}(0,r(i))$ (the component with index $i$) for all $i$ in $I$. The bornology and the coarse structure are compatible.
In this way we have described a bornological coarse space $X$.

We consider $I$ with the discrete coarse structure and the minimal bornology. Then we   have a natural inclusion $$\iota:I\hookrightarrow X$$
as the centers of the balls.

We claim that {$\iota$} is a homotopy equivalence with inverse the projection $\pi:X\to I$.
Note that $\pi\circ \iota\cong \id_{I}$. We must exhibit a homotopy between $\iota\circ \pi$ and $\id_{X}$.
We let $(i,x)$ denote the point $x\in B^{n}(0,r(i))$ of the $i$'th component.
 We consider the functions $p_{-}:=0$ and $p_{+}(i,x):=\|x\|$ and form the coarse cylinder
$I_{p}X$ associated to the pair $p:=(p_{-},p_{+})$. One can check that
$$h:I_{p}X\to X\, , \quad h(t,(i,x)):= \left\{ \begin{array}{cc}(i,x-t\frac{x}{\|x\|})&x\not=0\\
(i,0)&x=0\end{array}\right. $$
is a morphism. Since the functions $p_{\pm}$ are controlled this morphism is a homotopy from the composition 
$\iota\circ \pi$ and $\id_{X}$ as required.

We conclude that
$$\Yo^{s}(I)\simeq \Yo^{s}(X)\, .$$ 

If the function $r$ is unbounded, then $\iota$ and $\pi$ are not coarse equivalences.
\hB
\end{ex}
 
  \subsection{Axioms for a coarse homology theory}\label{erogerwrefreferfwef}

Let $\bC$ be a  cocomplete  stable   $\infty$-category, and let $$E:\BC\to \bC$$ be a functor.
For a big family $\cY=(Y_{i})_{i\in I}$ on a bornological coarse space $X$ we set $$E(\cY):=\colim_{i\in I} E(Y_{i})$$
and  $$E(X,\cY):=\Cofib(E(\cY)\to E(X))\ .$$

\begin{ddd}\label{rgljogreggregrege}
$E$ is a  coarse homology theory \index{classical coarse homology theory}\index{homology theory!classical coarse} if it has the following properties\index{axioms of coarse homology theories}:
\begin{enumerate}
\item (excision) For every complementary pair $(Z,\cY)$ on a bornological coarse space $X$  the natural morphism
$$E(Z,Z\cap \cY)\to E(X,\cY)$$ is an equivalence.
\item (coarse invariance)  If $X\to X^{\prime}$ is an equivalence  of bornological coarse spaces, then $E(X)\to E(X^{\prime})$ is an equivalence in ${\bC}$.
\item (vanishing on flasques)  If $X$ is a flasque bornological coarse space, then $0\simeq E(X) $.
\item\label{efofjewofop3294i0234324} ($u$-continuity) For every  bornological coarse space $X$ with coarse structure $\cC$ we have the equivalence $$E(X)\simeq \colim_{U\in \cC}E(X_{U})$$ induced by the collection of  canonical morphisms $X_{U}\to X$.
\end{enumerate}
\end{ddd}

\begin{rem}
The excision axiom is equivalent to the requirement, that for every  complementary pair $(Z,\cY)$ on a bornological coarse space $X$  the square
$$\xymatrix{E(Z\cap \cY)\ar[r]\ar[d]&E(Z)\ar[d]\\E(\cY)\ar[r]&E(X)}$$
is a push-out square.  \hB
\end{rem}

As a consequence of Corollary \ref{efiuweif8u23942342344} 
and \eqref{geruhgergiu34iut34t3t34t3t}  the $\infty$-category $\Sp\cX$ of motivic coarse spectra has the following universal property\index{universal property!motivic coarse spectra}:
\begin{kor}\label{ofwoief239042309444324}
For every large, small cocomplete stable $\infty$-category $\bC$   precomposition by $\Yo^{s}$ induces an equivalence between $ \Fun^{\colim}(\Sp\cX ,\bC)$ and the full subcategory of $\Fun( \BC,\bC)$ of  $\bC$-valued coarse homology theories.
\end{kor}
Note that $\colim$ stands for the small colimit-preserving functors.
In particular, $\Yo^{s}$ is a  $\Sp\cX$-valued coarse homology theory.

Usually the targets $\bC$ of our coarse homology theories are small and admit all very small colimits. To transfer motivic results to such coarse homology theories using Corollary~\ref{ofwoief239042309444324} we employ the following trick.

There exists a fully faithful, very small colimit preserving  inclusion $i:\bC\to \bC^{\la}$ of $\bC$ into a large  stable  $\infty$-category which admits all small colimits, see Remark \ref{qwrfglkjqodfewewfdewdqeqwedew}.

Here are two  typical arguments using this trick. 

If $E:\BC\to \bC$ is a coarse homology theory, then we define $$E^{\la}:=i\circ E:\BC\to \bC^{\la}\ .$$
Because $i$ preserves colimits, it is again a coarse homology theory. By Corollary \ref{ofwoief239042309444324} it corresponds essentially uniquely to a small  colimit preserving functor  $\Sp\cX\to \bC^{\la}$ for which we use the same symbol.
For $X$ in $\BC$ we then have $i(E(X))\simeq E^{\la}(\Yo^{s}(X))$.

Assume that $f:X\to Y$ is a morphism  in $\BC$.

\begin{kor}\label{qwrfglkjqodfewewfdewdqeqwedew1}If  
$\Yo^{s}(f)\colon \Yo^{s}(X)\to \Yo^{s}(Y)$ is an equivalence, then $E(f)\colon E(X)\to E(Y)$ is an equivalence.
\end{kor}
\begin{proof}
We use that  
 $i$ is fully faithful.    
 \end{proof}

 Assume that $E\to E'$ is a natural transformation of coarse homology theories,
 and let $X$ be in $\BC$. 
 
 \begin{kor}\label{qwrfglkjqodfewewfdewdqeqwedew2}
 If  
 $E^{\la}(\Yo^{s}(X))\to E^{\prime,la}(\Yo^{s}(X))$ is an equivalence, then  $E(X)\to E'(X)$ is an equivalence.
\end{kor}
\begin{proof}
We use that $E(X)\simeq  E^{\la}(\Yo^{s}(X))$ and again that  $i$ is fully faithful.    
\end{proof}

\begin{rem}\label{qwrfglkjqodfewewfdewdqeqwedew}
The construction of such an embedding $i$ is standard.  We let $\Sp^{\la}$ be the category of large spectra\index{large spectra}\index{spectra!large} and consider the stable Yoneda embedding 
$$\yo^{s}:\bC\to \Fun(\bC^{\op},\Sp^{\la})\ .$$
We let $S$ be the small set of the canonical morphisms   
$$\colim_{I} \yo^{s}(D)\to \yo^{s}(\colim_{ I} D)$$ in $ \Fun(\bC^{\op},\Sp^{\la})$ 
for all diagrams $D:I\to \bC$ indexed by very small categories $I$. Since   $\Fun(\bC^{\op},\Sp^{\la})$ is presentable 
we can define $\bC^{\la}$ as the category of $S$-local objects in $\Fun(\bC^{\op},\Sp^{\la})$.
It fits into a adjunction
$$ L:\Fun(\bC^{\op},\Sp^{\la})\leftrightarrows \bC^{\la}:\incl\ .$$
 Note  that $\yo^{s}(C)$ is $S$-local for every object $C$ of $\bC$. 
This follows from
\begin{eqnarray*}
\map_{\Fun(\bC^{\op},\Sp^{\la})}(\yo^{s}(\colim_{I}D),\yo^{s}(C))&\simeq&
\map_{\bC}(\colim_{I}D,C)\\&\simeq&\lim_{I}\map_{\bC}( D,C)\\&\simeq&
\lim_{I}\map_{\Fun(\bC^{\op},\Sp^{\la})}( \yo^{s}(D),\yo^{s}(C))\\&\simeq&
\map_{\Fun(\bC^{\op},\Sp^{\la})}(\colim_{I} \yo^{s}( D),\yo^{s}(C))\ ,
\end{eqnarray*}
where $\map(-,-)$ stands for the mapping spectrum in the respective stable $\infty$-category, and $D:I\to \bC$
is any diagram with a very small index category.
 
 We now set $i:=  \yo^{s}:\bC\to \bC^{\la}$. This immediately implies that $i$ is fully faithful.
%
We finally observe that $i$ preserves very small colimits. 
Let $C$ be in $\bC^{\la}$. 
Then we have the following chain of equivalences
\begin{eqnarray*}
\map_{\bC^{\la}}(i(\colim_{I}D),C)&\simeq&
\map_{\Fun(\bC^{\op},\Sp^{\la})}(\yo^{s}(\colim_{I}D),C)\\&\stackrel{!}{\simeq}&
\map_{\Fun(\bC^{\op},\Sp^{\la})}(\colim_{I}\yo^{s}(D) ,C)\\&\simeq&
\map_{\bC^{\la}}(L(\colim_{I}\yo^{s}(D)) ,C)\\&\simeq&
\map_{\bC^{\la}}( \mbox{$\colim^{\bC^{a}}_{I}$}\yo^{s}(D) ,C)
\end{eqnarray*}
where $\colim_{I}\yo^{s}(D)$ is interpreted in $\Fun(\bC^{\op},\Sp^{\la})$, and
$ \colim^{\bC^{a}}_{I}\yo^{s}(D)$ is interpreted in $\bC^{\la}$.
For the   equivalence marked by $!$ we use the assumption that $C$ is $S$-local.
\hB
\end{rem}

Let $E:\BC\to \bC$ be a coarse homology theory.
 \begin{kor}\label{foijfowffewfwefewfewf}
 The following statements are true:
 \begin{enumerate}
 \item\label{rglkeglregregporeigroepgergregregg} If $A$ is a subset of a bornological coarse space with the induced bornological coarse structure, then the natural morphism induces an equivalence 
 $E(A)\simeq E(\{A\})$.
 \item\label{wefifjweiofioww823232434} If $X$ is a  bornological coarse space which is  flasque in the generalized sense (Definition~\ref{ddd:rf435f}), then $E(X)\simeq 0$.
 \item \label{fijwoeifjoewifewfewfewf}
 $E$ sends homotopic morphisms (Definition \ref{jgerlgjreolgregregregreg}) to equivalent morphisms.
 \item \label{fwlfjwefeoifuwefwf3244} $E$ is coarsely excisive, i.e., for a coarsely excisive pair $(Y,Z)$ (Definition~\ref{efjwelkfwefoiu2or4234234}) on $X$ the square
 $$\xymatrix{E(Y\cap Z)\ar[r]\ar[d]&E(Y)\ar[d]\\ E(Z)\ar[r]&E(X)}$$ is cocartesian.
\end{enumerate}
\end{kor}

\begin{proof}
We use
  Corollary \ref{qwrfglkjqodfewewfdewdqeqwedew1} in order to transfer the motivic relations relations between values of $E$.
  
 The Assertion  \ref{wefifjweiofioww823232434} now follows from Lemma \ref{ewfifjewiofoiewfewfewfef}.
The Assertion   \ref{fijwoeifjoewifewfewfewf} follows from   Corollary  \ref{liwejfweioiewfwefewfewf1}.
The Assertion\ref{fwlfjwefeoifuwefwf3244} is a consequence of Lemma \ref{lgjrgrogijreogregregregegegrg}. And finally, the Assertion \ref{rglkeglregregporeigroepgergregregg} follows from
 Lemma \ref{leqwjfefeowewfew23r}.
\end{proof}

\section{Merging coarse and uniform structures}\label{efwljkfowpfjowefwefewfwef}

If a bornological coarse space has an additional compatible uniform structure and a big family of subsets, then we can define a new bornological coarse structure called the hybrid structure (introduced by Wright \cite[Sec.~5]{nw1}). This new structure will allow us to coarsely decompose simplicial complexes into cells and therefore to perform inductive arguments by dimension. Such arguments are at the core of proofs of the coarse Baum--Connes conjecture \cite{nw1}, but also an important ingredient in proofs of the Farell--Jones conjecture in various cases \cite{blr,Bartels:2011fk}. 
 
 In Section \ref{rgoihrwiui23orr234r} we introduce the hybrid coarse structure. Our main technical results are the decomposition theorem and the homotopy theorem shown in Sections \ref{woihfiow356} and \ref{efoiweiof34t3455}.
 These results will be applied in Section \ref{fiwehjfi892938u9r23r23rr} in order to decompose the motives of simplicial complexes.
In Section \ref{ekfjweofjweoifewfwefwefewf} we provide sufficient conditions for the flasqueness of hybrid spaces which will be used to study  coarsening spaces in Section \ref{rgpojeroi34t35546456}.

In the subsequent papers \cite{ass} and \cite{equicoarse} we formalize bornological coarse spaces with additional uniform structures
by introducing the category of uniform bornological coarse spaces $\mathbf{UBC}$. 
 
 \subsection{The hybrid structure}\label{rgoihrwiui23orr234r}

 We start with recalling the notion of a uniform structure.
 
 Let $X$ be a set.
 \begin{ddd} A uniform structure\index{uniform!structure} on $X$ is a 
  subset $\cT$\index{$\cT$|see{uniform structure}} of $\cP(X\times X)$  satisfying  following conditions 
   \begin{enumerate}
   \item For every $W$ in $\cT$ we have $\diag_{X}\subseteq W$.
   \item $  \cT$ is closed under taking supersets.     
   \item $\cT$ is closed under taking inverses.
 \item $\cT$ is closed under finite intersections.
   \item For every $W$ in $\cT$ there exists $W^{\prime}$ in $\cT$ such that $W^{\prime}\circ W^{\prime}\subseteq W$.
\end{enumerate}
\end{ddd}
The elements of $\cT$ will be called uniform entourages\index{uniform!entourages}. We will occasionally use the term coarse entourage for elements of a coarse structure on $X$ in order to clearly distinguish them from uniform entourages.
 
A uniform space is a pair     $(X,\cT)$
of a set $X$ with a uniform structure $\cT$\index{$(X,\cT)$}\index{uniform!space}.

Let $\cT$ be a uniform structure on a set $X$.
\begin{ddd}
A uniform structure $\cT$ is Hausdorff\index{Hausdorff!uniform structure} if $\bigcap_{W\in \cT} W=\diag(X)$.
\end{ddd}

\begin{ex}\label{efwwelfewfop}
If $(X,d)$ is a metric space, \index{uniform!structure!from a metric} then we can define a uniform structure $\cT_{d}$.
It is the smallest uniform structure {containing} the subsets $U_{r}$ for $r$ in $(0,\infty)$  given by \eqref{ffwefewfe24543234wfefewef}. The uniform structure $\cT_{d}$ is Hausdorff.
\hB
\end{ex}

Let $X$ be a set with a uniform structure $\cT$ and a coarse structure $\cC$. 
 \begin{ddd}\label{defn:sd23d9023}
 We say that  $\cT$  and $\cC$ are compatible\index{compatible!uniform and coarse structure}  if $\cT\cap \cC\not=\emptyset$. 
\end{ddd} 

In other words, a uniform and a coarse structure are compatible if there exist a controlled uniform entourage.

\begin{ex}\label{ex:jnksdf34} 
If $(X,d)$ is a metric space, then 
 the coarse structure $\cC_{d}$ introduced in Example \ref{welifjwelife89u32or2}  and the uniform structure $\cT_{d}$ {from} Example \ref{efwwelfewfop} are compatible.
\hB
\end{ex}

\begin{ex}
Let $X$ be a set. The minimal coarse structure  $\cC_{min} $ on $X$ is only compatible with  the discrete uniform structure\index{uniform!structure!discrete}\index{discrete!uniform structure} $\cT_{\disc}=\cP(X\times X)$. On the other hand, the maximal coarse structure $\cC_{max} $ is compatible with every uniform structure on $X$.
\hB
\end{ex}

Let $(X,\cC,\cB)$ be a bornological coarse space equipped with an additional compatible uniform structure $\cT$ and with a big family $\cY=(Y_{i})_{i\in I}$. We consider $\cT$ as a filtered poset using the opposite of the inclusion relation
\[U\le U^{\prime} \::=\: U^{\prime}\subseteq U\,.\]

Let $I,J$ be partially ordered sets, and let $\phi:I\to J$ be a map between the underlying sets.
\begin{ddd}
The map $\phi$ is called a function\index{function between posets}  if it is order preserving.
It {is} cofinal\index{cofinal!function between posets}, if in addition for every $j$ in $J$ there exists $i$ in $I$ such that $j\le \phi(i)$.
\end{ddd}
 
\begin{ex}
{We} assume that the uniform structure $\cT_{d}$ is induced by a metric $d$ (see Example~\ref{efwwelfewfop}). For a positive real number  $r$  we can consider the uniform entourages $U_{r}$ as defined in \eqref{ffwefewfe24543234wfefewef}. We equip $(0,\infty)$ with the usual order and let $(0,\infty)^{\op}$ denote the same set with the opposite order.
Consider a function $\phi:I\to (0,\infty)^{\op}$. Then   $I\ni i\mapsto U_{\phi(i)}\in \cT_{d}$ is cofinal if and only if $\lim_{i\in I} \phi(i)=0$.
\hB
\end{ex}

We consider $\cP(X\times X)$ as a poset with the inclusion relation and we consider a function
$\phi:I\to \cP(X\times X)^{\op}$, where $I$ is some poset. 
\begin{ddd}
The function $\phi$ is $\cT$-admissible\index{$\cT$-admissible} if for every $U$ in $\cT$ there exists $i$ in $I$ such that $\phi(i)\subseteq U$.
\end{ddd}
Note that we do not require that $\phi$ takes values in $\cT$.

Let   $\cY=(Y_{i})_{i\in I}$ be a filtered family of subsets of a set $X$. For a function  $\phi:I\to  \cP(X\times X)^{\op}$ we consider the following subset of $X\times X$:
\begin{equation}\label{dqdqw2345r2t2444t3t4t}
U_{\phi}:= \{(x,x')\in X\times X :  (\forall i\in I :  x,x'\in Y_{i} \text{ or } (x,x')\in \phi(i)\setminus ( Y_{i}\times  Y_{i}))\}\, .
\end{equation}
Let us spell {this} out in words: a pair  $(x,x') $ belongs to $U_{\phi}$ if and only if for every $i$ in $I$  both entries $x,x'$ belong to $Y_{i}$ or, if not, the pair $(x,x')$ belongs to the subset $\phi(i)$ of $X\times X$.

 Let $(X,\cC,\cB)$ be a bornological coarse space equipped with a  compatible uniform structure $\cT$ and  a big family $\cY=(Y_{i})_{i\in I}$.

\begin{ddd}\label{defn:sdfuh8934t}
We define the hybrid coarse structure\index{coarse structure!hybrid}\index{hybrid coarse structure}  on $X$ by\index{$\cC_{h}$|see{hybrid coarse structure}}
\[\cC_{h}:=\cC\langle\{U_{\phi}\cap V :  V\in \cC \text{ and } \phi:I\to \cP(X\times X)^{\op} \text{ is $\cT$-admissible}\}\rangle\,.\]
and set
\[X_{h}:=(X,\cC_{h},\cB)\,.\]
\end{ddd}

By definition we have the inclusion $\cC_{h}\subseteq \cC$. This implies in particular that $\cC_{h}$ is compatible with the bornological structure $\cB$. The triple $(X,\cC_{h},\cB)$ is thus a bornological coarse space. Furthermore, the identity of the
underlying sets is a morphism of bornological coarse spaces
\[(X,\cC_{h},\cB)\to (X,\cC,\cB)\,.\]
Note that the hybrid coarse structure depends on the  original coarse structure $\cC$, the uniform structure $\cT$, and the big family $\cY$.


\begin{rem}\label{fwoejweofewfewfwefwefwe34345345345345}
If the uniform structure $\cT$ is determined by a metric and we have $I\cong \nat$ as partially ordered sets, then we can describe the hybrid structure  $\cC_{h}$ as follows:

\emph{A subset $U$ of $X\times X$  belongs to $\cC_{h}$ if $U$ in $\cC$ and
if   for every $\epsilon$ in $(0,\infty)$ there exists an $i$ in $I$ such that
$U\subseteq (Y_{i}\times Y_{i})\cup U_{\epsilon}$. }

It is clear that any hybrid entourage\index{entourage!hybrid} satisfies this condition.
In the other direction, it is easy to check that a subset $U$ satisfying the condition above is contained in $U\cap U_{\phi}$ for an appropriate $\phi$. In order to construct such a $\phi$ one uses the cofinal family $(U_{1/n})_{n\in \nat\setminus \{0\}}$ in $\cT$ and that $I=\nat$ is well-ordered.
\hB
\end{rem}

\begin{ex}\label{ex:df8734r}
Let $(X,\cC,\cB)$ be a bornological coarse space equipped with an additional uniform structure $\cT$. 
Then we can define the hybrid structure associated to $\cC$, $\cT$ and the big family $\cB$ (see Example \ref{4oj3iotiu43otu43t4390t3t430t}). This hybrid structure is called the $\cC_{0}$-structure\index{$\cC_{0}$-structure} by Wright \cite[Def.~2.2]{nw1}.
  \hB
\end{ex}

We now consider  functoriality of the hybrid structure.

Let $(X,\cT)$ and $(X^{\prime},\cT^{\prime})$ be uniform spaces, and let $f:X\to X^{\prime}$ be a map of sets.
\begin{ddd}\label{lfjwfefoewfefwefewf}
The map $f$ is uniformly continuous\index{uniformly continuous}
if for every $V^{\prime}$ in $\cT^{\prime}$ there exists $V$ in $\cT$ such that $(f\times f)(V)\subseteq V^{\prime}$.
\end{ddd}


Let $X$, $X^{\prime}$ be bornological coarse spaces which come equipped with uniform structures~$\cT$ and $\cT^{\prime}$ and big families $\cY=(Y_{i})_{i\in I}$ and $\cY^{\prime}=(Y^{\prime}_{i^{\prime}})_{i^{\prime}\in I^{\prime}}$. Let $f  :  X\to X^{\prime}$ be a morphism of bornological coarse spaces.
In the following we introduce the condition of being compatible in order to express that $f$ behaves well with respect to the additional structures. The main motivation for imposing the conditions is to ensure Lemma \ref{oepwfweoifewofiewofewf} below.

\begin{ddd}\label{ejwiofeoifefewfewf}
The morphism $f$ is compatible\index{compatible!morphism}\index{morphisms!compatible} if the following conditions are satisfied: \begin{enumerate}
\item $f$ is uniformly continuous. 
\item For every $i$ in  $I$ there exists $i^{\prime}$ in $I^{\prime}$ such that $ f(Y_{i})\subseteq Y^{\prime}_{i^{\prime}}$.
\end{enumerate}
\end{ddd}

\begin{lem}\label{oepwfweoifewofiewofewf}
If $f$ is compatible, then it is a morphism  $f :  X_{h}\to X_{h}^{\prime}$.
\end{lem}

\begin{proof}
It is clear that $f$ is proper since we have not changed the bornological structures.

We must check that $f$ is controlled with respect to the hybrid structures.

Let $\phi:I\to \cP(X\times X)^{\op}$ be a $\cT$-admissible function, and let $V$ be a coarse entourage of $X$. We then consider the entourage
$U_{\phi}\cap V$ of $X_{h}$.
We have
$$(f\times f)(U_{\phi}\cap V)\subseteq (f\times f)(U_{\phi})\cap (f\times f)(V)\, .$$ Since
$(f\times f)(V)$ is an entourage of $X^{\prime}$ it suffices to show that there exists
a suitable $\cT^{\prime}$-admissible function $\phi^{\prime}:I^{\prime}\to \cP(X^{\prime}\times X^{\prime})^{\op}$ such that we have
$(f\times f)(U_{\phi})\subseteq U_{\phi^{\prime}}$.
%
Since $f$ is compatible   
we can choose a map of sets
$\lambda:I\to I^{\prime}$ such that
$Y_{i}\subseteq Y^{\prime}_{\lambda(i)}$ for all $i$ in $I$.
We now define a map of sets
$\tilde \phi^{\prime}:I^{\prime}\to  {\cP(X^{\prime}\times X^{\prime})}$
by $$\tilde \phi^{\prime}(i^{\prime}):= \bigcap_{i\in \lambda^{-1}(i^{\prime})}(f\times f)( \phi(i))\, .$$
Note that the empty intersection of subsets of $X^{\prime}$ is by convention equal to $X^{\prime}\times X^{\prime}$.
Then we construct the order preserving function
$\phi^{\prime}: I^{\prime}\to {\cP(X^{\prime}\times X^{\prime})}^{\op}$ by
 $$\phi^{\prime}(i^{\prime}):=\bigcap_{j^{\prime}\le i^{\prime}} \tilde \phi^{\prime}(j^{\prime})\, .$$

We claim that $\phi^{\prime}$ is $\cT^{\prime}$-admissible. Let $W^{\prime}$ in $\cT^{\prime}$ be given. Since $f$ is uniformly continuous we can choose $W$ in $\cT$ such that $(f\times f)(W)\subseteq W^{\prime}$. Since  $\phi$ is $\cT$-admissible we can choose $i$ in $I$ such that $\phi(i)\subseteq W$. Then
 $\phi^{\prime}(\lambda(i))\subseteq W^{\prime}$ by construction.


Assume that $(x,y)$ is in $U_{\phi}$.  
We must show that $(f(x),f(y))\in U_{\phi^{\prime}}$.
We consider   $i^{\prime}$ in $I^{\prime}$ and have the following two cases:
\begin{enumerate}
\item $(f(x),f(y))\in Y^{\prime}_{i^{\prime}}\times Y^{\prime}_{i^{\prime}}$.
\item $(f(x),f(y))\not\in Y^{\prime}_{i^{\prime}}\times Y^{\prime}_{i^{\prime}}$. In this case we must show that 
$(f(x),f(y))\in \phi^{\prime}(i^{\prime})$. Let $j^{\prime}$ in $I^{\prime}$ be any element such that $j^{\prime}\le i^{\prime}$. 
Then we must show that $(f(x),f(y))\in \tilde \phi^{\prime}(j^{\prime})$. {We have two subcases:}
\begin{enumerate} 
\item $j^{\prime}$ is in the image  of  $\lambda$. For any choice  of a preimage
  $j$ in $\lambda^{-1}(j^{\prime})$ we have  $(x,y)\not\in Y_{j}\times Y_{j}$ since otherwise  $(f(x),f(y))\in Y^{\prime}_{j^{\prime}}\times Y^{\prime}_{j^{\prime}}\subseteq Y^{\prime}_{i^{\prime}}\times Y^{\prime}_{i^{\prime}}$. Hence for all 
  $j$ in $\lambda^{-1}(j^{\prime})$ we have
$(x,y)\in \phi(j)$. Then $(f(x),f(y))\in \tilde \phi^{\prime}(j^{\prime})$ as required.
 \item $j^{\prime}$ is not in the image of $\lambda$. Then $(f(x),f(y))\in \tilde \phi^{\prime}(j^{\prime})=X^{\prime}\times X^{\prime}$.
\end{enumerate}
\end{enumerate}{All the above shows that $f$ is controlled, which finishes this proof.}
%
\end{proof}

\begin{ex}\label{ifuweiofeowfew982342343434}
A typical example of a bornological coarse space with a $\cC_{0}$-structure\index{$\cC_{0}$-structure} is the cone $\cO(Y)$\index{$\cO(-)$|see{cone}}\index{cone} over a uniform space $(Y,\cT_{Y})$. We consider $Y$ as a bornological coarse space with the maximal structures. We then  form the bornological coarse space 
$$ ([0,\infty)\times Y,\cC,\cB):=[0,\infty) \otimes  Y\, ,$$ see Example \ref{eiofweoifwefuewfieuwf9wwfwef}.  The metric   {induces a} uniform structure {$\cT_d$  on $[0,\infty)$, and we form the} product uniform structure\index{product!uniform structure}
$\cT:=\cT_{d}\times \cT_{Y}$ on $[0,\infty)\times Y$.   The cone over $(Y,\cT_{Y})$ is  now defined as
\[\cO(Y):=([0,\infty) \otimes  Y,\cC_{0},\cB)\,,\]
where one should not overlook that we use the $\cC_0$-structure.
 
Note that instead of all bounded subsets we could also take the big family
$$\cY:=([0,n]\times Y)_{n\in \nat}\, .$$ The resulting hybrid structure is equal to the $\cC_{0}$-structure since  $\cY$ is cofinal in $\cB$.
  
  Note that $ [0,\infty) \otimes  Y$ is flasque by Example~\ref{feijweifewfoiuwe9fuewfewfwf}, but the cone $\cO(Y)$
  it is not flasque in general  since it has the smaller $\cC_0$-structure.
  
Let $\bU$\index{$\bU$|see{uniform space}} denote the category of uniform spaces\index{uniform!space} and uniformly continuous maps. It follows from Lemma \ref{oepwfweoifewofiewofewf}
  that the cone construction determines a functor
  $$\cO:\bU \to \BC$$
which is of great importance in coarse algebraic topology, see \cite{ass} and \cite{equicoarse}.
\hB
\end{ex}

 Let $(X,\cC,\cB)$ be a bornological coarse space equipped with a  compatible uniform structure $\cT$ and  a big family $\cY=(Y_{i})_{i\in I}$. We consider furthermore a subset $Z$ of $X$. It has an induced bornological coarse structure $(\cC_{|Z},\cB_{|Z})$, an  induced uniform structure $\cT_{|Z}$, and  a big family $Z\cap \cY:=(Z\cap Y_{i})_{i\in I}$. We let $(\cC_{|Z})_{h}$ denote the hybrid coarse structure on $Z$  defined by these induced structures. Let furthermore $(\cC_{h})_{|Z}$ be the coarse structure  on $Z$ induced from the hybrid coarse structure $\cC_{h}$ on $X$.
 
 \begin{lem}\label{folrje099045u45z4}
  We have an equality $ (\cC_{h})_{|Z}= (\cC_{|Z})_{h}$.
\end{lem}

\begin{proof}
It is clear that $(\cC_{h})_{|Z}\subseteq (\cC_{|Z})_{h}$.

Let now $\phi:I\to \cP(Z\times Z)^{\op}$ be $ \cT_{|Z}$-admissible, and let $V$ in $\cC_{|Z}$ and $U_{\phi}$ be as in \eqref{dqdqw2345r2t2444t3t4t}.
We must show that $U_{\phi}\cap V\in (\cC_{h})_{|Z}$. Let $\psi$ be the composition of $\phi$ with the inclusion $\cP(Z\times Z)^{\op}\to \cP(X\times X)^{\op}$. Then $\psi$ is $\cT$-admissible  and
$$U_{\phi}\cap V=(U_{\psi}\cap V)\cap (Z\times Z)\, ,$$ hence
 $U_{\phi}\cap V\in (\cC_{h})_{|Z}$.
\end{proof}

\begin{rem}In this proof it is important to allow that $\cT$-admissible functions can have values which do not belong to 
$\cT$. In fact, the inclusion $\cP(Z\times Z)\to \cP(X\times X)$ in general does not map $\cT_{|Z}$
to $\cT$.\hB
\end{rem}

\subsection{Decomposition theorem}\label{woihfiow356}
  
In this subsection we show the decomposition theorem.  It roughly states that a uniform decomposition 
of a space with a hybrid structure  associated to a family $\cY$ induces a decomposition of its  motivic  coarse spectrum relative to this family.
The precise formulation is Theorem \ref{weiofjewi98u3298r32r32rr}. As a consequence we deduce an excisiveness result for the cone functor.

\subsubsection{Uniform decompositions and statement of the theorem}

We consider a uniform space $(X,\cT)$. For a uniform entourage  $U$ in $\cT$ we write $\cT_{\subseteq U}$ (or $\cP(X\times X)^{\op}_{\subseteq U}$, respectively) for the partially ordered subset
of uniform entourages (or subsets, respectively) which are contained in $U$.

%
%
%
%
%
 
Let $(Y,Z)$ be a pair of two subsets of $X$.
\begin{ddd}\label{fewoifweifoew2324234234}
The pair $(Y,Z)$  is called a uniform decomposition\index{uniform!decomposition}\index{decomposition!uniform} if:
\begin{enumerate}
\item $X=Y\cup Z$.
\item There exists $V$ in $\cT$ and a function $s:\cP(X\times X)^{\op}_{\subseteq V}\to \cP(X\times X)^{\op}$
such that:
\begin{enumerate}
\item The restriction $s_{|\cT_{\subseteq V}}:\cT_{\subseteq V}\to  \cP(X\times X)^{\op}$ is  $\cT$-admissible.
\item For every $W$ in  $\cP(X\times X)_{\subseteq V}^{\op}$  we have  the relation
\begin{equation}\label{eqwe98798dv}
W[Y]\cap W[Z]\subseteq s(W)[Y\cap Z]\, .
\end{equation}
\end{enumerate}
\end{enumerate}
\end{ddd}

\begin{rem}
Note that the condition in Definition \ref{fewoifweifoew2324234234} is apparently stronger than the condition (compare with Definition \ref{efjwelkfwefoiu2or4234234} of a coarsely excisive decomposition)
that for every $W$ in $\cP(X\times X)_{\subseteq V}^{\op}$ there exists  $W'$ in $\cP(X\times X) $ such that
$W[Y]\cap W[Z]\subseteq W'[Y\cap Z]$. 
First of all $W'$ must become small if $W$ becomes small. And even stronger, the choice of $W'$ for given $W$ must be expressible by a function as stated.
 \hB
\end{rem}

\begin{ex}\label{ifweoifewoufuew98fu9wefweff}
Assume that $(X,d)$ metric space such that the restriction of $d$ to every component of $X$ is a path-metric. We further assume    that there exists a constant $c>0$ such that  the distance between every two components of $X$ is  greater than $ c$. We consider a pair
  $(Y,Z)$  of closed subsets such that
$Y\cup Z=X$. If we equip $X$ with the uniform structure $\cT_{d}$ induced by the metric, then $(Y,Z)$ is a uniform decomposition. 

Here is the argument:

For every $r$ in $[0,\infty)$ we define the subset  $$\bar U_{r}:=\{(x,y)\in X\times X: d(x,y)\le r \}$$ of $X\times X$. Note that $\bar U_{r}$ is a uniform entourage provided $r$ is in $r>0$.
We define the uniform entourage $$V:=\bar U_{c}\cap \bar U_{1}\, ,$$
{where $c$ is as above.}
We consider the subset  $$Q:= \{0\}\cup \{n^{-1}: n\in \nat\ \&\  n\not=0\}$$ of $[0,1]$. 
 For $W\in \cP(X\times X)^{\op}_{\subseteq V}$ we let $d(W)$ in $Q$ be the minimal element such that
 $W\subseteq \bar U_{d(W)}$. Note that $d(W)$ exists, and if $W$ in $\cT_{V}$ is such that $W^{\prime}\le  W$ (i.e., $W\subseteq W^{\prime}$) then
 $d(W)\le d(W^{\prime})$.
 We then define the function
 $$s:\cP(X\times X)^{\op}_{\subseteq V} \to \cP(X\times X)^{\op}\, , \quad  s(W):=\bar U_{d(W)}\, .$$
 Since 
$s(\bar U_{r})=\bar U_{r}$ for $r$ in $Q$
this function is $\cT_{d}$-admissible.

Let us check that this function $s$ satisfies \eqref{eqwe98798dv}. We consider an element $x$ of $ W[Y]\cap W[Z]$. Since $Y \cup Z = X$ we have $x \in Y$ or $x \in Z$. Without loss of generality we can assume that $x \in Y$. Then there exists $z$  in $Z$ such that $(x,z) \in W$. We choose a path $\gamma$ from $x$ to $z$ realizing the distance from $x$ to $z$ (note that $W \subseteq V \subseteq \bar U_{c}$ and therefore $x$ and $z$ are in the same path component of $X$). Let now $x^\prime$ be the first point on the path $\gamma$ with $x^\prime \in Y \cap Z$. Such a point exists since $Y$ and $Z$ are both closed, $x \in Y$ and $z \in Z$. Then $(x,x^\prime) \in s(W)$, which finishes the proof of \eqref{eqwe98798dv}.
\hB
\end{ex}

 We consider a bornological coarse space $(X,\cC,\cB)$ with an additional compatible uniform structure  $\cT $, big family $\cY=(Y_{i})_{i\in I} $, and uniform decomposition  $Y,Z$.  We let $X_{h}$ denote the associated bornological coarse space with the hybrid structure. Moreover, we let $Y_{h}$\footnote{This should not be confused with the notation  $Y_{i}$ for the members of the big family $\cY$.} and $Z_{h}$ denote the subsets $Y$ and $Z$ equipped with the bornological coarse structure induced from $X_{h}$. By Lemma~\ref{folrje099045u45z4} this induced structure coincides with the hybrid structure defined by the induced uniform structures and big families on these subsets.

 Recall the pair notation \eqref{fkhwiufuiz823zr824234242424234234234234}.
\begin{theorem}[Decomposition theorem]\label{weiofjewi98u3298r32r32rr}
\index{Decomposition Theorem}We assume that $I=\nat$ and that the uniform structure on $X$ is Hausdorff. Then the following square in $\Sp\cX$ is cocartesian
\begin{equation}\label{dqwdqd3kjnr23kjrn23r32r32}
\xymatrix{((Y\cap Z)_{h},  Y\cap Z\cap \cY)\ar[r]\ar[d]&(Y_{h},Y\cap\cY )\ar[d]\\(Z_{h},Z\cap \cY )\ar[r]&(X_{h},\cY)}
\end{equation}
\end{theorem}

The proof of the theorem will be given in the next section.

\begin{rem}
The assumption in the decomposition theorem that $I=\nat$ seems to be technical. It covers our main applications  to cones or the coarsening spaces. Therefore we have not put much effort into a generalization.
\hB
\end{rem}

\subsubsection{Proof of the decomposition theorem}

We retain the notation from the statement of the theorem.

In order to simplify the notation we omit the subscript $h$ at $Y$ and $Z$. All subsets of $X$ occuring in the following argument are considered with the induced bornological coarse structure from $X_{h}$.

In general we use the symbol $I$ for the index set, but at the place where the assumption $I=\nat$ is relevant, we use $\nat$ instead.

For every $i$ in $I$ we have a push-out square (compare with \eqref{hoijiojdio32ddd23ddd} and see also \eqref{erkjnkjnvkfvfvfvdscac} for notation)
$$\xymatrix{\Yo^{s}(\{Y\cup Y_{i}\}\cap (Z\cup Y_{i}) ) \ar[r]\ar[d]&\Yo^{s}(\{Y\cup Y_{i}\})\ar[d]\\\Yo^{s}(Z\cup Y_{i})\ar[r]&\Yo^{s}(X_{h})}$$
associated to the complementary pair $(Z\cup Y_{i},\{Y\cup Y_{i}\})$.
Taking the cofibre of the obvious map from the push-out square $$\xymatrix{\Yo^{s}(Y_{i})\ar[r]\ar[d]&\Yo^{s}(Y_{i})\ar[d]\\\Yo^{s}(Y_{i})\ar[r]&\Yo^{s}(Y_{i})}$$
we get the push-out square (with self-explaining pair notation)
$$\xymatrix{ (\{Y\cup Y_{i}\}\cap (Z\cup Y_{i}) ,Y_{i}) \ar[r]\ar[d]&\ (\{Y\cup Y_{i}\},Y_{i})\ar[d]\\ (Z\cup Y_{i},Y_{i})\ar[r]& (X_{h},Y_{i})}$$
  We now form the colimit over $i$ in $I$ and get a push-out square  $$\xymatrix{ (\{Y\cup \cY\}\cap  (Z\cup \cY) ,\cY) \ar[r] \ar[d] & (\{Y\cup \cY\},\cY)\ar[d] \\ ( Z\cup \cY ,\cY)\ar[r]& (X_{h},\cY)}$$
  Note that the combination of $\{-\}$ and $\cY$ indicates a double colimit.
In the following we show that this square is equivalent to the square \eqref{dqwdqd3kjnr23kjrn23r32r32} by identifying the corresponding corners.
In Remark~\ref{jhfjkiuewfweew23424} below we will verify the  equivalence
\begin{equation}\label{eqio24rfed}
(Z\cup \cY,\cY)\simeq (Z,Z\cap \cY)\, .
\end{equation}
Using Corollary~\ref{fijweiofjwefoejufoeifueofeffff} and this remark again (for $Y$ in place of $Z$) we get the equivalence
$$ (\{Y\cup \cY\},\cY)\simeq  ( Y\cup \cY ,\cY)\simeq (Y,Y\cap \cY)\, .$$
 The core of the argument is to observe the first equivalence in the chain
$$(\{Y\cup \cY\}\cap  (Z\cup \cY),\cY)\simeq (\{Y\cap Z\}\cap Z\cup   \cY ,\cY)\simeq ((Y\cap Z)\cup \cY,\cY)\simeq (Y\cap Z,Y\cap Z\cap \cY)\, ,$$
where the remaining equivalences are  other instances of the cases considered above. 
We must show that
\begin{equation}\label{dqwdq31241234213}
\Yo^{s}(\{Y\cup \cY\}\cap  (Z\cup \cY))\simeq \Yo^{s}( \{Y\cap Z\}\cap Z\cup   \cY)\, . 
\end{equation} 
Note  that both sides are filtered colimits over families  {of} motives of subsets of $X_{h}$. We will show that
these families of subsets  dominate each other  in a cofinal way. This then implies the equivalence of the colimits.

One direction is easy. Let $V$ be a coarse entourage of $X_{h}$ containing the diagonal and let $i$ be in $I$. Then clearly
$$V[Y\cap Z]\cap Z\cup Y_{i}\subseteq V[Y\cup Y_{i}]\cap (Z\cup Y_{i})\, .$$
The other direction is non-trivial.

 From now on we write $\nat$ instead of $I$ and choose $i$ in $\nat$.
We furthermore consider a coarse entourage $U$ of $X$ and  a $\cT$-admissible function
$$\phi:\nat\to \cP(X\times X)^{\op}\, .$$
This data defines a coarse entourage $$V:=U\cap U_{\phi}$$ of $X_{h}$.  
We must show that there exists $j$ in $I$ and an entourage $W$ of $X_{h}$ such that
$$V[Y\cup Y_{i}]\cap (Z\cup Y_{i})\subseteq W[Y\cap Z]\cap Z\cup Y_{j}\ .$$

The uniformity of the decomposition $(Y,Z)$ provides a uniform entourage $C$ in $\cT$ (called $V$ in Definition \ref{fewoifweifoew2324234234})  and a $\cT$-admissible function
$$s:\cP(X\times X)^{\op}_{\subseteq C}\to \cP(X\times X)^{\op}\, .$$  
Since the uniform and coarse structures of $X$ are compatible  we can choose  an entourage $E$
which is uniform and coarse. Using that $s$ is $\cT$-admissible   we can assume (after decreasing $C$ if necessary) that $s$ takes values in $\cP(X\times X)_{\subseteq E}^{\op}$, so in particular in coarse entourages. Without loss of generality we can furthermore assume that $\diag(X)\subseteq s(W)$ for every $W$ in $\cP(X\times X)^{\op}_{\subseteq C}$ (otherwise replace $s$ by $s\cup \diag(X)$).

We define the function $$\phi_{C}:\nat \to \cP(X\times X)^{\op}_{\subseteq C}\, ,\quad  \phi_{C}(i):=(C\cap \phi(i))\cup \diag(X)\, .$$
We furthermore define the maps $$  k^{\prime}:\nat\to \nat\cup\{\infty
\}\, , \quad  \psi:\nat \to \cP(X\times X)^{\op} $$ as follows. For $\ell\in \nat$  we consider the set 
$$N(\ell):=\big\{k^{\prime}\in \nat\::\: \big(\forall k\in \nat\:|\: k\le k^{\prime}\Rightarrow Y_{k}\times s(\phi_{C}(k-1))^{-1}[Y_{k}]\subseteq Y_{\ell}\times Y_{\ell}\big)\big\}\, .$$
\begin{enumerate}
\item If $N(\ell)$ is not empty and bounded, then we set  $$k^{\prime}(\ell):=\max N(\ell)\, , \quad \psi(\ell):=s(\phi_{C}(k^{\prime}(\ell)))\, .$$
\item If $N(\ell)$ is unbounded,  then we set
$$ k^{\prime}(\ell):=\infty\, , \quad 
 \psi(\ell):=\diag(X)\, .$$
 \item If $N(\ell)$ is empty   (this can only happen for small $\ell$), then we set $$k^{\prime} (\ell):=0\, , \quad \psi(\ell):=X\times X\, .$$ 
\end{enumerate}

We show that $\psi$ is a $\cT$-admissible function. 

We  first observe that $k'$ is monotonously increasing since the family $(Y_{\ell})_{\ell\in I}$ is so.
Since $s$ and $\phi_{C}$ are functions,  $\psi$ is also a function.

Let $Q$ in $\cT$ be given. We must show that there exists $\ell$ in $\nat$ such that $\psi(\ell)\subseteq Q$.
  Since $s_{|\cT_{\subseteq C}}$ is $\cT$-admissible we can find $W$ in $\cT_{\subseteq C}$ such that
$s(W)\subseteq Q$.  Since $\phi_{C}$ is $\cT$-admissible we can chose $i$ in $\nat $ such that $\phi_{C}(i)\subseteq W$.
Using that $\cY$ is big and that $s$ is bounded by $E$ we see that $\lim_{\ell\to \infty} k'(\ell)=\infty$. Hence we can
 choose $\ell$ in $\nat $ so large that  $k^{\prime}(\ell)\ge i$. 
 Then
$\psi(\ell)\subseteq Q$.  This finishes the proof that $\psi$ is $\cT$-admissible.

We now fix $i'$ in $\nat $ so large that $i\le i'$ and $\phi(i')\subseteq C$. This is possible by the $\cT$-admissibility of $\phi$. We then set
  $$W:=(\psi(i')\cap U_{\psi})\cup\diag(X)\in \cC_{h} \, .$$ Using that $\cY$ is big we further choose $j$ in $\nat $ such that $i'\le j$ and $U[Y_{i'}]\subseteq Y_{j}$.

The following chain of inclusions finishes the argument:
$$V[Y\cup Y_{i}]\cap (Z\cup Y_{i})\subseteq  V[Y\cup Y_{i'}]\cap (Z\cup Y_{i'}) \subseteq W[Y\cap Z]\cap Z\cup Y_{j}\, .$$
The first inclusion is clear since $i\le i'$ and therefore $Y_{i}\subseteq Y_{i'}$. It remains to show the second inclusion.

We show the claim by the following case-by-case discussion. Let $x$ be in $V[Y\cup Y_{i'}]\cap (Z\cup Y_{i'})$.
\begin{enumerate}
\item If $x\in Y_{i'}$, then $x\in Y_{j}$ since $Y_{i'}\subseteq U[Y_{i'}]\subseteq Y_{j}$ by the choice of $j$.
\item Assume now that $x\in Z\setminus Y_{i'}$.
Then there exists $b$ in $Y\cup Y_{i'}$ such that $(x,b)\in V$.
\begin{enumerate}
\item If $b\in Y_{i}$, then $x\in U[Y_{i'}]\subseteq Y_{j}$.
\item Otherwise $b\in Y\setminus Y_{i}$.  In view of the definition \eqref{dqdqw2345r2t2444t3t4t}
of $U_{\phi}$ we then have $(x,b)\in \phi(i')$. By our choice of $i'$ this also implies $(x,b)\in \phi_{C}(i')$. 
We distinguish two cases:   
\begin{enumerate}
\item There exists $k$ in $\nat$ such that $(x,b)\in Y_{k}\times Y_{k}$. We can assume that $k$ is minimal with this property. Then $(x,b)\in \phi_{C}(k-1)$, where we set $\phi_{C}(-1):=C$. 
Consequently,  $x\in \phi_{C}(k-1)[Y]\cap Z$, and by the defining property of $s$ there exists $c$ in $Y\cap Z$ such that $(x,c)\in s(\phi_{C}(k-1))$.
For every $k^{\prime}$ in $\nat$ we then have exactly one the following  two cases:
\begin{enumerate}
\item $k^{\prime}\le k-1$: Then $(x,c)\in s(\phi_{C}(k^{\prime}))$.
\item $k^{\prime}\ge k$: Then $(x,c)\in Y_{k}\times s(\phi_{C}(k-1))^{-1}[Y_{k}]$.
\end{enumerate}
Let now $\ell$ in $\nat$ be given. We let $k^{\prime}:=k^{\prime}(\ell)$ be as in the definition of $\psi$.
If $(x,c)\not\in \psi(\ell)=s(\phi_{C}(k^{\prime}))$, then $k^{\prime}\ge k$ and we have case B:
$$(x,c)\in    Y_{k}\times s(\phi_{C}(k-1))^{-1}[Y_{k}]\subseteq Y_{\ell}\times Y_{\ell} \, ,$$ where the last inclusion holds true by the definition of $k^{\prime}(\ell)$.
This implies that $x\in W[Y\cap Z]\cap Z$.
\item For every $k$ in $\nat$ we have $(x,b)\not\in Y_{k}\times Y_{k}$.
Then we have
$$(x,b)\in \bigcap_{k\in \nat} \phi(k)=\diag(X)$$ (the last equality holds true since  $\phi$ is $\cT$-adimissible and $\cT$ is Hausdorff by assumption), i.e., $x=b$.
It follows that $x\in Z\cap Y$, hence $x\in W[Y\cap Z]\cap Z$.
\end{enumerate}
\end{enumerate}
\end{enumerate}
The proof is therefore finished.

\begin{rem}\label{jhfjkiuewfweew23424}
We have to verify the equivalence \eqref{eqio24rfed}.

We retain the notation introduced for Theorem \ref{weiofjewi98u3298r32r32rr}. We have the following chain of equivalences:
\begin{alignat*}{2}
(Z & ,Z\cap \cY)\\
&\simeq \colim_{i\in I} (Z,Z\cap Y_{i}) & \quad & \mbox{definition}\\
&\simeq \colim_{i\in I}  (Z,Z\cap \{Y_{i}\}) && \mbox{coarse invariance}\\
&\simeq  \colim_{i\in I}  \colim_{j\in I, j\ge i}(Z\cup Y_{j},   (Z\cup  Y_{j}) \cap \{Y_{i}\})  && \mbox{excision}\\
&\simeq  \colim_{i\in I}  \colim_{j\in I, j\ge i} \colim_{U\in \cC_{h}}(Z\cup Y_{j},   (Z\cup  Y_{j}) \cap U[Y_{i}]) && \mbox{definition}\\
&\simeq \colim_{i\in I}  \colim_{U\in \cC_{h}}  \colim_{j\in I, j\ge i , Y_{j}\supseteq U[Y_{i}]} (Z\cup Y_{j}, U[Y_{i}]) && \mbox{$\cY$ is big and Lemma \ref{leqwjfefeowewfew23r}}\\
&\simeq  \colim_{i\in I}  \colim_{j\in I, j\ge i }(Z\cup Y_{j},     Y_{i} ) && \mbox{coarse invariance, Corollary \ref{kjeflwfjewofewuf98ewuf98u798798234234324324343}.\ref{iweufhf89wfu89ewfew245}} \\
&\simeq \colim_{i\in I} (Z\cup Y_{i},Y_{i}) && \mbox{cofinality} \\
&\simeq (Z\cup \cY,\cY) && \mbox{definition}
\end{alignat*}

This verifies the equivalence \eqref{eqio24rfed}.
\hB
\end{rem}

\subsubsection{Excisiveness of the cone-at-infinity}
\label{ofiweifewufou9823u92343432434}

Recall the cone functor   $\cO\colon \bU \to \BC$   from Example \ref{ifuweiofeowfew982342343434}.

  If $f:(X,\cT_{X})\to (X^{\prime},\cT_{X^{\prime}})$ is a morphism of uniform spaces, then by Lemma~\ref{oepwfweoifewofiewofewf} we get an induced morphism $\cO(f)\colon \cO(X)\to \cO(X^{\prime})$ in $\BC$. We can thus define a functor
\begin{equation}\label{bbgboitgiotuoitutrttttbb}\index{$\cO^{\infty}$}
\cO^{\infty}\colon \bU\to \Sp\cX\, , \quad (X,\cT_{X})\mapsto (\cO(X),\cY)\, ,
\end{equation}
where we use
\begin{equation}\label{ervgiuheiogeerev}
\cY:=([0,n]\times X)_{n\in \nat}
\end{equation}
as the big family.



\begin{kor}\label{wfuhweuifhewiufewf243}
If the uniform structure on $X$ is Hausdorff and  $(Y,Z) $  is a uniform decomposition of $X$, then the 
 square $$\xymatrix{\cO^{\infty}(Y\cap Z)\ar[r]\ar[d]&\cO^{\infty}(Y)\ar[d]\\\cO^{\infty}(Z)\ar[r]&\cO^{\infty}(X)}$$
is cocartesian.
\end{kor}

\begin{proof}
Note that  $([0,\infty)\times Y,[0,\infty)\times Z)$ is a uniform decomposition of $[0,\infty)\times X$ with the product uniform structure (which is again Hausdorff), and the big family  $\cY$ in \eqref{ervgiuheiogeerev} is indexed by $\nat$.
\end{proof}

\subsection{Homotopy theorem}\label{efoiweiof34t3455}

In this subsection we state and prove the homotopy theorem. It asserts that the motivic  coarse  spectrum of a bornological coarse space
with the hybrid structure associated to a uniform structure and big family
is  homotopy invariant.  Here the notion of homotopy is  compatible with the uniform structure.
The idea is to relate this notion of homotopy with the notion of coarse homotopy introduced in Section \ref{kjfwoeff2243324324324}. The details are surprisingly complicated.

\subsubsection{Statement of the theorem}

Let $X$ be a bornological coarse space with an additional compatible uniform structure $\cT$ and a big family $\cY=(Y_{i})_{i\in I}$. We consider  the unit interval $[0,1]$ as a bornological coarse space with the structures induced by the metric (hence with the maximal structures) and the compatible metric uniform structure.  We abbreviate $IX:=[0,1]\ltimes X$, see Example~\ref{hffweiufweif}.
The projection $IX\to X$ is a morphism.
The bornological coarse space $IX$ has a big family
$I \cY:=([0,1]\times Y_{i})_{i\in I}$ and a compatible product uniform structure. 
We let $(IX)_{h}$ denote the corresponding hybrid bornological coarse space (Definition \ref{defn:sdfuh8934t}).
The projection is compatible (Definition \ref{ejwiofeoifefewfewf}) so that by Lemma \ref{oepwfweoifewofiewofewf} it is also a morphism \begin{equation}\label{grtoijoig3gtgertg}
(IX)_{h}\to X_{h}\, .
\end{equation}

Let $X$ be a bornological coarse space with a compatible uniform structure $\cT$ and a big family $\cY=(Y_{i})_{i\in I}$.

\begin{theorem}[Homotopy theorem]\label{fewijwefio23ri3ohewkjfwefewf}
\index{Homotopy Theorem}
Assume:
\begin{enumerate} 
\item \label{fgroiwfewfwefewff} $I=\nat$.
\item  For every bounded subset $B$ of $X$ there exists $i$ in $I$ such that $B\subseteq Y_{i}$.\end{enumerate}
Then the projection \eqref{grtoijoig3gtgertg} induces an equivalence
$$\Yo^{s}((IX)_{h})\to  \Yo^{s}(X_{h})\, .$$
\end{theorem}

\subsubsection{Proof of the homotopy theorem}

The rough idea is to identify $(IX)_{h}$ with  a coarse cylinder over $X_{h}$ and then to apply Proposition \ref{kfwejfefjewfejwofewfw3343}.  In the following we make this idea precise.
Point \ref{efwijfiewfoiefe3u40934332r} of Corollary~\ref{kjeflwfjewofewuf98ewuf98u798798234234324324343} gives an equivalence
$$ \Yo^{s}((IX)_{h})\simeq \colim_{U\in I\cC_{h}}\Yo^{s}((IX)_{U})\, ,$$
where the colimit runs over the poset $I\cC_{h}$ of the  entourages   of the hybrid coarse structure of $(IX)_{h}$. See also Example \ref{wfeoihewiufh9ewu98u2398234324234} for notation.

We consider $[1,\infty)$ as a bornological coarse space with the structures induced from the metric. 
    Let $q:X\to [1,\infty)$ be a function.
  Then we define the pair of functions $p:=(q,0)$ and   consider the cylinder
$I_{p}X$. We  define the maps of sets
$$f:I_{p}X\to[0,1]\times X\, , \quad f(t,x):=(t/q(x),x) $$
and 
$$g:[0,1]\times X\to I_{p}X\, , \quad g(t,x):=(tq(x),x)$$
by rescaling the time variable. They are inverse to each other.

 By Lemma \ref{erogjreeo43t34t}, for every coarse entourage $V$ of $X$ we can find a function $q$ as above such that  $q:X_{V}\to [1,\infty)$  is bornological and controlled  and 
   \begin{equation}\label{wefe42245524}\lim_{i\in I}\sup_{x\in X\setminus Y_{i}} q(x)^{-1} =0\, .\end{equation}   
 Note that we take the supremum  for subsets in the ordered set $[0,\infty]$, and that the supremum of the empty set is equal to $0$.
This is relevant if there exists {an} $i$ in $I$ with $Y_{i}=X$.

It suffices to show:
\begin{enumerate} 
\item  For any entourage $V_{h}$ in $\cC_{h}$ and   controlled and bornological function $q\colon X_{V_{h}}\to [1,\infty)$  with  \eqref{wefe42245524} there exists an entourage $U_{h}$ in $I\cC_{h}$ such that
$f\colon I_{p}X_{V_{h}}\to (IX)_{U_{h}}$ is a morphism.
\item  For every $U_{h}$ in $I\cC_{h}$  there exists  $V_{h}'$ in $\cC_{h}$ and a   controlled and bornological function $q'\colon X_{V'_{h}}\to [1,\infty)$  with  \eqref{wefe42245524}   such that $g'\colon (IX)_{U_{h}}\to I_{p'}X_{V'_{h}}$ is a morphism.
\end{enumerate}
We {then} choose $U_{h}'$ for $V_{h}'$ similarly.
We can assume that $V_{h}\subseteq V_{h}'$ and $U_{h}\subseteq U_{h}'$.
We obtain a commuting diagram  
\begin{align*}
\mathclap{
\xymatrix{
\cdots\ar[rr]\ar[dr]&& \Yo^{s}((IX)_{U_{h}})\ar[dr]^-{g'}\ar[rr]&&\Yo^{s}((IX)_{U_{h}^{\prime}})\ar[rr]&&\cdots\\
&\Yo^{s}(I_{p}X_{V_{h}})\ar[ur]^-{f}\ar[dr]_-{\simeq}&&\Yo^{s}(I_{p'}X_{V'_{h})}\ar[ur]^-{f'}\ar[dr]_-{\simeq}&&&\\
\cdots\ar[rr]&&\Yo^{s}(X_{V_{h}})\ar[rr]&&\Yo^{s}(X_{V^{'}_{h}})\ar[rr]&&\cdots
}
}
\end{align*}
where the upper and lower horizontal morphisms are induced by the identities of the underlying sets, and the morphisms indicated as equivalences are the projections, see 
Proposition \ref{kfwejfefjewfejwofewfw3343}.
%
The diagram induces by a cofinality consideration  the middle equivalence between the  colimits in the following chain
$$\Yo^{s}(X_{h})\simeq \colim_{V_{h}\in \cC_{h}} \Yo^{s}(X_{V_{h}})\simeq \colim_{U_{h}\in I\cC_{h}} \Yo^{s}((IX)_{U_{h}})\simeq \Yo^{s}((IX)_{h})\, .$$

 The maps $f$ and $g$ are clearly proper. So we must discuss under which conditions  they are controlled.
 
 We now take advantage that we can explicitly parametrize 
 a cofinal set of coarse entourages of $I\cC_{h}$ in the form
$ U_{h}:= IV\cap  V_{\phi}$. In detail, we take 
 a coarse entourage $V$ of $X$ and define the entourage $IV:=[0,1]\times [0,1]\times V$ of the product  $[0,1]\times X$. We also  choose a pair   $\phi=(\kappa,\psi)$ of a function $\kappa: I\to (0,\infty)$ with $\lim_{i\in I}\kappa(i)=0$ and a $\cT$-admissible function $\psi:I\to \cP(X\times X)^{\op}$.  
Using the metric on $[0,1]$ the function $\phi$ determines a function $I\to \cP(IX\times IX)^{\op}$ which is admissible with respect to the product uniform structure on $IX$. In particular, $\phi$ defines  a subset
$V_{\phi}\subseteq \cP(IX\times IX)$ as in \eqref{dqdqw2345r2t2444t3t4t}. 
In detail, $((s,x),(t,y))\in V_{\phi}$ if for every $i$ in $I$ we  have $(x,y)\in Y_{i}\times Y_{i}$ or
$|s-t|\le \kappa(i)$ and $(x,y)\in V\cap \psi(i)$.

We will now analyse $f$.  We fix the data for the entourage $V_{h}$. To this end we  
 fix 
  $V$ in $\cC$ and  a  $\cT$-admissible function $\psi:I\to \cP(X\times X)^{\op}$.
  Then we set $V_{h}:=V\cap V_{\psi}$ in $\cC_{h}$. 
  We furthermore fix  $e$ in $(0,\infty)$ and let $U_{e}$ be the corresponding metric entourage on $\R$.
  Then $\tilde V_{h}:= U_{e}\times V_{h}$ (we omit to write the restriction from $\R\times X$ to $I_{p}X$) generates the coarse structure of  $I_{p}X_{V_{h}}$. 
  Our task is now to find an entourage 
  $U_{h}$ in $I\cC_{h}$ such that $(f\times f)(\tilde V_{h})\subseteq U_{h}$.


Since $q:X_{V_{h}}\to [1,\infty)$ is controlled and $V_{h}\subseteq V$ there exists a constant $C$ in $(0,\infty)$ such that
$(x,y)\in V$ implies that $|q(x)-q(y)|\le C$.
We define the function $\kappa:\nat\to  [0,\infty)$ by   $$\kappa(i):=(C+e)\cdot\sup_{x\in X\setminus Y_{i}} q(x)^{-1} \, .$$
By \eqref{wefe42245524}  we have $\lim_{i\in I}\kappa(i)=0$. We now set
$\phi:=(\kappa,\psi)$
and define $U_{h}:=IV\cap V_{\phi}$.

We  now show that $(f\times f)(\tilde V_{h})\subseteq U_{h}$.
We assume that $((s,x),(t,y))$ is in $\tilde V_{h}$. Then  $|t-s|\le e$ and 
$(x,y)\in V\cap V_{\psi}$. 
For every $i$ in $I$ we have two cases:
\begin{enumerate}
\item If $(x,y)\in V\cap (Y_{i}\times Y_{i})$. Then  $((s,x),(t,y))\in ([0,1]\times Y_{i})\times  ([0,1]\times Y_{i})$.
 \item Otherwise,  $(x,y)\in V\cap \phi(i)$ and   w.l.o.g $x\not\in Y_{i}$. Then
 $$|s/q(x)-t/q(y)|\le e/q(x)+t(|q(x)-q(y)|) q(x)^{-1}q(y)^{-1}\le (e+C) q(x)^{-1}\le \kappa(i) \, .$$
 \end{enumerate}
This shows that $(f(s,x),f(t,y))\in U_{h}$ as required.

We will now investigate $g$. We choose a coarse entourage $V$ of $X$ and $\phi=(\kappa,\psi)$ as above.
These choices define the entourage
$U_{h}:=IV\cap V_{\phi}$ in $I\cC_{h}$. 
We must find $V_{h}'$  in $\cC_{h}$ and a function $q'$ appropriately. 
We take $V_{h}':=V\cap V_{\psi}$. In order to find $q'$ we use the following lemma 
which we will prove afterwards.

\begin{lem}\label{erogjreeo43t34t} For every coarse entourage $V$ of $X$
there exists a controlled and bornological  function $q:X_{V}\to (0,\infty)$  such that the sets of real numbers
\[ \left\{ q(x) \cdot \Big( \inf_{ i\in I :  x\not\in Y_{i} } \kappa(i) \Big)  :  x\in X \right\} \text{ and } \bigg\{ q(x)  :  x\in \bigcap_{i\in I} Y_{i} \bigg\} \]
are bounded by some positive real number $D $ and \eqref{wefe42245524} holds.
\end{lem}

Let $q'$  (called $q$ in the lemma) and   $D$  be as in the lemma.
Since $q'$ is controlled we can choose    $C$ in $(0,\infty)$ such that
$(x,y)\in V$ implies that $|q'(x)-q'(y)|\le C$.
We define the entourage $\tilde V_{h}:=U_{C+D}\times V_{h}'$ of $I_{p'}X_{V'_{h}}$.
Assume that $((s,x),(t,y))$ belongs to $U_{h}$. For every $i$ in $I$ we distinguish two cases:
\begin{enumerate}
\item\label{jfwelfeiofewfewf} $(x,y)\in V\cap(Y_{i}\times Y_{i})$.  
 \item\label{jfwelfeiofewfewf1} Otherwise we can w.l.o.g. assume that $x\not\in Y_{i}$. Then
 \begin{equation}\label{jhfjkwehfuiwezfiiew}
 \left| sq'(x)-tq'(y) \right| \le q'(x) \cdot \Big( \inf_{ i\in I  :  x\not\in Y_{i} }\kappa(i) \Big)  +|q'(x)-q'(y)|\le D+C\end{equation}
 and $(x,y)\in V\cap \psi(i)$.
\end{enumerate}
If    $(x,y)\in V\cap(Y_{i}\times Y_{i})$ for all $i\in I$, then $$|sq'(x)-tq'(y)|\le D\le C+D$$ by the second bound claimed in Lemma~\ref{erogjreeo43t34t}.  

 The estimates thus show the inclusion 
$(g\times g)(U_{h})\subseteq \tilde V_{h}$. 
This completes the proof of the Homotopy Theorem \ref{fewijwefio23ri3ohewkjfwefewf}.

\begin{proof}[Proof of Lemma~\ref{erogjreeo43t34t}]
It is here where we use that $I=\nat$.
In a first step we define the  monotonously decreasing function $$\lambda: \nat \to (0,\infty)\, , \quad \lambda(i):=  \sup_{i\le j} \kappa(j)\, .$$ It satisfies $\kappa\le \lambda$ and $\lim_{i\to \infty} \lambda(i)=0$.

    We now choose by  induction a new function
$\mu:\nat\to (0,\infty)$ and an auxilary function $a:\nat\to \nat$ as follows.

We set $\mu(0) := \lambda(0)$ and $a(0) := 0$. 

Assume now that $a$ is defined on all numbers $\le n$ and
 $\mu$ is defined for all numbers $\le a(n)$.
Then we choose $a(n+1)$ in $\nat$ such that $a(n+1)> a(n)$ and $V[Y_{a(n)}]\subseteq Y_{a(n+1)}$.
We then define for all $k$ in $(a(n),a(n+1)]$
$$\mu(k):=\max\Big\{\frac{\mu(a(n))}{\mu(a(n))+1}, \lambda (a(n))  \Big\}\ .$$
This has the effect that $\mu$ is constant on the intervals $(a(n),a(n+1)]$ for all $n$ in $\nat$, and that $\mu^{-1}$ increases by at most one if one increases the argument from $a(n)$ to $a(n)+1$.
We furthermore have the estimate
 $\lambda(k)\le \mu(k)$ for all $k$ in $\nat$,   $\mu$ is monotonously decreasing, and it
 satisfies $\lim_{i\to \infty} \mu(i)=0$.
 We  define the function
$$q:X\to [1,\infty)\, , \quad q(x):=1+\frac{1}{\sup\{\mu(i) :  x\in Y_{i}\}}\, .$$
 Note that $q$ is well-defined since for every $x$ in $X$ the index set of the supremum is non-empty and $\mu$ is positive.
Let $B$ be a bounded subset of $X$. Then we can find $i$ in $I$ such that
$B\subseteq Y_{i}$.  Then $\sup_{x\in B} q(x)\le \mu(i)^{-1}+1$. This shows that $q$ is bornological. For $x$ in $\bigcap_{i\in I} Y_{i}$ we have
$q(x)\le \mu(i)^{-1}+1$ for every $i$.  This in particular shows   the second estimate claimed in the lemma.

We now argue that $q$ is controlled.  We will show that
for 
$(x,y)$ in $V$ we have   $|q(x)-q(y)|\le 1$. We can choose $n$ in $\nat$ such that $x\in Y_{a(n)}$ and $q(x)=\mu(a(n))^{-1}+1$, and $m\in \nat$ such that
$q(y)=\mu(a(m))^{-1}+1$. After flipping the roles of $x$ and $y$ if necessary we can assume that $n\le m$.
 Then $y\in V[Y_{a(n)}]\subseteq Y_{a(n+1)}$.   We then have the following options:
\begin{enumerate}
\item If $n=m$, then $q(x)=q(y)$.
\item If $n<m$, then $m=n+1$ (here we use that $\mu$ decreases monotonously).
In this case $q(y)-q(x)\le 1$.
\end{enumerate}

We now consider  $x$ in $X$ and $i$ in $I$ such that $x\not\in Y_{i}$.
We have $\kappa(i)q(x)\le \mu(i) q(x)$. We must minimize over $i$.
Assume that $n$ is such that $x\in Y_{a(n+1)}\setminus Y_{a(n)}$.
We see that $$ {\inf_{i\in I :  x\not\in Y_{i}} \kappa(i)q(x) \le} \inf_{i\in I :  x\not\in Y_{i}} \mu(i)q(x)= 
\mu(a(n)) (1+
\frac{1}{\mu(a(n+1)) })\le 2\mu(a(0))\ .$$
This finishes the 
  proof.
\end{proof}

\begin{rem}
The assumption that $I=\nat$ in the statement of the homotopy theorem was used in the proof of Lemma  \ref{erogjreeo43t34t}.
At the moment we do not know if this assumption is only technical or  really necessary. 
\hB
\end{rem}

 
\subsubsection{Uniform homotopies and the cone functors} 
 In the category of uniform spaces we have a natural notion of a uniform homotopy.
 
 Let $X$ be a uniform space.

\begin{ddd}
 A uniform homotopy is a uniformly continuous map $[0,1]\times X\to X$, where the product is equipped with the product uniform structure.
\end{ddd}

Let $X$ and $X^{\prime}$ be bornological coarse spaces   equipped with uniform structures $\cT$ and $\cT^{\prime}$ and big families $\cY=(Y_{i})_{i\in I}$ and $\cY^{\prime}=(Y^{\prime}_{i^{\prime}})_{i^{\prime}\in I^{\prime}}$, respectively.   Let $f_{0},f_{1}:X\to X^{\prime}$  be two compatible morphisms.
\begin{ddd}
We say that $f_{0}$ and $f_{1}$ are compatibly   homotopic\index{compatibly homotopic} if there is
a compatible morphism
$h\colon IX\to X^{\prime}$ such that
$f_{i}=h_{|\{i\}\times X}$ for $i$ in $\{0,1\}$.
\end{ddd}
Note that compatibly homotopic maps are close to each other and uniformly homotopic.

 Let $X$ and $X^{\prime}$ be bornological coarse spaces   equipped with uniform structures $\cT$ and $\cT^{\prime}$ and big families $\cY=(Y_{i})_{i\in I}$ and $\cY^{\prime}=(Y^{\prime}_{i^{\prime}})_{i^{\prime}\in I^{\prime}}$, respectively.  
  
\begin{kor} \label{efwjweiofiu97249873942342423423424}
Assume that $I=\nat$ and that 
   for every bounded subset $B$ of $X$ there exists $i$ in $I$ such that $B\subseteq Y_{i}$. 
If $f_{0},f_{1}:X\to X^{\prime}$ are  compatibly homotopic to each other, we have an equivalence
$$\Yo^{s}(f_{0})\simeq \Yo^{s}(f_{1}):\Yo^{s}(X_{h})\to \Yo^{s}(X^{\prime}_{h})\, .$$ 
\end{kor}

\begin{proof}
Let $h:[0,1]\times X\to X^{\prime}$ be a compatible homotopy between $f_{0}$ and $f_{1}$.
By the Lemma~\ref{oepwfweoifewofiewofewf} it induces a  morphism
$$(IX)_{h} \to X^{\prime}_{h}$$
which yields the morphisms 
$f_{i}:X_{h}\to X_{h}^{\prime} $ by precomposition.
The assertion now follows from the Homotopy Theorem~\ref{fewijwefio23ri3ohewkjfwefewf}.
\end{proof}
  
 We continue Example~\ref{ifuweiofeowfew982342343434} and the cone-at-infinity from Section~\ref{ofiweifewufou9823u92343432434}.
 We consider two maps  $f_{0},f_{1}:Y\to Y^{\prime}$ between uniform spaces.
 \begin{kor}
 If  $f_{0}$ and $f_{1}$ are uniformly homotopic, then we have an equivalence $ \Yo^{s}(\cO(f_{0}))\simeq  \Yo^{s}(\cO(f_{1}))$.
 \end{kor}
Therefore the cone functor $\Yo^{s}\circ \cO \colon  \bU\to \Sp\cX$ sends uniformly homotopy equivalent uniform spaces to equivalent motivic spectra. It shares this property with the functor $\cO^{\infty} \colon  \bU\to \Sp\cX$ introduced in \eqref{bbgboitgiotuoitutrttttbb} which is in addition excisive by Corollary~\ref{wfuhweuifhewiufewf243}.

\begin{kor}\label{efjwefu23u2394u2394324234234}
The functor $ \cO^{\infty} \colon  \bU\to \Sp\cX$ sends uniformly homotopy equivalent uniform spaces to equivalent motivic spectra. 
\end{kor}   

The following corollary to the combination of Corollaries \ref{wfuhweuifhewiufewf243} and \ref{efjwefu23u2394u2394324234234} says that $\cO^{\infty}$ behaves like a $\Sp\cX$-valued homology theory on  {the} subcategory $\bU_{\sep}$  of $\bU$ of uniform Hausdorff spaces (see Mitchener \cite[Thm.~4.9]{mit} for a related result).
\begin{kor}
The functor $\cO^{\infty}\colon \bU_{\sep}\to \Sp\cX$ is uniformly homotopy invariant and satisfies excision for  uniform  decompositions.
\end{kor}

This observation is the starting point of a translation of the homotopy theory on $\bU$ into motivic coarse spectra. It will be further discussed in  {the} subsequent paper \cite{ass}.

\subsection{Flasque hybrid spaces}\label{ekfjweofjweoifewfwefwefewf}

We consider a bornological coarse space $(X,\cC,\cB)$ with a compatible uniform structure~$\cT$ and a big family $\cY=(Y_{i})_{i\in I}$. These data determine the hybrid space $X_{h}$ (Definition \ref{defn:sdfuh8934t}). In this section we analyse which additional structures on $X$ guarantee that $X_{h}$ is flasque in the generalized sense (see Definition \ref{ddd:rf435f}).
 
We assume that the uniform structure $\cT$ is induced from a metric $d$. We will denote the uniform entourage of radius $r$ by $U_{r}$, see \eqref{ffwefewfe24543234wfefewef}.

We furthermore assume that the big family is indexed by $\nat$ and determined through a function
 $e\colon X\to [0,\infty)$ by $Y_{i}:=\{e\le i\}$ for all $i$ in $\nat$. Moreover, we assume that    $e$  has the following properties:
 
\begin{enumerate}
\item $e$ is uniformly continuous.
\item $e$ is controlled. This ensures that the family of subsets { $(Y_i)_{i \in \IN}$} as defined above is big. Indeed, let $V$ be a coarse entourage of $X$. Then we can find an $s$ in $\nat$ such that  $|e(x)-e(y)|\le s$ for all $(x,y)$ in $V$. For $i$ in $\nat$ we then have $V[Y_{i}]\subseteq Y_{i+s}$.
\item $e$ is bornological. This implies that for every bounded subset $B$ of $X$ there exists $i$ in $\nat$ such that $B\subseteq Y_{i}$.
\end{enumerate}

The reason why we make these restrictions on the index set of the big family and $\cT$ is that we are going to use the simple characterization of hybrid entourages given in Remark~\ref{fwoejweofewfewfwefwefwe34345345345345}. We do not know whether this restriction is of technical nature or essential.

The essential structure is a map $\Phi  :  [0,\infty)\ltimes X\to X$. We write $\Phi_{t}(x):=\Phi(t,x)$.
 Recall that the semi-direct product has the product coarse structure and the bounded sets are generated by $[0,\infty)\times B$ for bounded subsets $B$ of $X$. We assume that  $\Phi$  has the following properties:

\begin{enumerate}
\item $\Phi$ is a morphism in $\BC$.
\item $\Phi$ is uniformly continuous.
\item $\Phi_{0}=\id$.
\item\label{sdjnfwersfd43} For every bounded subset $B$ of $X$ there exists $t$ in $\R$ such that
$\Phi([t,\infty)\times X)\cap B=\emptyset$.
\item\label{fwijwefoi2434545345} $\Phi$ is contracting in the following sense: For every $i$  in $ \IN$, $\epsilon$ in $(0,\infty)>0$ and coarse entourage $V$ of $X$   there exists $T$  in $(0,\infty)$ such that we have
\[\forall x,y\in X, \ \forall t\in [T,\infty) \: : \:   (x,y)\in V\cap (Y_{i}\times Y_{i}) \Rightarrow {(\Phi_{t}(x),\Phi_{t}(y))} \in U_{\epsilon}\,.\]
\item\label{qrgijoqwfdeqdwqedq} $\Phi$ is compatible with the big family in the sense that for all $s,t$ in $[0,\infty)$   we have  $\Phi_{t}(\{e\le s\})\subseteq \{e\le s+t\}$.
\end{enumerate}

Recall the Definition \ref{ddd:rf435f} of flasqueness in the generalized sense.
\begin{theorem}\label{thm:sdf09234}
Let all the above assumptions be satisfied. Then $X_h$ is flasque\index{flasque!hybrid space} in the generalized sense.
\end{theorem}

\begin{proof}
For every $S$ in $\nat \setminus \{0,1\}$ we consider the map $r_{S} :  [0,\infty)\to [0,\log(S)]$ given by
\[r_{S}(t):=\left\{\begin{array}{ll}
\log(S)-\frac{t}{S} & \text{ for } 0\le t\le \frac{S}{\log(S)}\\
\log(S)- \frac{1}{\log(S)} -\big( t- \frac{S}{\log(S)} \big) & \text{ for }\frac{S}{\log(S)}\le t\le \log(S)- \frac{1}{\log(S)}+ \frac{S}{\log(S)}\\
0 & \text{ for } t\ge \log(S)- \frac{1}{\log(S)}+\frac{S}{\log(S)}
\end{array}\right.\]
We further set
$r_1 :=  0$. The family of maps $(r_{S})_{S\in \nat \setminus \{0\}}$ is uniformly controlled and uniformly equicontinuous.

\begin{figure}[ht]
\centering
\includegraphics[scale=0.4]{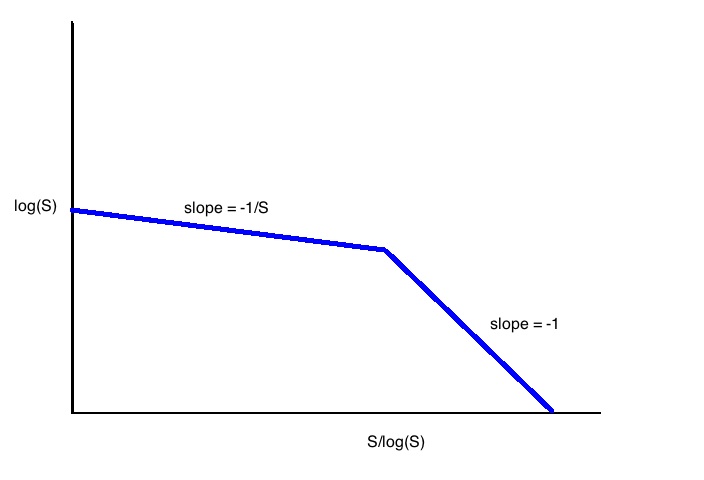}
\caption*{Function $r_{S}(t)$}
\end{figure}

 We consider the composition of maps
\[\Psi_{S}  :  X\xrightarrow{(e,\id)}[0,\infty)\times X \xrightarrow{(r_{S},\id)}  [0,\log (S)]\ltimes X\xrightarrow{\Phi} X\,.\]
 
The first map is controlled since $e$ is controlled. Furthermore it is      proper since its second component is the identity. Hence it is
a morphism in $\BC$.  The second map is also a morphism (since we use the semi-direct product it is proper).
The last map is a morphism by assumption since it is a restriction of a morphism to a subset. 
All three maps are uniformly continuous and we have $\Psi_S(\{e \le s\}) \subseteq \{e \le s + \log(S)\}$. So the maps $\Psi_S$ are compatible in the sense of Definition~\ref{ejwiofeoifefewfewf} and therefore are morphisms
\[\Psi_{S}  :  X_{h}\to X_{h}\,.\]

Our claim is now that these maps implement flasqueness of $X_h$ in the generalized sense. More precisely,  we will show that the family of morphisms $(\Psi_S)_{S \in \nat \setminus \{0\}}$ satisfies the four conditions of Definition~\ref{ddd:rf435f}.

\textbf{Condition \ref{sdfn34r}}

Since $r_{1}  \equiv 0$ we get $\Psi_1 = \id$.

\textbf{Condition \ref{asdfgjtzkj56}}

We must   show that the union $$\bigcup_{S \in \nat \setminus \{0\}} (\Psi_S \times \Psi_{S+1})(\diag_{X_h})$$ is an entourage of $X_h$. Hence we must control the pairs
\[(\Psi_{S}(x),\Psi_{S+1}(x)) \text{ {for} } x \text{ in } X \text{  {and} } S \text{ in }  \nat \setminus \{0\}\, .\]
 
Since the family $(r_{S})_{ \nat \setminus \{0\}}$ is uniformly equicontinuous the family $(S\mapsto \Psi_{S}(x))_{x\in X}$ of functions  is uniformly equicontinuous. Hence there exists a coarse entourage $V$ of $X$ such that for all $x $ in $X$ and $S$ in $ \nat \setminus \{0\}$
$$ (\Psi_{S}(x),\Psi_{S+1}(x))\in  V\, .$$
 
In view of Remark \ref{fwoejweofewfewfwefwefwe34345345345345} it remains to show that for every $\epsilon$ in $(0,\infty)$ there exists $j$ in $\nat$ such that
\begin{equation}\label{velkvkekorevrevevev}
\forall x\in X\  \forall S\in  \nat \setminus \{0\}   \: : \:  (\Psi_{S}(x),\Psi_{S+1}(x))\in (Y_{j}\times Y_{j})\cup U_{\epsilon} \, .
\end{equation}

First observe using the explicit formula for $r_{S}$ that  {there} exists a $C$ in $(0,\infty)$ such that for all $t$ in $[0,\infty)$ and $S$ in $\nat$ we have the inequality  \[|r_{S}(t)-r_{S+1}(t)|\le \frac{C}{\min\{1,\log(S)\}} \, . \]
 
We now fix $\epsilon$ in $(0,\infty)$. Using the above estimate and that $\Phi$ is uniformly continuous  we can find $S_{0}$ in $ \nat \setminus \{0,1\}$ sufficiently large such that
\begin{equation}\label{dqwdqwd2342342343242432424}
 \forall x\in X\  \forall S\in \nat \ \mbox{with}\ S\ge S_{0}\: :  \:(\Psi_{S}(x),\Psi_{S+1}(x))\in U_{\epsilon} 
\end{equation}
It remains to control the pairs  with $1\le S\le S_{0}$.

Set $$i := \sup_{1 < S \le S_0} \Big( \log(S)- \frac{1}{\log(S)}+\frac{S}{\log(S)} \Big)$$ (note that $i$ is finite). Then we have $r_{S}(e(x)) = 0$ for all $x$ in $X \setminus Y_{i}$ and all $S $  in $[1,S_0]$. Consequently, $\Psi_S(x) = x$ for all $x$ in $X \setminus Y_{i}$ and all $S$ in $[1,S_0]$.  

We now choose
the integer $j$ in $\nat$ such that $\Psi_{S}(Y_{i})\subseteq Y_{j}$ for all $S$ in $[1,S_0]$.
Then  we have 
\begin{equation}\label{qweqweqwewwe}
 \forall x\in X\  \forall S\in \nat \ \mbox{with}\ 1\le S\le  S_{0} \: : \:(\Psi_{S}(x),\Psi_{S+1}(x)) \in (Y_{j}\times Y_{j})\cup \diag_{X} \, .
\end{equation}
The combination of \eqref{qweqweqwewwe} and \eqref{dqwdqwd2342342343242432424} implies \eqref{velkvkekorevrevevev}.

\textbf{Condition \ref{sdfn3443243t6}}

Let $U$ be a coarse entourage of $X_{h}$. We have to show that the union $\bigcup_{S \in  \nat \setminus \{0\}} (\Psi_S \times \Psi_S)(U)$ is again a coarse entourage of $X_h$.

By definition of the hybrid structure and  Remark \ref{fwoejweofewfewfwefwefwe34345345345345} there is a coarse entourage $V$ of $X$ such that for every $\epsilon$ in $(0,\infty)$ we can find $i$ in $\nat\setminus \{0\}$  with
\[U\subseteq V\cap ((Y_{i}\times Y_{i})\cup   U_{\epsilon}).\]
 
First of all we observe that the family $(r_{S})_{S\in  \nat \setminus \{0\}}$ is uniformly controlled. Since
 $\Phi$ is controlled with respect to the original structures there exists an entourage $V^{\prime}$ in $\cC$ such that
\begin{equation}\label{ferfmo34jrtoi34ut43tu43tu89}
(\Psi_{S}\times \Psi_{S})(V)\subseteq V^{\prime}
\end{equation} for all $S$ in $ \nat \setminus \{0\}$.

We now fix $\delta$ in $(0,\infty)$.
 Since the family $(r_{S})_{S \in  \nat \setminus \{0\}}$ is uniformly equicontinuous the same is true for the family $(\Psi_{S})_{S \in  \nat \setminus \{0\}}$.
Hence there exists $i $ in $ \IN$ such that
\begin{equation}
\label{eq:sdfnkj345sef}
(\Psi_{S} \times \Psi_S)(U\setminus (Y_{i}\times Y_{i}))\subseteq U_{\delta}
\end{equation}
 for all $S\in \nat \setminus \{0\}$.
 
Assume now that is $(x,y)$ in $(Y_{i}\times Y_{i})\cap V$.
The control of $e$ yields a positive number $R$ (which only depends on $V$) such that 
$|e(x)-e(y)| \le R$. Furthermore, we have
$e(x),e(y)\le i$.  

In the following steps we are going to choose $S_{0}$ in  $\nat \setminus \{0,1\}$ by step increasingly bigger such that certain estimates   hold true.

In the first step we choose $S_{0}$ such that
$$i\le \frac{S_{0}}{\log(S_{0})}\, .$$
Then for all $S$ in $\nat$ with $S\ge S_{0}$ and $(x,y)$  in  $V\cap(Y_{i}\times Y_{i})$ we have \begin{equation}\label{fwegreegergerg43535eopwfe2}
|r_{S}(e(x))-r_{S}(e(y))|\le \frac{R}{S_{0}}\, .
\end{equation}

We observe that
\begin{equation}\label{qdq23242323535}
\operatorname*{lim}_{S_{0}\to \infty} \inf_{S\ge S_{0},x\in Y_{i}} r_S(e(x))=\infty\, .
\end{equation}


We have by definition
\[\big( \Psi_S(x), \Psi_S(y) \big) = \big( \Phi_{r_S(e(x))}(x), \Phi_{r_S(e(y))}(y) \big).\]
Since $\Phi$ is uniformly continuous, in view of \eqref{fwegreegergerg43535eopwfe2} we can choose $S_{0}$ so large that \begin{equation}\label{r34rkljr34ktt43oitoiuu9} \forall S\ge S_{0} ,\ \forall (x,y)\in V\cap (Y_{i}\times Y_{i}) \:\: :\:\:
\big( \Phi_{r_S(e(x))}(x), \Phi_{r_S(e(y))}(x) \big)\in U_{\delta/2}\, .
\end{equation}

We 
  employ that $\Phi_t$ is contracting and \eqref{qdq23242323535} in order to see that we can increase $S_{0}$ even more such that      \begin{equation}\label{3tiojriojoijjjoj3jgoi34t}
 \forall S\ge S_{0} ,\ \forall (x,y)\in V\cap (Y_{i}\times Y_{i}) \:\: :\:\:\big( \Phi_{r_S(e(y))}(x), \Phi_{r_S(e(y))}(y) \big) \in U_{\delta / 2} \, .
\end{equation} 
The equations \eqref{r34rkljr34ktt43oitoiuu9} and \eqref{3tiojriojoijjjoj3jgoi34t} together give
$$  \forall S\ge S_{0} ,\ \forall (x,y)\in V\cap (Y_{i}\times Y_{i}) \:\: :\:\: \big( \Phi_{r_S(e(x))}(x), \Phi_{r_S(e(y))}(y) \big) \in U_{\delta}\, .$$

This implies \begin{equation}\label{gergergergegeg43r3r3r34r23r2}\forall S\ge S_{0}\:\::\:\:
(\Psi_{S} \times \Psi_S)((Y_i \times Y_i) \cap V) \subseteq U_\delta \, .
\end{equation}


We now use the compatibility of $\Phi$ with the big family and \eqref{ferfmo34jrtoi34ut43tu43tu89}   in order to find    an integer $j$  in $\nat$ such that for all  $S $ in $[1,S_{0}]$
\[    (\Psi_{S} \times \Psi_S)((Y_i \times Y_i) \cap V) \subseteq V^{\prime}\cap(Y_{j}\times Y_{j})\, .\]
 

In conclusion we have shown that  there is $V^{\prime}$ in $\cC$ and for given $\delta$ in $(0,\infty)$ we can choose $j$ in $\nat$ appropriately such that
\[\forall S\in  \nat\setminus \{0\}\:\::\:\: (\Psi_{S} \times \Psi_S)(U)\subseteq V^{\prime}\cap ((Y_{j}\times Y_{j})\cup U_{\delta})\, .\] 
In view of  Remark \ref{fwoejweofewfewfwefwefwe34345345345345} this proves that  $\bigcup_{S \in \IN} (\Psi_S \times \Psi_S)(U)$ is   an entourage of $X_h$.

\textbf{Condition \ref{fwoiejweoifeowi23}}

If $B$ is a bounded subset of $X$, then $e_{|B}$ is bounded. We now use \eqref{qdq23242323535} and
  Property~\ref{sdjnfwersfd43} of $\Phi$ in order to verify the condition.
\end{proof} 


\subsection{Decomposition of simplicial complexes}\label{fiwehjfi892938u9r23r23rr}

The  main result of the present section is Theorem \ref{qdhdqwidqwihdqwdqwd}. 
It roughly states that the motivic coarse spectrum  of a simplicial complex with a hybrid structure  relative to an exhaustion 
can be expressed as a colimit  of motivic coarse spectra of discrete bornological coarse spaces.

\subsubsection{Metrics on simplicial complexes}

     By a simplicial complex we understand a topological space 
presented as a geometric realization of an abstract  simplicial complex.

We can model the $n$-dimensional simplex $\Delta^{n}$ as the intersection $S^{n}\cap [0,\infty)^{n+1}$ in $\R^{n+1}$. The restriction of the Riemannian distance on the sphere to $\Delta^{n}$ is called the spherical metric $d_{\sph}$.  
If we model $\Delta^{n}$ by the subset $\{x\in \R^{n}\:|\: \sum_{i=0}^{n}x_{i}=1\}\cap  [0,\infty)^{n+1}$ of $\R^{n+1}$, then the induced metric is the {E}uclidean metric and will be denoted by $d_{eu}$. Another metric 
  is the $\ell^{1}$-metric $d_{\ell^{1}}$ on $\Delta^{n}$ given in the same model by 
$d(x,x'):=\sum_{i=0}^{n}|x_{i}-x_{i}'|$.

We let $X$ be a  simplicial complex with a metric $d$.
\begin{ddd}\label{fwoiefjeoiwfefewf}
The metric $d$ is {good}\index{good metric}\index{metric!good}  if it has the following propertries:
\begin{enumerate}
\item\label{fekwjefeoiewfioewfwefwef} The restriction of $d$ to every component of $X$      a path metric.
\item \label{fekwjefeoiewfioewfwefwef1}There exists $c$ in $(0,\infty)$  such that any two components of $X$ have distance at least $c$.
\item \label{fjwefwefewlfeioo} For every $n$ in $\nat$ there exists $C_{n}$ in $(0,\infty)$ such that   the metric induced on every  $n$-simplex $\Delta\subseteq X$ satisfies $ C_{n}^{-1} \cdot d_{\sph}\le d_{|\Delta}\le C_{n} \cdot d_{\sph} $. 
\end{enumerate}
\end{ddd}

Note that we allow that different components of $X$ have finite or infinite distances, i.e., our metrics are actually quasi-metrics. 
 
\begin{ex}  \label{gwerjgiorejgoregrferwfwerfwfrwfrfwrf}
An example of a good metric  on a simplicial complex is
the path-metric induced from the  spherical metrics on the simplices.  In this case we can take $C_{n}=1$ for all $n$ in $\nat$ and the components have infinite distance from each other  (see Wright \cite[Appendix]{nw1}).  

In the case of  finite-dimensional simplicial complexes we can also start with the {E}uclidean or the $\ell^{1}$-metric to generated a good path metric.
\hB
\end{ex}

%

\begin{ex}\label{ioefjewoifeu23r23r32r}
If $\Delta$ is the standard simplex, then  for $\mu$ in $(0,1]$ we let $\mu\Delta\subseteq \Delta$ be the image of the radial scaling (with respect to the barycenter) by the factor $\mu$. 

Assume that $X$ is equipped with a good metric.  For $n$ in $\nat$ and $\mu$ in $(0,1)$
let $Z$ be the subset of $X$ given by   the disjoint union of the $\mu$-scalings of the $n$-simplices of $X$ with the induced metric.
After reparametrization of the simplices  we can consider $Z$ as a simplicial complex which is  topologically  a disjoint union of $n$-simplices. The restriction of the good metric of $X$ to $Z$ is again good.

Two simplices of $Z$ which are scalings of simplices in the same component of $X$ are in finite distance.
\hB
\end{ex}

 \begin{ex} \label{ergi3234r534555}
 In the following we describe an example of a simplicial complex $X$ with the property that the identity of $X$ does not induce a morphism $X_{d_{e}}\to X_{d_{s}}$ of coarse spaces, where $d_{e}$ and $d_{s}$ are the path metrics induced by the Euclidean and the spherical metrics on the simplices, respectively. That this can happen for infinite-dimensional simplicial complexes has essentially been observed by Wright \cite[Appendix]{nw1}.
 
 The complex $X$ is defined as the quotient of $\bigsqcup_{n\in \nat}\Delta^{n}$ by identifying for all $n$ in $\nat$ the last face $\partial_{n}\Delta^{n}$ of $\Delta^{n}$ with the first face of the first face $\partial_{0}\partial_{0}\Delta^{n+1}$ of $\Delta^{n+1}$. The entourage $ \{d_{e}\le 1\}$ in $\cC_{d_{e}}$ does not belong to $\cC_{d_{s}}$, as shown below.
\begin{figure}[ht]
\centering\includegraphics[scale=0.7]{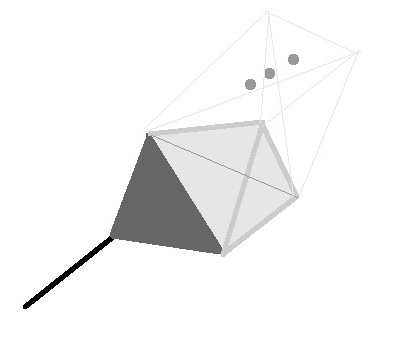}
\caption*{The complex $X$ in Example~\ref{ergi3234r534555}} \end{figure}

For  the barycenters $b_{n}$ and $b_{n+1}$ of $\Delta^{n}$ and $\Delta^{n+1}$ we have
\begin{equation}\label{rvelkj34oitjr3o4}d_{e}(b_{n},b_{n+1})=\sqrt{\frac{1}{n(n+1)}}+\sqrt{\frac{2}{n(n+2)}}\, ,
\end{equation}
while
$$d_{s}(b_{n},b_{n+1})=\pi/2\, .$$
Define $k:\nat \to \nat$ by
$$k(n):=\max\{m\in \nat: d_{e}(b_{n},b_{n+m})\le 1\}\, .$$
As a consequence of \eqref{rvelkj34oitjr3o4} we have $\lim_{n\to \infty} k(n)=\infty$. 
Hence we have
$$\forall n\in \nat\:\::\:\:d_{e}(b_{n},b_{n+k(n)})\le 1$$
while 
\[d_{s}(b_{n},b_{n+k(n)})=\frac{k(n)\pi}{2}\xrightarrow{n\to\infty}  \infty\, .\]
{Hence the identity of $X$ does not induce a controlled map $X_{d_{e}}\to X_{d_{s}}$.}
\hB
\end{ex}

\subsubsection{Decomposing simplicial complexes}\label{triohjorgvwevfdvsfdv}

Assume that $X$ is equipped with a good metric $d$.  
We let  $\cT_{d}$ be   the uniform structure on $X$ determined by the metric. 
We further assume that $X$ comes with a bornological coarse structure $(\cC,\cB)$ such that
$\cC$ and $\cT_{d}$ are compatible (see Definition \ref{defn:sd23d9023}). For example, $\cC$ could be (but this is not necessary) the structure $\cC_{d}$ induced by the metric. Note that the compatibility  condition  and Condition~\ref{fjwefwefewlfeioo} of Definition \ref{fwoiefjeoiwfefewf} imply that the size of the simplices of a fixed dimension  is uniformly bounded with respect to $\cC$.

 Finally we assume that $X$ comes with a big family $\cY=(Y_{i})_{i\in I}$ of subcomplexes
 such that
 every bounded subset of $X$ is contained in $Y_{i}$ for some $i$ in  $I$. In this situation we can consider the hybrid structure $\cC_{h}$ (see Definition \ref{defn:sdfuh8934t}).
We write $X_{h}$ for the corresponding bornological coarse space.

We  now assume that $X$ is $n$-dimensional. 
For $k$ in $\nat$ we let $X^{k}$ denote the $k$-skeleton of $X$.
 We let  the subset $Z$ of $ X$ be obtained as in Example \ref{ioefjewoifeu23r23r32r} by scaling all $n$-simplices by the factor $2/3$. Furthermore we let
$Y$  be the subspace of $X$ obtained from $X$ by removing the interiors of the unions
of the $1/3$-scaled $n$-simplices.
 Then $Y\cup Z=X $ is a decomposition of $X$ into two closed subsets. 
 
\begin{figure}[ht]
\centering
\includegraphics[scale=0.29]{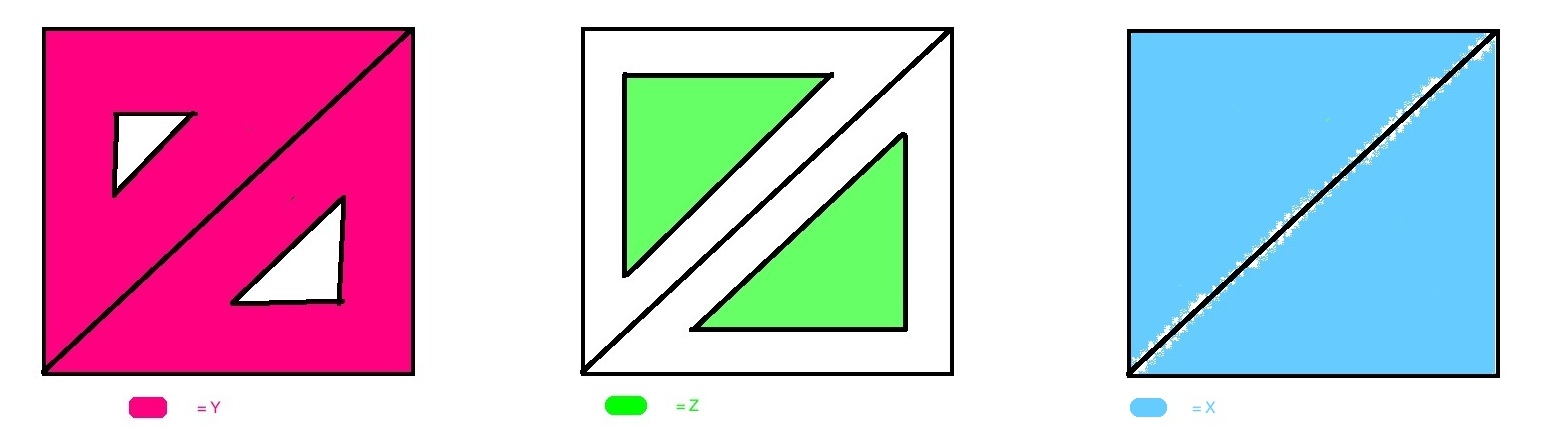}
\caption*{Decomposition $X=Y\cup Z$}
\end{figure}

Conditions  
\ref{fekwjefeoiewfioewfwefwef} and  \ref{fekwjefeoiewfioewfwefwef1} of Definition \ref{fwoiefjeoiwfefewf} ensure  
by Example \ref{ifweoifewoufuew98fu9wefweff} that the pair $(Y,Z)$ is a uniform decomposition (Definition \ref{fewoifweifoew2324234234}) of $X $. If $I=\nat$, then the Decomposition Theorem~\ref{weiofjewi98u3298r32r32rr}  asserts that 
we have a push-out square in $\Sp\cX$
\begin{equation}\label{fjhwefuiewfhzewfzz89z29834234234}
 \xymatrix{((Y\cap Z)_{h},  Y\cap Z\cap \cY)\ar[r]\ar[d]&(Y_{h},Y\cap\cY )\ar[d]\\(Z_{h},Z\cap \cY )\ar[r]&(X_{h},\cY)}
\end{equation}

In the following we analyse the corners of \eqref{fjhwefuiewfhzewfzz89z29834234234}. Let $C $ be the set of the barycenters of the $n$-simplices of $X$ with the  induced bornological coarse structure (see Example \ref{wfijweifjewfoiuewoifoewfu98fewfwfw}). 
The restriction of the good metric to $C$ induces the    discrete uniform structure $\cT_{\disc}$. 
For $i$ in $I$ we let   
$C_{i}$  denote the  subset of  $C $ of the barycenters of the $n$-simplices contained in~$Y_{i}$. 
Then $(C_{ i})_{i\in I}$ is a big family in $C$. We can consider the  associated hybrid  space $C_{h}$.

\begin{lem}\label{efwiofweife89u9ewf324423434}
If $I=\nat$, then in $\Sp\cX$ we have an equivalence
$$(C_{h} ,(C_{ i})_{i\in I})\simeq (Z_{h},Z\cap \cY)\, .$$
\end{lem}

\begin{proof}
The inclusion $g:C\to Z$ and the projection $p:Z\to C$ which sends each component to the center
are compatible morphisms. For properness of $p$ we use the fact noticed above that the size of the simplices of $X$ is uniformly bounded with respect to $\cC$. Furthermore the compositions $g\circ p$ and $p\circ g$ are compatibly homotopic to the respective identities. By Lemma \ref{oepwfweoifewofiewofewf} we obtain morphisms
$g :  C_{h}\to Z_{h}$ and $p :  Z_{h}\to C_{h}$.  By Corollary \ref{efwjweiofiu97249873942342423423424}
$g$  induces  an equivalence
$\Yo^{s}(C_{h})\simeq \Yo^{s}(Z_{h})$. 
In general, $g$ is not a coarse equivalence, but
one  could look at Example \ref{efwiojweiofeofewfefewfewf} in order to understand the mechanism giving this equivalence. 
  The restriction  $g_{|C_{i}}  :C_{i}\to Z_{h}\cap Y_{i}$ is already an equivalence in $\BC$ for every $i$ in $I$, where $C_{i}$ has the bornological coarse structure induced from $C_{h}$. In fact, the   bornological coarse structures on $C_{i}$ and $Z_{h}\cap Y_{i}$ induced from the hybrid structures coincide with the bornological coarse structures induced from the original bornological coarse structures $C$ or $Z$, respectively.
\end{proof}


Recall the Definition \ref{ijwieorfjuwefuwe9few435345}
 of the notion of a discrete bornological coarse space.
 \begin{lem}\label{fekwfewpkfeokeopefefefewfe}
There exists a filtered family  $(W_{j},\cW_{j})_{j\in J}$ of  pairs of discrete bornological coarse spaces  and big families such that in $\Sp\cX$ we have an equivalence $$ (C_{h}, (C_{i})_{i\in I})\simeq \colim_{j\in J} (W_{j},\cW_{j}) \, .$$
\end{lem}

\begin{proof}
Using Point \ref{efwijfiewfoiefe3u40934332r} of Corollary \ref{kjeflwfjewofewuf98ewuf98u798798234234324324343} we have the equivalence of coarse motivic spectra
$$ (C_{h}, (C_{i})_{i\in I})\simeq \colim_{U\in \cC_{h}} (C_{U},(C_{i})_{i\in I} )\, .$$
Let $U$ in $\cC_{h}$ be an entourage of $C_{h}$ of the form $V\cap U_{\phi}$ for
a coarse entourage $V$ of $C$ and a cofinal function $\phi:I\to \cT_{\disc}$.  This cofinality implies that there exists $i$ in $I$ such that
$\phi(i)=\diag_{C} $. We let $T:=C_{U}\setminus C_{i}$.
Using excision (Point~\ref{fwejiofjweiofuewofewf234} of Corollary~\ref{kjeflwfjewofewuf98ewuf98u798798234234324324343}) with respect to the complementary pair $(T, (C_{i})_{i\in I} )$ we have the equivalence
$$ (C_{U},(C_{i})_{i\in I} )\simeq (T,T \cap  (C_{i})_{i\in I})\, ,$$
where   $T$ has the structures induced from $C_{U}$.
The crucial  observation is now that the coarse structure induced from $C_{U}$ on  $T$
is discrete.
\end{proof}

We now study the pair $(Y_{h},Y\cap \cY)$.  
\begin{lem}\label{eifjweiofwefwefwefewf} If $I=\nat$, then in $\Sp\cX$
we have an equivalence
$$(X^{n-1}_{h},X^{n-1}\cap \cY )\simeq (Y_{h},Y\cap \cY)\, .$$
\end{lem}
\begin{proof}
 We  consider the   retraction $r:Y\to X^{n-1}$ which is locally modeled by the radial retraction of $\Delta^n \setminus \inter(1/3 \Delta^n)$ to the boundary $\partial \Delta^n$. This is a morphism in $\BC$ which is compatible. It provides a compatible homotopy inverse to the inclusion
 $X^{n-1}\to Y$. 
 By  Corollary \ref{efwjweiofiu97249873942342423423424} this inclusion induces an equivalence
 $\Yo^{s}(X_{h}^{n-1})\simeq \Yo^{s}(Y_{h})$. We again observe that for every
 $i$ in $I$ the original bornological coarse structure  and the hybrid structure induce the same bornological coarse structure on
 $Y\cap Y_{i}$ and $X^{n-1}\cap Y_{i}$. Moreover, the restrictions of $r$ and the inclusion are inverse to each other 
 equivalences   of bornological coarse spaces.
 This implies the assertion. 
\end{proof}

 We let $D\subseteq X$ be the subset which locally is modeled by $\partial(1/2\Delta^n)\subseteq \Delta^n$. 
 After reparametrization we can consider $D$ as an $(n-1)$-dimensional simplicial complex equipped with the induced bornological coarse structure and the induced good metric from   $X$. It furthermore comes with a  big family $D\cap \cY$.
 
 \begin{lem}\label{fiojweoifjioefefefewfwf}
  If $I=\nat$, then in $\Sp\cX$ we have an equivalence
 $$(D_{h},D \cap \cY)\simeq ((Z\cap Y)_{h},(Z\cap Y\cap \cY))\, .$$
\end{lem}

\begin{proof}
We have a natural inclusion $D\to  Z\cap Y$ and we have also a (locally radial) retraction $r:Z\cap Y\to D$.
This retraction is compatible and a  compatible homotopy inverse to the inclusion.
The assertion of the lemma now follows from  Corollary \ref{efwjweiofiu97249873942342423423424}
\end{proof}

\subsection{Flasqueness of the coarsening space}\label{rgpojeroi34t35546456}

In this section we will introduce the coarsening space of Wright \cite[Sec.~4]{nw1}. 
The main result is that the coarsening space is flasque in the generalized sense if it is equipped with a suitable hybrid structure. This fact has first  been shown by Wright  \cite[Sec.~4]{nw1}. In the present section we will adapt the arguments of Wright to our setting.

\subsubsection{Construction of the coarsening space}\label{rtkohwrgwgvwerwevfds}

Let $(X,\cC,\cB)$ be a bornological coarse space,  and let $\cU=(U_{i})_{i\in I}$ be a cover of $X$, i.e., a family of subsets of $X$ such that $\bigcup_{i\in I} U_{i}=X$.  
\begin{ddd}
\mbox{}
\begin{enumerate}
\item The cover $\cU$ is bounded\index{bounded!cover} by   $V$ in $\cC$   if for every $i$ in $I$ the subset $U_{i}$ is $V$-bounded (see Definition \ref{jfewiofjweofwef234}). We say that $\cU$ is   bounded if it is $V$-bounded for some $V$ in $\cC$.
\item An entourage $V$ in $\cC$ is a Lebesgue\index{Lebesgue entourage}\index{entourage!Lebesgue} entourage\footnote{This name should resemble the notion of Lebesgue number in the case of metric spaces.} of $\cU$ if for every $V$-bounded subset $B$ of $X$ there exists $i$ in $I$ such that $B\subseteq U_{i}$.
 \end{enumerate}
\end{ddd}
The nerve $N(\cU)$ of a cover\index{nerve of a cover}
   $\cU$ is a simplicial set whose (non-degenerate) $n$-simplies are the $(n+1)$-tuples  $(i_{0},\dots,i_{n})$ in $I^{n+1}$  (without repetitions)  such that
   $\bigcap_{j=0}^{n} U_{i_{j}}\not=\emptyset$. The geometric realization of the nerve is   a  simplicial complex which  will be denoted by $\|\cU\|$.  It will be equipped 
with the path-metric induced from the spherical metric on the simplices. 
The choice of the spherical metric is required in order to have Point \ref{ewfjweofoewiio23423424} of Lemma \ref{grejhegiurz38t3z384rfwerfwfe}
which would not be true in general for other metrics, e.g., the Euclidean metric (see Example~\ref{ergi3234r534555}) {or the $\ell^1$-metric, if $\|\cU\|$ is infinite-dimensional.}
 If $\|\cU\|$ is not connected,  then we define the metric component-wise and let different components have  infinite distance. 
As in Example \ref{welifjwelife89u32or2} the metric defines a coarse structure $\cC_{\cU}$ on $\|\cU\|$.

We consider a cover $\cU=(U_{i})_{i\in I}$ of $X$. Note that we can identify the index set $I$ with the zero skeleton $\|\cU\|^{0}$ of the simplicial space  $\|\cU\|$.
 For every $x$ in $X$ we can choose $i$ in $I$ such that
$x\in U_{i}$.  This choice defines a map 
$$f:X\to \|\cU\|$$  
which actually maps $X$ to the zero skeleton.
In the other direction we define a map $$g:\|\cU\|\to X$$ such that it maps a point in the interior 
of a simplex in $\|\cU\|$ corresponding to a non-degenerate simplex $(i_{0},\dots,i_{n})$ in $N(\cU)$  to a point in the intersection
 $\bigcap_{j=0}^{n} U_{i_{j}}$.
 We use the map
$g:\|\cU\|\to X$ in order to induce a bornology $\cB_{\cU}:=g^{*}\cB$ on $\|\cU\|$ (Example~\ref{wfijweifjewfoiuewoifoewfu98fewfwfw}). 
 Recall the Definition~\ref{efijwefoie98u32234234} of a controlled map between sets equipped with coarse structures.
 
The following lemma is an easy exercise. 
\begin{lem}\label{grejhegiurz38t3z384rfwerfwfe}
\mbox{}
\begin{enumerate} 
\item The composition $f\circ g$ is $\pi/4$-close to the identity of $\|\cU\|$. \item
If $V$ in $\cC$ is a Lebesgue entourage of $\cU$, then $f:(X,\cC\langle V\rangle) \to (\|\cU\|,\cC_{\cU})$ is controlled. \item \label{ewfjweofoewiio23423424}
If  $\cU$ is bounded by $W$ in $\cC$, then $g:(\|\cU\|,\cC_{\cU})\to  (X,\cC\langle W\rangle)$ is controlled. Moreover $g\circ f$ is $W$-close to the identity of $X$. 
\item If $\cU$ is bounded, then the bornology $\cB_{\cU}$ is compatible with the coarse structure $\cC_{\cU}$.
\end{enumerate}
\end{lem}
 
The above lemma immediately implies the following corollary.

 \begin{kor}\label{lem:sdf234r2}
 \mbox{}
 \begin{enumerate}
 \item  If $\cU$ is a bounded cover, then $(\|\cU\|,\cC_{\cU},\cB_{\cU})$
 is a bornological coarse space and $g : (\|\cU\|,\cC_{\cU},\cB_{\cU})\to (X,\cC,\cB)$ is a morphism.
 \item If $V$ is a Lebesgue entourage and $W$ is a bound of the cover $\cU$ such that  $ \cC\langle V\rangle=\cC\langle W\rangle$,   then $f$ and $g$ are inverse to each other equivalences in $\BC$ between
$X_{V}$ (see Example \ref{wfeoihewiufh9ewu98u2398234324234}) and $(\|\cU\|,\cC_{\cU},\cB_{\cU})$. 
\end{enumerate}
\end{kor}

%
%

Let $\hat \cV = (\cV_i)_{i \in I}$ be a family of covers of $X$ indexed by a partially ordered set $I$.
\begin{ddd}
\index{anti-\v{C}ech system}
\label{wojwfooewfofewf234324}
 We call $\hat \cV$ an anti-\Cech system if the following two conditions are satisfied:
\begin{enumerate}
\item  For all $i,j$ in $I$ with $j<i$ exists a Lebesgue entourage of $\cV_{i}$ which is a bound of $\cV_{j}$. 
\item Every entourage of $X$ is a Lebesgue entourage for some member  of the family $\hat \cV$. 
 \end{enumerate}
\end{ddd}

\begin{ex}\label{wefhweofwfwefwefw}
Note that the notion of an anti-\Cech system only depends on the coarse structure. Let $(X,\cC)$ be a coarse space. For every entourage $U$ in $\cC$ we consider the cover $\cV_U$ consisting of all $U$-bounded subsets of $X$. Then $(\cV_{U})_{U\in \cC}$ is an anti-\Cech system. In this case $\cV_{U}$ is $U$-bounded and has  $U$ as a Lebesgue entourage.
\hB
\end{ex}

We will now construct the coarsening space.
We   consider a bornological coarse space $(X,\cC,\cB)$  {and} assume that 
$\cC=\cC\langle U\rangle$ for  an entourage $U$ in $\cC$.   We consider an 
anti-\Cech system $\hat \cV = (\cV_n)_{n \in \nat}$ indexed by the natural numbers  {and} assume that $U$ is a Lebesgue entourage for $\cV_{n}$ for all $n$ in $ \nat$. Our assumption on $X$ implies that   $\cC\langle U\rangle=\cC\langle V\rangle$ for every coarse entourage $V$ of $X$ containing $U$. Hence $(X,\cC,\cB)$ and $(\|\cV_{n}\|,\cC_{\cV_{n}},\cB_{\cV_{n}})$ are equivalent bornological coarse spaces by Corollary~\ref{lem:sdf234r2}.


For every $n$ in $\nat$ we can choose a refinement $\cV_{n}\to \cV_{n+1}$. In detail, let  $\cV_{n}=(V_{i})_{i\in I_{n}}$.  Then a refinement is a map
$\kappa_{n}:I_{n}\to I_{n+1}$ such that $V_{i}\subseteq V_{\kappa(i)}$ for all $i$ in $I_{n}$. Such a refinement determines a morphism of simplicial sets $N(\cV_{n})\to N(\cV_{n+1})$ and finally a map
  \begin{equation}\label{gerij234oijo4ifjofnewofewfew34r4}
\phi_n :  \|\cV_n\| \to \|\cV_{n+1}\|
\end{equation}
between the geometric realizations of the nerves. Note that these maps are contracting {since we are using the spherical metric on the simplices (they would be also contracting for the $\ell^1$-metric on the simplices, but in general not for the Euclidean metric).}
  
%

We  construct the coarsening space $(\|\hat \cV\|,\cC_{\hat \cV},\cB_{\hat \cV})$ following Wright \cite[Def.~4.1]{nw1}. It will furthermore be equipped with a uniform structure $\cT_{\hat \cV}$.
  The underlying topological space of the coarseninig space\index{coarsening space}\index{space!coarsening} is given by
 \begin{equation}\label{fwhhefwifiui23r2r23423434}
\|\hat \cV\| := \big( [0,1] \times \|\cV_0\| \big) \cup_{\phi_0} \big( [1,2] \times \|\cV_1\| \big) \cup_{\phi_1}([2,3]\times \|\cV_{2}\|)\cup_{\phi_{2}} \cdots.
\end{equation} 
We triangulate the products $[n,n+1] \times \|\cV_n\|$ for every $n\in \nat$ in the standard way. The structure of a simplicial complex provides a path-metric. These path-metrics induce a path-metric on $\|\hat \cV\|$, which in turn induces a coarse structure $\cC_{\hat \cV}$ and a uniform structure~$\cT_{\hat \cV}$. 
The bornology $\cB_{\hat \cV}$ is 
 generated by sets of the form $[n,n+1] \times B$ for all $n$ in $\nat$ and   $B$ in $\cB_{\cV_{n}}$.
 One checks that this bornology is compatible with the coarse structure. We thus get a bornological coarse space $(\|\hat \cV\|,\cC_{\hat \cV},\cB_{\hat \cV})$.
 Furthermore, the uniform structure $\cT_{\hat \cV}$  is compatible, see Example \ref{ex:jnksdf34}.  {Since we are using the spherical metric on the simplices, we get that for every $n$  in $\IN$ the subspace $\{n\} \times \|\cV_n\|$ of $\hat \cV$ with the coarse structure induced from $\cC_{\hat \cV}$ is equivalent to $\|\cV_n\|$ with the coarse structure $\cC_{\cV_{n}}$ \cite[Lem.~A.5]{nw1}. It is not known to the authors whether this property also holds if we equip our simplicial complexes with the $\ell^1$-metric or with the Euclidean metric.}
 
 If we choose another family of refinements  $(\kappa_{n}^{\prime})_{n\in \nat}$, then we get
 a different simplicial complex $\|\hat \cV \|^{\prime}$. It consists of the same pieces $[i,i+1)\times \|\cV_{i}\|$ which are glued differently. One can check that the obvious (in general non-continuous) bijection which identifies these pieces is an equivalence of bornological coarse spaces between  $\|\hat \cV \|$ and $\|\hat \cV \|^{\prime}$.
 
\begin{rem}
   Philosophically, the choice of the family of refinements $(\kappa_{n})_{n\in \nat}$ should be part of  the data of an  anti-\Cech system. Then the notation $\|\hat \cV \|$ would be unambiguous. 
   But we prefer to use the definition as stated since in arguments below we need the freedom to change the
   family  of refinements. 
   Also Example \ref{wefhweofwfwefwefw} would need some adjustment.
\hB
\end{rem}

We show below flasqueness of the coarsening space for two different coarse structures, which we define now.

We consider the   projection to the first coordinate
\begin{equation}\label{fwehfwuiefhiewufuiewhfiewf2343}
\pi:\|\hat \cV\|\to [0,\infty)
\end{equation}
and   define the subsets $\|\hat \cV\|_{\le n}:=\pi^{-1}([0,n])$. We observe that $(\|\hat \cV\|_{\le n})_{n\in \nat}$ is a big family.
Following  Wright \cite{nw1} we introduce the following notation:  
  \begin{itemize}
\item By  $\|\hat \cV\|_0$ we denote $\|\hat \cV\|$ equipped with the $\cC_0$-structure (see Example \ref{ex:df8734r}). 
\item\label{wfeopfjweoif234434} By  $\|\hat \cV\|_h$ we denote $\|\hat \cV\|$ with the hybrid structure (see Definition \ref{defn:sdfuh8934t}) associated to the big family $(\|\hat \cV\|_{\le n})_{n \in \IN}$.
\end{itemize}

\subsubsection{Flasqueness for the \texorpdfstring{$\cC_0$}{C0}-structure}

We consider a bornological coarse space $(X,\cC,\cB)$,  an entourage $U$ of $X$, and an anti-\v{C}ech system $\hat \cV$ for $X$.
\begin{theorem}[{\cite[Thm.~4.5]{nw1}}]
\label{thm:sdf934v}
 We assume that
 \begin{enumerate}
  \item $\cC=\cC\langle U\rangle$,
  \item  $\cB$ is generated by  $\cC$ (see Example \ref{wfijiofweofewfwefewfewf}), and
  \item  for every $n$ in $\nat$ there exists a bound $U_{n}$ in $\cC$ of $\cV_{n}$ such that $U_{n}^{4}$ is a Lebesgue entourage of $\cV_{n+1}$.
  \end{enumerate}
   Then  $\|\hat \cV\|_0$ is flasque in the generalized sense.
\end{theorem}

\begin{proof}
We will apply Theorem~\ref{thm:sdf09234}. Every connected component of $\|\hat \cV\|$ meets the front face $  \{0\} \times \|\cV_0\|$.
For every component we choose a base point in this front face.
 We then define the function
$$e:\|\hat \cV\| \to [0,\infty)\, , \quad e(x) := d(*_{x},x)\, ,$$ 
 where $*_{x}$ denotes  the chosen base point $*_{x}$ in the component of $x$. 
 This function is Lipschitz continuous with Lipschitz constant~$1$. It is  therefore  uniformly continuous and controlled. 
 By Corollary  \ref{lem:sdf234r2} the bornology on the pieces $\|\cV_{n}\|$ is the bornology of metrically bounded subsets.
 This implies that $e$ is also  bornological.
 Consequently,   
 $e$ satisfies the three  conditions listed in Section \ref{ekfjweofjweoifewfwefwefewf}.   
 Furthermore note that the family $(\{e \le n\})_{n \in \IN}$ of subsets  is a cofinal sequence in the bornology $\cB_{ \hat \cV}$. Hence the hybrid structure determined by $e$ is the $\cC_{0}$-structure.

It remains to construct the  morphism $\Phi  \colon  [0,\infty)\ltimes \|\hat \cV\|\to \|\hat \cV\|$ with the properties listed in  Section \ref{ekfjweofjweoifewfwefwefewf}. First note that we are free to choose the  family of  refinements $(\kappa_{n}\colon  \cV_n\to\cV_{n+1})_{n\in \nat}$
since being flasque in the generalized sense is invariant under equivalences of bornological coarse spaces.

We will choose the refinements later after we have explained which properties they should have. For the moment we adopt some choice. Then  
we define  \cite[Def.~4.3]{nw1}  the map $\Phi  :  [0,\infty)\ltimes \|\hat \cV\| \to \|\hat \cV\|$ by
\begin{equation}
\label{eq:fweh9432r}
\Phi(t,(s,x)) :=
\begin{cases}
\big( t,(\phi_{\lfloor t \rfloor -1 } \circ \cdots \circ \phi_{\lfloor s \rfloor})(x) \big) & \text{ for }  \lfloor s \rfloor < \lfloor t \rfloor \\
\big( t,  \phi_{\lfloor s \rfloor}(x) \big) & \text{ for }  s \le t,   \lfloor s \rfloor = \lfloor t \rfloor \\
 \big( s,x \big)& \text{ for } s \ge t
\end{cases}
\end{equation}
(see \eqref{gerij234oijo4ifjofnewofewfew34r4} for $\phi_{n}$).
Note that  \begin{equation}\label{vfsvrewwesdfvsfdvsdfverfv}
\Phi(t,\Phi(s,(u,x))=\Phi(\max\{t,s\},(u,x))\, .
\end{equation} 
We must check that $\Phi$ has the properties required for  Theorem~\ref{thm:sdf09234}. 
By definition of $\Phi$ we have $\Phi_{0}=\id$ and  Conditions \ref{sdjnfwersfd43} and \ref{qrgijoqwfdeqdwqedq} are satisfied. Using that 
the maps $\phi_{n}$ are contracting we see that $\Phi$ is uniformly continuous.
Since the maps $\phi_{n}$ are coarse equivalences one furthermore checks that
$\Phi$ is a morphism in $\BC$.

The only non-trivial one is that $\Phi$ is contracting. Below we show that we can choose the refinements for the anti-\v{C}ech system such that $\Phi$ is contracting. 
 So by Theorem~\ref{thm:sdf09234} we conclude that $\|\hat \cV\|_0$ is flasque in the generalized sense.

We now reprove \cite[Lem.~4.4]{nw1} in our somewhat more general setting, which will imply that   $\Phi$ is contracting {(we will give the argument for this at the end of this proof).} The statement is that the refinements $\kappa_n  \colon  I_n \to I_{n+1}$ can be chosen with the following property: for every bounded subset $B$ of $X$ which is contained in a single coarse component (see Definition \ref{ekhqiuheiuqe123122342334}) there exists an integer $N$ such that {for all $k$ in $\nat$ with $k \le N$.}
\begin{equation}
\label{eq:dsf98234}
(\kappa_{N}\circ \kappa_{N-1}\circ \dots \circ {\kappa_{k}})(\{i\in {I_{k}} :   U_{i}\cap B\not=\emptyset \}) \subseteq I_N
\quad \text{consists of a single point}\, .
\end{equation}
 
As a first step we observe that we can assume that $X$ is coarsely connected. Indeed we can do the construction below for every coarse component separately. 
%
%

We now assume that $X$ is coarsely connected and we fix a base point $x_{0}$ in $X$. For every $n$ in $\nat $ we let $J_{n}$  be the index subset of $I_{n}$ of members of the covering $\cV_{n}$ which intersect   $U_{n}[x_{0}]$ non-trivially. We now define the connecting maps  with  the property that $\kappa_{n}(J_{n})$ consists of a single point: 
Since $U_{n}$ is a bound of $\cV_{n}$ the union $\bigcup_{j\in J_{n}} V_{j}$ is bounded by $U_{n}^{4}$. Since this is by assumption a Lebesgue entourage of $\cV_{n+1}$ there exists $i$ in $I_{n+1}$ such that
$V_{j}\subseteq V_{i}$ for all $j$ in $J_{n}$. We define $\kappa_{n}(j):=i$ for all $j$ in $J_{n}$. Then we extend $\kappa_{n}$ to $I_{n}\setminus J_{n}$.

Assume now that $B$ is bounded. Since we assume that the bornology is generated by the coarse structure  there exists an entourage $U_{B}$ in $\cC$ such that
$B\subseteq U_{B}[x_{0}]$. If we choose $N$ in $\nat$ so large that $U_{B}\subseteq {U_{N}}$, then 
\eqref{eq:dsf98234} holds.

We now show now that \eqref{eq:dsf98234} implies that $\Phi$ is contracting. Given $i$ in $\IN$, $\varepsilon$ in $(0,\infty)$   and a coarse entourage $V$ of $\|\hat \cV\|$ we must provide a number  $T$ in $(0,\infty)$ with the property
\[\forall x,y\in \|\hat \cV\|, \ \forall t\in [T,\infty) \: : \:   (x,y)\in V \cap (Y_{i}\times Y_{i}) \Rightarrow (\Phi_{t}(x),\Phi_{t}(y)) \in U_{\epsilon}\, .\]
Here $Y_i = \{e \le i\}$ is a member of the big family which is used to define the $\cC_0$-structure. Now $Y_i$ is a bounded subset of $\|\hat \cV\|$. In particular, it is contained in a finite union of generating bounded subsets which are of the form $[n,n+1] \times B$ with $B$  in $\cB_{\cV_n}$ for $n$ in $\nat$.
%
%

Let $x$ and $y$ be points in  $Y_i$. There are an integer~$j$, natural numbers $n_1, \ldots, n_j  $ and bounded subsets $B_l$ in $\cB_{\cV_{n_l}}$ for $1 \le l \le j$ such that $Y_i \subseteq \bigcup_{l = 1}^j [n_l, n_l + 1] \times B_{n_l}$. Corollary~\ref{lem:sdf234r2} provides an equivalence of bornological coarse spaces $\omega_0 : X \to \|\cV_0\|$, and we get equivalences $\omega_r : X \to \|\cV_r\|$ by composing $\omega$ with the equivalence $\phi_{r-1} \circ \cdots \circ \phi_0$, see \eqref{gerij234oijo4ifjofnewofewfew34r4}. We get a bounded subset
\[B_X := \omega^{-1}_{n_0}(B_{n_0}) \cup \ldots \cup \omega^{-1}_{n_l}(B_{n_l})\]
of $X$. By \eqref{eq:dsf98234} we get a corresponding integer $N$, and then $T := \max\{N + 1, i+1\}$   does the job. 
\end{proof}

\begin{rem}
If $\hat \cV$ is an anti-\v{C}ech system on a bornological coarse space $X$ whose coarse structure has a single generating entourage and whose bornology is the underlying one, then by passing to a subsystem we can satisfy the third condition of Theorem \ref{thm:sdf934v}.
\hB
\end{rem}

\subsubsection{Flasqueness for the hybrid structure}

We now discuss flasqueness of $\|\hat \cV\|_{h}$.

Let $(X,\cC)$ be a coarse space.
\begin{ddd}[{\cite[Def.~2.7]{dadarlat_guentner}}]
\index{exact space}\index{space!exact}
\label{ljweofijweoifuewfewu9fewfwefewfewfewfewf}
  $(X,\cC)$ is exact, if for all  $U$ in $\cC$ and all $\epsilon$ in $(0,\infty)$ there exists a   bounded 
cover $\cW = (W_i)_{i\in I}$   and a partition of unity $(\varphi_i)_{i \in I}$   subordinate to $\cW$ such that
\begin{equation}
\label{eq:xvcjnk34}
\forall (x,y)\in U :\:\: \sum_{i \in I} |\varphi_i(x) - \varphi_i(y)| \le \epsilon\, .
\end{equation}
\end{ddd}

\begin{rem}\label{dqjoiqiwqwodu98u89123213}
We assume that  $(X,d)$ is a metric space and   consider the coarse structure $\cC:=\cC_{d}$ on $X$ defined by the metric. In this situation Dadarlat--Guentner \cite[Prop.~2.10]{dadarlat_guentner} proved that we have the chain of implications
$$\xymatrix{&\text{finite asymp.\ dim.}\ar@{=>}[d]&\\\text{Property A}\ar@{=>}[r]& \text{exact}\ar@{=>}[r]& \text{uniformly embeddable}  }$$
Here uniformly embeddable\index{uniformly!embeddable} abbreviates uniformly embeddable into a Hilbert space \cite{yu_embedding_Hilbert_space}. 

Property A\index{Property A} has been introduced by Yu~\cite[Sec.~2]{yu_embedding_Hilbert_space}, where he also showed that finitely generated, amenable groups have Property A \cite[Ex.~2.3]{yu_embedding_Hilbert_space}.

A metric space {$(X,d)$} is uniformly discrete if there exists an $\epsilon$ in $(0,\infty)$   such that $d(x,y) > \epsilon$ for all $x, y$ in $X$ with $ x\not= y $.
{Such a space} has bounded geometry\index{bounded geometry!of a discrete metric space} if for all $r $ in $(0,\infty)$ we have $$\sup_{x \in X} |B_d(x,r)| < \infty\, ,$$ where $|A|$ denotes the number of points in the set $A$.
Equivalently, {a uniformly discrete metric space $X$ has bounded geometry, if} the bornological coarse space $X_{d}$ has strongly bounded geometry in the sense of Definition \ref{ejfwjefklwejlkewfwefwefwefewf}.
If $(X,d)$ is {uniformly} discrete and has bounded geometry, then exactness is equivalent to Property A {\cite[Prop.~2.10(b)]{dadarlat_guentner}}. {Tu proved that under the bounded geometry assumptions there are many more equivalent reformulations of Property A \cite[Prop.~3.2]{tu}.}

Higson--Roe \cite[Lem.~4.3]{higson_roe_propA} proved  that bounded geometry metric spaces with finite asymptotic dimension\index{finite asymptotic dimension}  {(Definition~\ref{frjewifoewfewfewf})} have Property A.  {Hence such spaces are exact.}

 {Willett \cite[Cor.~2.2.11]{w_exact} has shown that a 
 metric space of finite asymptotic dimension is exact. Note that the cited corollary states that spaces of finite asymptotic dimension have Property A under the assumption   of bounded geometry. A close examination of the proof reveals that it is first shown that spaces of finite asymptotic dimension are exact, and then  bounded geometry  is used to verify   Property A.}

Exactness has the flavor of being a version of coarse para-compactness. The relationship of Property A to versions of coarse para-compactness was further investigated in \cite{cencelj_dydak_vavpetic} and \cite{dydak}.
\hB
\end{rem}

 {
Let $(X,\cC)$ be a coarse space.
\begin{ddd}[{{\cite[Sec.\ 1.E]{gromov}}}]\label{frjewifoewfewfewf}
  $(X,\cC )$ has finite asymptotic dimension\index{asymptotic dimension}\index{finite asymptotic dimension} if it admits an anti-\Cech system $\hat \cV=(\cV_{n})_{n\in \nat}$ with $\sup_{n\in \nat} \dim \|\cV_{n}\|<\infty$.
\end{ddd}
}

We consider a bornological coarse space $X$ with coarse structure $\cC$ and bornology $\cB$. Furthermore, let     $\hat \cV=(\cV_{n})_{n\in \nat}$ be an  anti-\Cech system on $X$.
\begin{theorem}[Wright~{\cite[Thm.~5.9]{nw1}}]
\label{thm:sdfdvw4v}
Assume:\begin{enumerate}
  \item There exists  $U$ in $\cC$ such that  $ \cC=\cC\langle U\rangle$.
  \item \label{oifjqoefewfqwfqewf} $\sup_{n\in \nat} \dim \|\cV_{n}\|<\infty$.
\end{enumerate}
Then   the bornological coarse space $\|\hat \cV\|_h$ is flasque in the generalized sense.
\end{theorem}

\begin{proof}
Since $X$ admits an anti-\Cech system $\hat \cV=(\cV_{n})_{n\in \nat}$ with $\sup_{n\in \nat} \dim \|\cV_{n}\|<\infty$  
it has finite asymptotic dimension.

We will apply Theorem~\ref{thm:sdf09234}. We let $e:=\pi$   be the projection defined in \eqref{fwehfwuiefhiewufuiewhfiewf2343}. It is straightforward to check that $e$ is uniformly continuous, controlled and bornological, i.e., $e$ satisfies all  three conditions  listed in Section \ref{ekfjweofjweoifewfwefwefewf}. 

We must  provide the morphism $\Phi_{h}  :  [0,\infty)\ltimes \|\hat \cV\| \to \|\hat \cV\|$ (we use the subscript $h$ in order to distinguish it from the morphism $\Phi$ in Equation \eqref{eq:fweh9432r}).

Corollary~\ref{lem:sdf234r2} provides an equivalence of bornological coarse spaces $\omega_0 : X \to \|\cV_0\|$, and we get equivalences $\omega_r : X \to \|\cV_r\|$ by composing $\omega$ with the equivalence $\phi_{r-1} \circ \cdots \circ \phi_0$, see \eqref{gerij234oijo4ifjofnewofewfew34r4}.  {Since asymptotic dimension is a coarsely invariant notion it follows that  each $\|\cV_r\|$ also has finite asymptotic dimension. Furthermore, the zero skeleton $\|\cV_r\|^{0} $ with the induced bornological coarse structure is coarsely equivalent to $\|\cV_r\|$ and therefore also has finite asymptotic dimension. By Willett \cite[Cor.~2.2.11]{w_exact} a quasi-metric space of finite asymptotic dimension is exact, hence $\|\cV_r\|^{0}$ is exact.} For each $r$ we fix an inverse equivalence $\omega_r^0$ to the inclusion $\|\cV_r\|^{0} \to \|\cV_r\|$.

The argument below will involve the following construction.
It depends on an integer $T$ in $\nat$, a number $\varepsilon$ in $(0,\infty)$ and an entourage $V$ of $\|\cV_{T}\|^{0}$. 
The construction then produces an integer $S$ with $T\le S$ and a map $ \Psi_{T,S} :  \|\cV_T\| \to \|\cV_S\|$
such that \begin{equation}
\label{eq_43wwwt23e3}
\forall (x,y)\in V : \operatorname{dist}(\Psi_{T,S}(x),\Psi_{T,S}(y)) \le C\cdot \varepsilon\, ,
\end{equation}
where $C > 0$ is a constant only depending on the dimension of $\|\hat \cV\|$.  At this point we use Assumption~\ref{oifjqoefewfqwfqewf}.
Note that
\eqref{eq:xvcjnk34} gives an estimate with respect to the $\ell^{1}$-metric and the constant arrises since we work with the spherical metric.

 We now describe the construction. 
Let  $T$ in $\IN$,  an entourage $V$ {of  $\|\cV_T\|^{0}$}, and $\epsilon$ in $(0,\infty)$ be given.   By  exactness of $\|\cV_T\|^{0}$ we can find a   bounded cover $\cW_T$ of $\|\cV_T\|^{0}$ and a subordinate partition of unity $(\varphi_W)_{W \in \cW_T}$ with Property \eqref{eq:xvcjnk34} for the given $\varepsilon$ and entourage $V$ (in place of $U$). We define a map $\Psi_T :  \|\cV_T\|^{0} \to \|\cW_T\|$~by
\[\Psi_T(x) := \sum_{W \in \cW_T} \varphi_W(x)[W]\, .\]

Note that we can identify $\|\cV_{T}\|^{0}=I_{T}$.
Since $\hat \cV$ is an anti-\Cech system and $\cW_T$ is   bounded, there is an $S$ in $\IN $ with $S > T$ such that for every member $W$ of $\cW_T$  there exists an element denoted by $\kappa(W)$ in $  I_{S}$ such that $V_{i}\subseteq V_{\kappa(W)}$ for all $i$ in $W$. The map $\kappa$ from the index set of $\cW_{T}$ to $I_{S}$ determines a map  $\phi_{T,S} :  \|\cW_T \|\to \|\cV_S\|$.    The composition $\phi_{T,S}\circ \Psi_T :  \|\cV_T\|^{0} \to \|\cV_S\|$ can be extended linearly to a map $\Psi_{T,S} :  \|\cV_T\| \to \|\cV_S\|$. 
%
%
%

We now repeatedly apply the construction above. {We start with the data $T_0 := 0$, $\varepsilon_{0} := 1$ and $V := ((\omega_0^0 \circ \omega_0) \times (\omega_0^0 \circ \omega_0))(U)$, where the entourage $U$ of $X$ is as in the statement of this theorem. We get  the integer $T_1 $  and a map $\Psi_{0,T_1} : \|\cV_0\| \to \|\cV_{T_1}\|$. In the $n$'th step we use the data $T_n$, $\varepsilon_{n} := 1/n$ and $V := ((\omega_{T_n}^0 \circ \omega_{T_n}) \times (\omega_{T_n}^0 \circ \omega_{T_n}))(U^n)$, and get a  integer $T_{n+1}$  and a map $\Psi_{T_n,T_{n+1}} : \|\cV_{T_n}\| \to \|\cV_{T_{n+1}}\|$.}

 
Let $i$ in $I$ be given and assume that $(s,x)$ is a point in $[T_{i-1},T_{i-1}+1) \times \|\cV_{T_{i-1}}\|$. If 
$\Phi$ is the map from Equation~\eqref{eq:fweh9432r}, then
  we define $\Theta_{i-1}(s,x) \in \{T_{i}\} \times \|\cV_{T_{i}}\|$ by linear interpolation  (writing $s=T_{i-1}+\bar{s}$)
\[\Theta_{i-1}(s,x) := (1-\bar{s})(T_{i},\Psi_{T_{i-1},T_{i}}(x)) + \bar{s}\Phi(T_{i},(s,x))\, ,\] 
where $\Phi$ is the map from Equation~\eqref{eq:fweh9432r}.

We now define the map $\Phi_{h}  :  [0,\infty)\ltimes \|\hat \cV\| \to \|\hat \cV\|$ by
\[\Phi_{h}(t,(s,x)) :=
\begin{cases}
\Phi(t,  \Psi_{T_{i(t)-1},T_{i(t)}} \circ \Phi_{T_{i(t)-1}}(s,x))& \text{ for }s \le T_{i(t)-1}\\
{\Phi(t,\Theta_{i(t)-1}(s,x))} & {\text{ for } T_{i(t)-1} \le  s < T_{i(t)-1}+1}\\
\Phi(t,(s,x))& {\text{ for } T_{i(t)-1}+1 < s \le t}\\
(s,x)& \text{ for } s > t
\end{cases}\]
where we have used the notation
\begin{equation}\label{eq1243trgdfwqe}
i(t) := \sup_{T_i \le t} i\, .
\end{equation}
In order to check {the} continuity at $s=T_{i(t)-1}+1$ we use  the relation \eqref{vfsvrewwesdfvsfdvsdfverfv}.

We have to check that $\Phi_{h}$ has the required properties listed above Theorem~\ref{thm:sdf09234}. The functions
 $\Phi$ and $\Psi$ are $1$-Lipschitz. This implies that $\Phi_{h}$ is uniformly continuous. The only remaining non-obvious property is contractivity.

It remains to show that $\Phi_h$ is contractive. Given $j$ in $\IN$, $\varepsilon$ in $(0,\infty)$, and a coarse entourage $V$ of $\|\hat \cV\|$ we have to provide a number  $T  $ in $(0,\infty)$ with the property
\begin{align*}
\forall (s,&x) ,(r,y) \in \|\hat \cV\|, \ \forall t\in [T,\infty) \: :\\
& ((s,x),(r,y))\in V \cap (Y_{j}\times Y_{j}) \Rightarrow (\Phi_h(t,(s,x)),\Phi_h(t,(r,y))) \in U_{\epsilon}\, .
\end{align*}
We define a map (which is not a morphism in $\BC$) $\Omega : \|\hat \cV\| \to X$  which sends the point $(s,x)$  in $\|\hat \cV\|$   to $\omega^\prime_m(x)$ for $s$  in $[m,m+1)$, where $\omega^\prime_m$ is a choice of inverse equivalence to $\omega_m$. Let $n$ in $\IN$ be such that $(\Omega \times \Omega)(V \cap (Y_{j}\times Y_{j})) \subseteq U^n$. We choose $T $ in $\IR$ with
\begin{enumerate}
\item $i(T)-1>\max\{\lceil C/ \varepsilon\rceil,n\}$ (see \eqref{eq1243trgdfwqe}), and
\item $T_{i(T)-1}>j$.
\end{enumerate}
The second inequality implies  that the application of $\Phi_h$ to $(t,(s,x))$ and to $(t,(r,y))$ will fall into the first case of the definition of $\Phi_h$, i.e., the function $\Psi_{T_{i(T)-1},T_{i(T)}}$ will be applied. But the latter function, by the first inequality,  maps pairs of points lying in the entourage $(\omega_{T_{i(T)-1}} \times \omega_{T_{i(T)-1}})(U^n)$ to pairs of points lying a distance at most $\varepsilon$ from each other {due to \eqref{eq_43wwwt23e3} and the choice of $\varepsilon_{n}=1/n$ in the construction above.} Since $\Phi$ is $1$-Lipschitz we get the desired statement for $\Phi_h$.
\end{proof}

\begin{rem}
We do not know whether the   assumption  {of finite asymptotic dimension}  in Theorem~\ref{thm:sdfdvw4v} is necessary. It would be interesting to construct a space $X$ such that its coarsening space equipped with the hybrid coarse structure is not flasque. One could try to adapt the constructions of surjectivity counter-examples to the coarse Baum--Connes conjecture (\`{a} la Higson \cite{counterex_coarse_BC}) in order to show that there exists a non-trivial element in the coarse $K$-homology $\KX_{\ast}(\|\hat \cV\|_h)$ in the case of an expander $X$. But we were not able to carry out this idea.
\hB
\end{rem}

\subsection{The motivic coarse spectra of simplicial complexes and coarsening spaces}

 Recall Definition \ref{defn:sdf09232}.
Let $\cA_{\disc}$ denote the set of discrete bornological coarse spaces. 
 
   We consider a simplicial complex $X$ equipped with a bornological coarse structure $(\cC,\cB)$, a metric $d$, and an ordered family of subcomplexes $\cY:=(Y_{n})_{n\in \nat}$. Let $\cT_{d}$ denote the uniform structure associated to the metric. In the statement of the following theorem $X_{h}$ denotes the set $X$ with the hybrid structure defined in Definition \ref{defn:sdfuh8934t}.
 We furthermore use the  notation \eqref{fkhwiufuiz823zr824234242424234234234234}.
 \begin{theorem}\label{qdhdqwidqwihdqwdqwd}
We assume:  \begin{enumerate}
\item $d$ is good (Definition \ref{fwoiefjeoiwfefewf}). 
\item   $\cC$ and $\cT_{d}$ are compatible (Definition \ref{defn:sd23d9023}). 
\item The family $(Y_{n})_{n\in \nat}$ is big (Definition \ref{ifjweijwoiefewfewf}). 
\item For every  $B$ in $\cB$  there exists $i$ in $\nat$ such that $  B\subseteq Y_{i}$.
\item {$X$ is finite-dimensional.}
\end{enumerate}
 Then we have 
 $ (X_{h},\cY)\in  \Sp\cX\langle\cA_{\disc}\rangle$
 \end{theorem}
 \begin{proof}
 We argue by induction on the dimension. We fix   $n$ in $\nat$ and  assume that the theorem has been shown for all dimensions $k$ in $\nat$ with $k< n$. We now assume that  $\dim(X)=n$.
We   consider the decomposition  \eqref{fjhwefuiewfhzewfzz89z29834234234}.
Combining Lemmas \ref{efwiofweife89u9ewf324423434} and \ref{fekwfewpkfeokeopefefefewfe}
 we see that the lower left corner of  \eqref{fjhwefuiewfhzewfzz89z29834234234} belongs to  $\Sp\cX\langle\cA_{\disc}\rangle$.
In view of the Lemmas \ref{eifjweiofwefwefwefewf} and \ref{fiojweoifjioefefefewfwf} the induction hypothesis implies
that the two upper corners of the push-out square belong to $\Sp\cX\langle\cA_{\disc}\rangle$.
We conclude that $(X_{h},\cY)\in \Sp\cX\langle\cA_{\disc}\rangle$.
\end{proof}
 

Let $(X,\cC)$ be a coarse space. 
\begin{ddd}\label{wegoijobgwtwferfrewfer}
$X$ has weakly finite asymptotic dimension if here exists a cofinal subset   $\cC^{\prime}$ of $ \cC$ such that
for every $U$ in $ \cC^{\prime}$ the coarse space  $(X, \cC\langle U\rangle)$ has finite asymptotic dimension {(Definition~\ref{frjewifoewfewfewf}).}
\end{ddd}
If $X$ is a bornological coarse space, then we apply this definition to its underlying coarse space.

We consider a bornological coarse space $X$.
\begin{theorem}\label{jfweofjwoeifjewfoewfewfewfewf}
 If $X$ has weakly finite asymptotic dimension, then we have $\Yo^{s}(X )\in \Sp\cX\langle\cA_{\disc}\rangle$.
\end{theorem}

\begin{proof}
Let $\cC$ be the coarse structure of $X$, and let $\cC'$ be as in Definition \ref{wegoijobgwtwferfrewfer}.
By $u$-continuity we have
 $$\Yo^{s}(X)\simeq \colim_{U\in \cC^{\prime}} \Yo^{s}(X_{U})\, .$$
 It {therefore} suffices to show that $\Yo^{s}(X_{U})\in   \Sp\cX\langle\cA_{\disc}\rangle$ for all $U$ in $\cC^{\prime}$.

 We now consider $U$ in $\cC^{\prime}$. The bornological coarse space $X_{U}$ admits an anti-\Cech system $\hat \cV$ satisfying the condition formulated in Definition \ref{frjewifoewfewfewf}. 
 
   We consider the coarsening space 
 $\|\hat \cV\|_{h}$ with the big family
 $( \|\hat \cV\|_{\le n})_{n\in \nat}$. By~Theorem~\ref{thm:sdfdvw4v} and Lemma \ref{ewfifjewiofoiewfewfewfef} we have $\Yo^{s}(\|\hat \cV\|_{h})\simeq 0$.    By Point \ref{ifjweifjewiojwefw231} of Corollary 
 \ref{kjeflwfjewofewuf98ewuf98u798798234234324324343} this implies
 $$\Sigma \Yo^{s}(( \|\hat \cV\|_{\le n})_{n\in \nat})\simeq (\|\hat \cV\|_{h} ,(\|\hat \cV\|_{\le n})_{n\in \nat})\, .$$
 We     apply Theorem  \ref{qdhdqwidqwihdqwdqwd} to the pair $(\|\hat \cV\|_{h} ,(\|\hat \cV\|_{\le n})_{n\in \nat})$ in order to conclude that $$\Sigma \Yo^{s}(( \|\hat \cV\|_{\le n})_{n\in \nat})\in  \Sp\cX\langle\cA_{\disc}\rangle\, .$$
 We use that
 $X_{U}\to \|\hat \cV\|_{0} \to  \|\hat \cV\|_{\le n}$ is an equivalence in $\BC$ for every $n$ in $\nat$ and consequently
   $\Yo^{s}(X_{U})\simeq  \Yo^{s}(( \|\hat \cV\|_{\le n})_{n\in \nat})$ (see \eqref{wefweew254} for notation) in $\Sp\cX$.
It follows that  
 $\Yo^{s}(X_{U})\in   \Sp\cX\langle\cA_{\disc}\rangle$ as required.
\end{proof}

\part{Coarse and locally finite homology theories}

\section{First examples and comparison  of coarse homology theories}\label{feiojoiwfiowfu234234324}

The notion of a coarse homology theory was introduced in Definition \ref{rgljogreggregrege}. We first show that 
the condition of $u$-continuity can be enforced. This result may be helpful for the construction of coarse homology theories. 
We furthermore discuss some additional additivity properties for coarse homology theories.  

We then construct first examples of coarse homology theories, namely coarse ordinary homology and the coarsification of stable homotopy. 

Before we discuss further  examples of coarse homology theories in subsequent sections we
 state and prove the comparison theorems which motivated this book.

\subsection{Forcing \texorpdfstring{$u$}{u}-continuity}\label{egihfiebgbsdbsbfdbs}

Let $\bC$ be a  cocomplete stable   $\infty$-category.
In the following we describe a construction which turns a functor 
$E \colon \BC \to {\bC}$  into  a $u$-continuous\index{$u$-continuous}\index{continuous!$u$-} one $E_u\colon \BC \to {\bC}$.

We will have a natural trans\-for\-ma\-tion $E_{u}\to E$ which is an equivalence 
on  most bornological coarse spaces of interest (e.g., path metric spaces like Riemannian manifolds, or finitely generated groups with a word metric). {Furthermore,} if $E$ satisfies {one of the conditions} coarse invariance, excision or vanishes on flasque spaces, then $E_u$ retains these properties.

The obvious idea to define $E_u : \BC \to {\bC}$ is to set
\begin{equation}
\label{erwe999424}
E_u(X ) := \colim_{U \in \cC} E(X_U)\ , 
\end{equation}
where $\cC$ denotes the coarse structure of $X$ and we use the notation introduced in Example \ref{wfeoihewiufh9ewu98u2398234324234}.
This formula  defines the value of $E_u$  on objects, only.  In order to define $E_{u}$ as a functor   we first introduce the category $\BC^{\cC}$. Its objects are pairs $(X ,U)$ of a bornological coarse space $X$ and a coarse  entourage $U$ of  $X$. A morphism
$f:(X ,U)\to(X^{\prime} ,U^{\prime})$ in $\BC^{\cC}$ is a morphism $f:X \to X^{\prime} $ in $\BC $ such that $(f\times f)(U)\subseteq U^{\prime}$. We have a forgetful functors $$\BC\stackrel{q}{\leftarrow} \BC^{\cC}\stackrel{p}{\to} \BC\, , \quad X_{U} \leftarrow (X,U)\to X \, .$$
We now consider the diagram
$$\xymatrix{\BC^{\cC}\ar[rr]^-{E\circ q}\ar[d]^{p}&&{\bC}\\
\BC\ar@{--}[rru]^{E_u}\ar@/_{1cm}/[urr]^-{E}&&}$$
 The functor $E_u$  is then obtained from $E\circ q$ by a left Kan extension along $p$.  

There is a natural transformation $q\to p$ of functors $\BC^{\cC}\to \BC$ given on $(X,U)$ in $\BC^{\cC}$ by the morphism $X_{U}\to X$.  Applying $E$ we get a transformation $E\circ q\to E\circ p$.
This now induces the   natural transformation $E_{u}\to E$ via the universal property of the left Kan extension.

The pointwise formula for left Kan extensions gives
$$E_u(X)\simeq \colim_{\BC^{\cC}/X} E\circ q \, ,$$
where $\BC^{\cC}/(X,\cC,\cB)$ is the slice category of objects from $\BC^{\cC}$ over the bornological coarse space $X $. We observe that the subcategory of objects $(( X ,U) ,\id_{X})$ for $U$ in $\cC$   is cofinal in
$\BC^{\cC}/X $. Hence we can restrict the colimit to this subcategory. We then get the formula \eqref{erwe999424} as desired.

\begin{prop} \label{ergkjwergrereferfwrfwerfw}For a bornological coarse space $X$
the transformation $E_{u}(X)\to E(X)$ is an equivalence if the coarse structure  of $X$ is generated by a single entourage.

The following properties of $E$ are inherited by $E_{u}$:
\begin{enumerate}
\item  coarse invariance,
\item  excision,  
\item \label{ewrgijewroigergrefwreferf} vanishing on flasque  bornological coarse    spaces.\end{enumerate}
\end{prop}

\begin{proof}
The first claim is clear from \eqref{erwe999424}.

We now show that  if $E$ is coarsely invariant, then  $E_{u}$ is coarsely invariant. 
It suffices to show that for every bornological coarse space $X$ the projection $\{0,1\}\times X\to X$ induces an equivalence  \begin{equation}\label{fiohiufo34f3f}
E_{u}(\{0,1\} \otimes  X)\to E_{u}(X)\, ,
\end{equation} 
where $\{0,1\}$ has the maximal bornology and coarse structure.

For every entourage $U$ of $X$ we have an isomorphism of bornological coarse spaces
$$(\{0,1\} \otimes  X)_{\tilde U}  \cong \{0,1\} \otimes  X_{U}\, ,$$ where $\tilde U:=\{0,1\}^{2}\times U$. Moreover, the   entourages
of $\{0,1\} \otimes  X$ of the form $\tilde U$ are cofinal in all entourages of this bornological coarse space.
We now have an equivalence
$$E((\{0,1\} \otimes  X)_{\tilde U})\simeq E(\{0,1\} \otimes  X_{U})\simeq E(X_{U})\, ,$$
where for the second we use that $E$ is coarsely invariant. Forming the colimit of these equivalences over all entourages $U$ of $X$ and using the description \eqref{erwe999424} of the evaluation of $E_{u}$ we get the desired equivalence \eqref{fiohiufo34f3f}.


If $E$ is coarsely excisive, then $E_u$ is also coarsely excisive. This follows from the fact that complementary pairs on $X$ are also   complementary pairs on $X_U$.

Assume that  $E$ vanishes on flasque spaces. Let $X$ be a flasque space. In order to show that $E_{u}(X)\simeq 0$ it   suffices to show that
 $E(X_{W})\simeq 0$ for a cofinal set of entourages $W$ of $X$. Let   $f : X \to X$ be a morphism  implementing   flasqueness. Then there exists an entourage $V$ of $X$ such that  $f$ is $V$-close to $\id_X$. For any entourage $U$ of $X$ consider the entourage $$W :=  \bigcup_{k \in \IN} (f^k \times f^k)(U \cup  V)\, .$$ Then $X_W$ is flasque with flasqueness implemented by the morphism $f$ and $U\subseteq W$.  Hence we have $E(X_{W})\simeq 0$.
\end{proof}

Let $\bC$ be a cocomplete stable $\infty$-category, and let $E:\BC\to \bC$ be a functor.

\begin{kor}
Assume:
\begin{enumerate}
\item $E$ is coarsely invariant.
\item $E$ vanishes on flasques.
\item $E$ satisfies  excision.
\end{enumerate}
Then: 
\begin{enumerate}  
\item$E_{u}$ is a coarse homology theory.
\item If $E$ is  coarse homology theory, then $E_{u}\to E$ is an equivalence.
\item \label{rgoihewroigregrefgwef} If the coarse structure of $X$ is generated by a single entourage, then 
$E_{u}(X)\to E(X)$ is an equivalence.
\end{enumerate}
\end{kor}

 Note that Assertion \ref{rgoihewroigregrefgwef} applies, e.g.\ if $X$ is a path metric space.

\subsection{Additivity and coproducts}\label{jkndsfjh23}

Note that the axioms for a coarse homology theory do not include an additivity property for, say, infinite coproducts (finite coproducts are covered by excision). We refer to the discussion in Example \ref{jwefjweofewfewfewfewf} and Lemma~\ref{efiwofoewf234234244}. But most of the known coarse homology theories are in one way or the other compatible with forming infinite coproducts and / or infinite free unions. In this section we are going to formalize this.

\subsubsection{Additivity}

Let $E:\BC\to \bC$ be a coarse homology theory and assume  that $\bC$ is in addition  complete.
 
\begin{ddd}\label{jknsd23cs}
  $E$ is strongly additive\index{strongly additive!coarse homology theory}\index{coarse!strongly additive homology theory}\index{additive!strongly} if for every family $(X_i)_{i \in I}$ of bornological coarse spaces the morphism
\[E \big( \bigsqcup^{\free}_{i \in I} X_i \big) \to \prod_{i \in I} E(X_i)\]
induced by the collection of morphisms \eqref{uewhwiuhiewufheihfi2342341} is an equivalence.
\end{ddd}

\begin{ddd}\label{cdqedoihqiuqwdqwd}
We call $E$ additive\index{additive!coarse homology theory}\index{coarse!additive homology theory} if for every set $I$ the morphism
\[E\big( \bigsqcup^{\free}_{I}* \big)\to \prod_{I} E(*)\]
is an equivalence.
\end{ddd}

Clearly, a strongly additive coarse homology theory is additive. The notion of additivity features in Theorem~\ref{wekfhweuifhweiui23u4242342rf}.

\begin{ex}
Coarse ordinary homology (Proposition \ref{ewgiowjgwoergfrefrwfref}), the coarsification of stable homotopy (Proposition \ref{jnksdui2332}) and coarsifications of locally finite homology theories (Proposition \ref{cegiojiojergergeg}) are strongly additive. 
In {\cite{Bunke:ad}} we show that  coarse $K$-homology $\KX$  is strongly additive. {We refer} to \cite[Ex.~10.5]{ass} for an extended list of strongly additive coarse homology theories which have been constructed in sequels
to the present book.

We were not able to show that the quasi-local version  $\KXql$ of coarse $K$-homology is strongly additive. But both $\KX$ and $\KXql$   are additive (Corollary~\ref{cor_additivitiy_KX_KXql}).

In Example~\ref{hhfweweiufufwefwefwef} we will provide a coarse homology theory which is not additive.
\hB
\end{ex}

Let $E \colon \BC \to {\bC}$ be a coarse homology theory, and let $(X_i)_{i \in I}$ be a family of bornological coarse spaces.

\begin{lem}\label{jknsdiu2323}
If $E$ is strongly additive, then we have a fibre sequence
\begin{equation}
\label{dqwhdqwdqwiudqwidu}
\bigoplus_{i \in I} E(X_i) \to E\big( \bigsqcup^{\mixed}_{i \in I} X_{i} \big) \to \colim_{J \subseteq I \:\textit{finite}} \prod_{i \in I \setminus J} E(X_{i,\disc})\, .
\end{equation}
\end{lem}

\begin{proof}This 
follows from Lemma~\ref{efiwofoewf234234244} and the fact that we can rewrite the  remainder term \eqref{vefjkbnerkjvbrjebvjkjvnrekvvrrev} using strong additivity of $E$ (note that for discrete bornological coarse spaces the mixed and the free union coincide).
\end{proof}

 
\begin{ex}

Let $E$ be an additive coarse homology theory with $E(*)\not\simeq 0$.
Such a coarse homology theory exists, e.g., $\HX$ defined in {Section}~\ref{ewifjewoifoi232344234}. Applying \eqref{dqwhdqwdqwiudqwidu} to $E$ in the case that $X_{i}=*$ for all $i$ in $I$ (in this case it suffices to assume that  $E$   is  just additive in order to have the fibre sequence \eqref{dqwhdqwdqwiudqwidu}), we get the fibre sequence
$$\bigoplus_{I}E(*)\to  \prod_{I} E(*) \to \colim_{J \subseteq I \textit{finite}} \prod_{i \in I \setminus J} E(*)\, ,$$
which shows that the remainder term is non-trivial provided  $I$ is infinite. Here we have again used that for discrete bornological spaces the mixed and free union coincide.
\hB
\end{ex}

\begin{ex}
We refer to Theorem \ref{qrgiojoergergrefwrefref} for further examples of additive coarse homology theories. \hB
\end{ex}

\subsubsection{Coproducts}

Let   $E \colon \BC \to {\bC}$ be a coarse homology theory.

\begin{ddd}\index{coproducts!preserved!in $\BC$}
 $E$ preserves coproducts if for every family $(X_i)_{i \in I}$ of bornologial coarse spaces the canonical morphism
\[\bigoplus_{i \in I} E(X_i) \to E \big( \coprod_{i \in I} X_i \big)\]
is an equivalence.
\end{ddd}

\begin{rem}\label{jnksdfhiu23}
Due to the nature of the remainder term \eqref{wefuhiuhj4fhjiuwefqef} derived in the proof of Lemma~\ref{efiwofoewf234234244}  in order to show that a coarse homology theory preserves coproducts it suffices to check this only on discrete bornological coarse spaces.
\hB
\end{rem}
 
\begin{ex}
We show in this paper that coarse ordinary homology and  both versions of coarse $K$-homology preserve coproducts. Coarsifications of locally finite homology theories preserve coproducts if the locally finite theory does so. 
\hB
\end{ex}

\newcommand{\bP}{\mathbf{P}}

\begin{ex} 
Let $E\colon\BC\to \bC$ {be} a coarse homology theory. In \cite{ass} we construct a new coarse homology theory denoted by 
$E\cO^{\infty}\bP\colon \BC\to \bC$.  It serves as the domain of the coarse assembly map $\mu_{E}\colon E\cO^{\infty}\bP\to \Sigma E $ which is the main topic of  {that} paper.
Using information on this assembly map and   Lemma \ref{efiwofoewf234234244}  in \cite{ass} we will observe  {the following theorem}. 
 \begin{theorem}\label{qrgiojoergergrefwrefref} Assume that $E$ is strong  {(see \cite[[Def.~4.19]{equicoarse})}.
\begin{enumerate}
\item If $E$ is additive, then so is $E\cO^{\infty}\bP$.
\item If $E$ preserves coproducts, then so does $E\cO^{\infty}\bP$.
\end{enumerate}
\end{theorem}  In particular, we get more examples of coarse homology theories which preserve coproducts. \hB 
\end{ex}
\newcommand{\LF}{\mathrm{LF}}
 
In \cite{equicoarse} we introduce the notion of continuity. 
If $X$ is a bornological coarse space, then
by $\LF(X)$ we denote its poset of locally finite subsets, see Definition \ref{defn:asdfwewt}. The elements of $\LF(X)$ will be equipped with the bornological coarse structure induced from $X$. Note that by definition the induced bornology is  the minimal one.

Let $\bC$ be a cocomplete stable $\infty$-category, and let $E:\BC\to \bC$ be a functor.
\begin{ddd}
$E$ is continuous if for every $X$ in $\BC$ the canonical map
$$\colim_{L\in \LF(X)} E(L)\to E(X)$$ is an equivalence.
\end{ddd}

\begin{rem}\label{wethgrtgffsvfdvsvsdfv}
It is {straightforward} to check that a continuous coarse homology theory preserves coproducts \cite[Lem.~5.17]{equicoarse}. 

The coarse ordinary homology (Definition \ref{rekgopegergwfwerf}) and the two versions $\KX$ and $\KXql$ of coarse $K$-homology (Definition \ref{oiwefewiofu9234234234324})
are continuous.  

In general, given a coarse homology theory, using a method similar to the one in Section~\ref{egihfiebgbsdbsbfdbs}
one can construct a new coarse homology which is continuous  \cite[Lem.~5.26]{equicoarse}. 
So this construction produces further examples of coarse homology theories which preserve coproducts. 

Continuity of coarse homology theories is relevant for their application in the study of assembly maps.
In the present book we will not discuss continuity further and refer to \cite{equicoarse} and \cite{injectivity}
for more information and applications. \hB
 \end{rem}

\subsection{Coarse ordinary homology}\label{ewifjewoifoi232344234}
 
In this section we introduce a coarse version  $\HX$ of ordinary homology.   After defining the functor our main task is the verification of the {properties} listed in Definition \ref{rgljogreggregrege}. The fact that the coarse homology  $\HX$ exists and is non-trivial is a first indication that  $\Sp\cX$ is non-trivial. 

Since the following constructions and  arguments are in principle well-known we will be sketchy at various points in order to keep this section short. {Our task here will be mainly to put the well-known construction into the setup developed in this paper.} The structure of this section is  very much the model for the analogous Sections \ref{kjoijfoewifueofieuwfewfwefewfef} and  \ref{sec:jkbewrgh} dealing with coarsification and coarse $K$-{homology}.

It is hard to trace back the first definition of coarse ordinary homology. One of the first definitions was the uniform version {(i.e., a so-called rough homology theory)} by Block and Weinberger \cite{block_weinberger_1}.

Let $\Ch$\index{$\Ch$} denote the small category of very small chain complexes. Its objects are chain complexes of very small abelian groups. The morphisms in $\Ch$ are chain maps. The target  category of the coarse ordinary homology will be the stable $\infty$-category $\Ch_{\infty}$ obtained by a Dwyer--Kan localization
\begin{equation}\label{qewflkklefmklqwefqwefqfqefqwefe}
\ell: \Ch\to \Ch[W^{-1}]=:\Ch_{\infty}\, ,
\end{equation}
where $W$ is the set of quasi-isomorphisms in $\Ch$. The homotopy category of $\Ch_{\infty}$ is the usual unbounded derived category of abelian groups.  It is well known that $\Ch_{\infty}$ is stable, very small complete and cocomplete. We refer to \cite[Sec.~1.3]{HA} for details.
We will use the following properties of the functor $\ell$:
\begin{enumerate}
\item\label{ewfoifewoifeowfewfewf} $\ell$ sends chain homotopic maps to equivalent morphisms.
\item\label{ewfoifewoifeowfewfewf12} $\ell$ preserves filtered colimits.
\item \label{fewiofjoifewfewfewf} $\ell$ sends short exact sequences of chain complexes to fibre sequences of spectra.
\item \label{weoifewoifowiowefio90uwe908fwefewf} $\ell$ preserves sums and products.
\end{enumerate}

We start by defining a functor $$\CX\colon \BC\to \Ch$$\index{$C\cX$}which associates to a bornological coarse space $X$ the chain complex of locally finite controlled chains.  

In order to talk about properties of chains on {a bornological coarse space} $X$ we introduce the following language elements.
Let $n$ be in $ \nat$ and $B$   be a subset of $X$. We say that a point $ (x_{0},\dots,x_{n})$ in $ X^{n+1}$ meets $B$ if there exists an index $i$ in $\{0,\dots,n\}$ such that $x_{i}\in B$. 

Let $U$ be an entourage of $X$. We say that $(x_{0},\dots,x_{n})$ is $U$-controlled if for every two indices $i,j$ in $\{0,\dots,n\}$ we have $(x_{i},x_{j})\in U$.

An $n$-chain on $X$ is by definition just  a function $c:X^{n+1}\to \Z$.
Its support\index{support!of a chain} is the set
$$\supp(c):=\{x\in X^{n+1}\:|\: c(x)\not=0\}\, .$$

Let $X$ be a bornological coarse space with an entourage $U$. Let furthermore  $c$  be an $n$-chain on $X$.   
\begin{ddd}\mbox{}
\begin{enumerate}
\item $c$ is locally finite\index{locally finite!chain}, if for every bounded $B$ the set of points in $\supp(c)$ which meet $B$ is finite.
\item $c$ is $U$-controlled\index{$U$-!controlled chain} if every point in $\supp(c)$ is $U$-controlled.
\end{enumerate}
We say that an $n$-chain is controlled\index{controlled!chain} if it is $U$-controlled for some entourage $U$ of $X$.
\end{ddd}

Let $X$ be a bornological coarse space.
\begin{ddd} We define $\CX_{n}(X)$\index{$C\cX_{n}(-)$} to be the abelian group of locally finite and controlled $n$-chains.
\end{ddd}
 
It is often useful to represent $n$-chains on $X$ as (formal, {potentially infinite}) sums
$$\sum_{x\in X^{n+1}} c(x)x\, .$$
For $i$ in $\{0,\dots,n\}$ we define $\partial_{i}:X^{n+1}\to X^{n}$ by
$$\partial_{i}(x_{0},\dots,x_{n}):=(x_{0},\dots \hat x_{i},\dots,x_{n})\, .$$
So $\partial_{i}$ is the projection leaving out the $i$'th coordinate.

If $x$ is $U$-controlled, then so is $\partial_{i}x$.

Using that the elements of $\CX_{n}(X)$ are locally finite and controlled
 one checks that
$\partial_{i}$ extends linearly {(also for infinite sums)} in a {canonical} way to a map
$$\partial_{i}:\CX_{n}(X)\to \CX_{n-1}(X)\, .$$
We consider the sum $\partial:=\sum_{i=0}^{n} (-1)^{i}\partial_{i}$. One easily verifies that
$\partial\circ \partial=0$. The family $(\CX_{n}(X))_{n\in \nat}$ of abelian groups  together with
the differential $\partial$ is the desired chain complex 
$\CX(X)$ in $\Ch$. 

Let now $f:X\to X^{\prime}$ be a morphism of bornological coarse spaces.
The map
$$(x_{0},\dots,x_{n})\mapsto (f(x_{0}),\dots,f(x_{n}))$$
extends linearly {(also for infinite sums)} in a {canonical} way to a map
$$\CX(f):\CX_{n}(X)\to \CX_{n}(X^{\prime})\, .$$
This map involves sums over fibres of $f$ which become finite since $f$ is proper and controlled and since we apply it to locally finite chains.
Using that $f$ is controlled we conclude that the induced map sends controlled chains to controlled chains.
It is easy to see that $\CX(f)$ is a chain map.
 This finishes the construction of the functor
$$\CX\colon \BC\to \Ch\, .$$

%

\begin{ddd}\label{rekgopegergwfwerf}
\index{coarse!ordinary homology}\index{ordinary coarse homology}
We define the functor $\HX\colon\BC\to \Ch_{\infty}$\index{$\HX$}
 by $$\HX:=\ell\circ \CX\, .$$ 
\end{ddd}

\begin{theorem}\label{oifjweoifewfefwefwefewf}
$\HX$ is a $\Sp$-valued coarse homology theory.
\end{theorem}

\begin{proof}
In the following four propositions we will verify the {properties} listed in Definition~\ref{rgljogreggregrege}, which will complete the proof.
\end{proof}

\begin{prop}\label{fweofwefewfewfewfwef}
$\HX$ is coarsely invariant.
\end{prop}

\begin{proof}
Let $f,g :  X\to X^{\prime}$ be two morphisms which are close to each other (Definition~\ref{ioewfjweoifoijwfw32424}). We must show that the morphisms $\HX(f)$ and $\HX(g)$ from $\HX(X)$ to $\HX(X^{\prime})$ are equivalent.
We will construct a chain homotopy from $\CX(g)$ to $\CX(f)$ and then apply {Property \ref{ewfoifewoifeowfewfewf} of the functor $\ell$}.

For $i$ in $\{0,\dots,n\}$ we consider the map
$$h_{i}:X^{n+1}\to {{X^\prime}^{n+2}}\, , \quad (x_{0},\dots,x_{n})\mapsto (f(x_{0})\dots,f(x_{i}),g(x_{i}),\dots,g(x_{n}))\, .$$
We observe that it extends linearly {(again for possibly infinite sums)} to
$$h_{i}:\CX_{n}(X)\to \CX_{n+1}(X^{\prime})$$ in a canonical way.
Here we need the fact that $f$ and $g$ are close to each other  in order to control the pairs $(f(x_{i}),g(x_{j}))$.
We now define
$$h:=\sum_{i=0}^{n} (-1)^{i} h_{i}\, .$$
One checks in a straightforward manner that
$$\partial\circ h+h\circ \partial=\CX(g)-\CX(f)\, .$$
Hence $h$ is a chain homotopy between $\CX(g)$ and $\CX(f)$ as desired.
\end{proof}

\begin{prop}
$\HX$ satisfies excision.
\end{prop}
\begin{proof}
Let $X$ be a bornological coarse space, and let $(Z,\cY)$ a complementary pair (Definition~\ref{jfwoifjweofeiwf234234324}) with the big family $\cY=(Y_{i})_{i\in I}$.
By {Property \ref{ewfoifewoifeowfewfewf12} of the functor $\ell$} we have an equivalence
$$\HX(\cY)\simeq \ell(\colim_{i\in I} \CX(Y_{i}))\, .$$
Since for an injective map $f$ between bornological coarse spaces the induced map $\CX(f)$ is obviously injective we can
interpret the colimit on the right-hand side as a union of subcomplexes in $\CX(X)$.
Similarly we can consider $\CX
(Z)$ as a subcomplex of $\CX(X)$ naturally.

We now consider the following diagram
$$\xymatrix{0\ar[r]&\CX(Z\cap \cY)\ar[d]\ar[r]&\CX(Z)\ar[r]\ar[d]&\CX(Z)/\CX(Z\cap \cY)\ar[d]\ar[r]&0\\
0\ar[r]&\CX(\cY)\ar[r]&\CX(X)\ar[r]&\CX(X)/\CX(\cY)\ar[r]&0}$$
We want to show that the left square gives a push-out square after applying $\ell$.
Using {Property \ref{fewiofjoifewfewfewf} of the functor $\ell$} we see that it suffices to show that the {right} vertical map is a quasi-isomorphism. 
We will actually show that it is an isomorphism of chain complexes.

For injectivity we consider a chain $c$ in $\CX(Z)$ whose class $[c]$ in $\CX(Z)/\CX(Z\cap \cY)$ is sent to zero.
Then $c\in \CX(\cY)\cap \CX(Z)=\CX(Z\cap \cY)$. Hence $[c]=0$.

For surjectivity  consider  $c$ in $\CX(X)$. We fix $i$ in $I$ such that $Z\cup Y_{i}={X}$. We can find a coarse entourage $U$ of $X$ such that $c$ is $U$-controlled.
Let $c_{Z}$ be the restriction of $c$ to $Z^{n+1}$, where $n$ is the degree of $c$.
Using that $\cY$ is big we choose 
 $j$ in $I$   such that $U[Y_{i}]\subseteq Y_{j}$. 
 Then one can check that $c-c_{Z}\in \CX(Y_{j})$.
 Consequently, the class  $[c_{Z}]$ in $\CX(Z)/\CX(Z\cap \cY)$ is mapped to the class of $[c]$. 
\end{proof}

\begin{prop}
If $X$ is a flasque bornological coarse space, then $\HX(X)\simeq 0$.
\end{prop}
\begin{proof}
Let $f:X\to X$ implement flasqueness of $X$ (Definition \ref{efijewifjewiofwifwfew322423424}).
For every $n$ in $\nat$ we define
$$S\colon X^{n+1}\to \CX_{n}(X)\, , \quad S(x):=\sum_{k\in \nat} \CX(f^{k})(x) \, .$$
Using the properties of $f$ listed in Definition \ref{efijewifjewiofwifwfew322423424}  one checks that $S$ extends linearly
to a chain map
$$S\colon \CX (X)\to \CX
(X)\, .$$
For example, we use Point~\ref{fwoiejweoifeowi231} of Definition \ref{efijewifjewiofwifwfew322423424} in order to see that $S$ takes values in locally finite chains. We further use {Point~\ref{sdfn3443243t61} of Definition} \ref{efijewifjewiofwifwfew322423424} in order to see that $S$ maps controlled chains to controlled chains.

From the construction of $S$ we immediately conclude that
$$\id_{\CX(X)}+\CX(f)\circ S=S$$
and applying $\ell$ we get
$$ \id_{{\HX(X)}}+\HX(f)\circ \ell(S)\simeq \ell(S)\, .$$
We now use that $f$ is close to $\id_{X}$ (\ref{efijewifjewiofwifwfew322423424}.\ref{asdfgjtzkj561}) and Proposition \ref{fweofwefewfewfewfwef} to conclude that
$$\id_{{\HX(X)}}+  \ell(S)= \ell(S)\, .$$
This implies that $\id_{{\HX(X)}}\simeq 0$ and hence $\HX(X)\simeq 0$. \end{proof}

\begin{prop}
$\HX$ is $u$-continuous.
\end{prop}

\begin{proof}
Let $X$ be a bornological coarse space with coarse structure $\cC$.
Then $\CX(X)$ is the union of its subcomplexes of $U$-controlled locally finite chains over all $U$ in $\cC$.
Furthermore, for fixed $U$ in $\cC$ the union of the $U^{n}$-controlled chains, where $n$ runs over $n$ in $\nat$, is the image of $\CX(X_{U})$.
Using {Property \ref{ewfoifewoifeowfewfewf12} of the functor $\ell$} we immediately get the desired equivalence
$$\HX(X)\simeq \ell(\colim_{U\in \cC} \CX(X_{U}))\simeq  \colim_{U\in \cC}\ell( \CX(X_{U}))\simeq \colim_{U\in \cC} \HX(X_{U})\, .$$
\end{proof}

\begin{prop}\label{ewgiowjgwoergfrefrwfref}
$\HX$ is strongly additive.
\end{prop}

\begin{proof}
{Let $(X_i)_{i \in I}$ be a family of bornological coarse spaces. An inspection of the definitions shows that controlled  simplices on  $\bigsqcup^{\free}_{i\in I} X_i$ 
can not mix the components.  It follows that
$$\CX \big(  \bigsqcup^{\free}_{i \in I} X_i \big) \cong \prod_{i \in I} \CX(X_i)\, .$$
We now use Property \ref{weoifewoifowiowefio90uwe908fwefewf} of the functor $\ell$ in order to conclude the equivalence
\[\HX \big( \bigsqcup^{\free}_{i \in I} X_i \big) \simeq \prod_{i \in I} \HX(X_i)\, .\]
One easily checks that this equivalence is induced by the correct morphism.}
\end{proof}

\begin{lem}
$\HX$ preserves coproducts.
\end{lem}

\begin{proof}
Let $(X_i)_{i \in I}$ be a family of bornological coarse spaces. Since coarse chains are locally finite, we conclude that every coarse chain on $\coprod_{i \in I} X_i$ is  supported   on finitely many components $X_i$. Furthermore, a controlled simplex is supported on a single component. So we get an isomorphism of complexes
\[\bigoplus_{i \in I} \CX(X_i) \cong \CX\big( \coprod_{i \in I} X_i \big)\]
and the assertion of the lemma follows with Property~\ref{weoifewoifowiowefio90uwe908fwefewf} of the functor $\ell$.
\end{proof}

In the following we consider abelian groups as chain complexes concentrated in degree $0$.
\begin{lem}
If $X$ is a set considered as a bornological coarse space with the maximal structures, then
 $\HX(X)\simeq \ell(\Z)$.
\end{lem}

\begin{proof}
We have $\HX(*)\simeq \ell(\Z)$. Indeed, $$\CX(*)\cong \dots0 \leftarrow\Z\stackrel{0}{\leftarrow} \Z\stackrel{1}{\leftarrow} \Z\stackrel{0}{\leftarrow} \Z\stackrel{1}{\leftarrow} \Z\stackrel{0}{\leftarrow} \Z\dots$$ 
and the inclusion of the left-most $\Z$ sitting in degree $0$ induces an equivalence after applying $\ell$.
The inclusion $*\to X$ is an equivalence. Conseqently,
$$\ell(\Z)\simeq \HX(*)\simeq \HX(X)\, ,$$
as claimed.
\end{proof}


\subsection{Coarsification of {stable homotopy}}\label{kjoijfoewifueofieuwfewfwefewfef}

Let  $\Sp$ denote the small stable $\infty$-category of very small spectra.
 In usual homotopy theory the sphere spectrum $S$ in $\Sp$ represents the homology theory $$ \Sigma^{\infty,\topp}_{+}:\Top\to \Sp $$  called the stable homotopy theory.  
The goal of the present section is to construct a coarse homology theory $$Q:\BC\to \Sp $$
which is a coarse version of stable homotopy theory.
The construction uses the process of  coarsification of a locally finite homology theor{y} which was introduced by Roe \cite[Def.~5.37]{roe_lectures_coarse_geometry} using anti-\v{C}ech systems. We employ spaces of probability measures as Emerson and Meyer \cite[Sec.~4]{em}. In the present section we go directly to the construction of $Q$. In the subsequent Section \ref{ihjiovqceqwecewcqwecq} we will develop the theory of locally finite homology theories and their coarsification  in general,  generalizing  the ideas of the present subsection.

The  coarsification of stable homotopy theory $Q$ gives rise to may further coarse homology theories.
If $\bC$ is a cocomplete  stable   $\infty$-category{, then} it is tensorized over $\Sp$, i.e, we have {a bi-}functor
\begin{equation}\label{wervewrvwevfevfdvsfdvdfsvsfv}
\bC\times \Sp\to \bC\ , \quad  (C,R)\mapsto C\wedge R
\end{equation} 
which preserves very small colimits in both arguments. Every  object $C$ in $\bC$ gives rise to a
$\bC$-valued homology theory
$$C\wedge  \Sigma^{\infty,\topp}_{+}  :\Top \to \bC\ , \quad X\mapsto C\wedge  \Sigma^{\infty,\topp}_{+}(X) $$
whose value on a point is $C$.  
Similarly, we can define a $\bC$-valued coarse  homology theory
$$C\wedge  Q:\BC\to \bC\, , \quad X\mapsto C\wedge Q(X) \, ,$$
see Corollary \ref{goirjgio334t43gregergrege}.

%
%
%
%
%

Let $\bC$ be a complete and cocomplete  $\infty$-category. In some arguments below 
  we  need the following property of $\bC$.

Let  $(J_{i})_{i\in I}$ be a very small family of very small filtered categories. Then   $J:=\prod_{i\in I}J_{i}$ is again a very small filtered category.  If $(C_{i}:J_{i}\to \bC)_{i\in I}$    is a  family of diagrams, then we can form a diagram $\prod_{i\in  I}C_{i}:J\to \bC$. 

\begin{ddd}[{{\cite[Sec.~1.2.(iv)]{TAC_continuous_cat}}}]\label{wekrgergfvsfgsgfdg}
$\bC$ has the property that  {products distribute over  filtered colimits,}
\index{products distribute over  filtered colimits}\index{distribute!products  over colimits }
if for every family of diagrams $(C_{i}:J_{i}\to \bC)_{i\in I}$ as above
the canonical morphism
$$\colim_{J}\prod_{i\in I} C_{i}\to \prod_{i\in I} \colim_{J_{i}} C_{i}$$
is an equivalence.
\end{ddd}

\begin{rem}\label{qerugiergregrefrefwerf}
It is easy to see that the category $\Set$ has this property. 
 Using the fact that  the  functors $\pi_{i}:\Spc\to \Set$  (the group structure is not relevant here) for $i$ in $\nat$  preserves filtered colimits and products, and that they  
  jointly detect equivalences we can conclude that $\Spc$ has the property, too. 
One can then conclude that every compactly generated $\infty$-category $\bC$ has the property defined in Definition \ref{wekrgergfvsfgsgfdg}.
To this end one uses the functors $\Map_{\bC}(K,-):\bC\to \Spc$ for compact objects $K$ in $\bC$ in order to reduce to the case of $\Spc$.

 In particular, in the $\infty$-categories $\Sp$ and $\Ch_{\infty}$   {products distribute over  filtered colimits.}
 
We will also need {the} dual property:
  {sums distribute over cofiltered limits} \index{sums distribute over cofiltered}\index{distribute! sums over limits  } in $\bC$ if {products distribute over  filtered colimits,}   in $\bC^{\op}$.  
\hB
\end{rem}

\subsubsection{Rips complexes and a coarsification of stable homotopy}\label{eriogjeoigerfwerfwerfwrf}

A set $X$ can be considered as a discrete topological space. It is locally compact and Hausdorff. Its compact subsets are the finite subsets. We let $C_{0}(X)$ be the Banach space closure of the space of compactly supported, continuous (this a void condition) functions on $X$ with respect to the $\sup$-norm. For every point $x$ in $X$ we define the function $$e_{x}\in C_{0}(X)\, , \quad e_{x}(y):=\left\{\begin{array}{cc}1&y=x\\0&\text{else}\end{array}\right.$$
We can consider $(X,\cP(X))$ as the  Borel measurable space associated to the topological space $(X,\cP(X))$. By  $P(X)$\index{$P(-)$} we denote the space of regular probability measures on $X$. {Note that a regular probability measure is in this case just a positive function on the set~$X$ with $\ell^1$-norm $1$.} 
The topology on $P(X)$  is induced from the inclusion of $P(X)$ into the unit sphere of the Banach dual  of $C_{0}(X)$ which is compact {and Hausdorff} in the weak $*$-topology.
A map of sets $f:X\to X^{\prime}$ functorially induces a continuous map
$$f_{*}:P(X)\to P(X^{\prime})\, , \quad \mu\mapsto f_{*}\mu\, .$$
The support\index{support!of a probability measure} of a probability measure $\mu$ in $P(X)$ is defined by
$$\supp(\mu):=\{x\in X\:|\: \mu(e_{x})\not=0\}\, .$$
We assume that $X$ is a bornological coarse space with coarse structure $\cC$ and bornology $\cB$. For an entourage $U$ in $\cC$ we let
\begin{equation}\label{gegljio3tuio38tut3t3t34}
P_{U}(X):=\{\mu\in P(X) :  \supp(\mu) \text{ is }U\text{-bounded}\}
\end{equation}
(see Definition \ref{jfewiofjweofwef234} for the notion of $U$-boundedness).

We claim that $P_{U}(X)$ is a closed subspace of $P(X)$.  We will show that $P(X)\setminus P_{U}(X)$ is open. 
Let $\mu$ be in the complement $P(X)\setminus P_{U}(X)$. Then there exist two points $x,y$ in $X$ such that $(x,y)\not\in U$ and $\mu(e_{x})\not=0$ and $\mu(e_{y})\not=0$. Then $\{\nu\in P(X)\:|\: \nu(e_{x})\not=0 \, \& \,\nu(e_{y})\not=0\}$ is an open neighbourhood of $\mu$ contained in $P(X)\setminus P_{U}(X)$. This shows  the claim.

Let $X'$ be a second bornological coarse space with structures $\cC'$ and $\cB'$, and let $U'$ be in $\cC'$.
If $f :  X\to X^{\prime}$ is  map of sets and $(f\times f)(U)\subseteq U^{\prime}$, then $f_{*}$ restricts to a map
$$f_{*} :  P_{U}(X)\to P_{U^{\prime}}(X^{\prime})\, .$$ 
If $Y$ is a subset of $X$, then we get a map 
$$P_{U_{Y}}(Y)\to P_{U}(X)\ ,$$
where  $U_{Y}:=U\cap (Y\times Y)$. 
In the following we will simplify the notation and write  $P_{U}(Y)$ instead of $P_{U_{Y}}(Y)$. 

 

In the following we use the functor \index{$\Sigma_{+}^{\infty,\topp}$} 
\begin{equation}\label{vervelkvjekvevweqvevwv}
\Sigma_{+}^{\infty,\topp}:\Top \to \Sp\, . 
\end{equation} 
It is  defined as the composition \begin{equation}\label{frknfkejnfkjfrefwefwerfwref}
\Top\xrightarrow{\sing} \sSet \xrightarrow{\ell}  \sSet [W^{-1}]\stackrel{!}{\simeq}\Spc  \xrightarrow{\Sigma_{+}^{\infty}}
\Sp\, .
\end{equation}
Here $\sing$ sends a topological space to the very small simplicial set of singular simplices. 
The functor $\ell$ is the Dwyer--Kan localization at the set $W$ of weak homotopy equivalences.
The  equivalence  marked  by $!$ depends on the choice of a model for $\Spc$ and could be taken as a definition of $\Spc$.
Finally, the functor $\Sigma^{\infty}_{+}$ is the small version of the left-adjoint in \eqref{cnjkdsncjkkddscasdcadsc}.
In order to distinguish \eqref{vervelkvjekvevweqvevwv} from this functor we added the superscript ``$\topp$''.

 In the arguments below we will use the following well-known  properties of this functor:

\begin{enumerate}
\item \label{efwkjwoefwefewfwfq2223} $\Sigma_{+}^{\infty,\topp}$ sends homotopic maps to equivalent maps.
\item \label{efwkjwoefwefewfwfq22231} If $(T_{i})_{i\in I}$ is a filtered family of subspaces of some space such that for every $i$ in $I$ there exists $j$ in $J$ such that $T_{j}$ contains an open neighbourhood of $T_{i}$, then $$\colim_{i\in I} \Sigma_{+}^{\infty,\topp} (T_{i})\simeq \Sigma_{+}^{\infty,\topp} \Big( \bigcup_{i\in I} T_{i} \Big)\, .$$
\item \label{ewifjweoifiewfewfwfewfewfwfwf}If $(U,V)$ is an open covering of a topological space $X$, then we have a push-out square in spectra
$$\xymatrix{\Sigma_{+}^{\infty,\topp}(U\cap V)\ar[r]\ar[d]&\Sigma_{+}^{\infty,\topp}V\ar[d]\\\Sigma_{+}^{\infty,\topp}U\ar[r]&\Sigma_{+}^{\infty,\topp}X}$$
\end{enumerate}

\begin{ddd}\label{efifjweiof89247u23553435534543}
We define the functor\index{coarsification!of a bornological coarse space}\index{$Q$} $$Q :  \BC\to \Sp$$ {on objects} by
\begin{equation}\label{erwerjh23ui423423424}
 Q(X):=\colim_{U\in\cC} \ \lim_{B\in \cB} \ \Cofib \big( \Sigma_{+}^{\infty,\topp}P_{U}(X\setminus B)\to \Sigma_{+}^{\infty,\topp}P_{U}(X) \big)\, ,
\end{equation}
where $\cC$ and $\cB$ denote the coarse structure and the bornology of $X$.
For the complete definition also on morphisms see Remark \ref{fjwefoijewojf923ur9u23refewfwefewf} below.
\end{ddd}

 \begin{ex}
 We have $Q(*)\simeq S$, where $S$ in $\Sp$ is the sphere spectrum.\hB
 \end{ex}

\begin{rem}\label{fjwefoijewojf923ur9u23refewfwefewf}
The formula  \eqref{erwerjh23ui423423424} defines $Q$ on objects. 
In the following  we recall the standard procedure to interpret this formula as a definition of a functor.
Let $\BC^{\cC,\cB}$ be the category of triples  $(X,U,B)$ of a bornological coarse space $X $, a coarse entourage $U$ of $X$,   and a bounded subset $B$ of $X$. 
A morphism $$f:(X,U,B)\to(X^{\prime} ,U^{\prime},B^{\prime})$$ in $\BC^{\cC,\cB}$ is
a morphism $f:X \to X^{\prime} $ in $\BC$ such that $(f\times f)(U)\subseteq U^{\prime}$ and $f^{-1}(B^{\prime})\subseteq B$. 

A similar category $\BC^{\cC}$ of pairs $(X ,U)$ has been considered in Subsection \ref{egihfiebgbsdbsbfdbs}.

We have a chain of forgetful functors
$$\BC^{\cC,\cB}\to \BC^{\cC}\to \BC$$
where the first forgets the bounded subset, and the second forgets the entourage.
We now consider the diagram 
$$
\xymatrix{\BC^{\cC,\cB}\ar[rrd]^{Q''}\ar[d]&&\\ \BC^{\cC}\ar@{..>}[rr]^{Q'}\ar[d]&&\Sp\\\BC\ar@{--}[rru]^{Q}&&}
$$
The functor $Q^{''}$ is defined as  the composition
$$\BC^{\cC,\cB}\to   \Mor(\Sp^{sm})\stackrel{\Cofib}{\to} \Sp^{sm}\, ,$$
where the first functor sends
$(X,U,B)$ to the morphism of spectra
$$
 \Sigma^{\infty,\topp}_{+}P_{U}(X\setminus B)\to \Sigma^{\infty,\topp}_{+}P_{U}(X)\, ,$$
 and the second takes the cofibre. Therefore, $Q^{\prime\prime}$  is the functor
 $$(X,U,B)\mapsto \Cofib \big( \Sigma^{\infty,\topp}_{+} P_{U}(X\setminus B \big) \to \Sigma_{+}^{\infty,\topp}P_{U}(X))\, .$$
 We now define $Q^{\prime}$ as the right Kan extension of $Q^{\prime\prime}$ along the forgetful functor
\[\BC^{\cC,\cB}\to \BC^{\cC}\, .\]
By the pointwise formula for the evaluation of the Kan extension on objects we have
 $$Q^{\prime}(X,U):=\lim_{(X,U)/\BC^{\cC,\cB}} Q^{\prime\prime} \, ,$$
  {where ${(X ,U)/\BC^{\cC,\cB}}$ is the category of objects from $\BC^{\cC,\cB}$ under $(X,U)$.}
One now observes that the subcategory of  objects $\big( (X,U,B ), \id_{(X,U)} \big)$ with $B$  in $\cB$ is final in $(X,U)/\BC^{\cC,\cB}$. Hence we can restrict the limit to this final subcategory and get
$$Q^{\prime}(X,U)\simeq \lim_{B\in \cB} \ \Cofib \big( \Sigma_{+}^{\infty,\topp}P_{U}(X\setminus B)\to \Sigma_{+}^{\infty,\topp}P_{U}(X) \big)\, .$$
The functor $Q$ is then obtained from $Q^{\prime}$ by a left Kan extension along the forgetful functor $$\BC^{\cC}\to \BC\, .$$ The pointwise formula gives
$$Q(X)\simeq \colim_{\BC^{\cC}/X} Q^{'} \, ,$$
 {where $\BC^{\cC}/X$ is the category of objects from $\BC^{\cC}$ over the bornological coarse space $X$.}
As in Subsection \ref{egihfiebgbsdbsbfdbs} we observe that the subcategory of objects $\big(  (X,U) ,\id_{X} \big)$ for $U$ in $ \cC$   is cofinal in
$\BC^{\cC}/X$. Hence we can restrict the colimit to this subcategory. We then get the formula \eqref{erwerjh23ui423423424} as desired.
\hB 
\end{rem}


\begin{theorem}\label{efwifjwifou23984u23984u3294823424234}
The functor $Q$ is an $\Sp$-valued coarse homology theory.
\end{theorem}

The proof of Theorem~\ref{efwifjwifou23984u23984u3294823424234} will be given in Section~\ref{sec_proof_Q_coarsehomology} below.

\begin{rem}
In the literature  one  often works with the subspace
$P^{\fin}_{U}(X)\subseteq P_{U}(X)$ of probability measures whose support is $U$-bounded and in addition finite (also we do in \cite{ass}, \cite{equicoarse}). The constructions above go through with $P_{U}(X)$ replaced by the finite version  $P_{U}^{\fin}(X)$. We have no particular reason to prefer the larger space.

Note that if $X$ has the property that its $U$-bounded subsets are finite, then  we have an equality $P^{\fin}_{U}(X)= P_{U}(X)$. This is the case, for example, if $X$ has strongly locally bounded geometry (Definition \ref{oijioegewr}).
\hB
\end{rem}

\subsubsection{Proof of Theorem~\ref{efwifjwifou23984u23984u3294823424234}}
\label{sec_proof_Q_coarsehomology}

In view of Definition~\ref{rgljogreggregrege} the Theorem~\ref{efwifjwifou23984u23984u3294823424234} follows from the following four lemmas which verify the required four properties.

\begin{lem}\label{fewkljlwefewfewfewfwfwef} $Q$ is coarsely invariant. 
\end{lem}

\begin{proof}
Let $X$ be a bornological coarse space. We consider the two inclusions $$i_{0},i_{1}:X\to \{0,1\} \otimes X\, .$$
We must show that
 $Q(i_{0})$ and $Q(i_{1})$ are equivalent.
 
For every entourage $U$ of $X$ we have a continuous map
$$[0,1]\times P_{U}(X)\to  P_{\tilde U}(\{0,1\}\times X)\, , \quad (t,\mu)\mapsto t i_{0,*}\mu+(1-t)i_{1,*}\mu\, ,$$
where $\tilde U:=[0,1]^{2}\times U$. This shows that  
$ P_{U}(i_{0})$ and $P_{U}(i_{1})$ are naturally homotopic. Using Property \ref{efwkjwoefwefewfwfq2223} of the functor $\Sigma_{+}^{\infty,\topp}$
we see that $ \Sigma_{+}^{\infty,\topp}P_{U}(i_{0})$ and $\Sigma_{+}^{\infty,\topp}P_{U}(i_{1})$ are naturally equivalent. Using the naturality of the homotopies with respect to $X$ we see that this equivalence survives
{the}  limit over $\cB$ and the colimit over $\cC$ as well as the cofibre functor. We conclude that 
 $Q(i_{0})\simeq Q(i_{1})$ are equivalent. 
\end{proof}

\begin{lem}\label{lefjefewjfieofeifw24}
$Q$ is excisive.
\end{lem}

\begin{proof}
Let $X$ be a bornological coarse space, $U$ an entourage of $X$, and let
  $\cY=(Y_{i})_{i\in I}$ be a big family on $X$.
For every $i$ in $I$  the inclusion $Y_{i}\to X$ induces an inclusion of a closed subset
$P_{U}(Y_{i})\to P_{U}(X)$. Since $\cY$ is big we can find    $j$ in $I$ be such that $U[Y_{i}]\subseteq Y_{j}$.
We claim that $P_{U}(Y_{j})$ contains an open neighbourhood of $P_{U}(Y_{i})$.
Let $\mu$ be in $P_{U}(Y_{i})$, and let $x$ be in $\supp(\mu)$. Then we consider the open neighbourhood
$$W:=\{\nu\in P_{U}(X) : |\mu(e_{x})-\nu(e_{x})|< \mu(e_{x})/2\} $$
of $\mu$. If $\nu$ belongs to this neighbourhood, then $x\in \supp(\nu)$ and hence $\supp(\nu)\subseteq U[x]\subseteq  Y_{j}$.
This shows that $W\subseteq P_{U}(Y_{j})$.
Consequently, the union $$P_{U}(\cY):=\bigcup_{i\in I}P_{U}(Y_{i})$$ is an  open subset  of  $ P_{U}(X)$.
Furthermore, it implies by {Property \ref{efwkjwoefwefewfwfq22231} of the functor $\Sigma_{+}^{\infty,\topp}$} that
$$ \colim_{i\in I}\Sigma^{\infty,\topp}_{+}P_{U}(Y_{i})\simeq \Sigma_{+}^{\infty,\topp} P_{U}(\cY)\, .$$

If $(Z,\cY)$ is a complementary pair on $X$, then
$(P_{U}(\{Z\}),P_{U}(\cY))$ is an open decomposition of $P_{U}(X)$.
By {Property \ref{ewifjweoifiewfewfwfewfewfwfwf} of the functor $\Sigma_{+}^{\infty,\topp}$} this gives a push-out square
 $$\xymatrix{\Sigma^{\infty,\topp}_{+}P_{U}(\{Z\}\cap \cY)\ar[r]\ar[d]&\Sigma^{\infty,\topp}_{+}P_{U}(\cY)\ar[d]\\\Sigma^{\infty,\topp}_{+}P_{U}(\{Z\})\ar[r]&\Sigma^{\infty,\topp}_{+}P_{U}(X)}$$
The same can be applied to the complementary pair  $(Z\setminus B,\cY\setminus B)$ on $X\setminus B$ for every $B$ in $\cB$. We now form the limit of the corresponding cofibres of the map of push-out squares induced by
 $X\setminus B\to X$  over $B$ in $\cB$ and the colimit over $U$ in $\cC$. Using stability (and hence that push-outs are pull-backs and cofibres are fibres up to shift) we get the push-out square
 $$\xymatrix{Q(\{Z\}\cap \cY)\ar[r]\ar[d]&Q(\cY)\ar[d]\\Q(\{Z\})\ar[r]&Q(X) }$$
 By Lemma \ref{fewkljlwefewfewfewfwfwef} we have an equivalence  $Q(Z)\simeq Q(\{Z\})$. Similarly we have an equivalence
 $Q(Z\cap \cY)\simeq Q(\{Z\}\cap \cY)$. We conclude that   $$\xymatrix{Q(Z\cap \cY)\ar[r]\ar[d]&Q(\cY)\ar[d]\\Q(Z)\ar[r]&Q(X) }$$
 is a push-out square.
\end{proof}

\begin{lem}\label{ekfjweiofu9o4r34555}
$Q$ vanishes on flasque bornological {coarse} spaces.
\end{lem}

\begin{proof}
Let $X $ be a bornological coarse space, {and let $U$  be an entourage of $X$.} For a subset
$Y$ of $ X$ with the induced structures we write
 $$P_{U}(X,Y):= \Cofib \big( \Sigma_{+}^{\infty,\topp}P_{U}(Y)\to \Sigma_{+}^{\infty,\topp}P_{U}(X) \big)\, .$$
Let us now assume that flasquenss of $X$ is implemented by the endomorphism $f:X\to X$.
We consider the diagram
\[\xymatrix{
\colim_{U\in\cC}\lim_{B\in \cB} P_{U}(X,X\setminus B) \ar[rr]^-{\id_{X,*}} \ar@{=}[d] & & \colim_{U\in\cC}P_{U}(X,X)  \\
\colim_{U\in\cC} \lim_{B\in \cB} P_{U}(X,X\setminus B) \ar@{=}[d] \ar[rr]^-{\id_{X,*}+f_{*}} & & \colim_{U\in\cC}P_{U}(X,f(X))\ar[u] \\
\colim_{U\in\cC}\lim_{B\in \cB} P_{U}(X,X\setminus B) \ar@{=}[d] \ar[rr]^-{\id_{X,*}+f_{*}+f^{2}_{*}} & & \colim_{U\in\cC}P_{U}(X, f^{2}(X))\ar[u] \\
\vdots \ar@{=}[d] & & \vdots \ar[u] \\
\colim_{U\in\cC}\lim_{B\in \cB} P_{U}(X,X\setminus B) \ar@{=}[d] \ar[rr]^-{\sum_{k}f^{k}_{*} } & & \colim_{U\in\cC}\lim_{n\in \nat}P_{U}(X,f^{k}(X)) \ar[u]\ar[d] \\
Q(X) \ar[rr]^-{F} & & Q(X)
}\]
All the small cells in the upper part commute. For example the filler of the upper square is  an equivalence $f_{*}\simeq 0$. It   is obtained from the factorization   of $f$ as $X\to f(X) \to X$. Indeed, we have the diagram $$\xymatrix{\Sigma^{\infty,\topp}_{+}P_{U}(X\setminus B)\ar[r]\ar[d]&\Sigma^{\infty,\topp}_{+}P_{U}(f(X))\ar[d]\ar[r]&\Sigma^{\infty,\topp}_{+}P_{U}(f(X))\ar[d]\\\Sigma^{\infty,\topp}_{+}P_{U}(X)\ar[d]\ar[r]&\Sigma^{\infty,\topp}_{+}P_{U}(f(X))\ar[r]\ar[d]&\Sigma^{\infty,\topp}_{+}P_{U}(X)\ar[d]\\P_{U}(X,X\setminus B)\ar[r]&0\ar[r]&P_{U}(X,f(X))}$$
where the columns are pieces of fibre sequences and the lower map is the definition of the induced map $f_{*}$.
The fillers provide the desired equivalence $f_{*}\simeq 0$.
 
The upper part of the diagram 
defines the map denoted by the suggestive symbol   $\sum_{k}f^{k}_{*}$.
This map further induces $F$.

 We now observe that $$ Q(f)+F\simeq Q(f)+ Q(f)\circ F\simeq F\, .$$
This implies $Q(f)\simeq 0$. Since $Q$  maps close maps to equivalent morphisms and $f$ is close  to $\id_{X}$   we get the first equivalence in $\id_{Q(X)}\simeq Q(f)\simeq 0$.
\end{proof}

\begin{lem}
For every $X $ in  $\BC$  with coarse structure $\cC$ we have
$$Q(X)=\colim_{U\in \cC} Q(X_{U})\, .$$
\end{lem}

\begin{proof}
This is clear from the definition.
\end{proof}

  
 

\subsubsection{Further properties of the functor \texorpdfstring{$Q$}{Q} and generalizations}

In this section we show that  $Q$ is strongly additive. We furthermore show how $Q$ can be used to construct more examples of coarse homology theories.

\begin{prop}\label{jnksdui2332}
The coarse homology theory $Q$ is strongly additive.
\end{prop}

\begin{proof}
Let $(X_i)_{i \in I}$ be a family of bornological coarse spaces and set $X := \bigsqcup_{i \in I}^{\free} X_i$. For an entourage $U$ of $X$ every point of $P_U(X)$ is a probability measure supported on one of the components. We conclude that
\begin{equation}
\label{njdsfui232}
P_U(X) \cong \coprod_{i \in I} P_{U_i} (X_i)\, ,
\end{equation}
where we set $U_i := U\cap (X_i \times X_i)$.

We now analyze $$\Cofib(\Sigma_{+}^{\infty,\topp}P_{U}(X\setminus B)\to \Sigma_{+}^{\infty,\topp}P_{U}(X))$$ for a bounded subset $B$ of $X$.
We set $B_{i}:=B\cap X_{i}$ and define the set $J_{B}:=\{i\in I: B_{i}\not=\emptyset\}$. Then  we have an equivalence
\begin{align}
\Cofib(\Sigma_{+}^{\infty,\topp}P_{U}(X & \setminus B)\to \Sigma_{+}^{\infty,\topp}P_{U}(X))\label{jnkd823jnsd}\\
& \simeq \Cofib \Big( \Sigma_{+}^{\infty,\topp} \big( \coprod_{i \in J_B} P_{U_i}(X_i \setminus B_i) \big) \to \Sigma_{+}^{\infty,\topp} \big( \coprod_{i \in J_B} P_{U_i}(X_i) \big) \Big)\, .\notag
\end{align}
By the definition of the bornology of a free union,  $J_B$ is a finite set. We can use Property~\ref{ewifjweoifiewfewfwfewfewfwfwf} of the functor $\Sigma_{+}^{\infty,\topp} $ to identify \eqref{jnkd823jnsd} with
\begin{equation}
\label{knjdvjnk3sd}
\bigoplus_{i \in J_B} \Cofib \big( \Sigma_{+}^{\infty,\topp} P_{U_i}(X_i \setminus B_i) \to \Sigma_{+}^{\infty,\topp} P_{U_i}(X_i) \big)\, .
\end{equation}
In order to get $Q(X) $ we now  take the limit over   $B$ in $\cB$   and then the colimit over $U$ in $\cC$. We can restructure the index set of the limit  replacing $\lim_{B\in \cB}$ by the two limits $\lim_{J \subseteq I, J \textit{finite}} \lim_{ (B_{i})_{i \in J} \in \prod_{i\in J}\cB_{i}}$ and arrive at
\begin{align*}
Q(X) &\simeq \colim_{U \in \cC} \lim_{\substack{J \subseteq I\\ J \textit{finite}}}   \lim_{ (B_{i})_{i \in J} \in \prod_{i\in J}\cB_{i}} \bigoplus_{i \in J} \Cofib \big( \Sigma_{+}^{\infty,\topp} P_{U_i}(X_i \setminus B_{i}) \to \Sigma_{+}^{\infty,\topp} P_{U_i}(X_i) \big)\\ &\simeq \colim_{U \in \cC} \lim_{\substack{J \subseteq I\\ J \textit{finite}}} \bigoplus_{i \in J} \lim_{B_{i} \in \cB_i} \Cofib \big( \Sigma_{+}^{\infty,\topp} P_{U_i}(X_i \setminus B_{i}) \to \Sigma_{+}^{\infty,\topp} P_{U_i}(X_i) \big)\\
& \simeq \colim_{U \in \cC} \prod_{i \in I} \lim_{B_{i} \in \cB_i} \Cofib \big( \Sigma_{+}^{\infty,\topp} P_{U_i}(X_i \setminus B_{i}) \to \Sigma_{+}^{\infty,\topp} P_{U_i}(X_i) \big)\, .
\end{align*} 
In the last step we want to commute the colimit of $U$ in $\cC$ with the product. To this end we note that in view of the definition of the coarse structure of a free union  we have a bijection 
$\prod_{i\in I}\cC_{i}\cong \cC$  sending the family $(U_{i})_{i\in I}$ of entourages to the entourage $\bigsqcup_{i\in I}U_{i}$.
The inverse sends $U$ in $\cC$ to the family $(U\cap(X_{i}\times X_{i}))_{i\in I}$. Using that in $\Sp$ products distribute over filtered colimits (see Remark \ref{qerugiergregrefrefwerf})
 we  get an equivalence \begin{eqnarray*}
\lefteqn{
  \colim_{U \in \cC} \prod_{i \in I} \lim_{B_{i} \in \cB_i} \Cofib \big( \Sigma_{+}^{\infty,\topp} P_{U_i}(X_i \setminus B) \to \Sigma_{+}^{\infty,\topp} P_{U_i}(X_i) \big)}&&\\&\simeq&  \prod_{i\in I} \colim_{U_{i}\in \cC_{i}}\lim_{B_{i} \in \cB_i} \Cofib \big( \Sigma_{+}^{\infty,\topp} P_{U_i}(X_i \setminus B) \to \Sigma_{+}^{\infty,\topp} P_{U_i}(X_i) \big)\\&\simeq&
  \prod_{i\in I}Q(X_{i})
\end{eqnarray*}
{which proves the assertion of the lemma.}
\end{proof}

Let $\bC$ be a complete and cocomplete stable $\infty$-category. We now use the tensor structure  \eqref{wervewrvwevfevfdvsfdvdfsvsfv}.

%
Let $C$ be an object of $\bC$. 
\begin{kor}\label{goirjgio334t43gregergrege}
\begin{equation}
\label{eq:sfd98230990}
C\wedge Q:\BC\to \Sp\, , \quad X\mapsto (C\wedge Q(X))
\end{equation}
is a $\bC$-valued coarse homology theory. 
\end{kor}

\begin{proof}The proof
follows from Theorem~\ref{efwifjwifou23984u23984u3294823424234} and  the fact that the functor $C\wedge {-}$ preserves colimits.
\end{proof}

\begin{ex}\label{hhfweweiufufwefwefwef}  
We can use the coarsification in order to produce an example of a non-additive coarse homology theory. We choose a spectrum $C$ in $\Sp$ such that the functor
$C\wedge - \colon \Sp\to \Sp$ does not send
$\prod_{i\in I}S$  to $\prod_{i\in I}C$. Here $S$ is the sphere spectrum and $I$ is some suitable infinite set.  
For example we could take $E:=\bigoplus_{p} H\Z/p\Z$, where $p$ runs over all primes.
Then $C\wedge Q$ is not additive.
\hB
\end{ex}


\subsection{Comparison of coarse homology theories}\label{ewfiweifioeuoiu29834u24434}

Let $\bC$ be a cocomplete stable $\infty$-category. In the present subsection we discuss the problem to which extend a coarse homology theory 
 $E:\BC\to \bC$  is determined by the value $E(*)$, or at least by the values $E(X)$  on discrete bornological coarse spaces $X$.

\begin{ex}\label{ewfioweufoo234234}
We consider the category $\Sp$ and the complex $K$-theory spectrum $KU$ in $\Sp$. In the present paper we construct various examples of coarse homology theories which all assume the value $KU$ on the one-point space:
\begin{enumerate}
\item $KU \wedge Q$ (Corollary \ref{goirjgio334t43gregergrege})
\item $Q(KU \wedge \Sigma^{\infty}_{+})^{\lf}$ (Definition \ref{wefiwiofuwe987u982523453453})
\item $QK^{\an,\lf}$ (Definition \ref{wefiwiofuwe987u982523453453} and Definition \ref{foiehweiofui23ur892342424})
\item $\KX$ and $\KXql$ (Definition \ref{oiwefewiofu9234234234324})
\end{enumerate}
A further example is constructed in \cite{ass}.
One could ask whether these homology theories are really different. 

Note that  $Q(KU \wedge \Sigma^{\infty}_{+})^{\lf}$ and $QK^{\an,\lf}$ are additive.
But we do not expect that $KU \wedge Q$ is additive for a {similar reason as in Example~\ref{hhfweweiufufwefwefwef}.}

If $X$ is a  bornological coarse spaces of locally bounded geometry  {(see Definition \ref{fwejfoiwejfewojfoefjewfewfewf})}, then we have
the  coarse assembly map {(see Definition \ref{wefijweiofoiwe245435})}
$$\mu \colon  QK_{*}^{\an,\lf}(X) \to \KX_\ast(X)\, .$$
The known counter-examples to the coarse Baum--Connes conjecture show that this map is   not always surjective.     
But this leaves open the possibility  {of} the existence of a different transformation which is an equivalence.

In general we have a natural transformation 
$$\KX\to \KXql\, .$$
The goal of the present subsection  is to provide positive results asserting that such a transformation is an equivalence on certain spaces.
%
\hB
\end{ex}

 In the following we will assume that $\bC$ is a  cocomplete  stable $\infty$-category. 
%
 We consider two  coarse homology theories $E,F:\BC\to \bC$  and a  natural transformation $$T :  E \to F\, .$$  

  Let $\cA$ be a set of objects in $\BC$ and recall Definition \ref{defn:sdf09232} of $\Sp\cX\langle \cA\rangle$.

 Let $X$ be a bornological coarse space.
\begin{lem}\label{fkjfwefowefweopfewfe}Assume:
\begin{enumerate}
\item $T\colon E(Y)\to F(Y)$ is an equivalence for all $Y$ in $\cA$.
\item $\Yo^{s}(X)\in \Sp\cX\langle \cA\rangle$.
\end{enumerate}
Then 
  $E(X)  \to  F(X)$ is an equivalence.
\end{lem}
\begin{proof}
We use  Corollary \ref{qwrfglkjqodfewewfdewdqeqwedew1}.
The assertion follows from the fact that the extension of $E$ and $F$ to functors 
 $E^{\la}:\Sp\cX\to \bC^{\la}$ and $F^{\la}:\Sp\cX\to \bC^{\la}$ preserve colimits. 
\end{proof}

  Combining Theorem  \ref{jfweofjwoeifjewfoewfewfewfewf}
 and Lemma  \ref{fkjfwefowefweopfewfe} we  get the following consequence.

 Let $X$ be a bornological coarse space.
 
\begin{kor}\label{thm:sdf98245csv}
Assume:
\begin{enumerate}
\item 
 $T\colon E(Y)\to F(Y)$ is an equivalence for all discrete bornological coarse spaces $Y$.
 \item $X$ has weakly finite asymptotic dimension (Definition \ref{wegoijobgwtwferfrewfer}).  
\end{enumerate}
Then $T\colon E(X)\to F(X)$ is 
   an equivalence.
\end{kor}

For a transformation between additive coarse homology theories we can simplify the assumption on the transformation considerably. 

{We assume that $\bC$ is a complete and cocomplete stable $\infty$-category.}
Let $T\colon E\to F $ be as above,  and let $X$ be a bornological coarse space.

\begin{theorem}\label{wekfhweuifhweiui23u4242342rf}Assume:
\begin{enumerate}
\item 
 $T:E(*)\to F(*)$ is an equivalence.
 \item $E$ and $F$ are additive.
\item  $X$ has weakly finite asymptotic dimension.
\item $X$ has the minimal compatible bornology.
 \end{enumerate}
Then $T\colon E(X)\to F(X)$ is 
   an equivalence.
\end{theorem}

\begin{proof}
Let $\cA_*$ be the class consisting of the motivic coarse spectra of the form $\Yo^s(\bigsqcup^{\free}_{I}*)$ for a set $I$. An inspection of the  intermediate steps of Theorem  \ref{jfweofjwoeifjewfoewfewfewfewf} shows that $\Yo^{s}(X)\in \Sp\cX\langle \cA_*\rangle$. We now conclude using Lemma \ref{fkjfwefowefweopfewfe}.
\end{proof}

%
%

\section{Locally finite homology theories and coarsification}\label{ihjiovqceqwecewcqwecq}

\subsection{Locally finite homology theories}\label{wkhwfhfwefuwefewfvewvewffewff}

In this section we review locally finite homology theories in the context of $\TopBorn$ of topological spaces with an additional bornology. 
It gives a more systematic explanation of the appearance of the limit over bounded subsets in \eqref{erwerjh23ui423423424}. We will see that the constructions of Section \ref{kjoijfoewifueofieuwfewfwefewfef} can be extended to a construction of coarse homology theories from locally finite homology theories.

Note that our notion of a locally finite homology theory differs from the one studied by Weiss--Williams \cite{ww_pro}, but our set-up contains a version of classical Borel--Moore homology of locally finite chains. There is also related work by Carlsson--Pedersen \cite{MR1634649}.

One  motivation to develop this theory in the present paper  is to capture the example $K^{\an,\lf}$ of analytic locally finite $K$-homology. In order to relate this with the homotopy theoretic version of the locally finite homology theory
(as discussed in Example \ref{weoifujweoifowe234234345}) associated to   spectrum $KU$ we derive the analogs of the classification results of Weiss--Williams \cite{ww_pro} in the  context of topological bornological spaces.
 
We think that the notions and results introduced     in this section  are interesting in itself, independent of their application to coarse homology theories.

The coarsification of the analytic locally finite $K$-homology is the domain of the assembly map {discussed} in Section \ref{sec:kjnbsfd981}.  
In \cite{ass} we will introduce the closely related concept of    a local homology theory which is better suited as a domain of the coarse assembly map in general.
On nice spaces it behaves like a locally finite homology theory.


\subsubsection{Topological bornological spaces}

In order to talk about locally finite homology theories we introduce the category $\TopBorn$ of topological bornological spaces.

Let $X$ be a set with a topology $\cS$\index{$\cS$|see{topology}}\index{topology} and a bornology $\cB$.
\begin{ddd}\label{fewilfjweiofoi4325345345}
$\cB$ is compatible\index{compatible!bornology and topology} with $\cS$ if $\cS \cap \cB$ is cofinal in $\cB$ and $\cB$ is closed under taking closures.
\end{ddd}
 
\begin{ex}\label{fljwlewkiou2ori23r2323}
 In words, compatibility means 
 that  every bounded subset   has a bounded closure and a bounded open neighbourhood.
 
A compatible bornology on a topological space $X$ contains all relatively compact subsets. Indeed, let $B$  be a relatively compact subset of $X$. By compatibility, for every point  $b$ in $X$ there is a bounded  open neighbourhood $U_{b}$ of $b$. Since $B$ is relatively compact, there  
  exists a finite subset $I$  of $\bar B$ such that $B\subseteq \bigcup_{b\in B} U_{b}$. But then $B$ is bounded since a finite union of bounded subsets is bounded and a bornology is closed under taking subsets.
\hB
\end{ex}

\begin{ddd} 
A topological bornological space\index{topological bornological space}\index{space!topological bornological} is a triple $(X,\cS,\cB)$\index{$(X,\cS,\cB)$} consisting of a set $X$ with a topology $\cS$ and a bornology $\cB$ which are compatible. 
\end{ddd}

In general we will just use the symbol $X$ to denote topological bornological spaces.

Let $X,X'$ be  topological bornological spaces, and let $f:X\to X'$ be a map between the underlying sets.

\begin{ddd}
 $f$ is a morphism   between topological bornological spaces\index{morphism!between topological bornological spaces} 
 if $f$ is continuous and proper.
 \end{ddd}
 
Note that the word \emph{proper} refers to the bornologies and means that $f^{-1}(\cB^{\prime})\subseteq \cB$.
 We let $\TopBorn$\index{$\TopBorn$} denote the small category of  very small topological bornological spaces and morphisms.

\begin{ex}\label{eigoffewerevfdsvsdfvsdfv}
A topological space is compact if it is Hausdorff and every open covering admits a finite subcovering.  A subspace of a topological spaces is relatively compact if its closure is compact. 
By convention, a locally compact topological space is a topological space which is Hausdorff and such that every point admits a compact neighbourhood.  

If $X$ is a locally compact topological space, then the family of relatively compact subsets of $X$ turns $X$ into a topological bornological space. We get an inclusion $$\Top^{\lc}\to \TopBorn$$\index{$\Top^{\lc}$}of the category of locally compact spaces and proper maps as a full subcategory. 

Assume that $X$ is locally compact and $U$  is an open subset of $X$. If we equip $U$ with the induced topology and bornology, then $U\to X$ is a morphism in $\TopBorn$. But in general, $U$ equipped with these induced structures  is not contained in the image of the functor  $\Top^{\lc}\to \TopBorn$ since the bornology on $U$ induced from $X$ is too large.

Consider, e.g., the open subset $(0,1)$ of $\R$. The induced bornology of $(0,1)$ is the maximal bornology, while e.g.\ $(0,1)$ is not relatively compact in $(0,1)$.

In order to define the notion of descent, on the one hand we want to work with open coverings. On the other hand the
inclusions of the open subsets should be morphisms in the category. In $\Top^{\lc}$ this is impossible and motivates to work in  $\TopBorn$.  This category just provides to minimal amount of structure for the consideration of locally finite homology theories.  \hB
\end{ex}

\begin{lem}
The category $\TopBorn$ has {all} very small products.\index{product!in $\TopBorn$}
\end{lem}

\begin{proof}
Let $(X_{i})_{i \in I}$ be a family of topological bornological spaces. 
Then the product of the family is represented by the topological bornological space  $ \prod_{i \in I} X_i$ 
whose underlying topological space is the product in of the underlying topological spaces, and whose bornology is the minimal one such that the projections to the factors are proper (compare with the bornology  of a product of bornological coarse spaces in see Lemma \ref{wleifjewfiewiofuew9fuewofewf}).
  \end{proof}

\begin{ex}
The category $\TopBorn$ is not complete. It does not have a final objects.
In order to get a complete and cocomplete category one could remove the condition on the bornology that all points are bounded, see \cite{Heiss:2019aa}.
\hB
\end{ex}

\begin{ex}\label{goiejrgoierjgeorgjreg43jut9034}
The category $\TopBorn$ has a symmetric monoidal structure\index{$-\otimes-$!topological bornological space}
\[(X,X^{\prime})\mapsto X\otimes X^{\prime}\, .\]
The underlying topological space of $X\times X^{\prime}$ is the product of the underlying topological spaces of $X$ and $X^{\prime}$. The bornology of $X\times X^{\prime}$ is generated by the sets $B\times B^{\prime}$ for all bounded subsets $B$ of $X$ and $B^{\prime}$ of $X^{\prime}$. The tensor unit is the one-point space.
As in the case of bornological coarse spaces (Example \ref{eiofweoifwefuewfieuwf9wwfwef})  the tensor product in general differs from the cartesian product.
\hB
\end{ex}

\begin{lem}
The category $\TopBorn$ has all very small coproducts.\index{coproducts!in $\TopBorn$}
\end{lem}

\begin{proof}
Let $(X_{i} )_{i\in I}$ be a family of topological bornological spaces.
Then the coproduct of the family is represented by the topological bornological space $\coprod_{i\in I} X_{i} $
whose underlying topological space is the coproduct of the underlying topological spaces, and whose bornology
is given by 
 \[\cB:=\{B\subseteq X\:|\: (\forall i\in I: B\cap X_{i}\in \cB_{i})\}\, ,\]
 where $\cB_{i}$ is the bornology of $X_{i}$.
\end{proof}

Note that the bornology of the coproduct is given by the same formula as in the case of bornological coarse spaces; see Definition \ref{ewoifweofe90wf09wef2334}.

{Let $(X_{i} )_{i\in I}$ be a family of topological bornological spaces.
\begin{ddd}\label{foi24r443535}
The free union\index{free union!in $\TopBorn$}\index{union!free!in $\TopBorn$} $\bigsqcup_{i\in I}^{\free} X_{i} $ of the family is the following topological bornological space:
\begin{enumerate}
\item The underlying set of the free union is the disjoint union $\bigsqcup_{i\in I} X_{i}$.
\item The topology of the free union is the one of the coproduct of topological spaces.
\item The bornology of the free union is given by $\cB\big\langle \bigcup_{i\in I} \cB_{i}\big\rangle$, where $\cB_{i}$ is the bornology of $X_{i}$.
\end{enumerate}
\end{ddd}}

\begin{rem}
The free union should not be confused with the coproduct. The topology of the free union is the same, but the bornology is smaller. The free union plays a role in the discussion of additivity of locally finite theories.
\hB
\end{rem}

 \subsubsection{Definition of locally finite homology theories}

We will use the following notation.
 Let $E $ be any functor from some category to a stable $\infty$-category. If $Y\to X$ is a morphism in the domain of $E$, then we write
$$E(X,Y):=\Cofib \big( E(Y)\to E(X) \big)\, .$$

We now introduce the notion of local finiteness. It is this property of a functor from topological bornological spaces
to spectra which involves the bornology and distinguishes locally finite homology theories amongst all homotopy  invariant  and excisive functors.
 
Let $\bC$ be a  complete stable $\infty$-category. We consider a functor $$E:\TopBorn\to {\bC}\, .$$
\begin{ddd}\label{kfjwekfjklewfefwefwfewf}
$E$ is locally finite\index{locally finite!functor} if the natural morphism
$$E(X)\to \lim_{B\in \cB}  E(X,X\setminus B)$$
is an equivalence for all $X$ in $\TopBorn$.
\end{ddd}

\begin{rem}\label{eofjewoifjefoewi23452345}
$E$ is locally finite if and only if
\[\lim_{B\in \cB} E(X\setminus B)\simeq 0\]
{for all $X$ in $\TopBorn$.}
\hB
\end{rem}

\begin{ex}
A typical feature which is captured by the notion of local finiteness is the following. Let $X$ be in $  \TopBorn$ and assume that its bornology $\cB$ has a countable cofinal subfamily $(B_{n})_{n\in \nat}$. 
  Then we have the equivalence
$$E(X)\simeq \lim_{n\in \nat} E(X,X\setminus B_{n})\, .$$
Assume that $E$ takes values in spectra $\Sp$. Then in homotopy groups this is reflected by the presence of Milnor $\lim^{1}$-sequences
$$0\to {\lim_{n\in \nat}}^{1} \pi_{k+1}(E(X,X\setminus B_{n}))\to \pi_{k}(E(X))\to \lim_{n\in \nat }  \pi_{k}(E(X,X\setminus B_{n}))\to 0$$
for all $k$ in  $\Z$.
The presence of the $\lim^{1}$-term shows that one can not construct graded abelian group valued locally finite homology as the limit
$\lim_{n\in \nat }  \pi_{k}(E(X,X\setminus B_{n}))$. In general the latter would not satisfy descent in the sense that we have Mayer--Vietoris sequences for appropriate covers. 
\hB
\end{ex}

{Let $\bC$ be a  complete stable  $\infty$-category, and let
$E:\Top\Born\to {\bC}$ be any functor.}

\begin{ddd}\label{fiojwoifjwoifuwe45435}
We define the locally finite evaluation\index{locally finite!evaluation} $$E^{\lf}:\TopBorn\to {\bC}$$
by
$$E^{\lf}(X) :=\lim_{B\in \cB} E(X,X\setminus B) .$$\index{$-^{\lf}$}
\end{ddd}

\begin{rem}
In order to turn this description of the locally finite evaluation on objects into a definition of a functor we use right Kan extensions as described in Remark \ref{fjwefoijewojf923ur9u23refewfwefewf}. We consider the category $\TopBorn^{\cB}$ of pairs $(X,B)$, where $X$ is a topological bornological space and $B$ is a bounded subset of $X$.  A morphism $f:(X,B)\to  (X^{\prime},B^{\prime})$ is a continuous map such that $f(B)\subseteq B^{\prime}$.
We have a forgetful functor $$p:\TopBorn^{\cB}\to \TopBorn \, , \quad (X,B) \mapsto X\, .$$
The locally finite evaluation is then defined as the right Kan extension of the functor
$$\tilde E: \TopBorn^{\cB}\to {\bC}\, , \quad (X,B)\mapsto E(X,X\setminus B)$$ as indicated in
the following diagram
$$\xymatrix{\TopBorn^{\cB}\ar[rr]^-{\tilde E }\ar[dd]^{p}&&{\bC}\\
&\ar@{::>}[ul]&\\
\TopBorn\ar@{..>}@/_1pc/[uurr]_-{E^{\lf}}
}$$
\hB
\end{rem}

We have a canonical natural transformation
\begin{equation}\label{fwefqewfewedqdqwedqwedq}
E\to E^{\lf}\, .
\end{equation}
By Definition \ref{kfjwekfjklewfefwefwfewf} the
functor $E$ is locally finite if and only if  the natural {transformation}  \eqref{fwefqewfewedqdqwedqwedq} is an equivalence.

{Let $\bC$ be a complete   stable  $\infty$-category, and let}
  $E:\Top\Born\to {\bC}$ be any functor. 
\begin{lem}\label{fjwelkfwjoii4234}
$E^{\lf}$ is locally finite.
\end{lem}

\begin{proof}
Let $X$ be a topological bornological space. We have
\begin{eqnarray*}
\lim_{B\in \cB} E^{\lf}(X\setminus B)&\simeq &\lim_{B\in \cB} \lim_{B^{\prime}\in \cB} \Cofib\big(E(X\setminus (B\cup B^{\prime}))\to E(X\setminus B)\big)\\
&\simeq & \lim_{B^{\prime}\in \cB} \lim_{B\in \cB}   \Cofib\big(E(X\setminus (B\cup B^{\prime}))\to E(X\setminus B)\big)\\&\simeq&
\lim_{B^{\prime} \in \cB} \lim_{B\in \cB, B^{\prime}\subseteq B}  \Cofib\big(E(X\setminus (B\cup B^{\prime}))\to E(X\setminus B)\big)\\&\simeq&\lim_{B^{\prime} \in \cB} \lim_{B\in \cB, B^{\prime}\subseteq B}  \Cofib\big(E(X\setminus B\ )\to E(X\setminus B)\big)\\&\simeq&
 0\, .
\end{eqnarray*}
By Remark \ref{eofjewoifjefoewi23452345} this equivalence implies that $E^{\lf}$ is locally finite.
\end{proof}

Note that Lemma \ref{fjwelkfwjoii4234} implies that the natural transformation   \eqref{fwefqewfewedqdqwedqwedq} induces an equivalence
\[E^{\lf}\simeq (E^{\lf})^{\lf}\, .\]
 
We now introduce the notion of homotopy invariance of functors defined on $\TopBorn$. 
To this end we equip the unit interval $[0,1]$ with its maximal bornology. Then for every topological bornological space $X$ the projection $[0,1]\otimes X\to X$ (see Example \ref{goiejrgoierjgeorgjreg43jut9034} for $-\otimes-$) is a morphism in $\TopBorn$. Moreover, the inclusions $X\to [0,1]\otimes X$ are morphisms. This provides a notion of homotopy\index{homotopy!topological bornological space} in the category $\TopBorn$. In the literature one often talks about proper homotopy. We will not add this word \emph{proper} here since it is clear from the context $\TopBorn$ that all morphisms are proper, hence also the homotopies.

Let $E:\TopBorn\to {\bC}$ be a functor.
\begin{ddd}
 $E$ is homotopy invariant  \index{homotopy invariant!topological bornological space} if the projection 
$[0,1] \otimes X\to X$ induces an equivalence $E([0,1] \otimes X)\to E(X)$.
\end{ddd}

Next we shall see that the combination of homotopy invariance  {and} local finiteness of a  functors allows Eilenberg swindle arguments. Similar as in the bornological coarse case we encode this in the property that
the functor vanishes on certain flasque spaces.

\begin{ddd}
A topological bornological space $X$ is flasque\index{flasque!topological bornological space}\index{topological bornological space!flasque} if it admits a morphism
$f:X\to X$ such that:
\begin{enumerate}
\item $f$ is homotopic to $\id$.
\item For every bounded $B$ in $X$ there exists $k$ in $\nat$ such that $f^{k}(X)\cap B=\emptyset$.
\end{enumerate}
 \end{ddd}
 
\begin{ex}
Let $X$ be a topological bornological space, 
and let $[0,\infty)$ have the bornology 
given by relatively compact subsets.
  Then $[0,\infty) \otimes  X$   is flasque. We can define the morphism $f:[0,\infty) \otimes  X\to [0,\infty) \otimes  X$ by $f(t,x):=(t+1,x)$.\hB
 \end{ex}

Let $\bC$ be a   complete stable  $\infty$-category, and let $E\colon \TopBorn\to {\bC}$ be a functor.
\begin{lem} \label{efewoifeuwfoiewufoiew89u98234234324}Assume:
\begin{enumerate}
\item 
$E$ is locally finite.
\item $E$ is   homotopy invariant.
\item  $X$ is flasque. 
\end{enumerate} 
Then $E(X)\simeq 0$.
\end{lem}

\begin{proof}
 The argument is very similar to the proof of Lemma \ref{ekfjweiofu9o4r34555}. Let $X$ be a topological bornological space with bornology $\cB$, and let $f\colon X\to X$ implement flasqueness.
We consider the diagram 
\[\xymatrix{
 \lim_{B\in \cB} E(X,X\setminus B) \ar[rr]^-{\id_{X,*}} \ar@{=}[d] & &E(X,X)  \\
 \lim_{B\in \cB} E(X,X\setminus B) \ar@{=}[d] \ar[rr]^-{\id_{X,*}+f_{*}} & &E(X,f(X))\ar[u] \\
 \lim_{B\in \cB} E(X,X\setminus B) \ar@{=}[d] \ar[rr]^-{\id_{X,*}+f_{*}+f^{2}_{*}} & & E(X, f^{2}(X))\ar[u] \\
\vdots \ar@{=}[d] & & \vdots \ar[u] \\
 \lim_{B\in \cB}E(X,X\setminus B)  \ar[rr]^-{\sum_{k}f^{k}_{*} } \ar@{=}[d] & & \lim_{k\in \nat}E(X,f^{k}(X)) \ar[u] \ar[d]^{!}\\
  \lim_{B\in \cB}E(X,X\setminus B)  && \lim_{B\in \cB}E(X,X\setminus B) \\
 E(X) \ar[rr]^-{F} \ar[u]_{\simeq} & &E(X)\ar[u]_{\simeq} 
}\]

In order to define the morphism marked by $!$ we note that for every $B$ in $\cB$ we can choose $k$ in $\nat$ such that
$f^{k}(X)\cap B=\emptyset$. Then we have a map of pairs $(X,  f^{k}(X))\to(X,X\setminus B)$.

All the small cells in the upper part commute. For example the filler of the upper square is  an equivalence $f_{*}\simeq 0$. It   is obtained from the factorization   of $f$ as $X\to f(X) \to X$. Indeed, we have the diagram $$\xymatrix{E(X\setminus B)\ar[r]\ar[d]&E(f(X))\ar[d]\ar[r]&E(f(X))\ar[d]\\E(X)\ar[d]\ar[r]&E(f(X))\ar[r]\ar[d]&E(X)\ar[d]\\E(X,X\setminus B)\ar[r]&0\ar[r]&E(X,f(X))}$$
where the columns are pieces of fibre sequences and the lower map is the definition of the induced map $f_{*}$.
The fillers provide the desired equivalence $f_{*}\simeq 0$.
 
The upper part of the diagram 
defines the map denoted by the suggestive symbol   $\sum_{k}f^{k}_{*}$.
This map further induces $F$.

The construction of $F$ implies that
$$E(\id_{X})+E(f)\circ E(F)\simeq E(F)\, .$$
Using that $E$ is homotopy  invariant    and $f$ is homotopic to    $\id_{X}$   we get 
$$E(\id_{X})+E(F)\simeq E(F)$$ and hence $E(\id_{X})\simeq 0$. This implies $E(X)\simeq 0$.
\end{proof}

We now discuss the notion of excision.
If $Y $ is a subset of a topological bornological space $X$, then we consider $Y$ as a topological bornological space with the induced structures.

In order to capture all examples we will consider three versions of excision. 

Let $\bC$ be a stable $\infty$-category, and let 
 $E:\TopBorn\to {\bC}$ be a functor. 
\begin{ddd}
 $E$ satisfies (open or closed) excision\index{excision!open}\index{excision!closed}\index{open excision}\index{closed excision} if for every (open or closed) decomposition $(Y,Z)$ of a topological bornological space
$X$ we have a push-out
$$\xymatrix{E(Y\cap Z)\ar[r]\ar[d]&E(Y)\ar[d]\\E(Z)\ar[r]&E(X)}$$
\end{ddd}

\begin{ex}\label{iowdqwioduqz9q2342343424}
We consider the functor
 $$\ell\circ   C^{\sing}\circ \cF_{\cB}  :\TopBorn\to \Ch_{\infty}\, ,$$ where $\ell$ is as in \eqref{qewflkklefmklqwefqwefqfqefqwefe},
 $\cF_{\cB}:\TopBorn \to \Top$ forgets the bornology, and 
 $C^{\sing}:\Top\to   \Ch$ is the   singular chain complex functor. 
 This composition is  homotopy invariant and well-known to   satisfy open excision.
Closed   excision fails.

The functor $$\Sigma_{+}^{\infty,\topp}\circ \cF_{\cB}:\TopBorn\to \Sp$$ (see \eqref{vervelkvjekvevweqvevwv}) is homotopy invariant and satisfies open excision (Property \ref{ewifjweoifiewfewfwfewfewfwfwf} of $\Sigma_{+}^{\infty,\topp}$). 
Again, closed excision fails for $\Sigma^{\infty,\topp}_{+}\circ \cF_{\cB}$.

We will see below {(combine  Assertion\ref{goegoeri45345345345} and \ref{fwifjweoifwefwefwef})} that analytic locally finite $K$-homology $K^{\an,\lf}$ satisfies closed excision. We do not know if it satisfies open excision.
\hB
\end{ex}

The following notion of weak excision is an attempt for a concept which comprises both open and closed excision.
Let $X$ be a topological space, and let $\cY:=(Y_{i})_{i\in I}$ be a filtered family of subsets. 
\begin{ddd}\label{rojepgokerggerg}
 $\cY$ is called a big family\index{big family!in $\TopBorn$} if for every $i$ in $I$ there exists
$i^{\prime}$ in $I$ such that $Y_{i^{\prime}}$ contains an open neighbourhood of $\bar Y_{i}$.
 \end{ddd}
\begin{ex}
The bornology of a topological bornological space is a big family. The condition introduced in Definition \ref{rojepgokerggerg} is satisfied by compatibility between the topology and the bornology, see Definition \ref{fewilfjweiofoi4325345345}.
\hB
\end{ex}

Let $\bC$ be a  cocomplete stable  $\infty$-category.
If $\cY$ is a big family on $X$ and $E:\TopBorn\to {\bC}$ is a functor, then we write $$E(\cY):=\colim_{i\in I}E(Y_{i})\, .$$
We say that a pair $(\cY,\cZ)$ of  two big families $\cZ$ and $\cY$  is a decomposition of $X$ if there exist members $Z$ and $Y$ of $\cZ$ and $\cY$, respectively,  such that $X=Z\cup Y$. 

Let $\bC$ be a cocomplete stable   $\infty$-category, and let
 $E:\TopBorn\to {\bC}$ be a functor.
\begin{ddd}
  $E$ satisfies weak excision\index{excision!weak}\index{weak excision} if for every decomposition $(\cY,\cZ)$ of a topological bornological space
$X$ into two big families  we have a push-out
$$\xymatrix{E(\cY\cap \cZ)\ar[r]\ar[d]&E(\cY)\ar[d]\\E(\cZ)\ar[r]&E(X)}$$
\end{ddd}
By cofinality arguments it is clear that open or closed excision implies weak excision.


{Let $\bC$ be a cocomplete and cocomplete  stable  $\infty$-category, and let $E:\TopBorn\to \bC$ be a functor.}
\begin{ddd}\label{rewfiewrjgoiegergreefwerfw}
 $E$ is a locally finite homology theory if it has the following poperties:\index{locally finite!homology theory}\index{homology theory!locally finite}
\begin{enumerate}
\item $E$ is locally finite.
\item $E$ is homotopy invariant.
\item $E$ satisfies weak excision.
\end{enumerate}
\end{ddd}
 
\begin{rem}
\index{wrong way maps}\index{shriek maps|see{wrong way maps}}\index{umkehr maps|see{wrong way maps}}
A locally finite homology theory $E$  gives rise to wrong way maps for bornological open inclusions. We consider a topological bornological space $ X$ with bornology $\cB_{X}$ and an open  subset $U$ in $X$. Let $\cB_{U}$ be some compatible  bornology on $U$ such that the inclusion $U\to X$ is bornological. So $\cB_{U}$ may be smaller than the induced bornology $\cB_{X}\cap U$. We further assume that every $B$ in $\cB_{U}$ has an open neighbourhood $V$ such that
$\cB_{U}\cap V=\cB_{X}\cap V$.

We let $\tilde U$ denote the topological bornological space with the induced topology and the bornology $\cB_{U}$. We use the tilde symbol in order to distinguish that space from the topological bornological space $U$ which has by definition the induced bornology $U\cap \cB_{X}$.

In contrast to $U\to X$ the inclusion $\tilde U\to X$ is in general not a morphism in $\TopBorn$. The observation in this remark is that
we have a wrong-way morphism
$$E(X)\to E(\tilde U)$$ defined as follows:
For a bounded closed subset $B$ in $\cB_{U}$ we get  open coverings $(X\setminus B,U)$ of $X$ and
$(U\setminus B,V)$ of $U$. Using open descent for $E$  twice  we get the excision equivalence \begin{equation}\label{gregjnergkhui34h53453545455}
E(X,X\setminus B)\simeq E(V,V\setminus B)\simeq E(\tilde U,\tilde U\setminus B)\, .
\end{equation}
The desired wrong-way map is now given by
$$E(X)\simeq \lim_{B\in \cB_{X}}E(X,X\setminus B)\to\lim_{B\in \cB_{U}} E(X,X\setminus B)\stackrel{\eqref{gregjnergkhui34h53453545455}}{\to} \lim_{B\in \cB_{U}} E(\tilde U,\tilde U\setminus B)\simeq E(\tilde U)\, ,$$
where the first morphism is induced by the restriction of index sets along $\cB_{U}\to \cB_{X}$ which is defined since the   inclusion $\tilde U\to X$ was assumed to be  bornological.
 
 A typical instance of this is the inclusion of an open subset $U$ into a locally compact space $X$.
 In general, the induced bornology $\cB_{X}\cap U$ is larger than the bornology $\cB_{U}$ of relatively compact subsets.
If $B$ is relatively compact in $U$, then it has a relatively compact neighbourhood. 
So in this case the wrong-way map
$E(X)\to E(\tilde U)$ is defined.
These wrong-way maps are an important construction in index theory, but since they do not play any role in the present paper we will not discuss them further. 

Note that in  {\cite[Sec.~2]{ww_pro}} the wrong way maps are encoded in a completely different manner
 by defining  the functors themselves on {the} larger category $\cE^{\bullet}$ instead of $\cE$ (notation in \cite{ww_pro}).
\hB
\end{rem}
 
 \subsubsection{Additivity}
 
In the following we discuss additivity. Let $ E:\TopBorn\to {\bC}$ be a functor which satisfies weak excision. We consider a family 
$(X_{i})_{i\in I}$ of topological bornological spaces. Recall the Definition \ref{foi24r443535} of the free union $\bigsqcup_{i\in I}^{\free} X_{i}$. For $j$ in $I$ we can form the two big one-member  families of  {$\bigsqcup_{i\in I}^{\free} X_{i}$} consisting of 
$X_{j}$ and  {$\bigsqcup_{i\in I\setminus \{j\}}^{\free} X_{i}$.}
Excision provides  the first equivalence in the following  definition of projection morphisms
$$E \Big(  {\bigsqcup_{i\in I}^{\free}} X_{i} \Big) \simeq E(X_{j}) \oplus  E \Big(  {\bigsqcup_{i\in I\setminus \{j\}}^{\free}} X_{i} \Big) \to E(X_{j})\, .$$
These projections for all $j$ in $I$ together induce a morphism 
\begin{equation}\label{frwihfiio24554435}
E \Big(  {\bigsqcup_{i\in I}^{\free}} X_{i} \Big) \to  \prod_{i\in I}E(X_{i})\, .
\end{equation}

{Let $\bC$ be a {complete} stable $\infty$-category, and let}  $E:\TopBorn\to {\bC}$ be a weakly excisive functor.
\begin{ddd}
 $E$ is additive\index{additive!locally finite homology theory}\index{locally finite!homology theory!additive} if \eqref{frwihfiio24554435} is an equivalence for all    families $(X_{i})_{i\in I}$ of topological bornological spaces.
\end{ddd}

\begin{lem}\label{dsf89888}
If $E$ is locally finite and satisfies weak excision, then $E$ is additive.
\end{lem}

\begin{proof}
We consider a family 
$(X_{i})_{i\in I}$ of topological bornological spaces and  {their free union $X := \bigsqcup_{i\in I}^{\free} X_{i}$.}
If $B$ is a bounded subset of $X$, then  the subset $J(B):=\{i\in I\:|\: B\cap X_{i}\not=\emptyset\}$  of $ I$   is finite. 
By excision we have an equivalence
$$E(X,X\setminus B)\simeq   \bigoplus_{j\in J(B)} E( {X_j}, X_{j}\setminus B)\, .$$
We now form the limit over $B$ in $\cB$ in two stages $$\lim_{B\in \cB}\dots\simeq \lim_{J\subseteq I}\lim_{B\in \cB, J(B)=J}\dots\, ,$$ where
$J$ runs over the finite subsets of $I$.
The inner limit gives by local finiteness of $E$
$$\lim_{B\in \cB, J(B)=J} \bigoplus_{j\in J } E( {X_j}, X_{j}\setminus B)\simeq  \bigoplus_{ {j}\in J } E(X_{ {j}} )\, .$$
Taking now the limit over the finite subsets $J$ of $ I$,   and using
$$\lim_{J\subseteq I}  \bigoplus_{j\in J(B)} E(X_{j} )\simeq \prod_{i\in I} E(X_{i})$$
 and the local finiteness of $E$ again we get the equivalence
$$E(X)\simeq \prod_{i\in I} E(X_{i})$$
as claimed.
\end{proof}

Let $\bC$ be a complete and cocomplete stable $\infty$-category, and let $E:\TopBorn\to \bC$ be a functor.

\begin{kor}\label{tgiortgergfewferfw}
If $E$ is a locally finite homology theory, then $E$ is additive.
\end{kor}

In general it is notoriously difficult to check that a given functor is locally finite if it is not already given as $E^{\lf}$.
In the following we show a result {which} can be used to deduce local finiteness from weak excision and additivity, see Remark \ref{eofjewoifjefoewi23452345}.

Let $\bC$ be a {complete} stable $\infty$-category, and let $E:\TopBorn\to {\bC}$ be a weakly excisive  functor. 

\begin{ddd} 
$E$ is countably additive\index{countably additive!locally finite homology theory}\index{locally finite!homology theory!countably additive} if \eqref{frwihfiio24554435} is an equivalence for all   countable  families $(X_{i})_{i\in I}$ of topological bornological spaces.
\end{ddd}

An additive functor is countably additive.

We consider an  increasing family $(Y_{k})_{k\in \nat}$ of subsets  {of} a topological bornological  space~$X$.
 \begin{lem}\label{lijfweiofjou4r3435434534534}
 Assume:
 \begin{enumerate}
 \item\label{dvfkjkqwxsadcaxsdac} For   every bounded subset $B$ of   $X$ there exists $k$ in $\nat$ such that $B\subseteq Y_{k}$.
 \item $E$ is weakly excisive.
 \item  $E$ is countably additive. \end{enumerate} Then  $\lim_{n\in \nat} E(X\setminus Y_{n})\simeq 0$.
 \end{lem}

\begin{proof}
For $k,\ell$ in $\nat$ and $k\ge \ell$  we let $$f_{k,\ell}:X\setminus Y_{k}\to X\setminus Y_{\ell}$$ be the inclusion.
 We have a fibre sequence
$$\to \lim_{n\in \nat} E(X\setminus Y_{n})\to \prod_{n\in \nat} E(X\setminus Y_{n})\stackrel{d}{\to}  \prod_{n\in \nat} E(X\setminus Y_{n})\to $$
where $d$ is  described in the language of elements  by  $$d((\phi_{n})_{n}):=(\phi_{n}-f_{n+1,n}\phi_{n+1})_{n}\, .$$
We must show that $d$ is an equivalence.

For every $m$ in $\nat $ we have a morphism of  topological bornological  spaces $$g_{m}:=\sqcup_{n\ge m} f_{n,m} : {\bigsqcup_{\substack{n\in \nat\\ n\ge m}}^{\free}} X\setminus Y_{n}\to X\setminus Y_{m}\, .$$ The condition \ref{dvfkjkqwxsadcaxsdac} ensures that $g_{m}$ is proper. 
 
 Using the additivity of $E$, for every $m$ in $\nat$ we can define the morphism
$$s_{m}:\prod_{n\in \nat} E(X\setminus Y_{n})\stackrel{proj}{\to} \prod_{\substack{n\in \nat\\ n\ge m}} E(X\setminus Y_{n}) \simeq
E\big({\bigsqcup_{\substack{n\in \nat\\ n\ge m}}^{\free}} X\setminus Y_{n}\big)\stackrel{g_{m}}{\to} E(X\setminus Y_{m})\, .$$
Let $\pr_{m}:\prod_{n\in \nat} E(X\setminus Y_{n})\to E(X\setminus Y_{m})$ denote the projection. By excision we have the relation \begin{equation}\label{fkjhnrekvjververevverve}
\pr_{m}+f_{m+1,m}s_{m+1}\simeq s_{m}\, .
\end{equation}
We further have the relation \begin{equation}\label{fkjhnrekvjververevverve1}
s_{m}((f_{n+1,n}\phi_{n+1})_{n})\simeq f_{m+1,m}(s_{m+1}((\phi_{n})_{n}))\end{equation}
We can now define 
$$h: \prod_{n\in \nat} E(X\setminus Y_{n})\to  \prod_{n\in \nat} E(X\setminus Y_{n})$$
by
$$h((\phi_{n})_{n})=(s_{m}((\phi_{n})_{n}))_{m}\, .$$
We calculate \begin{eqnarray*}
(d\circ h)((\phi_{n})_{n})&\simeq &d((s_{m}((\phi_{n})_{n}))_{m})\\&\simeq&
(s_{m}((\phi_{n})_{n})-f_{m+1,m}(s_{m+1}((\phi_{n})_{n})))_{m}\\&\stackrel{\eqref{fkjhnrekvjververevverve}}{\simeq}&
( \phi_{m})_{m} 
\end{eqnarray*}
 and
\begin{eqnarray*}
(h\circ d)((\phi_{n})_{n})&\simeq & h((\phi_{n}-f_{n+1,n}\phi_{n+1})_{n})\\&\simeq&
(s_{m}((\phi_{n})_{n}))_{m}- (s_{m}((f_{n+1,n}\phi_{n+1})_{n}))_{m}\\&\stackrel{\eqref{fkjhnrekvjververevverve1}}{\simeq}& 
(s_{m}((\phi_{n})_{n}))_{m}- (f_{m+1,m}(s_{m+1}((\phi_{n})_{n})))_{m}\\&\stackrel{\eqref{fkjhnrekvjververevverve}}{\simeq}& (\phi_{m})_{m}\, .
\end{eqnarray*}
This calculation shows that $d$ is an equivalence.
Consequently, $$\lim_{n\in \nat} E(X\setminus Y_{n})\simeq 0\, ,$$
finishing the proof.
\end{proof}

\begin{rem}
If a functor $E:\TopBorn\to {\bC}$ is countably additive, homotopy invariant, and satisfies
weak excision, then it vanishes on flasque spaces. In fact, the weaker version of local finiteness shown in Lemma \ref{lijfweiofjou4r3435434534534} suffices to make a modification of the proof of Lemma \ref{efewoifeuwfoiewufoiew89u98234234324} work.  \hB
\end{rem}


\subsubsection{Construction of locally finite homology theories}

Let $\bC$ be a complete stable $\infty$-category, and let $E\colon\TopBorn\to {\bC}$ be a functor. 
Recall the Definition \ref{fiojwoifjwoifuwe45435} of the locally finite evaluation.

\begin{lem}\label{qdkjwqhdjkio233r}
If $E$ is homotopy invariant, then $E^{\lf}$ is homotopy invariant.
\end{lem}

\begin{proof}
%
We consider the  projection
 $[0,1] \otimes X\to X$.
Then the subsets $[0,1]\times B$  for bounded subsets  $B$ of $X$ are cofinal in the bounded subsets of $[0,1] \otimes X$. So we can conclude that
$E^{\lf}([0,1] \otimes X)\to E^{\lf}(X)$ is the limit of equivalences (by homotopy invariance of $E$)
$$E([0,1] \otimes X, [0,1] {\times} (X\setminus B))\to E(X,X\setminus B)$$ and consequently an equivalence, too.
%
\end{proof}

Let $\bC$ be a complete and  cocomplete stable $\infty$-category, and let $E\colon\TopBorn\to {\bC}$ be a functor.
\begin{lem}\label{fwljwefjoiuoi3224234}
If $E$ satisfies (open, closed or weak) excision, then
$E^{\lf}$ also satisfies (open, closed or weak) excision.
\end{lem}

\begin{proof}
We discuss open excision. The other two cases are similar. Let $(Y,Z)$ be an open decomposition of $X$. For every bounded $B$ we get an open decomposition $(Y\setminus B,Z\setminus B)$ of $X\setminus B$. The cofibre of the corresponding maps of push-out diagrams (here we use that $E$ satisfies open excision) is the push-out diagram
$$\xymatrix{E(Y\cap Z,(Y\cap Z)\setminus B)\ar[r]\ar[d]&E(Y,Y\setminus B)\ar[d]\\E(Z,Z\setminus B)\ar[r]&E(X,X\setminus B)}$$
Using stability of the category {$\bC$} the limit of these push-out diagrams over $B$ in $\cB$ is again a push-out diagram.
\end{proof}

Let $\bC$ be a complete and cocomplete stable $\infty$-category,  and let $E\colon\TopBorn\to {\bC}$ be a functor.  \begin{prop}\label{prop:sfd89230}
If $E$ is homotopy invariant and satisfies weak excision, then its locally finite evaluation $E^{\lf}$ is a locally finite homology theory.
\end{prop}

\begin{proof}
By Lemma~\ref{fjwelkfwjoii4234} {the functor} $E^{\lf}$ is locally finite. Furthermore, by Lemma~\ref{qdkjwqhdjkio233r} shown below it is homotopy invariant. Finally, by Lemma~\ref{fwljwefjoiuoi3224234} also shown below it satisfies excision. This shows that $E$ is a locally finite homology theory. 
\end{proof}

Let $\bC$ be a complete and cocomplete stable $\infty$-category, and let $E\colon\TopBorn\to {\bC}$ be a functor.
\begin{lem}\label{fwl0943n4234}
Assume:
\begin{enumerate}
\item\label{regihjeogerferferfwrf} $\bC$ has the property that  filtered limits distribute over sums\index{limits distribute over sums}\index{distribute!limits over sums} (Remark \ref{qerugiergregrefrefwerf}). 
\item 
  $E$ preserves coproducts\index{coproducts!preserved!in $\TopBorn$}.
  \end{enumerate}
   Then
$E^{\lf}$ also preserves coproducts.
\end{lem}

\begin{proof}
Let $(X_i)_{i \in I}$ be a family of topological bornologial spaces. We must show that the canonical morphism
\[\bigoplus_{i \in I} E^{\lf}(X_i) \to E^{\lf} \big( \coprod_{i \in I} X_i \big)\]
is an equivalence. We have the chain of equivalences
\begin{align*}
E^{\lf} \big( \coprod_{i \in I} X_i \big) & := \lim_{B \in \cB} E \big( \coprod_{i \in I} X_i, \coprod_{i \in I} X_i \setminus B_i \big)\\
& \stackrel{!!}{ \simeq} \lim_{B \in \cB} \bigoplus_{i \in I} E ( X_i, X_i \setminus B_i )\\
&\stackrel{!}{ \simeq} \bigoplus_{i \in I} \lim_{B_i \in \cB_i} E ( X_i, X_i \setminus B_i )\\
& \simeq \bigoplus_{i \in I} E^{\lf} (X_i)\, .
\end{align*}
For the equivalence marked by $!!$ we use that $E$ preserves coproducts.
Furthermore, for
  the   equivalence marked by $!$ we use  that the bornology of $\cB$ of the coproduct is can be identified with   $\prod_{i \in I} \cB_i$, and Assumption \ref{regihjeogerferferfwrf}.
\end{proof}

\begin{rem}
Note that the condition on $\bC$ that filtered limits distribute over sums seems\index{limits distribute over sums}\index{distribute!limits over sums} to be quite exotic. 
It is satisfied, e.g.\ for $\bC=\Sp^{\op}$ since in $\Sp$ filtered colimits distribute over products.\index{colimits distribute over products}\index{distribute!colimits over products}
\hB
\end{rem}



\begin{ex}\label{weoifujweoifowe234234345}
The functor $\Sigma^{\infty,\topp}_{+}:\TopBorn\to \Sp$ (we omit to write the forgetful functor $\cF_{\cB}$) is homotopy invariant and satisfies open excision (see Example \ref{iowdqwioduqz9q2342343424} and the list of properties of this functor in Subsection \ref{eriogjeoigerfwerfwerfwrf}). 
Consequently, its locally finite evaluation $$\Sigma^{\infty,\lf}_{+}:=(\Sigma^{\infty,loc})^{\lf}:\TopBorn\to \Sp$$
is an additive locally finite homology theory.
\index{$\Sigmalf$}

Let $\bC$ be a complete and cocomplete stable $\infty$-category,  and let $C$ be an object of $\bC$.
Then
$$(C\wedge \Sigma^{\infty,\topp}_{+})^{\lf}:\TopBorn\to \bC$$ is a  locally finite homology theory. 

The order of applying $C\wedge-$ and forming the local finite evaluation in general matters.
If the functor
$C\wedge-:\Sp\to \bC$  commutes with limits, then we have an equivalence
$$ C\wedge \Sigma^{\infty, \lf }_{+}  \simeq (C\wedge \Sigma^{\infty,\topp}_{+})^{\lf}\, .$$
But if $C\wedge- $ does not preserve limits, then applying this functor will destroy local finiteness.

For example, 
$\bC=\Sp$ and $C$ is a dualizable spectrum, then $C\wedge -$ preserves limits. In
Example~\ref{hhfweweiufufwefwefwef} {we} gave an example  {of a} spectrum  $C$ where $C\wedge -$ does not commute with limits.
\hB
\end{ex}

\begin{ex}\label{joifwjweoifuoifuewioufoiwefewfewfefewfwef}
The functor
$$\ell\circ  C^{\sing}\colon \TopBorn\to  \Ch_{\infty}$$  is homotopy invariant  {and} satisfies open excision (see Example \ref{iowdqwioduqz9q2342343424}). 
Its locally finite evaluation
$$(\ell\circ C^{\sing})^{\lf}\colon \TopBorn\to  \Ch_{\infty}$$
is a  locally finite homology theory. 
It is the $\TopBorn$-version of Borel--Moore homology for locally compact spaces.
\hB
\end{ex}

\subsubsection{Classification of locally finite homology theories}\label{erjgiofggsfwergergsgsgr}

\newcommand{\bDelta}{\mathbf{\Delta}}
\newcommand{\bMor}{\mathbf{Mor}}

In this subsection we provide   a partial classification of locally finite homology theories. 
The result  Proposition \ref{ifjewifweofewifuoi23542335345}, and also our approach are
 very similar to the classification result of Weiss--Williams \cite{ww_pro}.

Let $X$ be a set,  and let $\cF$ be a subset of $\cP_{X}$. Then we can consider 
$\cF$ as a poset with respect to the inclusion relation. We define $X\setminus \cF:=\{X\setminus Y\:|\: Y\in \cF\}$.

We consider the full subcategory $\bF$ of $  \TopBorn$ of topological bornological spaces 
 which are homotopy equivalent to spaces $X$   which admit  a cofinal subset $\cB'$ of the bornology with the following properties:
 \begin{enumerate}
   \item \label{wertgkowergergefwref}For every $B$ in $\cB^{\prime}$ there exists $B^{\prime}$ in $\cB^{\prime}$ and a subset  $\{B\}$ of $\cP_{X}$ satisfying:
  \begin{enumerate}
  \item  $\{B\}$ is filtered and a big family (Definition \ref{rojepgokerggerg}).
  \item Every member $\tilde B$ of $\{B\}$ satisfies $B\subseteq \tilde B\subset B^{\prime}$.
  \item For every member $\tilde B$ of $\{B\}$ the inclusion $B\to \tilde B$ is a  homotopy equivalence.
  \item The family $X\setminus \{B\}:=\{X\setminus B'': B''\in \{B\} \}$ is filtered and a big family.  
    \end{enumerate}\item Every $B$ in $\cB^{\prime}$ is homotopy equivalent to a finite $CW$-complex.
\end{enumerate}
  
     The Condition  \ref{wertgkowergergefwref} on the members of ${\cB'}$ is a sort of cofibration condition. It is satisfied, e.g., if the inclusion $B\to X$ has a normal unit disc bundle $D$. Then we can take $B'=D_{1}$ and  $\{B\}:=\{D_{r}\:|\: r\in (0,1)\}$, where $D_{r}$ is the total space of the subbundle of discs of radius~$r$.

\begin{ex}\label{weriogwetgwergregegw9}
 A  locally finite simplicial complex with the topological and bornological structures induced by the spherical path metric  belongs to $\bF$.
  For $\cB^{\prime}$ we can take the family of finite subcomplexes.

 A typical example of such a complex is the coarsening  space $\|\hat \cY\|$, see \eqref{fwhhefwifiui23r2r23423434},  provided  the complex $\|\cY_{n}\|$ is finite-dimensional    for every $n$ in $\nat$. The space $ \|\hat \cY\|$ itself need not be finite-dimensional.
\hB
\end{ex}

The following is an adaptation of a result of Weiss--Williams \cite{ww_pro}.

Let $\bC$ be a   cocomplete stable  $\infty$-category. In the following we use the tensor structure \eqref{wervewrvwevfevfdvsfdvdfsvsfv} of $\bC$ over $\Sp$.

 
 Let $E:\TopBorn\to \bC$ be a functor.

\begin{prop}\label{ifjewifweofewifuoi23542335345}
If $E$ is homotopy invariant, then there exists a natural transformation
\[E(*) \wedge \Sigma^{\infty,\topp}_{+}  \to E \, .\]
If $E$ is a locally finite homology theory, then {the induced} transformation ${( \Sigma^{\infty}_{+} \wedge E(*))^{\lf}\to E}$ induces an equivalence on all objects of $\bF$.
\end{prop}

\begin{proof}
{Let $\bDelta$ denote the usual category of the finite posets $[n]$ for $n$ in $\nat$.}
We have a functor $$\Delta:\bDelta\to \TopBorn\, , \quad [n]\mapsto \Delta^{n}$$
which sends the poset $[n]$ to the $n$-dimensional topological simplex. 
The simplex has the maximal bornology since it is compact, see Example \ref{fljwlewkiou2ori23r2323}.
%

We consider the category $\TopBorn^{\simpp}$ whose objects are pairs $(X,\sigma)$ of a topological bornological space $X$ and a singular simplex $\sigma:\Delta^{n}\to X$. A morphism $(X,\sigma)\to (X^{\prime},\sigma^{\prime})$ is {a commutative} diagram
$$\xymatrix{\Delta^{n}\ar[r]^{\phi}\ar[d]^{\sigma}&\Delta^{n^{\prime}}\ar[d]^{\sigma}\\X\ar[r]^{f^{\prime}}&X^{\prime}}$$
where $f$ is a morphism in $\TopBorn$ and $\phi$ is induced by a morphism $[n]\to [n^{\prime}]$ in $\bDelta$.

We have    two forgetful functors
$$ \TopBorn\stackrel{p}{\leftarrow}\TopBorn^{\simpp}\stackrel{q}{\rightarrow} \TopBorn\, ,  \quad p(X,\sigma):=X\, , \quad q(X,\sigma):= \Delta^{n}\, .$$
We then define the functor 
$$  E^{\%}:  \TopBorn\to  \bC$$
by left-Kan extension 
 $$\xymatrix{\TopBorn^{\simpp}\ar[rr]^-{ E\circ q }\ar[dd]^{p}\ar@{::>}[dr]&& \bC \\& &\\
\TopBorn\ar@{..>}@/_1pc/[uurr]_-{E^{\%}}}
\, .$$ The objectwise formula for Kan extensions  gives 
$$    E^{\%}(X):=\colim_{ (\Delta^{n}\to X)}  E(\Delta^{n})  \, .$$
We have a natural transformation
\begin{equation}\label{ewfijiojfoi23r23r2r32r}
E\circ q\to E\circ p\, , \quad (X,\sigma)\mapsto E(\sigma):E(\Delta^{n})\to E(X)\, .
\end{equation} 
The universal property of the Kan extension provides an equivalence of mapping spaces in functor categories
$$\Map(E\circ q,E\circ p)\simeq \Map(E^{\%},E)\, .$$
Consequently, the transformation \eqref{ewfijiojfoi23r23r2r32r} provides a natural transformation
\begin{equation}\label{fwbgg3golkewi23ri23r23r23r}
E^{\%}\to E\, . 
\end{equation} 
 We use now that $E$ is homotopy invariant.
The projection $\Delta^{n}\to *$ is a homotopy equivalence  in $\TopBorn$ for every $n$ in $\nat$.
We get an equivalence of functors $\TopBorn^{\simpp}\to \bC$
\begin{equation}\label{elwfkjwklfio3t34t34t43t3}
E\circ q\stackrel{\simeq}{\to} \const(E(*))\, . 
\end{equation}  
We now use 
 the definition \eqref{frknfkejnfkjfrefwefwerfwref} of $\Sigma^{\infty,\topp}_{+}$ and 
 the well-know facts  that
$$\colim_{\Top^{simp}/X} \ell(\sing(*))\simeq \ell(\sing(X))$$ in $\Spc$
(where $\Top^{simp}$ similar as $\TopBorn^{simp}$ is the category of pairs $(X,\sigma)$ of topological spaces $X$ with a singular simplex)), and that $\Sigma^{\infty}_{+}$ preserves colimits.  By  the point-wise formula   
the left-Kan extension of $ \const(E(*))$ can be identified with the functor
$$X\mapsto     E(*)\wedge \Sigma^{\infty}_{+} \ell(\sing(X)) \simeq   E(*) \wedge \Sigma^{\infty,\topp}_{+}X \, .$$
The left Kan extension of  the equivalence  \eqref{elwfkjwklfio3t34t34t43t3}  therefore yields an equivalence $$ E^{\%}    \stackrel{\simeq}{\to}   E(*)  \wedge   \Sigma^{\infty}_{+} \, .$$ 
The composition of the inverse of this equivalence with \eqref{fwbgg3golkewi23ri23r23r23r} yields the asserted transformation
$$ E(*)   \wedge  \Sigma^{\infty}_{+} \to E\, .$$
We now use that $E$ is   locally finite. We apply the $(-)^{\lf}$-construction and use  Lemma \ref{fjwelkfwjoii4234} for the second equivalence in order to get the transformation between locally finite homology theories
\begin{equation}\label{hwefwhieowfioewf234534555345}
(    E(*)\wedge  \Sigma^{\infty}_{+} )^{\lf}\to E^{\lf}\simeq E\, .
\end{equation}
If one evaluates this transformation on the one-point space, then it induces an equivalence.
 By excision and homotopy invariance we get an equivalence on all bounded spaces which are homotopy equivalent   to finite $CW$-complexes.
 
Assume now that $X$  belongs to $\bF$. After replacing $X$ by a homotopy equivalent space we can assume that it admits a cofinal subset  $\cB'$ of its bornology satisfying   the assumptions stated above.
  Every member $B$ of $ \cB^{\prime}$ is homotopy equivalent to a finite $CW$-complex. 
   Furthermore there exists $B^{\prime}$ in $\cB^{\prime}$ such that
 $( X\setminus \{B\}, \{B^{\prime}\})$ is  a decomposition of $X$ into two big familes.
By excision we get a natural equivalence  \begin{eqnarray*}
 ( \Sigma^{\infty}_{+} \wedge E(*))^{\lf}(X, X\setminus \{B \}) &\simeq &
( \Sigma^{\infty}_{+} \wedge E(*))^{\lf}( \{B^{\prime}\},  \{B^{\prime}\}\setminus \{B \})\\&\simeq& E ( \{B^{\prime}\},  \{B^{\prime}\}\setminus \{B \})\\&\simeq& E (X, X\setminus \{B \})\, .
 \end{eqnarray*}
 The  two middle terms involve a double colimit $\colim_{\{B\}}\colim_{\{B'\}}$.
 
 In the middle equivalence we use that the natural transformation \eqref{hwefwhieowfioewf234534555345} induces an equivalence on spaces which are bounded and homotopy equivalent to finite $CW$-comlexes.
 
 If we now take the limit over $B$ in $\cB^{\prime}$ and  use the cofinality of $\cB'$ in $\cB$ and  the  local finiteness of $E$, then we get
 $$( \Sigma^{\infty}_{+} \wedge E(*))^{\lf}(X)\stackrel{\simeq}{\to} E(X)\, .$$
which is the desired equivalence. 
\end{proof}

%
%
%
%

\subsection{{Coarsification of locally finite theories}}
\label{jhksdf2323r}

In this section we explain the construction of coarsifying a locally finite homology theory in the sense of Definition \ref{rewfiewrjgoiegergreefwerfw}. This construction is a generalization of the construction 
of the coarse stable homotopy theory $Q$ in Definition \ref{efifjweiof89247u23553435534543}.

We refer to \cite{ass} for a similar construction whose input is a local homology theory on the category of uniform bornological coarse spaces.


Let $X$ be a bornological coarse space. If $U$ is an entourage of $X$, then we can consider the space of controlled probability measures $P_{U}(X)$ defined in  \eqref{gegljio3tuio38tut3t3t34} as a topological bornological space with the bornology generated by the subsets $P_{U}(B)$ for all bounded subsets $B$ of $X$. The arguments given in the proof of Lemma \ref{lefjefewjfieofeifw24}
 applied to the big family $\cB$ show that the bornology and the topology of $P_{U}(X)$ are compatible in the sense of Definition~\ref{fewilfjweiofoi4325345345}. 
From now we will consider $P_{U}(X)$ as an object of $\TopBorn$. 

Let $\BC^{\cC}$ be as in Remark \ref{fjwefoijewojf923ur9u23refewfwefewf} the category of pairs $(X,U)$ of a bornological coarse space $X$ and an entourage $U$ of $X$. We define the functor
$$\cP:\BC^{\cC}\to \TopBorn\, , \quad \cP(X,U):=P_{U}(X)\, .$$
On the level of morphisms $\cP$ is defined in the obvious way in terms of the push-forward of measures.
 
Let $\bC$ be a {cocomplete} stable $\infty$-category,  and  let $E:\TopBorn\to {\bC}$ be a functor. 
\begin{ddd}\label{wefiwiofuwe987u982523453453}
 We  define  the functor\index{$Q$}
$$  QE:\BC\to {\bC}$$ by left Kan-extension
$$\xymatrix{\BC^{\cC}\ar[rr]^-{E\circ \cP }\ar[dd]\ar@{::>}[dr]&&{\bC}\\
& &\\
\BC\ar@{..>}@/_1pc/[uurr]_-{  QE}
}\, .$$
We call $QE$ the coarsification of $E$.
\end{ddd}

\begin{rem}\label{eifewfioeu9ewu982742345345}
The evaluation of $QE$ on a bornological coarse space  $X$ with coarse structure $\cC$  is given by
\begin{equation}
\label{knj23123}
QE(X)\simeq \colim_{U\in \cC} E(P_{U}(X))\, .
\end{equation}
{This expression  follows} from the object-wise formula for the left Kan extension.
\hB
\end{rem}

Let $\bC$ be a complete and cocomplete stable $\infty$-category, and let $E\colon\TopBorn\to \bC$ be a functor.

\begin{prop}\label{cegiojiojergergeg}\mbox{}
\begin{enumerate}
\item \label{ergiooerfewrfrefwferf}
If  $E$ is a  locally finite homology theory, then $QE$ is a coarse homology theory.
\item \label{gfsdggsgsgfd} If in $\bC$ filtered colimits distribute over products\index{colimits distribute over products}\index{distribute!colimits over products} (Definition \ref{wekrgergfvsfgsgfdg}), then  $QE$ is strongly additive.
\item \label{wegioewfwrefweferfwref}
If $E$ preserves  coproducts, then so does $QE$.
\end{enumerate}
\end{prop}

\begin{proof}
The proof of \ref{ergiooerfewrfrefwferf}  is very similar to the proof of Theorem \ref{efwifjwifou23984u23984u3294823424234}.
  
Assume that  $E$ is a locally finite homology theory.  
We must verify  that $QE$ satisfies the four conditions listed in Definition \ref{rgljogreggregrege} (with $QE$ in place of $E$).

In order to see that $  QE$ is coarsely invariant we use that $E$ is homotopy invariant and then argue as in the proof of Lemma \ref{fewkljlwefewfewfewfwfwef}.

Let $X$ be in $\BC$, and let 
   $(Z,\cY)$ be a complementary pair on $X$. Let furthermore   $U$  be an entourage of $X$.  
 We have shown in the proof of Lemma \ref{lefjefewjfieofeifw24} that
$(P_{U}(\{Z\}),P_{U}(\cY))$ is a decomposition of the topological bornological space $P_{U}(X)$ into two big families.
We now use weak excision of $E$ and form the colimit of the resulting push-out diagrams over the coarse entourages $U $ of $X$
in order to conclude  that $$\xymatrix{  QE(\{Z\}\cap \cY)\ar[r]\ar[d]&  QE(\{Z\})\ar[d]\\  QE(\cY)\ar[r]&  QE(X)}$$
is a push-out diagram. We finally use that $  QE$ is coarsely invariant in order to replace $\{Z\}$ by $Z$.
  
Assume that $X$ is flasque and that flasqueness of $X$ is implemented by $f:X\to X$.
As in the proof of Proposition \ref{ergkjwergrereferfwrfwerfw}.\ref{ewrgijewroigergrefwreferf}
we can find a cofinal set 
  set of entourages $U$ of $X$ such that $(f\times f)(U)\subseteq U$ and $(\id_{X}\times f)(\diag_{X})\subseteq U$. For such $U$
 the map $P_{U}(f):P_{U}(X)\to P_{U}(X)$ implements flasqueness of $P_{U}(X)$ as a topological bornological space. We now use Lemma \ref{efewoifeuwfoiewufoiew89u98234234324} and that $E$ is homotopy invariant and locally finite in order to see that $E(P_{U}(X))\simeq 0$ for  those $U$. By taking the colimit over these entourages we conclude that $QE(X)\simeq 0$.
  
Finally, $QE$ is $u$-continuous by definition. {This finishes the argument that $QE$ is a coarse homology theory.}

We now show \ref{gfsdggsgsgfd}.  Let $(X_i)_{i \in I}$ be a family of bornological coarse spaces and set $X := \bigsqcup_{i \in I}^{\free} X_i$. Using \eqref{knj23123} and that $P_U(X) \cong \bigsqcup_{i \in I}^{\free} P_{U_i} (X_i)$ we get
\[QE(X) \simeq \colim_{U \in \cC} E \big(\bigsqcup_{i \in I}^{\free} P_{U_i} (X_i)\big)\, ,\]
where we set $U_i := U\cap (X_i \times X_i)$. Note that contrary to \eqref{njdsfui232}, where we were working   with topological spaces, here we have to write   the free union  in order to take  the bornology   into account  properly (which is not the bornology of the coproduct in this case). Since $E$ is additive  by Lemma~\ref{tgiortgergfewferfw}, we get
\[\colim_{U \in \cC} E \big(\bigsqcup_{i \in I}^{\free} P_{U_i} (X_i)\big) \simeq \colim_{U \in \cC} \prod_{i \in I} E (P_{U_i} (X_i))\, .\]
We finish the argument by interchanging the colimit with the product (see the end of the proof of Lemma \ref{jnksdui2332} for a similar argument.) At this point we use the additional assumption on $\bC$.

We now show \ref{wegioewfwrefweferfwref}. We assume that   $E$ preserves coproducts.  Let $(X_i)_{i \in I}$ be a family of bornological coarse spaces and set $X := \coprod_{i \in I} X_i$. We now have an isomorphism $P_U(X) \cong \coprod_{i \in I} P_{U_i} (X_i)$ (the coproduct is understood in topological bornological spaces) and therefore the following chain of equivalences
\[QE(X) \simeq \colim_{U \in \cC} \bigoplus_{i \in I} E(P_{U_i}(X_i)) \simeq \bigoplus_{i \in I} \colim_{U \in \cC} E(P_{U_i}(X_i)) \simeq \bigoplus_{i \in I} QE(X_i)\, .\]
This completes the proof of Proposition \ref{cegiojiojergergeg}.
\end{proof}
  
\begin{ex}\label{wklqdjqwkldjwqdjoiu3e232344324}
Let $\bC$ be a complete and cocomplete stable $\infty$-category, and let $C$ be an object of $\bC$.
 Then we can construct the   locally finite homology theory $(C\wedge \Sigma_{+}^{\infty,\topp})^{\lf}$ as in
  Example \ref{weoifujweoifowe234234345}. 
  
  By Proposition \ref{cegiojiojergergeg} we obtain  to obtain a  coarse homology theory $$Q(C\wedge \Sigma_{+}^{\infty,\topp})^{\lf}:\BC\to \bC\, .$$  
  If $\bC$ has the property that filtered colimits distribute over products\index{colimits distribute over products}\index{distribute!colimits over products},  then  $Q(C\wedge \Sigma_{+}^{\infty,\topp})^{\lf}$ is strongly additive.
 
 We have $$Q(C\wedge \Sigma_{+}^{\infty,\topp})^{\lf}(*)\simeq C\, .$$

  Note that $C\wedge Q$ (Corollary \ref{goirjgio334t43gregergrege}) is  also  coarse homology theory with the   value $C$ on the one-point space, but this coarse homology theory is   in general  not  even  additive.
  
  We always 
  have a natural transformation  \begin{equation}\label{weqfoihqiowefqewfewdqd}
C\wedge Q\to QC(\wedge \Sigma_{+}^{\infty,\topp})^{\lf}
\end{equation}
 of coarse homology theories. On the other hand, if the functor $C\wedge -$ preserves  limits, then the transformation \eqref{weqfoihqiowefqewfewdqd} is an equivalence, 
 see Example \ref{hhfweweiufufwefwefwef}.
\hB
\end{ex}

%
%


\subsection{{Analytic locally finite \texorpdfstring{$K$}{K}-homology}}
\label{kjsdf09232fds}

\subsubsection{Extending functors from locally compact spaces to \texorpdfstring{$\TopBorn$}{TopBorn}}\label{wergoiwergwergerfwfrwfwefwerf}

In order to deal with analytic locally finite $K$-homology we consider the full subcategory $\Top^{\lc}$ of $ \TopBorn$\index{$\Top^{\lc}$} of separable, locally compact spaces with the bornology of relatively compact subsets.  By definition,  morphisms in  $\Top^{\lc}$ are proper continuous maps.
{Analytic} locally finite $K$-homology is initially defined on $\Top^{\lc}$. So we must    extend this functor to 
$ \TopBorn$ preserving good properties. We first discuss such extensions in general and then apply the theory to $K$-homology.

In the following it is useful to remember the Examples 
  \ref{eigoffewerevfdsvsdfvsdfv} and   \ref{fljwlewkiou2ori23r2323}.  
 
Let $X$ be a topological bornological space, and  let  $C$  be a subset of $X$.
\begin{ddd}
The subset $C$ is $X$-locally compact\footnote{We omit saying ``separable'' in order to shorten the text.} if $C$ with\index{$X$-locally compact subset}\index{subset!$X$-locally compact} the induced topology  is separable and locally compact and the induced bornology is given by the relatively compact subsets of~$C$.
  \end{ddd}
 We let $\Loc(X)$\index{$\Loc(-)$} denote the poset of locally compact subsets of $X$.

  Let $X$ be a topological bornological space,  and let  $C$ be a subset of $X$.
 \begin{lem}\label{fewlfjweoifewfwewr435345} Assume:
 \begin{enumerate}
 \item $X$ is Hausdorff.
 \item  $C$ is $X$-locally compact.
 \end{enumerate} 
 Then $C$ is a closed subset of $X$.   \end{lem}

\begin{proof}
We show that $X\setminus C$ is open. To this end we consider a point $x$ in $X \setminus C$. Let $\cJ$ be the family of closed and bounded neighbourhoods of $x$ in $X$.
 Since $\{x\}$ is bounded such neighbourhoods exist by the compatibility of the  bornology with the topology of $X$ (see Definition \ref{fewilfjweiofoi4325345345}).
For every $W$ in $\cJ$ the intersection $W\cap C$ is   closed {in $C$} and relatively compact {in $C$}, therefore it is compact. We claim that  $\bigcap_{W\in \cJ} (W\cap C)=\emptyset$. Indeed, if $\bigcap_{W\in \cJ} (W\cap C)\not=\emptyset$, then because of $\bigcap_{W\in \cJ} W=\{x\}$ (since $X$ is Hausdorff by assumption)  we would have $x\in C$. By the claim and the compactness of the intersections $W\cap C$
 there exists $W$ in $\cJ$ such that $W\cap C=\emptyset$. Hence $W$ is a (closed) neighbourhood of $x$ contained in $X\setminus C$.
\end{proof}

If $C$ is a $X$-locally compact subset of $X$, then $C$ is Hausdorff.  
 A closed subset of a locally compact subset $C$ of $X$ is therefore again $X$-locally compact.
 By Lemma \ref{fewlfjweoifewfwewr435345} 
the $X$-locally compact subsets of $X$ contained in the $X$-locally compact subset $C$ are  precisely the closed subsets of $C$.

{Let $\bC$ be a stable $\infty$-category, and let} $E:\Top^{\lc}\to {\bC}$ be a functor. Then we can say that $E$ is countably additive, locally finite,  homotopy invariant or satisfies closed descent by interpreting the  corresponding definitions made for $\TopBorn$ in the obvious way.  For the first two properties we must assume that $\bC$ admits countable products.

\begin{rem} 
Since $\Top^{\lc}$ only has countable {free unions} (because of the separability assumption) it does not make sense to consider additivity for larger families.

Furthermore, by Lemma \ref{fewlfjweoifewfwewr435345} we are forced to consider closed descent. The condition of 
open descent could not even be formulated since open subsets of a locally compact space in general do not belong to $\Top^{\lc}$, see Example \ref{eigoffewerevfdsvsdfvsdfv}.

 The interpretation of the local finiteness condition is as
 $$\lim_{B\in \cB\cap \cS} E(C\setminus B)\simeq 0\, .$$
We must restrict the limit to open, relatively compact subsets
since then $C\setminus B$   is a closed subset of $C$ and therefore also locally compact.
Since a separable, locally compact space is $\sigma$-compact, the intersection $ \cB\cap \cS$ contains a countable cofinal subfamily.
\hB
\end{rem}

We now discuss the  left Kan extension of functors defined on $\Top^{\lc}$  to $\TopBorn$. 
{{Let $\bC$ be a   cocomplete stable $\infty$-category, and let} $E:\Top^{\lc}\to {\bC}$ be a functor.}
\begin{ddd}\label{weirghiwergregwegwergwerg9}
We define\index{$L(-)$}
$$L(E):\TopBorn\to {\bC}\, , \quad L(E)(X):=\colim_{C \in \Loc(X)} E(C)\, .$$
\end{ddd}

\begin{rem}
Using the Kan extension technique we can turn the above description of the functor on objects into a proper definition of a functor. We consider the category $\TopBorn^{\Loc}$ of pairs $(X,C)$, where $X$ is a topological bornological space and $C$ is an element of $\Loc(X)$.
 A morphism $f:(X,C)\to (X^{\prime},C^{\prime})$ is a morphism $f:X\to X^{\prime}$ of topological bornological spaces such that $f(C)\subseteq C^{\prime}$.
We have the  forgetful functors $$\TopBorn^{\Loc}\to \TopBorn\, ,  (X,C)\mapsto X\, ,\quad p:\TopBorn^{\Loc}\to \Top^{\lc}\, ,  (X,C)\mapsto C\, . $$
We then  define $L(E)$ as left Kan extension
$$\xymatrix{\TopBorn^{\Loc}\ar[rr]^-{ E\circ p }\ar[dd]\ar@{::>}[dr]&&{\bC}\\
& &\\
\TopBorn\ar@{..>}@/_1pc/[uurr]_-{L(E)}
}\, .$$ \hB
\end{rem}

If $C$ is a separable locally compact space, then $C$ is final in the poset $ \Loc(C)$. This implies 
that  \begin{equation}\label{uwfghwejk7823z8r23ir23wf}
 L(E)_{|\Top^{\lc}}\simeq E\, . 
\end{equation} 

\begin{lem}\label{wefoiuwef982345u25355}
If $E$ is homotopy invariant, then so is $L(E)$.
\end{lem}

\begin{proof}
The subsets $[0,1]\times C$ of $[0,1]\times X$ for $C$ in $\Loc(X)$ are cofinal in  $\Loc([0,1]\times X)$.
\end{proof}

\begin{lem}\label{fwifjweoifwefwefwef}
If $E$ satisfies  closed  excision, then so does $L(E)$.
\end{lem}

\begin{proof}
Here we use that for  $C$ in $\Loc(X)$   a closed decomposition of $X$ induces a closed decomposition of $C$. We thus get a colimit of push-out diagrams which is again a push-out diagram \end{proof}

\begin{lem}\label{jsdiu2}
{If $E$ satisfies closed excision, then $L(E)$ preserves all coproducts.}
\end{lem}

\begin{proof}
Let $X := \coprod_{i \in I} X_i$ be the coproduct of the family $(X_{i})_{i\in I}$ in the category $\TopBorn$ and let $C$ be in $\Loc(X)$.  Since the induced bornology on $C$ is the bornology of relatively compact subsets the set $I_C :=\{i \in I : C \cap X_i \not= \emptyset\}$ is finite. Furthermore, for every $i$ in $I$ the space $C_i := C \cap X_i$  belongs to $\Loc(X)$.  Since $E$ satisfies closed excision it commutes with finite coproducts. Therefore we get the chain of equivalences:
\begin{align*}
L(E)(X) & \simeq \colim_{C \in \Loc(X)} E(C)\\
& = \colim_{C \in \Loc(X)} E \big( \coprod_{i \in I_C} C_i \big)\\
& \simeq \colim_{C \in \Loc(X)} \bigoplus_{i \in I_C} E(C_i)\\
& \simeq \bigoplus_{i \in I} \colim_{C_i \in \Loc(X_i)} E(C_i)\\
& \simeq \bigoplus_{i \in I} L(E)(X_i)
\end{align*}
which finishes the proof.
\end{proof}
 
Let $\bC$ be a complete and cocomplete stable $\infty$-category, and $E\colon \Top^{\lc}\to {\bC}$ be a functor.

\begin{prop}\label{wetigjowerfrwefwerf}\mbox{}
\begin{enumerate}
\item \label{ewriutghwieurgrwefrfr}
If $E $ is homotopy invariant and satisfies closed excision, then
$L(E)^{\lf}\colon\TopBorn\to \bC$ is a locally finite homology theory.
\item  \label{ewriutghwieurgrwefrfr1} If in $\bC$ filtered limits distribute over sums\index{limits distribute over sums}\index{distribute!limits over sums}, then $ L(E)^{\lf}$ preserves coproducts.
\end{enumerate}
\end{prop}

\begin{proof} We show Assertion \ref{ewriutghwieurgrwefrfr}.
By Lemmas \ref{fwifjweoifwefwefwef} and  \ref{wefoiuwef982345u25355} the functor $L(E)$ satisfies
closed excision and is homotopy invariant. By Proposition \ref{prop:sfd89230} the functor $L(E)^{\lf}$
is a  locally finite homology theory.

In order to show Assertion  \ref{ewriutghwieurgrwefrfr1} we combine the 
  Lemmas \ref{jsdiu2} and    \ref{fwl0943n4234}.
\end{proof}

Note that $L(E)^{\lf}$ satisfies closed excision.

The following lemma ensures that   $L(E)^{\lf} $ really extends $E$ if it was locally finite.

\begin{lem}\label{ilfjiweofuwei9u9435345345}
If $E$ is locally finite, then  $L(E)^{\lf}_{|\Top^{\lc}}\simeq E$.
\end{lem}

\begin{proof}
Let $X$ be  in $\Top^{\lc}$.
For an open bounded subset $B$ of $C$    we then have  $X\setminus B\in \Top^{\lc}$.
 We now use that by compatibility of the bornology and the topology we can define the locally finite evaluation as a limit over open bounded subsets of $X$.   
This gives   $$L(E)^{\lf} (X)\simeq \lim_{B\in \cB\cap \cS}  L(E) (X,X\setminus B)\stackrel{\eqref{uwfghwejk7823z8r23ir23wf}}{\simeq} \lim_{B\in \cB\cap \cS} E(X,X\setminus B)
  \simeq E(X)\, ,$$
  where for  the last equivalence we  use that $E$ is locally finite.
\end{proof}

\subsubsection{Cohomology for \texorpdfstring{$C^{*}$}{Cstar}-algebras}\label{fweoifjewio982u4r53453453454}

In order to construct the analytic $K$-homology and its locally finite version we employ a $K$-cohomology functor for $C^{*}$-algebras.
In this subsection we recall the required  properties of such a cohomology functor  in an axiomatic way
and provide references to the corresponding literature. Our main {examples} are the $K$-cohomology {functors} {from} Definition \ref{eughierwgvfdsdfvdfvsdfvfsdvsdfvswtf}.

Let $C^{*}\Alg^{\sep}$\index{$C^{*}\Alg^{\sep}$} denote the category of separable (very small) $C^{*}$-algebras. 
These algebras need not be unital, and morphisms need  not preserve units in case they exist.

Let $\bC$ be a complete stable $\infty$-category.
We consider a   functor 
$$F:C^{*}\Alg^{\sep,\op} \to \bC\, .$$

\begin{ddd}\label{weiotgwegerfefwrefre}  We define the following properties of $F$.
\begin{enumerate}
\item  homotopy invariance:  For every $A$ in $C^{*}\Alg^{\sep}$  the inclusion $A\to C([0,1],A)$ as constant functions induces an equivalence $ F(C([0,1],A))\to F(A)$.
\item  countable additivity:  For a countable family $(A_{i})_{i\in I}$ in $ \Alg^{\sep}$ the canonical map
$$F(\bigoplus_{i\in I}A_{i})\to \prod_{i\in I} F(A_{i})$$ is an equivalence.
\item  exactness:  We have $F(0)\simeq 0$, and for every short exact sequence
\[
0\to I\to A\to Q\to 0
\]
of $C^{*}$-algebras in $\Alg^{\sep}$ (where $I$ is a closed ideal) we have 
a pull-back
$$\xymatrix{F(Q)\ar[r]\ar[d]&F(A)\ar[d]\\0\ar[r]&F(I)}$$
\end{enumerate}
If $F$ is  homotopy invariant, additive and exact, then we call $F$ a cohomology theory\index{cohomology theory!for $C^{*}$-algebras}.
\end{ddd}

In the following we use a spectrum-valued $\KK$-theory functor in order to construct an  example of a cohomology theory as in Definition \ref{weiotgwegerfefwrefre}. 

Our starting point  is a
 stable $\infty$-category $\bK\bK$\index{$\bK\bK$} whose objects are 
the  separable $C^{*}$-algebras as objects.  As usual, for $A,B$ in $C^{*}\Alg^{\sep}$ we let $\map_{\bK\bK}(A,B)$ in $\Sp$ denote the mapping spectrum in $\bK\bK$.
This category has the following properties. 
\begin{enumerate}
\item There is a functor $C^{*}\Alg^{\sep}\to \bK\bK$ which is the identity on objects.
 \item \label{wergjiogerfeffewfwerfwerfwref}
For   $A,B$ in $C^{*}\Alg^{\sep}$  we have an isomorphism of $\Z$-graded groups
$$ \pi_{*}(\map_{\bK\bK}(A,B))\cong KK_{*}(A,B)\, ,$$
where $\KK_{*}$ denotes Kasparovs bivariant $K$-theory groups  (see Blackadar \cite{blackadar}). 
This  isomorphism  identifies the   composition of morphisms  in $\bK\bK$ with the  Kasparov product   for $\KK_{*}$.
\item The zero algebra $0$ is a zero object in $\bK\bK$.
\item \label{qreoighjiogqrfrgergfsgsgsdg} We have an equivalence of spectra $\map_{\bK\bK}(\C,\C)\simeq KU$. 

More generally, for  $B$ in $C^{*}\Alg^{\sep}$ we have $ \map_{\bK\bK}(\C,B)\simeq \Kast(B)$, where $\Kast(B)$ is the $C^{*}$-algebra $K$-theory spectrum of $B$. In fact, one could take this as a definition of $\Kast(B)$. We refer to  Subsection \ref{wifjewiof2323443534} for  further discussion.
\item \label{werigoewgerffewferwfw}An exact sequence of   $C^{*}$-algebras in  $C^{*}\Alg^{\sep}$ (where $I$ is a closed ideal)
$$0\to I\to A\to Q\to 0$$ gives rise  to a
  push-out  diagram $$\xymatrix{I\ar[r]\ar[d]&A\ar[d]\\0\ar[r]&Q}$$
  in $\bK\bK$.
   \item\label{ewfjweoifuioe45} If $B$ is a separable $C^{*}$-algebra and   $(A_{i})_{i\in I}$ is a countable family {of} separable $C^{*}$-algebras, then 
\begin{equation}\label{qdwqdoijooir23oioir23r2}
\map_{\bK\bK}\big(\bigoplus_{i\in I} A_{i},B\big)\simeq \prod_{i\in I}  \map_{\bK\bK}(A_{i},B)\, .
\end{equation} 
Equivalently, $\bigoplus_{i\in I} A_{i}$ represents the coproduct  of the family      $(A_{i})_{i\in I}$   in $\bK\bK$.
\item The upper left corner inclusion $A\to A\otimes \mathbb{K}$ induces an equivalence in $\bK\bK$, where $ \mathbb{K}$ denotes the compact operators on the Hilbert space $\ell^{2}(\nat)$.
\item \label{qeriogorefqwefqewfewdewd} The inclusion $A\to C([0,1],A)$ as constant functions induces an equivalence in $\bK\bK$.
\end{enumerate}

There are various constructions of such a stable $\infty$-category  
\cite{Mahanta:2012aa,Barnea:2015aa,Joachim:2007aa}  {and} \cite[Def.~3.2]{Land:2016aa}.

The  equivalence  \eqref{qdwqdoijooir23oioir23r2}  follows from    \cite[Thm.~19.7.1]{blackadar} (attributed to J.~Rosenberg).  In this reference the isomorphism 
$$\KK_{0}\big(\bigoplus_{i\in I} A_{i},B\big)\cong \prod_{i\in I}  \KK_{0}(A_{i},B)$$
is shown. The corresponding isomorphism in degree $1$ can be obtained by replacing $B$ by its suspension $C_{0}(\R,B)$.
By Bott periodicity these two cases imply Equivalence~\eqref{qdwqdoijooir23oioir23r2}.

 
 Let $B$ be a separable $C^{*}$-algebra. 
 
 \begin{ddd}\label{eughierwgvfdsdfvdfvsdfvfsdvsdfvswtf}
 We define the  $K$-cohomology for $C^{*}$-algebras with coefficients in $B$  by $$K_{B}:=\map_{\bK\bK}(-,B)\colon C^{*}\Alg^{\sep,\op}\to \Sp\, .$$
 \end{ddd}

\begin{kor}\label{tgkioegefrefwef}
The functor 
  $K_{B}\colon \Alg^{\sep,\op}\to \Sp$ 
is a cohomology theory\index{cohomology theory!for $C^{*}$-algebras}. 
\end{kor}
\begin{proof} This functor is homotopy invariant by  Property \ref{qeriogorefqwefqewfewdewd} of $\bK\bK$, sends countable sums to products by  Property  \ref{ewfjweoifuioe45}, and is exact by  Property  \ref{werigoewgerffewferwfw}.
\end{proof}


 \subsubsection{Locally finite homology theories from cohomology theories for \texorpdfstring{$C^{*}$}{Cstar}-algebras}\label{wklgowergerfrfwferf}

If $X$ is a separable  locally compact topological space, then  $C_{0}(X)$ denotes   the $C^{*}$-algebra defined as the completion of
the $*$-algebra of compactly supported functions in the supremum norm.  Since 
 $X$ is a separable, this $C^{*}$-algebra   is separable, too. 
 Recall that morphisms in $\Top^{\lc}$ are proper continuous maps. If $X\to X'$ is a morphism in $\Top^{\lc}$, then the pull-back of functions
 provides a morphism $C_{0}(X')\to C_{0}(X)$ in $C^{*}\Alg^{\sep}$. We therefore get a functor 
 $$C_{0}:\Top^{\lc,\op}  \to C^{*}\Alg^{\sep}\, .$$
 This functor has the following properties:
 \begin{enumerate}
 \item  \label{ergioeriogegfgsfgsfg9} If $X$ is in $\Top^{\lc}$ and $Y$  is a closed subset of $X$, then the inclusion $Y\to X$ 
  is   a morphism in $\Top^{\lc}$, and $X\setminus Y$ belongs to $\Top^{\lc}$ (though the inclusion into $X$ is in general not a morphism).
  The inclusion  induces the  second map in the exact sequence
 $$0\to C_{0}(X\setminus Y)\to C_{0}(X)\to C_{0}(Y)\to 0\, .$$
 The  first map is given by extension by zero.
 \item\label{sbiojdboibsfdfvfvsfvsfvsdfv} If $(X_{i})_{i\in I}$ is a countable family in $\Top^{\lc}$, then $\bigsqcup^{\free}_{i\in I} X_{i}$ (interpreted in $\Top\Born$)
 is given by the coproduct of the family in $\Top$ which is again separable and locally compact and therefore an object of $\Top^{\lc}$. We have a canonical isomorphism
 $$C_{0}(\bigsqcup^{\free}_{i\in I} X_{i})\cong \bigoplus_{i\in I}C_{0}(X_{i})\, .$$
 \end{enumerate}

 Let $\bC$ be a complete stable $\infty$-category, and let $F\colon C^{*}\Alg^{\sep,\op}\to \bC$ be a functor.
\begin{ddd}
We define $F^{\an}:=F\circ C_{0}\colon \Top^{\lc}\to \bC$.
\end{ddd}

%
%

 \begin{prop} \label{qerugfihierwgegwerfwefwerfwerfw}
 If $F$ is a cohomology theory, then the functor $F^{\an}$ has the following properties:  \begin{enumerate}
\item  \label{goegoeri45345345345}
 homotopy invariant,   
 \item \label{goegoeri453453453451}   excisive (for closed decompositions),
 \item \label{wrthijwirogfrefwerfw} countably additive,
 \item  \label{wrthijwirogfrefwerfw1}     locally finite.
 \end{enumerate}
\end{prop}
\begin{proof} 
Assertion \ref{goegoeri45345345345}  follows from homotopy invariance of $F$ and the fact that
$$C([0,1],C_{0}(X))\cong C_{0}([0,1]\times X)\, .$$

For  Assertion \ref{goegoeri453453453451} we argue as follows.  Let $X$ be in $\Top^{\lc}$, and let 
 $(A,B)$ be a closed decomposition of  $X$.  Then we have an isomorphism $B\setminus (B\setminus A)\cong X\setminus A$ in $\Top^{\lc}$ and therefore an isomorphism of $C^{*}$-algebras
$$C_{0}(X\setminus A)\cong C_{0}(B\setminus (A\cap B))\, .$$ 
We consider the following commuting diagram in $C^{*}\Alg^{\sep}$:
$$\xymatrix{C_{0}( B\setminus (A\cap B)) \ar[r]&C_{0}(B) \ar[r]&C_{0}(A\cap B)  \\
C_{0}(X\setminus A) \ar[u]^{\cong}\ar[r]&C_{0}(X)\ar[r]\ar[u]&C_{0}( A)\ar[u]}$$
The right square comes from a commuting square in $\Top^{\lc}$, and the left square commutes by the explicit description of the left horizontal maps given in Property \ref{ergioeriogegfgsfgsfg9} of $C_{0}$ which also asserts that 
  horizontal sequences are exact.
We now apply $F$ and get the commuting diagram
$$\xymatrix{F^{\an}(B\cap A)\ar[d]\ar[r]&F^{\an}(B)\ar[d]\ar[r]&F^{\an}( B\setminus (A\cap B))\ar[d]^{\simeq}\\
F^{\an}(A)\ar[r]&F^{\an}(X)\ar[r]&F^{\an}(X\setminus A)}$$
Since $F$ is excisive, the horizontal sequences are pieces of fibre sequences. Since the right vertical morphism is an equivalence, the left square 
is a push-out. This implies closed excision.

We now show Assertion \ref{wrthijwirogfrefwerfw}.
 Let $(X_{i})_{i\in I}$ be a countable family in $\Top^{\lc}$.  Using Property \ref{sbiojdboibsfdfvfvsfvsfvsdfv} of $C_{0}$ and that $F$ is additive we get the chain of equivalences
\[F^{\an}\big({\bigsqcup_{i\in I}^{\free}} X_{i}\big)\simeq F(\bigoplus_{i\in I}C_{0}(X_{i}))\simeq \prod_{i\in I} F(C_{0}(X_{i}))\simeq \prod_{i\in I} F^{\an}(X_{i})\, .\]

Finally we show Assertion \ref{wrthijwirogfrefwerfw1}. 
Let $X$ be in $\Top^{\lc}$. Then there exists   an open exhaustion $(B_{n})_{n\in \nat}$ of $X$ such that $\bar B_{n}$ is compact for every $n$ in $\nat$. Then every relatively compact subset of $X$ is contained in $B_{n}$ for sufficiently large $n$.  
The same argument as in the proof 
 Lemma \ref{lijfweiofjou4r3435434534534}  (using that $F^{\an}$ is  additive and excisive)  shows that
 $\lim_{n\in \nat} F^{\an}(X\setminus B_{n})\simeq 0$. 
 The analog of 
 Remark \ref{eofjewoifjefoewi23452345}
 for $\Top^{\lc}$ then implies that $F^{\an}$ is locally finite.
  \end{proof}

Let $F\colon C^{*}\Alg^{\sep,\op}\to \bC$ be a functor  and assume that $\bC$ is a complete and cocomplete stable $\infty$-category.
The following Definition combines \ref{weirghiwergregwegwergwerg9}  and  \ref{fiojwoifjwoifuwe45435}.
\begin{ddd}\label{weoigjoiegrergfsvgsdfvfdvsdfvsddfsvsdfvsdfvsf}
We define the functor $$F^{\an,\lf}:=L(F^{\an})^{\lf}\colon\TopBorn\to \bC\, .$$
\end{ddd}

\begin{kor} Assume that $F$ is a cohomology theory.
\begin{enumerate}
\item \label{weigjowerfgerfrefwef}
$F^{\an,\lf}$ is a locally finite homology theory.
\item \label{weigjowerfgerfrefwef1} $F^{\an,\lf}_{{|\Top^{\lc}}}\simeq F^{\an}$.
\end{enumerate}
\end{kor}
\begin{proof}
Assertion \ref{weigjowerfgerfrefwef} follows from Propositions \ref{qerugfihierwgegwerfwefwerfwerfw}  and \ref{wetigjowerfrwefwerf}.
For \ref{weigjowerfgerfrefwef1} we use Lemma~\ref{ilfjiweofuwei9u9435345345}. 
\end{proof}

%
%
%
%

 We can apply Proposition \ref{ifjewifweofewifuoi23542335345} to
$F^{\an,\lf}$. Recall the subset of objects $\bF$  of $\TopBorn$ introduced at the beginning of Subsection \ref{erjgiofggsfwergergsgsgr}.  
By Example \ref{weriogwetgwergregegw9}  the set  $\bF$ contains, e.g. \  locally finite   simplicial complexes with the structures induced from the spherical path metrics. 
  Let $X$ be in $\TopBorn$.

\begin{kor}\label{lijfoieijwofoewefewfewwf} If 
 $X$ is homotopy equivalent to an object of the subcategory $\bF$, then we have a natural equivalence
$$(F(*)\wedge \Sigma^{\infty,\topp}_{+})^{\lf}{(X)}\simeq F^{\an,\lf}(X)\, .$$
\end{kor}

%
%
%
 
Let $B$ be in $\Calg^{\sep}$.
 
\begin{ddd}\label{foiehweiofui23ur892342424}
The analytic locally finite $K$-homology\index{$K$-homology!analytic locally finite}\index{analytic!locally finite $K$-homology}\index{locally finite!analytic $K$-homology} with coefficients in $B$ is the functor \index{$K_{B}^{\an,\lf}$}
$$K_{B}^{\an,\lf}\colon \TopBorn\to \Sp$$ 
obtained by applying Definition \ref{weoigjoiegrergfsvgsdfvfdvsdfvsddfsvsdfvsdfvsf} to the $K$-cohomology $K_{B}$ of $C^{*}$-algebras  with coefficients in $B$ (Definition \ref{eughierwgvfdsdfvdfvsdfvfsdvsdfvswtf}).
 \end{ddd}

In view of the Property \ref{qreoighjiogqrfrgergfsgsgsdg}  of $\bK\bK$ we have an equivalence 
$K_{B}^{\an,\lf}(*)\simeq \Kast(B)$.

Specializing Corollary \ref{lijfoieijwofoewefewfewwf} we get for $X$ in $\TopBorn$:
 \begin{kor} \label{wergkjeowrgerfwefref} If  $X$ in $\bF$, then we have an equivalence
$$(\Kast(B)\wedge \Sigma^{\infty,\topp}_{+})^{\lf}{(X)}\simeq K_{B}^{\an,\lf}(X)\, .$$
\end{kor}  

 In particular,  if $X$ belongs to $\bF$ and is separable and locally compact, then using Lemma \ref{ilfjiweofuwei9u9435345345} and Property \ref{wergjiogerfeffewfwerfwerfwref} of $\bK\bK$ we get an isomorphism of groups
 \begin{equation}\label{4thgoijoieflk245efv}
\pi_{*}(\Kast(B)\wedge \Sigma^{\infty,\topp}_{+})^{\lf}{(X)}\simeq   \KK_{*}(C_{0}(X),B)\, .
\end{equation}  
  
\begin{rem}
 We got the impression that   the isomorphism \eqref{4thgoijoieflk245efv}
 is known as folklore, but we were not able to trace down a reference for this fact. 
 
In the global analysis literature it is standard to consider the functor  $X\mapsto \KK_{*}(C_{0}(X),B)$ as the definition of the group-valued locally finite $K$-homology  functor with coefficients in {a $C^*$-algebra} $B$. Its homotopy theoretic version  
represented by the left-hand side of~\eqref{4thgoijoieflk245efv} is  {usually} not used.
\hB
\end{rem}

\subsection{Coarsification spaces}\label{sdf132fwd}


In this section we introduce the notion of a coarsification of a  bornological coarse space~$X$.  A coarsification of $X$ is a simplicial complex with a reference map from $X$ such that the value of a coarsification of a locally finite homology on $X$ can be calculated by applying the locally finite homology theory itself to the coarsification, see Proposition~\ref{prop:sdf09432g23}.

This is actually one of the origins of coarse algebraic topology. Assume that the classifying space $BG$ of a group~$G$ has a   model which is a finite simplicial complex. Gersten \cite[Thm.~8]{gersten} noticed that for such a group the group cohomology $H^\ast(G, \IZ G)$ is a quasi-isometry invariant. Using coarse algebraic topology, we may explain this phenomenon as follows. We have the isomorphism $H^\ast(G,\IZ G) \cong H_{c}^\ast(EG)$, see \cite[Prop.~VIII.7.5]{brown}. Furthermore,
we have an isomorphism $H_{c}^\ast(EG)\cong  H\!\cX^\ast(EG)$, where  the latter is Roe's coarse cohomology \cite[{Prop.~3.33}]{roe_coarse_cohomology}.

Let now $G$ and $H$ are two groups with the property that   $BG$ and $BH$  are homotopy equivalent to finite simplicial complexes.  We equip $G$ and $H$ with word metrics associated to choices of finite generating sets. Then a quasi-isometry $G \to H$ of the underlying metric spaces translates to a coarse equivalence $EG \to EH$. Since coarse cohomology is invariant under coarse equivalences we get a chain of isomorphisms relating $H^\ast(G,\IZ G)$ and $H^{\ast}(H,\IZ H)$.  Therefore the cohomology group  $H^\ast(G,\IZ G)$ is a quasi-isometry invariant of $G$.

The above observation, {translated to homology theories,} can be generalized to the fact that a locally finite homology theory evaluated on a finite-dimensional, uniformly contractible space is a coarse invariant of such spaces, see Lemma \ref{fweoifuweiofuioewfewuifewfwefewf}. So on such spaces locally finite homology theories behave like coarse homology theories. And actually, this observation originally coined the idea of what a coarse homology theory should be.


Let $f_{0},f_{1}:X\to X^{\prime}$ be maps between topological spaces, and let $A^{\prime}$ be a   be a subset of $ X^{\prime}$ containing the images of $f_{0}$ and $f_{1}$. We say that $f_{0}$ and $f_{1}$ are homotopic in $A^{\prime}$ if there exists a homotopy $h:I\times X\to A^{\prime}$ from $f_{0}$ to $f_{1}$ whose image is also contained in $A^{\prime} $.

Let $X$ be a metric space. For  a point $x$ of $X$ and   a positive real number   $R $  we denote by $B(R,x)$ the ball in $X$ of radius $R$ with center $x$.
\begin{ddd}[{{\cite[Sec.\ 1.D]{gromov}}}]\label{defn:sdf8923}
 $X$ is uniformly contractible\index{uniformly!contractible} if for all $R $ in $(0,\infty)$ exists an $S$  in $[R,\infty)$  such that {for all $x $ in $X$} the inclusion $B(R,x) \to {X}$ is homotopic  in $B(S,x)$  to a constant map.
\end{ddd}

Note that the property of being uniformly contractible only depends on the quasi-isometry class of the metric.

In the following a metric space $X$ will be considered as a bornological coarse space $X_{d}$ with the structures induced by the metric, see Example \ref{welifjwelife89u32or2}.
We equip a simplicial complex $K$ with a good metric  (see Definition \ref{fwoiefjeoiwfefewf})  and denote the associated bornological coarse  space by $K_{d}$. If $K$ is finite-dimensional, then the metric on $K$ is independent of the choices up to quasi-isometry. In particular, the   bornological coarse space $K_{d}$ is well-defined up to equivalence.
\begin{ddd}[{\cite[Def.~2.4]{roe_index_coarse}}]
\label{defn:sdf8wef923}
A coarsification\index{coarsification!of a bornological coarse space} of a bornological coarse space  $X $ is a pair $(K,f)$  consisting of a  finite-dimensional, uniformly contractible simplicial complex~$K$ and  an equivalence $f :  X\to K_{d}$  in $\BC$.
\end{ddd}

\begin{ex}
If $X$ is  a non-empty bounded bornological coarse space {with the maximal coarse structure (e.g., a bounded metric space)},  then the pair $(*,f : X \to *)$ is a coarsification of $X$.

A coarsification of $\IZ$ is given by $(\IR, \iota :  \IZ \to \IR)$.
Another possible choice is    $(\IR, -\iota)$. 
 
Let $K$ be  a finite-dimensional  simplicial complex with a good metric. We assume that $K$ is   uniformly contractible.    Its the zero skeletion $K^{(0)}$ has an  induced metric  and gives rise to a bornological coarse space $K_{d}^{(0)}$.
Then the inclusion $\iota : K^{(0)}_{d}\to K_{d}$  turns $(K_{d},\iota)$ into a coarsification of the bornological coarse    space $K^{(0)}_{d}$.

{The following example is due to Gromov \cite[Ex.\ 1.D$_1$]{gromov}:} We consider a finitely generated group $G$ and assume that it admits a model for {its} classifying space $BG$ which is a finite simplicial complex. The action of $G$ on the universal covering $EG$ of $BG$ provides a choice of a map $f:G\to EG$ which  depends on the choice of a base point. We equip $EG$ with a good metric.
Then $(EG_{d},f)$  is a coarsification  of~$G$ equipped with the bornological coarse structure introduced in Example \ref{ewfijwfioiouiou23roi2r32rrew}.
\hB    
\end{ex}

Let $K$ be a simplicial complex, and let $A$   be a subcomplex of $K$. Furthermore, let  $X$ be a  metric space, and let
$f:K_{d}\to X_{d}$ be a  morphism of bornological coarse spaces such that $f_{|A}$ is continuous.

\begin{lem}\label{lem:sdf9823}
If $K$ {is} finite-dimensional and $X$ is
uniformly contractible, 
  then $f$   is close to a morphism of bornological coarse spaces which extends $f_{|A}$ and  is in addition continuous.
\end{lem}

%
%
\begin{proof}
 We will define  a continuous map $g :  K \to X$ by induction over the relative skeleta of $K$ such that it is 
 close to $f$ and satisfies $f_{|A}=g_{|A}$. Then $g:K_{d}\to X_{d}$ is also a morphism of bornological coarse spaces.

On the zero skeleton $K^{(0)}\cup A$ {(relative to $A$)} of $K$ we define $g_0:=f_{|K^{(0)}\cup A}$. 

Let now $n\in \nat$  and 
assume that $g_n:K^{(n)}\cup A\to X$ is already defined such that $g_n$ is continuous  and close to the restriction of $f$ to $K^{(n)}\cup A$.  Since the metric on $K$ is good there exists a uniform bound    of the diameters of $(n+1)$-simplices of $K$.
 Since $g_{n}$ is controlled  there exists an $R$ in $(0,\infty)$ such that $g(\partial \sigma)$ is contained in some $R$-ball for every $(n+1)$-simplex $\sigma$ of $K$. 
 Let     $S$ in $[R,\infty)$ be
as in Definition \ref{defn:sdf8923}. 

{We now extend $g_{n}$ to the interiors of the $(n+1)$-simplices separately.   Let $\sigma$ be an $(n+1)$-simplex  in $K^{(n+1)}\setminus A$.} 
  Using a contraction of the simplex $\sigma$ to its center and the contractibility of $R$-balls of $X$ inside $S$-balls, we    can continuously extend   $(g_{n})_{|\partial \sigma}$ to $\sigma$ with image in
an $(S+R)$-neighbourhood of $g_{n}(\partial \sigma)$.

 In this  way we get a continuous extension  
$g_{n+1} :K^{(n+1)}\cup A\to X$ which is still close to $f_{|K^{(n+1)}\cup A}$.  
 \end{proof}

A finite-dimensional simplicial complex $K$ can  naturally be considered as a topological  bornological space with the bornology of  metrically bounded subsets. We will denote   this topological bornological space by $K_{t}$. A morphism of bornological coarse spaces $K_{d}\to K_{d}^{\prime}$  which is in addition continuous is a morphism $K_{t}\to K_{t}$ of topological bornological spaces.

We consider finite-dimensional simplicial complexes $K$ and $K^{\prime}$ and a morphism $f:K_{d}\to K_{d}^{\prime}$ of bornological coarse spaces.
\begin{lem}\label{fweoifuweiofuioewfewuifewfwefewf}
If $K$ and $K^{\prime}$ are  uniformly contractible, and 
 $f $ is an equivalence, then 
 $f$  is close to a homotopy equivalence 
 $ K_{t}\to K^{\prime}_{t}$, and any two choices of {such a}  homotopy equivalence    are homotopic to each other.
 \end{lem}
 \begin{proof}
 By Lemma \ref{lem:sdf9823} we can replace $f$ by a close continuous map which we will also denote by $f$. We claim that this is the desired homotopy equivalence $f:K_{t}\to K_{t}$. 
 Denote by $g  :  K_{d}^{\prime} \to K_{d}$ an inverse equivalence.
 By Lemma \ref{lem:sdf9823} we can assume  that $g$ is  continuous.
The compositions $f\circ g$ and $g\circ f$ are close to the respective identities.

We consider the complex $I\times K$ with the subcomplex $\{0,1\}\times K$.
We define  a morphism of bornological coarse spaces $I\times K_{d}\to K_{d}$ by $\id_{K}$ on $\{0\}\times K$ and by $g\circ f\circ \pr_{K}$ on
$(0,1]\times K$. By Lemma \ref{lem:sdf9823} we can replace this map by a close  continuous map which is a homotopy between $\id_{K_{t}}$ and $g\circ f:K_{t}\to K_{t}$.
 This shows that $g\circ f:K_{t}\to K_{t}$ and $\id_{K_{t}}$ are  homotopic.
 We argue similarly for $f\circ g$.
 
 Furthermore, we obtain homotopies between different choices of continuous replacements of $f$ by a similar argument. 
\end{proof}

Let $X$ be a bornological coarse space and consider two coarsifications $(K,f)$ and $(K^{\prime},f^{\prime})$ of $X$. 
\begin{kor}\label{cor:sdf8923}
There is a   homotopy equivalence $K_{t}\to K_{t}^{\prime}$  uniquely {determined} up to homotopy 
by the compatibility with the structure maps $f$ and $f^{\prime}$.
\end{kor}

\begin{proof}
Let $g:K_{d}\to X$ be an inverse equivalence of $f$.
Then we apply Lemma \ref{fweoifuweiofuioewfewuifewfwefewf} to the composition $f^{\prime}\circ g:K_{d}\to K_{d}^{\prime}$.
\end{proof}

%
%
%
%

%
%
%
%
%
%

Let $X$ be a bornological coarse space.
\begin{ddd}\label{ejfwjefklwejlkewfwefwefwefewf}
$X$ has strongly bounded geometry\index{strongly!bounded geometry}\index{bounded geometry!strongly} if it has the minimal compatible  bornology  and for every entourage $U$ of $X$ there exists a uniform finite upper bound on the cardinality of $U$-bounded subsets of $X$.
\end{ddd}

\begin{rem}
In the literature, the above notion of strongly bounded geometry is usually only given for discrete metric spaces and either just called ``bounded geometry'' (Nowak--Yu \cite[Def.~1.2.6]{nowak_yu}) or ``locally uniformly finite'' (Yu~\cite[Pg.~224]{yu_baum_connes_conj_coarse_geom}).
\hB
\end{rem}

Note that the above notion is not invariant under equivalences. The invariant notion is as follows {(cf.~Roe~\cite[Def.~2.3(ii)]{roe_index_coarse} or Block--Weinberger~\cite[Sec.~3]{block_weinberger_1}):}
\begin{ddd}\label{dsf892323234}
A bornological coarse space has bounded geometry\index{bounded geometry!of a bornological coarse space} if it is equivalent to a bornological coarse space of strongly bounded geometry.
\end{ddd}

Let $K$ be a simplicial complex.
\begin{ddd}\label{sdf8923nwer2fd}
$K$ has bounded geometry\index{bounded geometry!of a simplicial complex}  if there exist a uniform finite upper bound on the 
 cardinality of the set of vertices of the stars at all vertices of $K$.
\end{ddd}
Note that this condition implies that $K$ is finite-dimensional and that the number of simplices that meet at any vertex is uniformly bounded.

%
%

Let $K$ be a simplicial complex. If $K$ has bounded geometry (as a simplicial complex), then the associated bornological coarse space $K_{d}$ has bounded geometry (as a coarse bornological space).  In this case the set of vertices of $K$ with induced structures is an equivalent bornological coarse space of  strongly   bounded geometry.

\begin{rem}
In general, bounded geometry of $K_{d}$ does not imply bounded geometry of $K$ as a simplicial complex. Consider, e.g., the disjoint union $K:=\bigsqcup_{n\in \nat} \Delta^{n}$, where we equip $\Delta^{n}$
with the metric induced from the standard embedding into $S^{n}$. The main point is that the diameters of $\Delta^{n}$ for $n$ in $\nat$ are uniformly bounded.
\hB
\end{rem}

Let $K$ be a simplicial complex.
\begin{prop}\label{prop:sdf23d9823f}
If $K$ has bounded geometry and is uniformly contractible,   then we have a natural   equivalence $Q(K_{d}) \simeq \Sigmalf K_{t}$.
\end{prop}

\begin{proof}
We consider the inclusion $K^{(0)}\to K$ of the zero skeleton and let $K^{(0)}_{d}$ be the coarse bornological space with the induced structures.
 If $U$ is an entourage of $K_{d}^{(0)}$, then we consider the map
\begin{equation}\label{fwefiuhiu243rtwrg}
K^{(0)}\to  P_{U }(K_{d}^{(0)})
\end{equation} 
 given by Dirac measures. 
Since the metric on $K$ is good and $K$ is finite-dimensional we can choose $U$ so large that all simplices 
of $K$ are $U$-bounded. Using convex interpolation we can extend the map \eqref{fwefiuhiu243rtwrg} to a continuous map
$$f:K\to P_{U}(K_{d}^{(0)})\, .$$
By construction,
$$f:K_{d}\to  P_{U}(K_{d}^{(0)})_{d}$$ is an equivalence of bornological coarse spaces.
Since $K$  has bounded geometry, $K^{(0)}_{d}$ is of strongly    bounded geometry, and hence  $P_{U}(K^{(0)}_{d})$ is a  finite-dimensional simplicial complex.
Since $K_{d}$ is   uniformly contractible,   by Lemma \ref{lem:sdf9823}  we can choose a continuous inverse equivalence $g: P_{U}(K_{d}^{(0)})_{d}\to K_{d}$.

%
%
%
The following argument is taken from {Nowak--Yu} \cite[Proof of Thm.~7.6.2]{nowak_yu}. Since the coarse structure of $K_{d}^{(0)}$ is induced by a metric we
can  choose a cofinal sequence $(U_n)_{n \in \IN}$ of entourages of $K_{d}^{(0)}$. We let   $$f_n :  P_{U_n}(K_{d}^{(0)})_{d} \hookrightarrow P_{U_{n+1}}(K_{d}^{(0)})_{d}$$ be the inclusion maps. Each $f_n$ is an   equivalence of bornological coarse spaces. For every $n$ in $\nat$     we can choose an inverse equivalence    $$g_n :  P_{U_{n+1}}(K_{d}^{(0)})_{d} \to K_{d}$$   to the composition $f_n \circ \cdots \circ f_0 \circ f$ which {is} in addition continuous. Then the composition
 $$g_n \circ (f_n \circ \cdots \circ f_0 \circ f):K_{t}\to K_{t}$$ is  homotopic to $\id_{K_{t}}$.
 The following picture illustrates the situation.
\[\xymatrix{
K_{t} \ar@/^1pc/[r]^-{f} & P_U(K^{(0)}_{d})_{t}\ar@/^1pc/[r]^-{f_0} \ar[l]^<{g} & P_{U_1}(K^{(0)}_{d})_{t} \ar@/^1pc/[r]^-{f_1} \ar@/^1pc/[ll]^<{g_0} & P_{U_2}(K_{d}^{(0)})_{t} \ar@/^1pc/[r]^-{f_2} \ar@/^2pc/[lll]^<{g_1} & \cdots
}\]

Note that we are free to replace the sequence $(U_{n})_{n\in \nat}$ by a cofinal subsequence $(U_{n_{k}})_{k\in \nat}$.
We use this freedom in order to ensure that for every $n\in \nat$ the map  $f_n$ is  homotopic to the composition  $(f_n \circ \cdots \circ f_0 \circ f) \circ g_{n-1}$.

Since the functor $\Sigmalf
$ is homotopy invariant (see Example \ref{weoifujweoifowe234234345}), it follows that the induced map $$(f_n \circ \cdots \circ f_0 \circ f)_\ast :  \pi_\ast(\Sigmalf K_{t}) \to \pi_\ast(\Sigmalf P_{U_{n+1}}(K_{d}^{(0)})_{t})$$ is an isomorphisms onto the image of $(f_n)_\ast$. So the induced map
$$\Sigmalf K_{t}\to \colim_{n\in \nat}  \Sigmalf P_{U_{n}}(K^{(0)}_{d})_{t}\simeq  \colim_{U} \Sigmalf  P_{U}(K_{d}^{(0)})_{t}\stackrel{\text{Def.} \ref{efifjweiof89247u23553435534543}}{\simeq}   Q(K_{d}^{(0)}) $$ induces an  isomorphism in homotopy groups.
The desired equivalence is now given by
$$\Sigmalf K_{t}\simeq Q(K_{d}^{(0)})\simeq Q(K_{d})\, ,$$
where for the last equivalence we use that $Q$ preserves equivalences between bornological coarse spaces and that $K_{d}^{(0)}\to K_{d}$ is an equivalence.
\end{proof}

%
%
%

{Let $\bC$ be a complete and cocomplete stable   $\infty$-category.} Let $E:\TopBorn\to \bC$ be a locally finite homology theory, and let $X$ be a bornological coarse space.
\begin{prop}\label{prop:sdf09432g23}
If $X$ admits a coarsification $(K,f)$ of bounded geometry, 
then we have a canonical equivalence  $(QE)(X) \simeq E(K_{t})$.
\end{prop}

\begin{proof}
Since $f :  X \to K_{d}$ is an equivalence of bornological coarse spaces, it induces an equivalence $(QE)(X) \simeq (QE)(K_{d})$. An argument similar  to the proof of Proposition \ref{prop:sdf23d9823f}  (replace $\Sigmalf$ by $E$) gives the equivalence $(QE)(K_{d}) \simeq E(K_{t})$.
\end{proof}

\section{Coarse \texorpdfstring{$\boldsymbol{K}$}{K}-homology}
\label{sec:sdn934}

This final section of the book is devoted to the construction and investigation of coarse $K$-homology. 
This coarse homology theory and its applications to index theory, geometric group theory and topology
are one of the driving forces of development in coarse geometry. A central result is the  coarse Baum--Connes conjecture which asserts that under certain conditions on a bornological coarse space the assembly map  from  the coarsification of the locally finite version of $K$-theory to  the $K$-theory of a certain Roe algebra associated to the bornological coarse  space is an equivalence.

The classical construction of coarse $K$-homology works for proper metric spaces and produces $K$-theory groups. 
The construction involves the following steps:
\begin{enumerate}
\item   For a space $X$ one chooses an ample $X$-controlled Hilbert space.
\item One then defines a Roe subalgebra of the bounded operators on  that Hilbert space.
\item Finally one defines the coarse $K$-homology groups of $X$  as the topological $K$-theory groups of this Roe algebra. 
\end{enumerate} 
All these notion will be discussed in detail in the subsequent subsections.
 
In order to fit coarse $K$-homology into the general framework developed in the present book we require a spectrum valued coarse $K$-theory functor which is defined on the category $\BC$.
The necessity to choose an ample Hilbert space causes two problems.
First of all, as we shall see, ample Hilbert spaces do not always  exist.
Furthermore, even for proper metric spaces (where ample Hilbert spaces exist) it is not clear how to ensure functoriality on the spectrum level. 

Our solution is to define for every $X$ in $\BC$  a Roe $C^{*}$-category in a functorial way  which comprises all possible choices of $X$-controlled Hilbert spaces.  We then   define the coarse $K$-homology spectrum as the topological $K$-theory  spectrum of this $C^{*}$-category. 
This solves the functoriality problem.  Furthermore,
because the Roe category ``absorps'' all $X$-controlled Hilbert spaces we do not need the existence of one ample $X$-controlled Hilbert space which absorps all other.

 We will  define actually two slightly different versions of coarse $K$-homology functors. The construction    and the verification that these functors  {satisfy}   the axioms of a  coarse homology theory will be completed   in Section~\ref{sec:jkbewrgh}.  
 In a subsequent paper \cite{Bunke:ad} we vastly generalize the constructions of the present section.
 We consider the equivariant case and allow coefficients in a $C^{*}$-category  {with} group action.
 These equivariant coarse $K$-homology theories are then used to obtain injectivity results for assembly maps.

In Sections~\ref{ofjewiofiouioru233244343244} till~\ref{sec:dfs8934} we discuss  the technical preliminaries needed for the definition of coarse $K$-homology.

In Section~\ref{ijfweoifweoifoi239804234342} we show that in   essentially  all cases of interest our new definition of coarse $K$-homology 
  coincides with the classical one roughly described above.

Since we define different versions of coarse $K$-homology, the question under which circumstances   they coincide immediately arises. A partial answer to this question is given by  an application of the comparison result  {stated as} Corollary \ref{thm:sdf98245csv}.  

In Section~\ref{subsec:drsgd} we will discuss some simple applications of our setup of coarse $K$-homology to index theory, and in Section~\ref{sec:kjnbsfd981} we construct the assembly map as a natural transformation between  {the} coarsification of locally finite $K$-homology and coarse $K$-homology. The comparison theorem then immediately proves the coarse Baum--Connes conjecture for spaces of finite asymptotic dimension --- a result first achieved by Yu \cite{yu_finite_asymptotic_dimension}.

\subsection{\texorpdfstring{$X$}{X}-controlled Hilbert spaces} \label{ofjewiofiouioru233244343244}

Let $X$ be a bornological coarse space. We can consider $X$ as a discrete topological space. We    let $C(X)$\index{$C(-)$} be the $C^{\ast}$-algebra of all bounded, continuous  (a void condition) $\C$-valued functions   on $X$ with the  {supremum}-norm and involution given by complex conjugation. 
For a subset $Y$ of $X$  we denote by $\chi_Y$  in $C(X)$ the characteristic function of $Y$.

\begin{ddd}\label{feiojwoeffewfwefewfwef}
An $X$-controlled Hilbert space\index{Hilbert space!controlled}\index{controlled!Hilbert space} is a pair $(H,\phi)$, where $H$ is a Hilbert space and $\phi :  C(X)\to B(H)$ is a unital $^\ast$-representation such that for every  bounded subset $B$ of $X$ the  subspace $\phi(\chi_{B}) H $ is separable. 
\end{ddd}

For a subset $Y$ of $ X$ we    introduce the notation\index{$H(-)$}
\begin{equation}
\label{eq_notation_fiber_Hilbertspace}
H(Y):=\phi(\chi_{Y})H
\end{equation}
for the subspace of vectors of $H$ ``supported on $Y$''.  The operator  $\phi(\chi_{Y})$ is the orthogonal projection from $H$ onto $H(Y)$. Since we require $\phi$ to be unital, we have $H(X) = H$. We will also often use the notation $\phi(Y):=\phi(\chi_{Y})$.  
Note that we do not restrict the size of the Hilbert space $H$ globally. But we require 
  that $H(B)$ is separable for every bounded subset $B$ of $X$, i.e., $H$ is locally separable on $X$. We will also need the following more restrictive local finiteness condition.

Let $X$ be a bornological coarse space, and let $(H,\phi)$ be an $X$-controlled Hilbert space.
\begin{ddd}\label{ewoifjweoifiufu89234234324}
$(H,\phi)$ is called locally finite\index{Hilbert space!locally finite}\index{locally finite!Hilbert space} if for every  bounded subset $B$ of $X$ the space $H(B)$ is finite-dimensional.
\end{ddd}

Let $X$ be a bornological coarse space, and let $(H,\phi)$ be an $X$-controlled Hilbert space.
\begin{ddd}\label{rgioerjog34t4t34t34t}
$(H,\phi)$ is determined on points\index{determined on points}\index{controlled!Hilbert space!determined on points}\index{Hilbert space!determined on points} if for every {vector} $h$ of $H$ the condition $\phi(\{x\})h=0$ for all  points $x$ of $ X$ implies  $h=0$.
\end{ddd}

 Let $(H,\phi)$ be an $X$-controlled Hilbert space.
Note that the direct sums below are   taken in the sense of Hilbert spaces and in general involve a completion.
\begin{lem} \label{fghiofweqewfqwed}The following are equivalent:
\begin{enumerate}
\item   \label{egkjweogerwgegewg1}  $(H,\phi)$ is determined on points,
\item  \label{egkjweogerwgegewg2}
The family  $(H(\{x\})\to H)_{x\in X}$ of canonical   inclusions induces an isomorphism  $ \bigoplus_{x\in X} H(\{x\})\cong H$ 
\item  \label{egkjweogerwgegewg3}
For every  partition $(Y_{i})_{i\in I}$ of $X$ the family of natural inclusions $H(Y_{i})\to H$ induces an  isomorphism $  \bigoplus_{i\in I} H(Y_{i})\cong H$.
\end{enumerate}
\end{lem}
\begin{proof}
It is clear that Assertion \ref{egkjweogerwgegewg3} implies Assertion  \ref{egkjweogerwgegewg2}, and that 
Assertion    \ref{egkjweogerwgegewg2} implies Assertion    \ref{egkjweogerwgegewg1}.

Assume now Assertion  \ref{egkjweogerwgegewg1} and let $(Y_{i})_{i\in I}$ be a pairwise disjoint  partition of $X$. It is clear that the canonical map 
 $\bigoplus_{i\in I} H(Y_{i})\to H$ is injective. We must show that its orthogonal complement is trivial. Assume that $h$ in $H$ is orthogonal to the image.
 If $x$ is in $X$, then we choose $i$ in $I$ such that $x\in Y_{i}$. Since $H(\{x\})\subseteq H(Y_{i})$
 we conclude that $\phi(\{x\})h=0$. Since this holds for every $x$ in $X$ we conclude that $h=0$.
 \end{proof}

\begin{ex} 
We consider the bornological coarse space $\nat_{min,max} $ given by the set $\nat$ with the minimal coarse and the maximal bornological structures,   see Example~\ref{welkfjwekfo23u4234234}. We choose a point  in the boundary of the Stone--\v{C}ech compactification of the discrete space $\nat$. The evaluation at this point is a character $\phi:C(\nat)\to \C=B(\C)$ which annihilates all finitely supported functions.
The $X$-controlled Hilbert space $(\C,\phi)$ is not determined on points.
\hB
\end{ex}

In the following we discuss a construction of $X$-controlled Hilbert spaces.
%

\begin{ex}\label{fjwefjewewofewfewfewfewf}
Let $X $ be a bornological coarse space.
  Let $D$  be a locally countable subset of $X$ (see Definition \ref{defn:asdfwewt}). We define the counting measure $\mu_{D}$\index{$\mu_{-}$}  on the $\sigma$-algebra $\cP(X)$ of all subsets of $X$ by   \[\mu_{D}(Y):= |(Y\cap D)|.\]
We set $H_{D}:=L^{2}(X,\mu_{D})$ and define the representation $\phi_{D}$ of $C(X)$ on $H_D$ such that $f$ in $C(X)$ acts as a multiplication operator.  
Then $(H_{D},\phi_{D})$ is an $X$-controlled Hilbert space which is determined on points. 

If $D $ is   locally finite (see Definition \ref{defn:asdfwewt}), then the Hilbert space $(H_{D},\phi_{D})$  is locally finite.

As usual we set
$\ell^{2}:=L^{2}(\nat,\mu_{\nat})$.  Then
  $(H_{D}\otimes \ell^{2},\phi_{D}\otimes \id_{\ell^{2}})$ is also an $X$-controlled Hilbert space. 
  If $D\not=\emptyset$, then it is not locally finite.
\hB
\end{ex}

\begin{rem}\label{weoigjoreogewrg9}
Observe that the notion of an $X$-controlled Hilbert space and the properties of being determined on points or locally finite only depend on the bornology of $X$.
\hB
\end{rem}

 
 Let $X$ be a bornological coarse space.
We consider two 
 $X$-controlled Hilbert spaces $(H,\phi)$, $(H^{\prime},\phi^{\prime})$  and a bounded linear operator  $ A\colon H\to H^{\prime}$.
 Let $\cP_{X}$ denote the power set of $X$.
\begin{ddd}\label{defn:iuwn3r2}\mbox{}
\begin{enumerate}
\item  
The support\index{support!of an operator} of $A$ is the subset of $X\times X$ given by 
\[\supp(A):=\bigcap_{\left\{U\subseteq X\times X :  \left(\forall Y\in \cP_{X} :  A(H(Y)) \subseteq H(U[Y])  \right)\right\}}U\,.\]
\item  $A$ has controlled propagation\index{controlled!propagation}\index{propagation} if $\supp(A)$ is a coarse entourage of $X$.
\end{enumerate}
\end{ddd}

  The idea of the support of $A$ is that if a vector $h$ in $H$ is supported on $Y$, then  $Ah$ is supported on $\supp(A)[Y]$.   So if $A$ has controlled propagation, it moves supports in a controlled manner.


Let $f :  X\to X^{\prime}$ be a morphism  between bornological coarse spaces, and  let $(H,\phi)$ be an $X$-controlled Hilbert space.
\begin{ddd} \label{ewljfewilewfiouwefwefew}
We define the $X^{\prime}$-controlled Hilbert space $f_{*}(H,\phi) $  
to be the Hilbert space $H$ with the control $\phi\circ f^{*}$.
 \end{ddd}

Indeed, {$(H, \phi \circ f^\ast)$ is an $X'$-controlled Hilbert space:} if $B^{\prime}$ is a  bounded subset  of $X^{\prime}$, then $f^{-1}(B^{\prime})$ is bounded in $X$ since $f$ is proper. Hence $(\phi\circ f^{*})(\chi_{B^{\prime}})H=\phi(\chi_{f^{-1}(B^{\prime})})H$ is separable. If $(H,\phi)$ is locally finite, then so is $f_{*}(H,\phi)$. Similarly, if $(H,\phi)$ is determined on points, then so is $f_{*}(H,\phi)$.

In the following we list some obvious properties of the support of operators between controlled Hilbert spaces
  The support  of operators has the following properties which are easy to check:
  \begin{enumerate}
\item 
  $\supp(A^{*})=\supp(A)^{-1}$.
  \item 
   If $A,A'$ are two operators between the same  $X$-controlled Hilbert spaces, and $\lambda$ is in $\C$, then $$\supp(A+\lambda A')\subseteq \supp(A)\cup \supp(A)\,.$$
   \item If $(H,\phi)$, $(H',\phi')$ and $(H'',\phi'')$ are $X$-controlled Hilbert spaces and $A\colon H\to H'$ and $B\colon H'\to H''$ are  bounded operators, then 
   $$\supp(B\circ A)\subseteq \supp(B)\circ \supp(A)\, .$$
   \item Let $f\colon X\to X'$ be a morphism in $\BC$. If $(H,\phi)$ and $(H',\phi')$ are two $X$-controlled Hilbert spaces and $A:H\to H'$ is a bounded operator, then we  
let $f_{*}A:=A$ be considered as an operator between $X'$-controlled Hilbert spaces. Then we have
$$\supp(f_{*})\subseteq f_{*}(\supp(A))\, .$$
\end{enumerate}

Let $(H,\phi)$ be an $X$-controlled Hilbert space, and let $H^{\prime}$ a closed subspace of $ H$.

\begin{ddd}\label{defn:sfd9823}
\index{locally finite!subspace}
$H^\prime$ is locally finite, if it admits an $X$-control $\phi^\prime$ such that:
\begin{enumerate}
\item  $(H^\prime, \phi^\prime)$ is locally finite (Definition \ref{ewoifjweoifiufu89234234324}).
\item  $(H^\prime, \phi^\prime)$ is determined on points.
\item  The inclusion $H^\prime  \to H $ has controlled propagation (Definition \ref{defn:iuwn3r2}).
\end{enumerate}
If we want to highlight $\phi^{\prime}$, we say that it recognizes $H^{\prime}$ as a locally finite subspace of $H$.
\end{ddd}

 \begin{ex} We consider $\Z$ as a bornological coarse space with the metric structures.
 Furthermore, we consider the $X$-controlled Hilbert space $H:=L^{2}(\Z)$ with respect to the counting measure.
 The control $\phi$ is given by multiplication operators. The $\Z$-controlled Hilbert space $(H,\phi)$ is locally finite.
For $n$ in $\Z$ we let  $\delta_{n}$ in $L^{2}(\Z)$ denote the corresponding basis vector of $H$.
 
 The one-dimensional subspace $H^{\prime}$ generated by the vector
 $\sum_{n\in \nat} e^{-|n|}\delta_{n}$ is not locally finite.
 
  On the other hand, the  $\infty$-dimensional subspace $H^{\prime\prime}$ generated by the family  $(\delta_{2n})_{n\in \Z}$  is locally finite. 
 \hB
 \end{ex}

\begin{ex}
We consider a bornological coarse space $X$ and a   locally countable subset $D$ of $  X$. Then we can form the $X$-controlled Hilbert space $(H_{D}\otimes \ell^{2},\phi_{D}\otimes \id_{\ell^{2}})$ as in Example \ref{fjwefjewewofewfewfewfewf}.
Let $F$  be some finite-dimensional subspace of $  \ell^{2}$, and $D^{\prime}$   let be a locally finite subset  of $D$. Then $H^{'}:=H_{D^{\prime}}\otimes F$ considered as a subspace of $H_{D}\otimes \ell^{2}$ in the natural way is a locally finite subspace. As control we can take the restriction of $\phi_{D}\otimes \id_{\ell^{2}}$.
\hB
\end{ex}

\subsection{Ample \texorpdfstring{$X$}{X}-controlled Hilbert spaces}
\label{oiwfiowioewioewfewfew}
 Let $(H,\phi)$ be an $X$-controlled Hilbert space. 
\begin{ddd}\label{fjwefewiojoi2jroi23jr23r23r23r}
\index{ample}\index{controlled!Hilbert space!ample}
 $(H,\phi)$ is ample if it is determined on points and there exists an entourage $U$ of $X$ such that
$H(U[x])$ is $\infty$-dimensional for every $x$ in $X$.
\end{ddd}

Note that an ample $X$-controlled Hilbert space is not locally finite.

\begin{rem}
Our notion of ampleness differs from the usual notion where one also demands that no non-zero function  in $C(X)$ acts by compact operators.

Consider for example the bornological coarse space $\Q$ with the structures induced from the embedding into the metric space $\R$. The $\Q$-controlled Hilbert space $(H_{\Q},\phi_{\Q})$ (see Example \ref{fjwefjewewofewfewfewfewf}) is ample. But the characteristic functions
of points (they are continuous since we consider $\Q$ with the discrete topology) act by rank-one operators and are in particular compact.
\hB
\end{rem}


In the following we collect some facts about ample $X$-controlled Hilbert spaces.
Let $X$ be a bornological coarse space, $D$   be a subset of $X$, and $U$ be an entourage of $X$.

\begin{ddd}\label{defn:okmnsf}\mbox{}
\begin{enumerate}
\item\label{defn:ojiwnetr445}
   $D$  is    dense\index{dense}\index{subset!dense} in $X$, if there exists an entourage $V$ of $X$ such that $V[D]=X$. Sometimes we say that $D$ is $V$-dense.
\item $D$ is  $U$-separated\index{$U$-!separated}\index{subset!$U$-separated}   if for every two distinct points $d,d^{\prime}$ in $D$ we have $d^{\prime}\not\in U[d]   $. 
\item \label{rfjhewwefwewfewfewf} $X$ is locally countable\index{locally!countable!space}\index{locally!finite!space}\index{space!locally countable}\index{space!locally finite} (or locally finite, respectively) if it admits an entourage $V$  such that every  $V$-separated subset of $X$ is locally countable (or locally finite, respectively).
 \end{enumerate} 
 \end{ddd}
 
\begin{rem}
One should not confuse the above notion of local finiteness of $X$ as a bornological coarse space with the notion of local finiteness of $X$ (as a subset of itself) as given in Definition~\ref{defn:asdfwewt}. {But if $X$ is locally finite in the latter sense, then it is locally finite in the sense of Definition~\ref{defn:okmnsf}.\ref{rfjhewwefwewfewfewf}.}
\hB
\end{rem}

%

\begin{ex}\label{ex:ert34rn394}
Let $D$   be a dense and locally countable subset of $X$ (see Definition~\ref{defn:asdfwewt}). Then the $X$-controlled Hilbert space $(H_{D}\otimes \ell^{2}, \phi_{D}\otimes \id_{\ell^{2}})$ constructed in Example \ref{fjwefjewewofewfewfewfewf} is ample.
\hB
\end{ex}

\begin{ex}\label{ewljfewilewfiouwefwefew1}
Let $f :  X\to X^{\prime}$ be a morphism between bornological coarse spaces, and   $(H,\phi)$ be an ample $X$-controlled Hilbert space.
Then $ f_{*}(H,\phi)$ is ample if and only if $f(X)$   is dense in $ X^\prime$.
\hB
\end{ex}

Let $X$ be a bornological coarse space. 
\begin{lem}\label{lem:iohwenrf}
For every symmetric entourage $U$ of $X$ there exists a $U$-separated and $U$-dense subset of $ X$. \end{lem}

\begin{proof}
The $U$-separated subsets of $X$ are partially ordered by inclusion. We can use the Lemma of Zorn to get a maximal $U$-separated subset  $D$ of $X$. Due to maximality of~$D$ and  since $U$ is symmetric we can conclude that $U[D] = X$.
\end{proof}

Let $X$ be a bornological coarse space.

\begin{prop}\label{fkwjlkwejlewfewfewfewf3242344}
$X$ admits an ample $X$-controlled Hilbert space if and only if $X$ is locally countable.
\end{prop}
  
\begin{proof}

Assume that $X$ is locally countable. Then we can find an entourage $U$ of $X$ such that every $U$-separated subset of $X$ is locally countable. After enlarging $U$ if necessary we can in addition assume that $U$ is symmetric.

 Using Lemma \ref{lem:iohwenrf} we get a $U$-separated subset $D$ with $U[D] = X$. By our assumption on $U$ the subset  $D$ is  locally countable. We then construct the ample $X$-controlled Hilbert space as in Example~\ref{ex:ert34rn394}.

We now show the converse. Let $(H,\phi)$ be an ample $X$-controlled Hilbert space and let $U$ be an entourage such that $H(U[x])$ is $\infty$-dimensional for all $x$ in $X$.    Consider a $U\circ U^{-1}$-separated subset $Y$. We must show that $Y\cap B$ is at most countable for every  bounded subset~$B$ of~$X$.

We have an inclusion
\[H \Big( \bigcup_{y\in Y\cap B}U[y] \Big) \subseteq H(U[B])\, .\]
Since $U[B]$ is bounded  we know that $H(U[B])$ is separable. Hence $H(\bigcup_{y\in Y\cap B}U[y])$
is separable, too. On the other hand $U[y]\cap U[y^{\prime}]=\emptyset$ for all distinct pairs $y,y^{\prime}$ in $Y$.
 Hence
 \[H \Big( \bigcup_{y\in Y\cap B}U[y] \Big) \cong \bigoplus_{y\in Y\cap B} H(U[y])\, .\]
 Since this is a separable Hilbert space and all $H(U[y])$ are non-zero, 
 the index set  $Y\cap B$ of the sum  is at most countable.
\end{proof}

 Let $X$ be a bornological coarse space and  $(H,\phi)$ and $(H^\prime, \phi^\prime)$ be two $X$-controlled Hilbert spaces. 
\begin{lem}\label{lem:drf8934}
If  $(H,\phi)$ and $(H^\prime, \phi^\prime)$ are ample, then there exists a unitary operator $H \to H^\prime$ with controlled propagation.
\end{lem}

\begin{proof}
The bornological coarse space
$X$ is locally countable by Lemma \ref{fkwjlkwejlewfewfewfewf3242344}. Let $U$ be a symmetric entourage such that every $U$-separated subset is locally countable. After enlarging $U$ if necessary we can also assume  that the Hilbert spaces  $H(U[x])$ and $H^\prime(U[x])$ are $\infty$-dimensional for every $x$  in $X$. Choose a maximal $U^{2}$-separated subset $D$  of $ X$ as in Lemma \ref{lem:iohwenrf}. 
Consider  the ample $X$-controlled Hilbert space $(H_{D}\otimes \ell^{2}, \phi_{D}\otimes \id_{\ell^{2}})$ from Example~\ref{ex:ert34rn394}. 

We can construct unitary operators $u_H \colon  H \to H_{D}\otimes \ell^{2}$ and $u_{H^\prime} \colon  H^\prime \to H_{D}\otimes \ell^{2}$ of controlled propagation. Then    $u_{H^\prime}^\ast u_H \colon  H \to H^\prime$ is the desired unitary operator.

In order to construct, e.g., $u_{H}$  we proceed as follows. Since $D$ is $U^{2}$-separated, $(U[d])_{d\in D}$  is a pairwise disjoint family of subsets of $X$. Because $D$ is maximal    the family of subsets, $(U^{2}[d])_{d\in D}$ covers $X$. We choose a well-ordering on $D$ and using transfinite induction
we define a partition  $(B_{d})_{d\in D}$ of $X$ such that
$$U[d]\subseteq B(d)\subseteq U^{2}[d]$$ for all $d$ in $D$. To this end we set
$$B_{d_{0}}:=U^{2}[d_{0}]\setminus \bigcup_{d\in D,d_{0}<d} U[d]$$ for the smallest element $d_{0}$ of $D$.
Let now $d$ in $D$ and assume that $B_{d^{\prime}}$ has been constructed for all $d^{\prime}$ in $D$ with $d^{\prime}<d$.
Then we set
$$B_{d}:=U^{2}[d]\setminus \Big(\bigcup_{d^{\prime}\in D, d^{\prime}<d} B_{d^{\prime}}\cup \bigcup_{d^{\prime\prime}\in D, d<d^{\prime\prime}}U[d^{\prime\prime}]\Big)\, .$$
Because $U[d]\subseteq B_{d}$ and $B_{d}$ is bounded, the Hilbert space
$H(B_{d})$ is both separable and $\infty$-dimensional. We now construct a unitary $u_{H}:H\to H_{D}\otimes \ell^{2}$ by choosing a unitary 
$H(B_{d})\to (H_{D}\otimes \ell^{2})(\{d\})$ for every $d$ in $D$. The propagation of $u_{H}$ is controlled by $U^{2}$.
\end{proof}

Next we show  that an ample $X$-controlled Hilbert space absorbs every other $X$-controlled Hilbert space.
In view of later applications we will show a slightly stronger result.

Let $X$ be a bornological coarse space and $(H,\phi)$,  $(H^{\prime\prime},\phi^{\prime\prime})$ be $X$-controlled Hilbert spaces. Let $H^{\prime}\subseteq H$ be a closed subspace.

 \begin{lem}\label{rwoijwoifjoeiwewfjewjfo} 
 Assume: \begin{enumerate}
 \item $(H,\phi)$ is ample.
 \item $H^{\prime}$ is locally finite.
 \item $(H^{\prime\prime},\phi^{\prime\prime})$ is determined on points.
 \end{enumerate}
 Then there exists an isometric inclusion $H^{\prime\prime}\to H\ominus H^{\prime}$ of controlled propagation.
 \end{lem}

(The notation $H\ominus H^{\prime}$ stands for the orthogonal complement of $H^{\prime}$ in $H$.)
\begin{proof}
Below we will construct an ample $X$-controlled Hilbert space $(H_{1},\phi_{1})$ and an isometric
inclusion of controlled propagation $H_{1}\to  H\ominus H^{\prime}$. By the arguments in the proof of Lemma \ref{lem:drf8934} we can find an isometric inclusion $H^{\prime\prime}\to H_{1}$ of controlled propagation. Then the composition
$H^{\prime\prime}\to H_{1}\to  H\ominus H^{\prime}$ is an isometric inclusion of  controlled propagation.

If view of the proof of Lemma \ref{lem:drf8934} we can assume that $(H,\phi)$ is of the form
$(H_{D}\otimes \ell^{2},\phi_{D}\otimes \ell^{2})$ for some locally countable $U$-dense subset $D$, where $U$ is some entourage of $X$.
Since $H^{\prime}$ is locally finite, for every $d$ in $D$  the subspace $$H_{1}(\{d\}):=H_{D}\{d\}\otimes \ell^{2}\cap H^{\prime,\perp}$$ is  {$\infty$}-dimensional. We define the closed subspace
$$H_{1}:= {\bigoplus_{d\in D} H_{1}(\{d\})}\subseteq H\ominus H^{\prime}$$
with the control $\phi_{1} := (\phi_{D}\otimes \ell^{2})_{|H_{1}}$. 
The $X$-controlled Hilbert space $(H_{1},\phi_{1})$ is ample and the propagation of {the} inclusion
$H_{1}\to H$ is controlled by $\diag_{X}$.
\end{proof}

\begin{ex}\label{rem:sdb7834rw}
Let $M$ be a Riemannian manifold, and let $E$ be a hermitian vector bundle on $M$. Then we can consider the Hilbert space $H:=L^{2}(M,E)$ of square integrable measurable sections of $E$, where $M$ is equipped with the Riemannian volume measure. This measure is defined on the Borel $\sigma$-algebra of $M$ or some completion of it, but not on $\cP(M)$ except if $M$ is zero-dimensional. 

We consider $M$ as a bornological coarse space with the bornological coarse structure induced by the Riemannian distance. By our convention $C(M)$ is the $C^{*}$-algebra of all bounded $\C$-valued functions on $M$ and
therefore contains non-measurable functions. Therefore it can not act simply as multiplication operators on $H$.
  To get around this problem we can apply the following general procedure.

We call a measure space $(Y,\cR,\mu)$ separable if the metric space $(\cR,d)$ with the metric $d(A,B) := \mu(A \Delta B)$ is separable. Note that $L^2(Y,\mu)$ is a separable Hilbert space if $(Y,\cR,\mu)$ is a separable measure space.

Let $(X,\cC,\cB) $ be a bornological coarse space such that the underlying set $X$ is also equipped with the structure of a  measure space $(X,\cR,\mu)$. We want that the Hilbert space $H := L^2(X,\mu)$ becomes $X$-controlled by equipping it with a representation of $C(X)$ which is derived from the representation by multiplication operators.

Assume the following:
\begin{enumerate}
\item\label{edgeuzuwgfwuefoi0} $\cB \cap \cR$ is cofinal in $\cB$.
\item\label{edgeuzuwgfwuefoi1} A subset $Y$ of $X$ is measurable if the intersection $Y \cap B$ is measurable for every $B$ in  $ \cB \cap \cR$.
\item\label{edgeuzuwgfwuefoi2} For every measurable, bounded subset $B$ of $X$ the measure space $(B,\cR \cap B, \mu_{|\cR \cap B})$ is separable.
\item\label{jdjofewfwefewfewfwf} There exists an entourage $U$ in $\cC$ and a partition $(D_\alpha)_{\alpha \in A}$ of $X$ into non-empty, pairwise disjoint, measurable, $U$-bounded subsets of $X$ with the property that the set $\{\alpha \in A :  D_\alpha \cap B \not= \emptyset\}$ is countable for every bounded subset $B$ of $X$.
\end{enumerate}
For every $\alpha$ in $A$ we    choose a point $d_\alpha$  in $D_\alpha$.

We consider the Hilbert space
$$H:=L^{2}(X,\mu)\, .$$
We furthermore define a map
$$\phi:C(X)\to B(H)\, , \quad 
 f\mapsto \phi(f) := \sum_{\alpha \in A} f(d_\alpha) \chi_{D_\alpha}$$
 where we understand the right-hand side as a multiplication operator by a function which we will also denote by $\phi(f)$ in the following.
 First note that $\phi(f)$ is a measurable function. Indeed, for every $B$  in $\cB \cap \cR$ the restriction $\phi(f)|_B$ is a countable sum of measurable functions (due to {Assumption~\ref{jdjofewfwefewfewfwf}}) and hence measurable. By {Assumption~\ref{edgeuzuwgfwuefoi1}} we then conclude that $\phi(f)$ is measurable.
 
Since $ (D_{\alpha})_{\alpha\in A}$ is a partition we see that $\phi$ is a homomorphism of $C^{*}$-algebras.

If $B$ is a bounded subset of $X$, then $$H(B) = \phi(\chi_{B})H \subseteq \chi_{U[B]}H \subseteq \chi_V H\, .$$ where $V$ is  some measurable subset of $X$ which  contains $U[B]$ and is   bounded. Such a~$V$    exists due to {Assumption~\ref{edgeuzuwgfwuefoi0}} and the compatibility of $\cB$ and $\cC$. Because of {Assumption~\ref{edgeuzuwgfwuefoi2}} we know that $\chi_V H$ is separable, and therefore   $H(B)$ is separable. So $L^2(X,\mu)$ is an $X$-controlled Hilbert space. {By construction it is determined on points.}

We continue  the example of our Riemannian manifold.
We show that we can apply the construction above in order to turn $H=L^{2}(M,E)$ into an $M$-controlled Hilbert space. The measure space $(M,\cR^{Borel},\mu)$ with the Borel $\sigma$-algebra
$\cR^{Borel}$ and the Riemannian volume measure $\mu$ clearly satisfies the Assumptions \ref{edgeuzuwgfwuefoi0}, \ref{edgeuzuwgfwuefoi1} and \ref{edgeuzuwgfwuefoi2}. 

We discuss {Assumption~\ref{jdjofewfwefewfewfwf}}.
 We consider the entourage $U_{1}$ in $\cC$ (see \eqref{ffwefewfe24543234wfefewef}). By Lemma~\ref{lem:iohwenrf} we can find a countable $U_{1}$-separated (Definition \ref{defn:okmnsf}) subset $A$ of $ M$  such that $U_{1}[A]=M$. We fix a bijection of $A$ with $\nat$ which induces an ordering of $A$. 
 Using   induction
we can construct a partition $(D_{\alpha})_{\alpha\in A}$ of $M$ such that $D_{\alpha}\subseteq U_{1}[\alpha]$ for every $\alpha$ in $A$. If $\alpha$ is the smallest element of $A$, then we set $D_{\alpha}:=U_{1}[\alpha]$.
Assume now that $\alpha$ is in $A$ and  $D_{\beta}$ has been defined already for every
$\beta$ in $A$ with $\beta<\alpha$. Then we set $$D_{\alpha}:=U_{1}[\alpha]\setminus \bigcup_{\{\beta\in A :  \beta<\alpha\}} D_{\beta}\, .$$ This partition satisfies {Assumption~\ref{jdjofewfwefewfewfwf}} for the entourage $U_{1}$.
It consists of Borel measurable subsets since $U_{1}[\alpha]\subseteq M$ is open for every $\alpha$ in $A$ and hence measurable, and  the set $D_{\alpha}$ is obtained from the collection of    $U_{1}[\beta]$ for $\beta$ in $A$ using  countably many set-theoretic operations.  Since $A$ was $U_{1}$-separated we have   $\alpha\in D_{\alpha}$ for every $\alpha$ in $A$ and therefore $D_{\alpha}$ is non-empty.
\hB
\end{ex}

\subsection{Roe algebras}\label{ejfoiwfwoefueoi2342343243}

Let $X$ be a bornological coarse space, and let $(H,\phi)$ be an $X$-controlled Hilbert space. In the present section we associate to this data various $C^{*}$-algebras, all of which are  versions of the Roe algebra.
They differ by the way the conditions of controlled propagation and local finiteness are implemented.

We start with the local finiteness conditions.  
  Recall Definition \ref{defn:sfd9823} {for} the notion of a locally finite subspace.

We consider two $X$-controlled Hilbert spaces $(H,\phi)$, $(H^\prime, \phi^\prime)$  and a bounded linear operator  $A \colon (H,\phi) \to (H^\prime, \phi^\prime)$.

\begin{ddd}\label{defn:ertcnvb} 
  $A$ is locally finite\index{locally!finite!operator}\index{operator!locally finite} if there exist orthogonal decompositions $H=H_{0}\oplus H_{0}^{\perp}$ and   $H^{\prime}=H_0^\prime \oplus H_{0}^{\prime,\perp}$ and a bounded operator $A_{0}:H_{0}\to H_{0}^{\prime}$ such that: 
\begin{enumerate} 
\item   $H_{0}$ and $H_{0}^{\prime}$ are locally finite.
\item $$A=\left(\begin{array}{cc}{A_0}&0\\0&0\end{array}\right)\, .$$  
\end{enumerate}
\end{ddd}

\begin{ddd}\label{defn:ertcnvb1}
 $A$ is locally compact\index{locally!compact}\index{operator!locally compact} if for every bounded subset
$B$ of $X$ the products $\phi^{\prime}( B )A$ and $A\phi(  B )$ are compact.
 \end{ddd}
 
 It is clear that a locally finite operator is locally compact.
 
 \begin{rem}
 One could think of defining local finiteness of $A$ similarly as local compactness by requiring that
 $\phi^{\prime}(B)A$ and $A\phi(B)$ are finite-dimensional for every bounded subset of $X$.
 It is not clear that this definition is equivalent to Definition \ref{defn:ertcnvb}.
 Our motivation for using \ref{defn:ertcnvb} is that it works in the Comparison Theorem \ref{fwefiwjfeiooi234234324434}.   \hB
\end{rem}

 We now introduce propagation conditions.

Let $X$ be a bornological coarse space.
  We consider two subsets $B$ and $B^{\prime}$ and an entourage $U$ of $X$.
\begin{ddd}
\index{subset!$U$-separated}\index{$U$-!separated!subsets}
 $B'$ is $U$-separated from $B$ 
 \[U[B]\cap  B^{\prime}=\emptyset\, .\]
 \end{ddd}

Let $(H,\phi)$ and $(H^{\prime},\phi^{\prime})$ be two $X$-controlled Hilbert spaces and $A:H\to H^{\prime}$ be a bounded linear operator, and let $U$ be an entourage of $X$.
  \begin{ddd}\label{bvcbdfg}\mbox{}
  \begin{enumerate}
  \item $A$ is $U$-quasi-local\index{quasi-local!operator}\index{operator!quasi-local} if   for every   $\varepsilon$ in $(0,\infty)$ there exists an integer $n$ in $\nat $   such that for any two  mutually $U^{n}$-separated subsets $B$ and $B^{\prime}$ of $X$     we have
 $\|\phi^{\prime}(B^{\prime})A\phi(B)\|\le \varepsilon$.
 \item $A$ is quasi-local if it is $U$-quasi-local for some entourage $U$ of $X$. \end{enumerate}
 \end{ddd}
  It is clear that an operator  with controlled propagation  quasi-local.
 Indeed, if $A$ is $U$-controlled, the $\phi^{\prime}(B^{\prime})A\phi(B)=0$ provided
 $B'$ is $U$-separated from $B$.

Another possible definition  of the notion of  quasi-locality is the following. We add the word ``weakly'' in order to distinguish it from Definition \ref{bvcbdfg}

\begin{ddd}\label{wetrhiwgfwregwergreg9}
$A$ is  weakly quasi-local if for every $\epsilon$ in $(0,\infty)$  there exists an entourage $U$ of $X$ such that for any two mutually $U$-separated subsets $B$ and $B^\prime$ of $X$ we have $\|\phi^{\prime}(B^{\prime})A\phi(B)\|\le \varepsilon$.
\end{ddd}

It is obvious that quasi-locality implies weak quasi-locality. The following is also clear.

\begin{kor}\label{wtegiuhwefrefwerfwe}
If the coarse structure of $X$ is generated by a single entourage, then the notions of weak quasi-locality and quasi-locality over $X$ coincide.
\end{kor}

 \begin{rem}\label{rem:sf834}
In general,    being weakly quasi-local is  strictly weaker than  being quasi-local  as the following example shows.

We equip the set $X:=\Z\times \{0,1\}$ with the minimal  bornology and the coarse structure $\cC\langle (U_{n})_{n\in \nat}\rangle$, where   for $n$ in $\Z$ the entourage $U_{n}$ is given by $$U_{n}:=\diag_{X}\cup \{((n,0),(n,1)),((n,1),(n,0))\}\, .$$

Let $(H,\phi)$ be the $X$-controlled Hilbert space as in Example \ref{fjwefjewewofewfewfewfewf} with $D=X$.
We further let $A$ in $B(H)$ be the operator
which sends $\delta_{(n,0)}$ to $e^{-|n|}\delta_{(n,1)}$ and is zero otherwise.

We claim that $A$ is weakly quasi-local, but not quasi-local.
Every entourage $U$ of $X$ is contained in a finite union of generators. Hence there exists $n$ in $\Z$
such that $$U\cap (\{n\}\times \{0,1\})^{2}=\{(n,n)\}\times \diag_{\{0,1\}}\, .$$
The points $(n,0)$ and $(n,1)$ are $U^{k}$-separated for all $k$ in $\nat$.
But $$\|\phi(\{(n,1)\})A\phi(\{(n,0)\})\|=e^{-|n|}$$ independently of $k$. So the operator $A$ is not quasi-local in the sense of  Definition 
\ref{bvcbdfg}. 

On the other hand, given $\varepsilon$ in $(0,\infty)$, we can choose the entourage $$U:=\bigcup_{n\in \Z, e^{|-n|}\ge\epsilon} U_{n}$$ of $X$.
If $B$ and $B^{\prime}$ are mutually $U$-separated subsets, then one easily checks that
$$\|\phi^{\prime}(B^{\prime})A\phi(B)\|\le \varepsilon\ .$$ {This shows}
 {that $A$ is  weakly quasi-local.}\hB
\end{rem}

 Let $(H,\phi)$ and $(H^{\prime},\phi^{\prime})$ be two $X$-controlled Hilbert spaces, and $U$ be an entourage of $X$.

 \begin{lem}\label{qwrifuq9erfewdqewdqewdqewd}\mbox{}
 \begin{enumerate}
 \item\label{werigjowergrwegreferfw} The space of locally compact operators is a closed subspace of $B(H,H')$.
\item \label{werigjowergrwegreferfw1} 
The space of $U$-quasi-local operators is a closed subspace of $B(H,H')$.  
\item \label{werigjowergrwegreferfw2} 
 The space of  weakly quasi-local operators is a closed subspace of $B(H,H')$.  
\end{enumerate}
 \end{lem}
 \begin{proof}
 Assertion \ref{werigjowergrwegreferfw} is clear since the compact operators on $H$ or $H'$ are closed in the corresponding spaces of bounded operators.
 
 We now show Assertion \ref{werigjowergrwegreferfw1}.
 Assume that $A:H\to H'$ is in the closure of the $U$-quasi-local operators. 
 We must show that $A$ itself is $U$-quasi-local.
 We fix $\epsilon$ in $(0,\infty)$. Then there exists a   $U$-quasi-local operator $A':H\to H'$ such that $\|A-A'\|\le \epsilon/2$.
 We can now choose  $n$ in $\nat$ sufficiently large  $X$ such that for all pairs $B,B'$ of mutually $U^{n}$-separated subsets of $X$
 we have $\|\phi(B')A'\phi(B)\|\le \epsilon/2$. Then we have 
 $$\|\phi'(B')A\phi(B)\|\le \|\phi'(B')(A-A')\phi(B)\|+\|\phi(B')A'\phi(B)\|\le \epsilon\, .$$
 
 For Assertion \ref{werigjowergrwegreferfw2} we argue similarly.
 Assume that $A:H\to H'$ is in the closure of the weakly quasi-local operators. 
 We must show that $A$ itself is weakly quasi-local.
 We fix $\epsilon$ in $(0,\infty)$. Then there exists a  quasi-local operator $A':H\to H'$ such that $\|A-A'\|\le \epsilon/2$.
 We can now choose  an entourage $U$ of $X$ such that for all pairs $B,B'$ of mutually $U$-separated subsets of $X$
 we have $\|\phi(B')A'\phi(B)\|\le \epsilon/2$. Then we have 
\[\|\phi'(B')A\phi(B)\|\le \|\phi'(B')(A-A')\phi(B)\|+\|\phi(B')A'\phi(B)\|\le \epsilon\, .\qedhere\]
\end{proof}

Our motivation to use the notion of quasi-locality in the sense of  Definition~\ref{bvcbdfg}  in contrast to the weak quasi-locality  as in Definition \ref{wetrhiwgfwregwergreg9} is that it ensures the $u$-continuity of the coarse $K$-theory functor $\KXql$, see Proposition \ref{hewifeiiu3984u2398244241}.

\begin{rem}
Note  that  by Lemma \ref{qwrifuq9erfewdqewdqewdqewd} the weakly  quasi-local operators  form a closed subspace of all operators. 
In contrast, the space of quasi-local operators defined as in Definition \ref{bvcbdfg} is not closed in general.
But by Corollary \ref{wtegiuhwefrefwerfwe}  it is closed if the coarse structure is generated by a single entourage.
 \hB \end{rem}


Let $X$ be a bornological coarse space and $(H,\phi)$ an $X$-controlled Hilbert space.
In the following we define four versions of Roe algebras. They are all closed $C^{*}$-subalgebras of
$B(H)$ which are generated by operators with appropriate local finiteness and propagation conditions.

\begin{ddd}\label{qekdfjqodqqwdqwdqwdqdwq}
\index{Roe!algebra}
\mbox{}
\begin{enumerate}
\item   $\cC^{*}(X,H,\phi)$\index{$\cC^{*}(-)$|see{Roe algebra}}  is  generated by the operators which have controlled propagation and are locally finite.  
\item   $\cClc^\ast(X,H,\phi)$\index{$\cClc^\ast(-)$} is   generated by the operators which have controlled propagation and are locally compact.  
\item    $\cCql^\ast(X,H,\phi)$\index{$\cCql^\ast(-)$} is  generated by the operators which are quasi-local and locally finite.  
\item   $\cClcql^\ast(X,H,\phi)$\index{$\cClcql^\ast(-)$}  is   generated by the operators which are quasi-local and locally compact.  
\end{enumerate}
\end{ddd}

\begin{rem}
If $(H,\phi)$ is determined on points and if the coarse structure of $X$ is generated by a single entourage, then in view of 
 Lemmas \ref{qwrifuq9erfewdqewdqewdqewd} and \ref{webgiojewrgerfewferf}
 we can replace  the phrase ``generated by the''  by ``is the algebra of'' in the case of $\cClcql^\ast(X,H,\phi)$. 
\hB \end{rem}

We have the following obvious inclusions of $C^{*}$-algebras:
\begin{equation}\label{ewfiwuehfeuiiu2432342342234}
\xymatrix{& \cClc^\ast(X,H,\phi) \ar[dr] &\\
\cC^{*}(X,H,\phi) \ar[ur] \ar[dr] & & \cClcql^\ast(X,H,\phi)\\
& \cCql^\ast(X,H,\phi) \ar[ur] &} 
\end{equation}

Let $(H,\phi)$ and $(H^{\prime},\phi^{\prime})$ be two $X$-controlled Hilbert spaces, and let $u\colon H \to H^{\prime}$ be an isometry of controlled propagation.   Then we get   induced injective homomorphisms
\begin{equation}
\label{eq:sdfhiu432}
\cC_?^{*}(X,H,\phi)\to  \cC_?^{*}(X,H^{\prime},\phi^{\prime})
\end{equation}
given in all four cases by the formula $A\mapsto u A u^{*}$. For justification in the case $?\in \{\emptyset,\ql\}$
note that
if  $H''$ is a locally finite subspace of $H$, then $u(H'')$ is a locally finite subspace of~$H'$.

\begin{rem}
Classically the Roe algebra is defined as the closure of the locally compact, controlled propagation operators. The reason why we use the locally finite version as our standard version of Roe algebra (it does not come with any subscript in its notation) is because the functorial definition of coarse $K$-homology in Section~\ref{sec:jkbewrgh} is modeled on the locally finite version and we want the Comparison Theorem  \ref{fwefiwjfeiooi234234324434}.

Proposition~\ref{lem:sdfbi23} shows that in almost all cases of interest (Riemannian manifolds and countable discrete groups endowed with a left-invariant, proper metric) the locally finite version of the Roe algebra coincides with its locally compact version.

But note that we do not know whether the analogous statement of Proposition~\ref{lem:sdfbi23} for the quasi-local case is true, i.e., under which non-trivial  conditions on $X$ the equality   $$\cCql^\ast(X,H,\phi) \stackrel{?}{=} \cClcql^\ast(X,H,\phi)$$
holds true. 
\hB
\end{rem}

\begin{rem}
It is an interesting question under which non-trivial conditions on $X$ we have equalities  $$\cC^{*}(X,H,\phi) = \cCql^{*}(X,H,\phi)\, , \quad \cClc^{*}(X,H,\phi) = \cClcql^{*}(X,H,\phi)\, .$$ As far as the authors know the only non-trivial space (i.e., not a bounded one) for which this was known is $\mathbb{R}^n$ \cite[Prop.~1.1]{lange_rabinovich}. The conjecture was that these equalities hold true for any space $X$ of finite asymptotic dimension; see {Roe} \cite[Rem.~on Page 20]{roe_index_coarse} where this is claimed (but no proof is given). Meanwhile\footnote{after the appearance of the first version of the present paper}, \v{S}pakula and Tikuisis \cite{st} have shown the {second} equality  under the even  weaker condition {that $X$ has} finite straight {decomposition complexity.}
\hB
\end{rem}

Let $X$ be a bornological coarse space, and let $(H,\phi)$ be an $X$-controlled Hilbert space.
If $(H,\phi)$ is locally finite (Definition \ref{ewoifjweoifiufu89234234324}), then it is obvious that
$$\cC^{*}(X,H,\phi) = \cClc^\ast(X,H,\phi)\, .$$
The following proposition  gives a generalization of this equality to cases where $(H,\phi)$ is not necessarily locally finite.

Let $X$ be a bornological coarse space.

\begin{ddd}\label{weofweoifu9824234343224}
$X$ is  separable\index{space!separable}\index{separable}\index{bornological coarse space!separable} if it admits an entourage $U$ for which there exists   a countable $U$-dense (Definition \ref{defn:okmnsf}) subset.
\end{ddd}

\begin{ex}
If $(X,d)$ is a separable metric space, then  the bornological coarse space $(X,\cC_{d},\cB)$ is separable, where $\cC_{d}$ is the coarse structure \eqref{hvhdgqwuztd8t1z87231231233} induced by the metric and $\cB$ is any compatible bornology.
\hB
\end{ex}

Let $X$ be a bornological coarse space, and let $(H,\phi)$ be an $X$-controlled Hilbert space.

\begin{prop}\label{lem:sdfbi23}Assume:
\begin{enumerate} 
\item 
 $X$ is separable and locally finite (Definition~\ref{defn:okmnsf}).
 \item  $(H,\phi)$ is determined on points (Definition \ref{rgioerjog34t4t34t34t}).
 \end{enumerate}  Then $\cC^{*}(X,H,\phi) = \cClc^\ast(X,H,\phi)$.
\end{prop}

\begin{proof}
Since $X$ is separable we can choose an entourage  $U^{\prime}$ and  a countable subset $D^{\prime}$    of $X$    such that $D^{\prime}$ is $U^{\prime}$-dense. Since $X$ is locally finite we can enlarge $U^{\prime}$ and in addition assume that every $U^{\prime}$-separated subset of $X$ is locally finite. Let $D$ be a maximal $U^{\prime}$-separated subset of $D^{\prime}$ and set $U:=U^{\prime,2}$. Then $D$ is $U$-dense and locally finite.
For simplicity of notation we only consider the case that $D$ is infinite.
Let $\nat\ni n\mapsto d_{n}\in D$ be a bijection. We define inductively $B_{0}:=U[d_{0}]$ and
$$B_{n}:=U[d_{n}]\setminus \bigcup_{m<n} B_{m}\, .$$
Then   $(B_{n})_{n\in \nat}$ is a  partition of $X$ into  $U$-bounded  subsets. After deleting empty members and renumbering we can assume that $B_{n}$ is not empty for every $n$ in $\nat$.

We consider an operator  $A$ in $\cClc^{*}(X,H,\phi)$. We will show that for every given $\epsilon$ in $(0,\infty)$ there exists a locally finite  
subspace $H^{\prime}$ of $  H$  
and an operator $A^{\prime}:H^{\prime}\to H^{\prime}$ such that
\begin{enumerate}
\item $\|A- A^{\prime}  \|\le \epsilon$ (we {implicitly} extend $A^{\prime}$ by zero on $H^{\prime,\perp}$).
\item $A^{\prime}$ has controlled propagation.
\end{enumerate}
Note that we verify a condition that is stronger than Definition \ref{defn:ertcnvb} since we take the same subspace as domain and target of $A^{\prime}$.

As a first step we choose a locally compact operator $A_{1}$ of controlled propagation such that $\|A-A_{1}\|\le \epsilon/2$. 
For every $n$ in $\nat$ we define the subset $$B_{n}^{\prime}:=\supp(A_{1})[B_{n}]\cup B_{n}\, .$$ We then have
$A_{1}B_{n}=B_{n}^{\prime}A_{1}B_{n}$ for every $n$ in $\nat$, where for a subset $B$ of $X$ we abbreviate the operator  $\phi(B)$ by $B$. Since $A_{1}$ is locally compact
 we can now find {for every $n$ in $\nat$} a finite-dimensional projection $P_{n}$ on $H(B_{n}^{\prime})$ such that
$$\|P_{n}B_{n}^{\prime}A_{1}B_{n}P_{n}-A_{1}B_{n}\|\le \epsilon \cdot 2^{-n-2}\, .$$

Then we define $H^{\prime}$ to be the subspace of $H$ generated by the images of $P_{n}$.
Let $H_{n}$    be the subspace of  $H^{\prime}$ generated by the images of $P_{1},\dots,P_{n}$.
We define $Q_{n}$ to be the orthogonal projection onto the subspace $H_{n}\ominus H_{n-1}$.
For every $n$ in $\nat$ we choose a point $b_{n}$ in $B_{n}$. Then we define the control on $H^{\prime}$ by
$\phi^{\prime}(f):=\sum_{n\in \nat} f(b_{n}) Q_{n}$.
With these choices  the inclusion $ (H^{\prime}, \phi^\prime) \to (H,\phi)$ has controlled propagation. 
Moreover, $(H^{\prime}, \phi^\prime)$ is a locally finite $X$-controlled Hilbert space. Indeed, if $B$ is a bounded subset of $X$, then $U^{-1}[B]\cap D$ is finite since $D$ is locally finite. Hence the set $\{n\in \nat\:|\: B\cap B_{n}\not=\emptyset\}$ is finite. Similarly, $\{n\in \nat\:|\: B\cap B^{\prime}_{n}\not=\emptyset\}$ is finite.
  This shows that $H^{\prime}$   is a locally finite subspace of $ H$.

  Furthermore,
$$\Big\| \sum_{n\in \nat}   P_{n}B_{n}^{\prime}A_{1}B_{n}P_{n}-A_{1} \Big\|\le \epsilon/2\, ,$$
where the sum  converges {in the strong operator topology}.
The operator $$A^{\prime}:=\sum_{n\in \nat}   P_{n}B_{n}^{\prime}A_{1}B_{n}P_{n}$$ is an operator on $H^{\prime}$
which has controlled propagation (when considered as an operator on  $(H,\phi)$ after extension by zero).

Finally, we observe that $$\|A^{\prime}-A\|\le \epsilon\, ,$$ 
which completes the proof.
\end{proof}






Let $X$ be a bornological coarse space, and let $(H,\phi)$, $(H^{\prime},\phi^{\prime})$ be two $X$-controlled Hilbert spaces.
\begin{lem}
 If both $(H,\phi)$ and $(H^{\prime},\phi^{\prime})$ are ample (see Definition \ref{fjwefewiojoi2jroi23jr23r23r23r}), then we have isomorphisms of $C^{*}$-algebras {(for all four possible choices for ? from Definition~\ref{qekdfjqodqqwdqwdqwdqdwq})}
\[\cC_?^{*}(X,H,\phi)\cong \cC_?^{*}(X,H^{\prime},\phi^{\prime})\]
{given by conjugation with a unitary operator $H \to H'$ of controlled propagation.}
\end{lem}
 
\begin{proof}
This follows from Lemma~\ref{lem:drf8934} together with \eqref{eq:sdfhiu432}.
\end{proof}

Finally, let $f:X\to X^{\prime}$ be a morphism between bornological coarse spaces, and $(H,\phi)$ be an $X$-controlled
Hilbert space. Then we can  form the $X^{\prime}$-controlled
Hilbert space $f_{*}(H,\phi):=(H,\phi\circ f^{*})$ (see Definition
 \ref{ewljfewilewfiouwefwefew}). The Roe algebras for $(H,\phi)$ and $f_{*}(H,\phi)$ are closed subalgebras of $B(H)$.

\begin{lem}\label{elkjqdqoiuo3432424324}
We have inclusions
\[\cC_?^{*}(X,H,\phi)\subseteq  \cC_?^{*}(X,H,\phi\circ f^{*})\]
{for all four possible choices for ? from Definition~\ref{qekdfjqodqqwdqwdqwdqdwq}.}
\end{lem}

\begin{proof}
We just show that the local finiteness and propagation conditions imposed on $A$ in $B(H)$ using the $X$-control $\phi$ are stronger than the ones imposed with the $X^{\prime}$-control $\phi\circ f^{*}$. In greater detail one argues as follows:

If $A$  is locally compact with respect to the $X$-control $\phi$, then it is locally compact with respect to the $X^{\prime}$-control $\phi\circ f^{*}$. This follows from the properness of $f$.

If $H^{\prime} $ is a locally finite subspace of $H$ with respect to the $X$-control $\phi$, then it is so
with respect to the $X^{\prime}$-control $\phi\circ f^{*}$. Indeed, if $\phi^{\prime}$ is an $X$-control of $H^{\prime}$ such that $(H^{\prime},\phi^{\prime})$ is locally finite and $H^{\prime}\to H$ has finite propagation, then $\phi^{\prime}\circ f^{*}$ is an $X^{\prime}$-control for $H^{\prime}$ which can be used to recognize $H^{\prime}  $ as a locally finite subspace of $H$ with respect to the $X^{\prime}$-control $\phi\circ f^{*}$. This reasoning implies that if $A$ is locally finite on $(H,\phi)$, then it is so on $f_{*}(H,\phi)$.

For the propagation conditions we argue similarly. If $A$ is of $U$-controlled propagation (or $U$-quasi-local) with respect to the $X$-control $\phi$, then it is of  $(f\times f)(U)$-controlled propagation (or $(f\times f)(U)$-quasi-local) with respect to the $X^{\prime}$-control $\phi\circ f^{*}$.
\end{proof}

\begin{rem}\label{rem43wer45er}
Let $(H,\phi)$ be an $X$-controlled Hilbert space. By definition this means that $\phi$ is a unital $^\ast$-representation of all bounded functions on $X$. But usually, e.g., if $X$ is a Riemannian manifold with a Hermitian vector bundle $E$,  we want to use the Hilbert space  $L^2(E,X)$ and the representation should be the one given by multiplication operators. But this representation is  not defined on all bounded functions on $X$, but only on the measurable ones. To get around this we can construct a new representation as in Example~\ref{rem:sdb7834rw}. But now we have to argue why the Roe algebra we get by using this new representation coincides with the Roe algebra that we get from using the usual representation by multiplication operators. We will explain this now in a slightly more general context than Riemannian manifolds but focus on ample Hilbert spaces.

Let $(Y,d)$ be a separable proper metric space,  and let $Y_{d}$ be the associated bornological coarse space. 
We let $Y_{t}$ denote the underlying locally compact topological space of $Y$ which we consider as a topological bornological   space, see Subsection \ref{wergoiwergwergerfwfrwfwefwerf}. Let furthermore $(H,\phi)$ be an ample $Y_{{d}}$-controlled Hilbert space (in the sense of Definition \ref{fjwefewiojoi2jroi23jr23r23r23r})  and $\rho:C_{0}(Y_{t})\to B(H)$ be a $*$-representation (see Subsection \ref{wklgowergerfrfwferf} for $C_{0}$)
such that the following conditions are satisfied:
\begin{enumerate}
\item If $\rho(f)$ is compact, then $f=0$.
\item There exists an entourage $U$ of $Y_{d}$ such that
$$\phi(U[\supp(f)])\rho(f)=\rho(f)\phi(U[\supp(f)]) =\rho(f)\, .$$
\item There exists an  entourage $U$ of $Y_{d}$ such that for every bounded subset $B$ of $Y_{d}$ there exists
$f$ in $C_{0}(Y_{t})$ with $\supp(f)\subseteq U[B]$ such that $\rho(f)\phi(B)={\phi(B)\rho(f)} =\phi(B)$.
\end{enumerate}
Since $Y_{d}$ has a countable exhaustion by bounded subsets the Hilbert space $H$ is separable.
Note that the third condition implies that $\rho$ is non-degenerate. Together with
the first condition it  ensures that $H$ is an ample or  standard Hilbert space in the  sense of Roe \cite[Sec.~4.1]{roe_coarse_cohomology} \cite[Sec.~4.4]{roe_lectures_coarse_geometry} and Higson--Roe \cite{hr}. The second and third conditions together imply that the local finiteness, local compactness, finite propagation and quasi-locallity conditions imposed using $\rho$ or $\phi$ are equivalent.

Returning to our example of a Riemannian manifold $X$ with a Hermitian vector bundle $E$, we denote by $H$ the Hilbert space $L^2(E,X)$ and by $\rho$ the representation of $C_0(X)$ by multiplication operators. Using the construction from Example~\ref{rem:sdb7834rw} we get a representation $\phi$ of all bounded functions on $X$. If $X$ has no zero-dimensional component, then   $(H,\phi)$  is an  ample $X$-controlled Hilbert space.  In this case the three conditions above are satisfied and   we have the equalities $\cC_?^{*}(X,H,\phi) = \cC_?^{*}(X,H,\rho)$ as desired.
\hB
\end{rem}

\subsection{\texorpdfstring{$K$}{K}-theory of \texorpdfstring{$C^{\ast}$}{C*}-algebras}\label{wifjewiof2323443534}

Let $\Calg$ denote the category of not necessarily unital $C^{*}$- algebras and not necessarily unit preserving homomorphisms. The construction of coarse $K$-homology theory  uses a $K$-theory functor
\begin{equation}\label{rojgrogergp9345345}
\Kast:\Calg\to \Sp
\end{equation}
for $C^{*}$-algebras (not to be confused with the $K$-cohomology functor introduced in Definition \ref{eughierwgvfdsdfvdfvsdfvfsdvsdfvswtf}).
In the following we list the properties of this functor which we will use in the constructions later.
The functor $\Kast$ has the following properties.
\begin{enumerate}
\item \label{ergoijuefoiqewfqwfewfqwefqwf} $\Kast$ represents $C^{*}$-algebra $K$-theory groups:
For every $C^{*}$-algebra $A$ the spectrum $\Kast(A)$ represents the topological $K$-theory of $A$, i.e.,
the homotopy groups $\pi_{*}\Kast(A)$  are naturally isomorphic to the $K$-theory groups~$K_{*}(A)$. References for the group-valued  $K$-theory  {of} $C^{*}$-algebras are Blackadar \cite{blackadar} and Higson--Roe \cite{higson_roe}.
\item\label{wergijowergergrgregwefewrfwerf} We have $\Kast(\C)\simeq KU$.
\item\label{wtgwergfervfdsv} $\Kast$ preserves filtered colimits: Given a filtered family $(A_{i})_{i\in I}$  in $\Calg$  we can form the colimit $A:=\colim_{i\in I}A_{i}$ in  $\Calg$. The natural morphism
$$\colim_{i\in I} \Kast(A_{i})\to \Kast(A)$$
is an equivalence.

For example, in the case of a family subalgebras of $B(H)$ the colimit  $A$ is the closure of the union $\bigcup_{i\in I}A_{i}$.  

\item \label{rgpojweogweggffrw} $\Kast$ preserves sums: If $(A_{i})_{i\in I}$ is a family of $C^{*}$-algebras, then
the canonical morphism $$\bigoplus_{i\in I} \Kast(A_{i})\to \Kast(\bigoplus_{i\in I}A_{i})$$ is an equivalence.
\item \label{wtgwergfervfdsv1}$\Kast$ is exact:  We have $\Kast(0)\simeq 0$ and $\Kast$ sends   exact sequences of $C^{*}$-algebras 
$0\to I\to A\to Q\to 0$ to cocartesian squares
\begin{equation}\label{fjhweuif8723452534}
\xymatrix{\Kast(I)\ar[r]\ar[d]&{\Kast(A)}\ar[d]\\0\ar[r]&\Kast(Q)}\end{equation} in $\Sp$.
Equivalently, we have a fibre sequence 
$$\cdots \to\Kast(I)\to \Kast(A)\to \Kast(Q)\to \Sigma \Kast(I)\to \cdots \, .$$
\item \label{wegoijweroigjergwerg} $\Kast$ is stable: For every $A$ in $\Calg$ the left upper corner inclusion $A\to A\otimes  \mathbb{K}(\ell^{2})$
induces an equivalence $\Kast(A)\to \Kast(A\otimes  \mathbb{K}(\ell^{2}))$. Here $ \mathbb{K}(\ell^{2})$ denotes the compact operators on the Hilbert space $\ell^{2}$.
  \item \label{wroigjiowertgerwgrewgwreg} $\Kast$ is homotopy invariant:
  For every $A$ in $\Calg$ the inclusion $A\to C([0,1],A)$ as constant functions induces an equivalence
  $\Kast(A)\to \Kast( C([0,1],A))$.
\item $\Kast$ is Bott periodic: For every $A$ in $\Calg$ there exists a natural equivalence $\Kast(A)\to \Sigma^{2}\Kast(A)$.
In order to check that a morphism   $\Kast(A)\to \Kast(B)$ is an equivalence it suffices to show that
$\pi_{i}\Kast(A)\to \pi_{i}\Kast(B)$ is an isomorphism for $i=0,1$.
\item\label{rgoijwergioregrerfwref} Assume that $h\colon A\to B$ is a  homomorphism between $C^{*}$-algebras, and that $u$ is a partial isometry in the multiplier algebra of
$B$ such that $hu^{*}u=h$. Then $h':=uhu^{*}\colon A\to B$ is a homomorphism of $C^{*}$-algebras, and we have an equality of  induced maps
$\pi_{*}\Kast(h)=\pi_{*}\Kast(h')\colon \pi_{*}\Kast(A)\to \pi_{*}\Kast(B)$.
\end{enumerate}

%
%
%
%
%
%

A construction of a $K$-theory functor \eqref{rojgrogergp9345345} was  {provided} by Joachim \cite[Def.~4.9]{joachimcat}.
He verifies \cite[Thm.~4.10]{joachimcat} that his  spectrum indeed represents $C^{*}$-algebra $K$-theory.
In the following we argue that this functor has the properties listed above, which are not stated explicitly in  this reference.

 Concerning homotopy invariance note {that} every spectrum-valued functor representing $C^{*}$-algebra $K$-theory has this property.
 This follows from the fact that the group-valued functor sends the  homomorphism $A\to C([0,1],A)$ 
  to an isomorphism. 
 
Similarly, continuity with respect to filtered colimits is implied by the corresponding known property of the group-valued $K$-theory functor.
 
 Finally, the group-valued $K$-theory functor sends
 a sequence \eqref{fjhweuif8723452534} to a long exact sequence of $K$-theory groups  {(which will be a six-term sequence due to Bott periodicity).}
 So once we know that the composition $K(I)\to K(A)\to K(A/I)$ is the zero map in spectra we can produce a morphism to
 $K(I)\to \Fib(K(A)\to K(A/I))$ and conclude, using the Five Lemma, that it is an equivalence.
 An inspection of  \cite[Def.~4.9]{joachimcat} shows that the  composition $K(I)\to K(A)\to K(A/I)$   is level-wise the constant map to the base point.

The Bott periodicity is actually a formal consequence of the other properties. 
The periodicity morphism of spectra can be constructed from the long exact sequence of a so-called Toeplitz extension
and an equivalence
$\Sigma \Kast(C_{0}(\R,A))\simeq \Kast(A)$.

An alternative option is to  construct
$\Kast$ as a left Kan-extension 
\begin{equation}\label{erfiofjweorijfwerfwerewdwdwed} \xymatrix{
\Calg^{\sep}\ar[rr]^-{\map_{\bK\bK}(\C,-)}\ar[d]^{\incl}&&\Sp\\
\Calg\ar[urr]_-{\Kast}&&}\, ,
\end{equation}
where $\bK\bK$ is  the stable $\infty$-category discussed in Subsection \ref{fweoifjewio982u4r53453453454}.
The point wise formula for the evaluation of this left Kan extension yields
$$\Kast(A):=\colim_{B\subseteq A} \ \map_{\bK\bK}(\C,B)\, ,$$
where    the colimit runs over all separable subalgebras {$B$} of $A$. 
We will use this description later to define the assembly map.
All properties except \ref{wtgwergfervfdsv} of $\Kast$ are then consequences of the properties of $\bK\bK$ and the functor $\Calg^{\sep}\to \bK\bK$ listed in Subsection \ref{fweoifjewio982u4r53453453454}.

\begin{rem}
The Property \ref{rgoijwergioregrerfwref} of $\Kast$ is  not so easy to find in the literature.  Hence it is reasonable to   provide the argument. 

Since $u$ is a partial isometry we have two projections $p:=u^{*}u$ and $q:=uu^{*}$ in the multiplier algebra of $B$. 
We consider the $C^{*}$-algebra $C:=pBp\oplus qBq$ and the homomorphism
$d:C\to M_{2}(B)$ given by $(b,b')\mapsto \diag(b,b')$. 
The assumption on $h$ is equivalent to the fact that $h$ has a factorization 
$$h:A\stackrel{\tilde h}{\to} pBp\to B\ ,$$ where the second homomorphism is the embedding. 
We then have an equality 
$$i_{11}\circ h=d\circ \tilde h: A\to M_{2}(B)\ , $$ where  $i_{11}$ is the   left-upper corner inclusion.   The matrix
$$v:=\left(\begin{array}{cc} 0&u^{*}\\u&0\end{array}\right)$$
can be considered as a unitary in the multiplier algebra of $C$.
We now calculate that
$$d\circ (v\tilde hv^{*})=i_{22}\circ h'\ ,$$
where $i_{22}$ is the lower right corner inclusion and $h':=uhu^{*}$. 
By \cite[Lem. 4.6.1]{hr}  we have the equality 
$$\pi_{*}\Kast(v\tilde hv^{*})=\pi_{*}\Kast(\tilde h)\ .$$
This implies by functoriality of $\pi_{*}\Kast$ that 
$\pi_{*}\Kast(i_{11}\circ h)=\pi_{*}\Kast(i_{22}\circ h')$. 
Since we also have the equality of isomorphisms $\pi_{*}\Kast(i_{11})=\pi_{*}\Kast(i_{22})$ 
we conclude that $\pi_{*}\Kast(  h)=\pi_{*}\Kast(  h')$. 
\hB
\end{rem}

\subsection{\texorpdfstring{$C^\ast$}{C*}-categories and their \texorpdfstring{$K$}{K}-theory}
\label{sec:dfs8934}

We have seen in Section \ref{ejfoiwfwoefueoi2342343243} that we can associate a Roe algebra (with various decorations) to a bornological coarse space $X$ equipped with an $X$-controlled Hilbert space. Morally, coarse $K$-homology of $X$ should be coarse $K$-theory of this Roe algebra. This definition obviously depends on the choice of the $X$-controlled Hilbert space.
In good cases $X$ admits an ample $X$-controlled Hilbert space. Since any two choices of an ample Hilbert space are isomorphic (Lemma \ref{lem:drf8934}), the Roe algebra is independent of the choice of the ample $X$-controlled Hilbert space up to isomorphism.

It is not good enough to associate to $X$ the $K$-theory spectrum of the Roe algebra for some choice of an  $X$-controlled ample Hilbert space. First of all this would not give a functor. Moreover, we want the coarse $K$-theory to be defined for all bornological coarse spaces, also for those which do not admit an $X$-controlled  ample Hilbert space (Lemma~\ref{fkwjlkwejlewfewfewfewf3242344}). 

\begin{rem}
The {naive} way out would be to define
the coarse $K$-theory spectrum of $X$ as the colimit over all choices of $X$-controlled Hilbert spaces of the $K$-theory spectra of the associated Roe algebras. The problem is that the index category of this colimit is not a filtered poset. It would give a wrong homotopy type of the coarse $K$-theory spectrum of~$X$ even in the presence of ample $X$-controlled Hilbert spaces.

One {could try the following solution:  one could consider} the category of $X$-controlled Hilbert spaces as a topological category, interpret the functor from $X$-controlled Hilbert spaces to spectra as a topologically enriched functor and take the colimit in this realm. 

{One problem with this approach is the following: in order to show functoriality, given a morphism $X \to Y$ and $X$-controlled, resp.~Y-controlled Hilbert spaces $(H,\phi)$ and $(H^\prime, \phi^\prime)$, we must choose an isometry $V\colon H \to H^\prime$ with a certain propagation condition on its support (cf.~Higson--Roe--Yu \cite[Sec.~4]{higson_roe_yu}). Then one uses conjugation by $V$ in order  to construct a map from the Roe algebra $\cC^{*}(X,H,\phi)$ to $\cC^{*}(Y,H^\prime,\phi^\prime)$. Now two different choices of isometries $V$ and $V^\prime$ are related by an inner conjugation. Therefore, on the level on $K$-theory groups of the Roe algebras conjugation by $V$ and $V^\prime$ induce the same map. But inner conjugations do not necessarily act trivially on the $K$-theory spectra. When we pursue this route further we would need the statement that the space of  such isometries $V$ is contractible. But up to now we were only able to show that this space is connected (see Theorem~\ref{fwefiwjfeiooi234234324434e}).}

In the present paper we pursue {therefore the following alternative idea.} 
Morally we perform the colimit over the Hilbert spaces on the level of $C^{*}$-algebras before going to spectra. Technically we consider the $X$-controlled Hilbert spaces as objects of a $C^{*}$-category whose morphism spaces are 
the analogues of the Roe algebras. The idea of using $C^{*}$-categories in order to produce a functorial $K$-theory spectrum is not new and has been previously  employed in different situations {, e.g., by Davis--L{\"u}ck \cite{davis_lueck}, Hambleton--Pedersen \cite{hamped}, and Joachim \cite{joachimcat}.}
\hB
\end{rem}

In the present Section we review $C^{*}$-categories  a their $K$-theory. The details of the application to coarse $K$-homology will be given in Section \ref{sec:jkbewrgh}.

\begin{rem}\label{ewtgwergregregrgwegergw}
Motivated by the usage of $C^{*}$-categories in the present section   several follow-up papers dealing with
aspects of the theory of $C^{*}$-categories appeared after  
 the first version of the book was finished. 
\begin{enumerate}
\item \cite{startcats}:  homotopy theory for unital $C^{*}$-categories 
\item \cite{crosscat}:  categorial properties of the category of non-unital $C^{*}$-categories and crossed products 
\item \cite{cank}: infinite orthogonal sums in $C^{*}$-categories and properties of the $K$-theory for $C^{*}$-categories, in particular that it preserves  infinite products
\item \cite{Bunke:ad}: coarse $K$-homology with coefficients in a $C^{*}$-category,  a generalization of the theory of {the} present section to the equivariant case and with the category of Hilbert spaces replaced by a general $C^{*}$-category\hB
\end{enumerate}
\end{rem}

In the sections below we will consider $C^{*}$-algebras and spectra in the small universe (as opposed to the very small universe). In this case we add a superscript $(-)^{\la}$ to the notation for the categories, but keep the notation for the functors.

\subsubsection{Definition of \texorpdfstring{$C^{*}$}{Cstar}-categories}

In the present paper we will consider small $C^{*}$-categories. The  category $\Ccat$\index{$\Ccat$} is a large category.
Before we give the {details of the} definition note that our basic example is the small category of very small Hilbert spaces and bounded operators.

\begin{ddd}\label{qwfijoafvffdfadf}
A $C^{*}$-category\index{$C^{*}$-category} $\bC$ is given by the following structures:
\begin{enumerate}
\item a (possibly non-unital) category $\bC$,
\item for every two objects $c,c'$ in $\bC$ a complex Banach space structure on the set $\Hom_{\bC}(c,c')$,
\item an involution $*:\bC^{\op}\to \bC$.
\end{enumerate}
These structures must satisfy the following conditions:
\begin{enumerate}
\item The involution fixes objects and acts isometrically and anti-linearly on the morphism spaces.
\item The $C^{*}$-condition is satisfied: For every morphism $A:c\to c'$ in $\bC$ we have
\[\|A^{*}A\|_{\Hom_{\bC}(c,c)}=\|A\|_{\Hom_{\bC}(c,c')}^{2}\,.\]
\item For every morphism $A:c\to c'$ in $\bC$ the operator $A^{*}A$ is a non-negative operator in the $C^{*}$-algebra $\End_{\bC}(c)$.
\end{enumerate}
\end{ddd}

Note that we will usually just write $\|-\|$ instead of $\|-\|_{\Hom_{\bC}(c,c')}$.
 
\begin{ddd}\label{weogijoerwgrefrefwerfrewfw}
 A functor between $C^{*}$-categories is a (not necessarily unit-preserving) functor between $\C$-linear categories which is compatible with the $*$-operation.
 \end{ddd}

\begin{rem}
In view of Definition  \ref{qwfijoafvffdfadf} it would be very natural to require a condition of compatibility of a functor with the norms on the morphism spaces. But it turns out that that a functor in the sense of Definition \ref{weogijoerwgrefrefwerfrewfw} is automatically norm decreasing \cite[Prop.~2.14]{mitchc}. Hence such a condition is redundant. 

We refer to \cite{crosscat} for an alternative, but equivalent definition of a $C^{*}$-category as a $\C$-linear $*$-category with additional properties (in contrast to additional Banach space structures).
\hB
\end{rem}

\begin{ddd}\mbox{}
\begin{enumerate}\item 
A $C^{*}$-category is  unital if every object has an identity morphism.
\item A functor between unital $C^{*}$-categories is unital if it preserves
identity morphisms. \end{enumerate}
\end{ddd}
 
\begin{rem}
In the   papers listed in Remark \ref{ewtgwergregregrgwegergw} the notation $\Ccat$ is reserved for the category of unital $C^{*}$-categories and unital functors. The notation for the non-unital case is $\Ccat^{\mathrm{nu}}$. 
In the present book we use the notation  $\Ccat$   for the non-unital case and say explicitly when we require a $C^{*}$-category or a functor to be unital. 
Since we only consider the large category of small $C^{*}$-categories we {do} not add the superscript ``$\la$''.
\hB
\end{rem}

A good reference for $C^{*}$-categories is \cite{mitchc}.
   We furthermore refer to Davis--L{\"u}ck \cite[Sec.~2]{davis_lueck} and Joachim \cite[Sec.~2]{joachimcat}.      

\subsubsection{From \texorpdfstring{$C^{*}$}{Cstar}-categories to \texorpdfstring{$C^{*}$}{Cstar}-algebras and \texorpdfstring{$K$}{K}-theory}
 
We let  $\Calg^{\la}$\index{$\Calg^{\la}$} denote the large category of (possibly non-unital)  small $C^{*}$-algebras and (not necessarily unit preserving) homomorphisms. Note that in contrast $\Calg$ is the small full subcategory of {$\Calg^{\la}$}  of very small $C^{*}$-algebras.

A   $C^{*}$-algebra can be considered as a $C^{*}$-category with a single object. This provides
an inclusion $\Calg^{\la}\to \Ccat$. It fits into an adjunction \begin{equation}\label{hbhjeuzefwefewfwe}
A^{f}:\Ccat\leftrightarrows \Calg^{\la}:\mathit{\incl}\, .
\end{equation}  
The functor $A^{f}$\index{$A^{f}$}  has first been introduced by Joachim \cite{joachimcat}.
Even if $\bC$ is a unital $C^{*}$-category the $C^{*}$-algebra $A^{f}(\bC)$ is non-unital in general. This happens, e.g.\ in the case when  $\bC$ has infinitly many objects.
Because of this fact we need the context of non-unital $C^{*}$-categories in order to present $A^{f}$ as a left-adjoint.

\begin{rem}
 In the following we describe  $A^{f}$ explicitly. The {functor $A^{f}$}  associates to a $C^{*}$-category~$\bC$ a $C^{*}$-algebra $A^{f}(\bC)$. The unit of the adjunction is a functor
$\bC\to A^{f}(\bC)$ between $C^{*}$-categories. The superscript $f$ stands for free which distinguishes the algebra from another $C^{*}$-algebra $A(\bC)$ associated to $\bC$ which is simpler, but not functorial in $\bC$. We will discuss $A(\bC)$ later.

  We first construct the free $*$-algebra $\tilde A^{f}(\bC)$ generated by the morphisms of $\bC$.
  If $A$ is a morphism in $\bC$, then we let $(A)$ in $\tilde A^{f}(\bC)$ denote the corresponding generator.
    We then form the quotient by an equivalence relation which reflects the sum, the composition, and the $*$-operator already present in $\bC$.       \begin{enumerate}
    \item \label{ifewfewifoewfuiou1} If $c,c^{\prime}$ are in $\Ob(\bC)$, $A,B$ are in $\Hom_{\bC}(c,c^{\prime})$, and $\lambda,\mu$ are in $\C$, then we have the relation
    $\lambda(A)+\mu(B) \sim (\lambda A+\mu B)$.
    \item \label{ifewfewifoewfuiou2} If $c,c^{\prime},c^{\prime\prime}$ are in $\Ob(\bC)$, $A$ is in  $\Hom_{\bC}(c,c^{\prime})$, and $B$ is in  $\Hom_{\bC}(c^{\prime},c^{\prime\prime})$, then we have the relation
    $(B)(A) \sim (B\circ A)$.
    \item \label{ifewfewifoewfuiou3} For every morphism $A$  in $\bC$ we have the relation $(A^{*}) \sim (A)^{*}$.
    \end{enumerate}
    We denote the quotient $*$-algebra of  $\tilde A^{f}(\bC)$ by this set of relations
    by $A^{f}(\bC)_{0}$. 
We equip $A^{f}(\bC)_{0}$ with the maximal $C^{*}$-norm and let $A^{f}(\bC) $ be the closure. In general, this $C^{*}$-algebra is only small (as opposed to very small) since $\bC$ is small. This is the reason for considering 
large $C^{*}$-algebras. 
 The natural  functor \begin{equation}\label{jkhwkejf8929348r23r23r3}
\bC\to A^{f}(\bC)
\end{equation} sends every object of $\bC$ to the single object of $A^{f}(\bC)$
and a morphism $A$ of $\bC$ to the image of the generator $(A)$ in $A^{f}(\bC)$ under the quotient map and the natural map into the completion.

It has been observed by {Joachim}  \cite{joachimcat} that this construction has the following universal property:
For every $C^{*}$-algebra $\bB$ (considered as a $C^{*}$-category with a single object) a functor $\bC\to \bB$ in $\Ccat$ has  a unique extension to a homomorphism $A^{f}(\bC)\to \bB$ of $C^{*}$-algebras.

Let now $\bC\to \bC^{\prime}$ be a functor between $C^{*}$-categories. Then we get a functor
$\bC\to A^{f}(\bC^{\prime})$ by composition with \eqref{jkhwkejf8929348r23r23r3}. It now follows from the universal property of $A^{f}(\bC)$ that this functor extends uniquely  to a homomorphism of $C^{*}$-algebras
$$A^{f}(\bC)\to A^{f}(\bC^{\prime})\, .$$
This finishes the construction of the functor $A^{f}$ on functors between $C^{*}$-categories.
The universal property of $A^{f}$ is equivalent to  the adjunction \eqref{hbhjeuzefwefewfwe}.
\hB
\end{rem}

To a  $C^\ast$-category $\bC$ we can also associate a simpler $C^{*}$-algebra $A(\bC)$. In order to construct this $C^{*}$-algebra  we start with the description of a $*$-algebra $A(\bC)_{0}$. The underlying $\C$-vector space of  $A(\bC)_0$ is the algebraic direct sum
\begin{equation}\label{efefwefeefefefewfef}
A(\bC)_0 := \bigoplus_{c,c^{\prime} \in \Ob(\bC)} \Hom_{\bC}(c,c^{\prime})\, .
\end{equation}
The product $(g,f) \mapsto gf$ on $A_{0}(\bC)$ is defined by  
\[gf := \begin{cases} g \circ f& \text{if }g\text{ and }f\text{ are composable\,,}\\ 0 & \text{otherwise\,.}\end{cases}\]
The $*$-operation on $\bC$ induces an involution on $A(\bC)_{0}$ in the obvious way. We again equip $A(\bC)_{0}$ with the maximal $C^{*}$-norm and define $A(\bC)$ as the closure of $A(\bC)_{0}$ with respect to this norm.\index{$A$}

\begin{ex}
If $\bC$ is the category of very small Hilbert spaces, then the vector space
$A(\bC)_{0}$ is the sum over all spaces of bounded operators between two very small Hilbert spaces.
This $*$-algebra belongs to the small universe and $A(\bC)$ is not very small. \hB
\end{ex}

%
%


Note that the assignment $\bC \mapsto A(\bC)$ is not a functor since non-composable morphisms in a $C^\ast$-category may become composable after applying a $C^\ast$-functor. However, $\bC \mapsto A(\bC)$ is functorial for $C^\ast$-functors that are injective on objects.

There is a canonical $C^{*}$-functor $\bC\to A(\bC)$. By the universal property of $A^f$  it extends uniquely to a homomorphism   $\alpha_{\bC}\colon A^f(\bC) \to A(\bC)$ of $C^{*}$-algebras.

 \begin{lem}
 If $f\colon \bC\to \bD$ is a functor between $C^{*}$-categories which is injective on objects, then the following square commutes:
 $$\xymatrix{A^{f}(\bC)\ar[r]^{A^{f}(f)}\ar[d]_{\alpha_{\bC}}&A^{f}(\bD)\ar[d]^{\alpha_{\bD}}\\A(\bC)\ar[r]_{A(f)}&A(\bD)}$$
 \end{lem}
 
\begin{proof}
 We consider the diagram
 $$\xymatrix{A^{f}(\bC)\ar[dd]_{\alpha_{\bC}}\ar[rrr]^{A^{f}(f)}&&&A^{f}(\bD)\ar[dd]^{\alpha_{\bD}}\\
 &\bC\ar[ul]\ar[dl]\ar[r]^{f}&\bD\ar[ur]\ar[dr]&\\A(\bC)\ar[rrr]_{A(f)}&&&A(\bD)}$$
 The left and the right triangles commute by construction of the transformation $\alpha_{\ldots}$.
 The upper and the lower middle square commute by construction of the functors $A^{f}(-)$ and $A(-)$.
 Therefore the outer square commutes.
\end{proof}

{Joachim} \cite[Prop.~3.8]{joachimcat} showed that if $\bC$ is unital and has only countably many objects, then $\alpha_{\bC}\colon A^f(\bC) \to A(\bC)$ is a stable homotopy equivalence and therefore induces an equivalence in $K$-theory.
 We will generalize this now to our more general setting.
 Let $\Kast\colon \Calg^{\la}\to \Sp^{\la}$ be the $K$-theory functor for small $C^{*}$-algebras discussed in Subsection \ref{wifjewiof2323443534}.

\begin{prop}\label{prop:dfs78934}
The homomorphism  $\alpha_{\bC}:A^f(\bC) \to A(\bC)$  induces an equivalence  of spectra
$$\Kast(\alpha_{\bC}):\Kast(A^{f}(\bC))\to\Kast (A(\bC))\, .$$
\end{prop}

\begin{proof}

If $\bC$ is unital and has countable may objects, then the assertion of the proposition has been shown by Joachim \cite[Prop.~3.8]{joachimcat}.

First assume that $\bC$ has countably many objects, but is possibly non-unital.  In the following we argue that   the arguments from the proof of \cite[Prop.~3.8]{joachimcat} are applicable and show that the canonical map $\alpha_{\bC}:A^{f}(\bC)\to A(\bC)$  is a stable homotopy equivalence. 
In the following we recall the construction of the stable inverse which  is a homomorphism  $$\beta:A(\bC)\to A^{f}(\bC)\otimes \IK$$  of $C^{*}$-algebras, 
where $\IK:=\IK(H)$ {are the compact operators on} the Hilbert space $$H := \ell^2(\Ob(\bC) \cup \{e\})\, ,$$  where  $e$ is an artificially added point.
The assumption on the cardinality of $\Ob(\bC)$ is made since we want that $\IK$ is the algebra of compact operators on a separable Hilbert space.
Two points $x,y$ in $\Ob(\bC) \cup \{e\}$ provide a rank-one operator $\Theta_{y,x}$ in  $\IK(H)$ which sends the basis vector corresponding to $x$ to the vector corresponding to $y$, and is zero otherwise.
By definition, the homomorphism $\beta$ sends  $A$  in $\Hom_{\bC}(x,y)$ to  $$\beta(A) := A \otimes \Theta_{y,x}\, .  $$
If $A$ and $B$ are composable morphisms, then the relation $\Theta_{z,y}\Theta_{y,x}=\Theta_{z,x}$ implies   that $\beta(B\circ A)=\beta(B)\beta(A)$. Moreover, 
if $A$,$B$ are not composable, then $\beta(B)\beta(A) = 0$. Finally, $\beta(A)^{*}=\beta(A^{*})$ since
$\Theta_{y,x}^{*}=\Theta_{x,y}$. It follows that  $\beta$ is a well-defined $^\ast$-homomorphism.

The argument now proceeds by showing that the composition
$(\alpha_{\bC}\otimes \id_{\IK(H)})\circ \beta$ is homotopic to $\id_{A(\bC)}\otimes \Theta_{e,e}$, 
and that the composition
$\beta\circ \alpha_{\bC}$ is homotopic to $\id_{A^{f}(\bC)}\otimes \Theta_{e,e}$.
Note that in our setting $\bC$ is not necessarily unital. In the following we directly refer to the  proof of \cite[Prop.~3.8]{joachimcat}.
The only step in the  proof of that proposition where the identity morphisms are used is the definition of the maps denoted by $u_x(t)$ in the reference. But they in turn are only used to define the map denoted by $\Xi$  later in that proof. The crucial observation is that we can define this map $\Xi$ directly without using any identity morphisms in $\bC$.

We  conclude that the morphism $$\Kast(\alpha_{\bC}):\Kast(A^{f}(\bC))\to\Kast (A(\bC))$$  {is an equivalence} for $C^\ast$-categories $\bC$ with countably many objects.

In order to extend this to all $C^{*}$-categories, we use that 
$$\bC\cong  \colim_{\bC^\prime} A^f(\bC^\prime)$$ in $\Ccat$, where the colimit runs over the small filtered poset of full subcategories with countably many objects. This can be shown  easily by verifying the universal property of a colimit, see
also \cite[Lem.~4.5]{crosscat}.

Since $A^{f}$ is a left-adjoint it preserves colimits and we get an isomorphism 
$$A^f(\bC) \cong \colim_{\bC^\prime} A^f(\bC^\prime)\, .$$  
The connecting maps of the indexing family   are functors which are injections on objects. 
By an inspection of the definition we see that the canonical map is an isomorphism
$$\colim_{\bC'} A(\bC')_{0}\cong A(\bC)_{0}\, ,$$  where
the colimit is interpreted in the category of pre-$C^{*}$-algebras \cite[Def. 2.14]{startcats}.  
The completion with respect to the maximal norm is a left-adjoint \cite[(3.17)]{startcats} and preserves colimits. Hence the canonical morphism   is an isomorphism\footnote{One could also check this  isomorphism directly as it was the approach in the first version of this book. But the theory developed in the reference formalizes some of the arguments nicely.} \begin{equation}\label{rgoijioqjefefqwefqwefqewfqf}
\colim_{\bC^{\prime}} A(\bC^{\prime})\cong A(\bC)\, .
\end{equation}

 Recall that  
the index category of these colimits is  a filtered poset. By Property \ref{wtgwergfervfdsv} of $\Kast$ saying that this functor  
  preserves  filtered colimits 
the morphism $$\Kast(\alpha_{\bC}):\Kast(A^{f}(\bC))\to \Kast(A(\bC))$$ is equivalent to the morphism
$$\colim_{\bC'} \Kast(\alpha_{\bC'}):\colim_{\bC^{\prime}} \Kast(A^{f}(\bC^{\prime}))\to \colim_{\bC^{\prime}} \Kast(A(\bC^{\prime}))\, .$$ Since the categories $\bC^{\prime}$ appearing in the colimit now have at most countable many objects we have identified
$\Kast(\alpha_{\bC})$  with a colimit of equivalences. Hence this morphism itself is an equivalence. 
  \end{proof}

     \begin{ddd}\label{wegjioegergewfewrfewef}
   We define the topological $K$-theory functor for $C^{*}$-categories as the composition
   $$\Kcat:\Ccat\stackrel{A^{f}}{\to} \Calg^{\la}\stackrel{\Kast}{\to} \Sp^{\la}\, .$$
   \end{ddd}

In the subsequent subsections we show the basic properties of this $K$-theory functor for $C^{*}$-categories.

\subsubsection{\texorpdfstring{$K$}{K}-theory preserves filtered colimits}

Note that the functor   $A^{f}\colon\Ccat\to \Calg^{\la}$ is a left-adjoint, see \eqref{hbhjeuzefwefewfwe}. Hence it 
 preserves all small colimits.  Furthermore, the functor  $\Kast\colon \Calg^{\la}\to \Sp^{\la}$   preserves small filtered colimits by Property \ref{wtgwergfervfdsv}. 
 
 \begin{kor}\label{wegkoerfrefewrfwerf} The functor $\Kcat$ preserves small filtered colimits.
 \end{kor}

  As shown in \cite[Thm.~4.1]{crosscat} the category $\Ccat$ admits all small colimits. 
 But in general it might be complicated to calculate a filtered colimit of $C^{*}$-categories explicitly. 
In the present book we only need the special case  formulated in Proposition \ref{erijgowegwergrefewrferfw}. 

We consider a $C^{*}$-category $\bC$ and a filtered family $(\bD_{i})_{i\in I}$ of subcategories
 of $\bC$ with the following special properties:
\begin{enumerate}
\item These subcategories all have the same set of objects as $\bC$. 
\item On the level of morphisms the inclusion $\bD_{i}\hookrightarrow \bC$ is an isometric embedding of a closed subspace.
\end{enumerate}
In this situation we can define a full subcategory
\begin{equation}\label{defwkjnnjjkwejkf89324234234}
\bD:=\colim_{i\in I}\bD_{i} 
\end{equation}
of $\bC$. It again has the same set of objects as $\bC$. For two objects $c,c^{\prime}$ of  $\Ob(\bC)$ we define
$\Hom_{\bD}(c,c^{\prime})$ to be the closure in $\Hom_{\bC}(c,c^{\prime})$ of the linear subspace $\bigcup_{i\in I}\Hom_{\bD_{i}}(c,c^{\prime})$.
One checks easily that $\bD$ is again a $C^{*}$-category.

\begin{prop}\label{erijgowegwergrefewrferfw}
The canonical morphism
$$\colim_{i\in I} \Kcat(\bD_{i})\to \Kcat(\bD)$$ is an equivalence.
\end{prop}
\begin{proof}
One easily checks by verifying the universal property that $\bD$ represents the colimit of the family
$(\bD_{i})_{i\in I}$. For a detailed argument see also  \cite[Lem.~4.5]{crosscat}. 
We can now apply Corollary \ref{wegkoerfrefewrfwerf}.
 \end{proof}

\subsubsection{\texorpdfstring{$K$}{K}-theory preserves unitary equivalences}

   Let $\bD$ be a $C^{*}$-category, and let $u:d\to d'$ be a morphism in $\bD$.
   \begin{ddd}\mbox{}
   \begin{enumerate}
   \item $u$ is a partial isometry if $uu^{*}$ and $u^{*}u$ are   projections.
   \item $u$ is an isometry if it is a partial isometry and $u^{*}u=\id_{d}$.
   \item $u$ is unitary, if it is an isometry and $u^{*}$ is an isometry, too. 
   \end{enumerate}
   \end{ddd}
 If $u:d\to d'$ is a unitary, then $d$ and $d'$ admit unit endomorphisms.
   
 Let $f,g:\bC\to \bD$ be two functors in $\Ccat$.
 \begin{ddd}
 A unitary isomorphism $u\colon f\to g$ is a natural transformation $u=(u_{c})_{c\in \bC}$
 such that $u_{c}\colon f(c)\to g(c)$ is unitary for every object $c$ in $\bC$. 
 \end{ddd}

  Let $f,g:\bC\to \bD$ be two functors between unital $C^{*}$-categories.
 
 \begin{lem}\label{ojewjfowfewfewfewf}
Assume: \begin{enumerate}
\item $\bD$ is unital.
\item  $f$ and $g$ are unitarily isomorphic.
 \item $\Kcat(f):\Kcat\bC)\to \Kcat ( \bD)$  is an equivalence.
 \end{enumerate}
 Then  
 $\Kcat(g):\Kcat(\bC))\to \Kcat(\bD)$ is an equivalence.
 \end{lem}
 \begin{proof}
 We first assume that $\bC$ and $\bD$ have countably many objects. 
 The difficulty of the proof is that a unitary isomorphism between the functors $f$ and $g$ is not well-reflected
 on the level of algebras $A^{f}(\bC)$ and $A^{f}(\bD)$. But it induces an isomorphism between the induced functors $f_{*}$ and $g_{*}$
 on the module categories of the $C^{*}$-categories $\bC$ and $\bD$.
 
The following argument provides this transition. 
  We use that the $K$-theory of $A^{f}(\bC)$ considered as a $C^{*}$-category with a single object is the same as the $K$-theory of $ A^{f}(\bC)$ considered as a $C^{*}$-algebra. Since the unit of the adjunction 
 \eqref{hbhjeuzefwefewfwe} is a natural transformation we have a commutative diagram
\[\xymatrix{\bC\ar[rr]^-{h}\ar[d] & & \bD\ar[d]\\
A^{f}(\bC)\ar[rr]^-{A^{f}(h)} & & A^{f}(\bD)}\]
 of $C^{*}$-categories, where $h$ is in $ \{f,g\}$.
 We apply $K$-theory $K_{*}=\pi_{*}\Kast$ and get the commutative square:
\[\xymatrix{K_{*}(\bC)\ar[rr]^-{K_{*}(h)}\ar[d] & & K_{*}(\bD)\ar[d]\\
 K_{*}(A^{f}(\bC))\ar[rr]^-{K_{*}(A^{f}(h))} & & K_{*}(A^{f}(\bD))}\]
By \cite[Cor.~3.10]{joachimcat}  the vertical homomorphism are isomorphisms. 
By assumption $K_{*}(A^{f}(f))$ is an isomorphism. We conclude that $K_{*}(f)$ is an isomorphism. 
The proof of \cite[Lem.~2.7]{joachimcat} can be modified to show that the functors $f_{*}$ and $g_{*}$ between the module categories of $\bC$ and $\bD$ are equivalent. So if $K_{*}(f)$ is an isomorphism, then so is $K_{*}(g)$.
 Finally, we conclude that $K_{*}(A^{f}(g))$ is an isomorphism. 
 This gives the result in the case of countable many objects.
 
 The general case then follows by considering filtered colimits over subcategories with countably many objects as in the proof of Proposition \ref{prop:dfs78934}.
 \end{proof}

 Let $f:\bC\to \bD$ be a functor between unital $C^{*}$-categories. 
 \begin{ddd}
 $f$ is a unitary equivalence if there exists a functor $g:\bD\to \bC$ such that $f\circ g$ and $g\circ f$  are unitary isomorphic to the respective identity functors.
 \end{ddd}
 
 The functor $g$ is called an inverse of $f$.

Let $f:\bC\to \bD$ be a functor between unital $C^{*}$-categories.
 \begin{kor}\label{ojewjfowfewfewfewf1}
 If $f $ is a unitary equivalence, then $\Kcat(f):\Kcat(\bC)\to \Kcat(\bD)$ is an equivalence.
 \end{kor}

\begin{proof}
Let $g:\bD\to \bC$ be an inverse of $f$.
Then $f\circ g$ (or $g\circ f$, respectively) is unitarily isomorphic to $\id_{\bD}$ (or $\id_{\bC}$, respectively).
By Lemma \ref{ojewjfowfewfewfewf} the morphisms   $\Kcat( g\circ f)$ and $\Kcat(f\circ g)$ are equivalences of spectra.
Since $\Kcat$ is a functor this implies that $\Kast( f)$ is an equivalence of spectra. 
\end{proof}

\subsubsection{Exactness of \texorpdfstring{$K$}{K}-theory}

 Let $\bD$ be a $C^{*}$-category.  We furthermore consider a wide $C^{*}$-subcategory $\bC$ of $\bD$,
 i.e.,  $\bC$ has the same objecs as $\bD$, and the morphism spaces of $\bC$ are closed subspaces of the morphism spaces of $\bD$ which are preserved by the $*$-operation.

\begin{ddd}\label{eiogegwergfewrfewfwefwe} $\bC$ is a closed  ideal\index{ideal in a $C^\ast$-category} in $\bD$ if
for all triples $x,y,z$ of objects of $\bD$,  
$f$ in $\Hom_\bD(x,y)$ and   $g$  in $\Hom_\bC(y,z)$ we have  $gf \in \Hom_\bC(x,z)$.
 \end{ddd}
The condition in Definition  \ref{eiogegwergfewrfewfwefwe} is the analogue of being a right-ideal. Since $\bC$ is preserved by the involution
this condition also implies an analoguous left-ideal condition.

If $\bC \to \bD$ is a {closed} ideal we can form the quotient category $\bD / \bC$ by declaring
\[\Ob(\bD / \bC) := \Ob(\bD) \text{ and } \Hom_{\bD / \bC}(x,y) := \Hom_{\bD}(x,y) / \Hom_{\bC}(x,y),\]
where the latter quotient is taken in the sense of Banach spaces.

 The following has been verify, e.g.\ in \cite[Cor.~4.8]{mitchc}.
\begin{lem}\label{iojfoiwef32904u23424324}
{If $\bC \to \bD$ is a closed ideal, then the} quotient category $\bD / \bC$ is a $C^\ast$-category.
\end{lem}
 
We will depict this situation by
$$0\to \bC\to \bD\to \bD/\bC\to 0$$
and call this an exact sequence of $C^{*}$-categories.

\begin{rem}
In \cite[Thm.~4.1]{crosscat} {it is} shown that $\Ccat$ is cocomplete, hence in particular admits push-outs.
The quotient $\bD/\bC$ fits into  a square
$$\xymatrix{\bC\ar[r]\ar[d]&\bD\ar[d]\\0[\Ob(\bD)]\ar[r]&\bD/\bC}$$
where  $0[X]$ denotes the  $C^{*}$-category with the set of objects $X$ and only zero morphisms.   
 This square is a push-out and a pull-back at the same time.
\hB
\end{rem}

\begin{prop}\label{qerijogergwregwregwerg9}
The functor $\Kcat$ sends exact sequences of $C^{*}$-categories to fibre sequences.
\end{prop}

The next lemma  prepares the proof of this proposition.

\begin{lem}\label{fjweoifjweo89234234234}
If  $\bC \to \bD$ is  a {closed} ideal, then we have a short exact sequence of $C^\ast$-algebras
\[0 \to A(\bC) \to A(\bD) \to A(\bD / \bC) \to 0\,.\]
\end{lem}
\begin{proof} 
Recall  that we define  the $C^{*}$-algebra $A(\bC)$ as the closure of the pre-$C^{*}$-algebra $A(\bC)_{0}$ given by \eqref{efefwefeefefefewfef} using the maximal $C^{*}$-norm. Since taking the closure with respect to the maximal $C^{*}$-norm is a left-adjoint it preserves surjections. Consequently the obvious 
surjection  $A(\bD)_{0}\to A(\bD/\bC)_{0}$ induces a surjection  $A(\bD) \to A(\bD/\bC)$.

The homomorphism $A(\bC)_0 \to A(\bD)_0$ of pre-$*$-algebras  induces a morphism $A(\bC) \to A(\bD)$  of $C^{*}$-algebras. We must show that this is the injection of a closed ideal which is precisely the kernel of the surjection $A(\bD)\to A(\bD/\bC)$.
 
In the following we argue that $ A(\bC)\to A(\bD)$ is isometric. We consider a finite subset of the set of objects of $\bD$ and  the full subcategories $\bC^{\prime}$ and $\bD^{\prime}$ of $\bC$ and $\bD$, respectively, on these objects. It has been observed in the proof of \cite[Lem.~3.6]{joachimcat} that $A(\bD^{\prime})_{0}$ with its topology as a finite sum \eqref{efefwefeefefefewfef} of Banach spaces is closed in $A(\bD)$. Consequently, the inclusion
$A(\bD^{\prime})_{0} \hookrightarrow A(\bD)$ induces the maximal norm on $A(\bD^{\prime})_{0}$.
We now use that $A(\bC^{\prime})_{0}$ is closed in $A(\bD^{\prime})_{0}$. Consequently, the maximal $C^{*}$-norm on
$A(\bC^{\prime})_{0}$ is the norm induced from the  inclusion $A(\bC^{\prime})_{0} \hookrightarrow A(\bD^{\prime})_{0}  \hookrightarrow  A(\bD)$. 
%

%
%
%

The norm induced on $A(\bC)_{0}$ from the inclusion $A(\bC)_{0}\hookrightarrow A(\bD)$ is bounded above by the maximal norm but induces the maximal norm on $A(\bC^{\prime})_{0}$. 
Since every element of $  A(\bC)_0$ is  contained in $A(\bC^{\prime})_{0}$ for some  full subcategory $\bC^{\prime}$ 
on a finite subset of objects, we conclude that the inclusion $A(\bC)_0 \to A(\bD)$ 
induces the maximal norm.

It follows that the $^\ast$-homomorphism $A(\bC) \to A(\bD)$ is isometric.
As an isometric inclusion the homomorphism of $C^{*}$-algebras $A(\bC)\to A(\bD)$ is injective.

The composition $$A(\bC) \to A(\bD) \to A(\bD / \bC)$$ vanishes by the functoriality of the process of taking completions with respect to maximal $C^{*}$-norms. 
 So it remains to show that this sequence is exact.
 We start with the obvious exact sequence
 $$A(\bC)_{0} \to A(\bD)_{0} \stackrel{\pi}{\to} A(\bD / \bC)_{0}\, .$$
If we restrict to the subcategories with finite objects we get an isomorphism
$$ A(\bD^{\prime})_{0}/A(\bC^{\prime})_{0}\to A(\bD^{\prime} / \bC^{\prime})_{0}\, .$$
This isomorphism is an isometry when we equip all algebras with their maximal $C^{*}$-norms.
Furthermore, by the same arguments as above  the inclusion 
$ A(\bD^{\prime})_{0}/A(\bC^{\prime})_{0}\to A(\bD / \bC)$
is an isometry.  Using the arguments from above we conclude that
$A(\bD )_{0}/A(\bC )_{0}\to A(\bD / \bC)$ is an isometry.
 This assertion implies   that the closure of $A(\bC)_{0}$  in the topology induced from $A(\bD)$
 is the kernel of  $A(\bD) \to A(\bD / \bC)$. 
 But above we have seen that this closure is exactly $A(\bC)$.   
 \end{proof}
 Note that Lemma \ref{fjweoifjweo89234234234} implies an isomorphism of $C^{*}$-algebras
\begin{equation}\label{sdfvsdfvkporgdfsvs}
A(\bD/\bC)\cong A(\bD)/A(\bC)\, .
\end{equation}

 \begin{proof}[Proof of Proposition \ref{qerijogergwregwregwerg9}]
 Let $$0\to \bC\to \bD\to \bD/\bC\to 0$$
 be an exact sequence of $C^{*}$-categories. Then we get the following commuting diagram
$$\xymatrix{\Kcat(\bC)\ar[r]\ar[d]^{\Kast(\alpha_{\bC})}_{\simeq}&\Kcat(\bD)\ar[r]\ar[d]_{\simeq }^{\Kast(\alpha_{\bD})}&\Kcat(\bD/\bC)\ar[d]^{\Kast(\alpha_{\bD/\bC})}_{\simeq}\\
\Kast(A(\bC))\ar[r]&\Kast(A(\bD))\ar[r]&\Kast(A(\bD/\bC))}$$
The vertical morphisms are equivalences by Proposition \ref{prop:dfs78934} and the definition $\Kcat:=\Kast\circ A^{f}$.
Using \eqref{sdfvsdfvkporgdfsvs} and Property \ref{wtgwergfervfdsv1} of $\Kast$
we see that the lower part of the diagram is a segment of a fibre sequence.
Hence the upper part is a segment of a fibre sequence, too.
\end{proof}

%
%
%

\subsubsection{Additivity of \texorpdfstring{$K$}{K}-theory}\label{ergioheriogergrewferfwef}
 
Let $\bD$ be a $C^{*}$-category, and let $d,d'$ be objects in $\bD$.\index{orthogonal sum in $C^\ast$-categories}\index{$C^\ast$-category!orthogonal sum}
\begin{ddd}\label{erigoowerffrefwefwerf}
An orthogonal sum of $d$ and $d'$ is triple $(d\oplus d',u,u')$ of an object $d\oplus d'$ of $\bD$ and isometries
$u:d\to d\oplus d'$ and $u':d'\to d\oplus d'$ such that $u'u^{*}=0$ and $uu^{*}+u'u^{\prime,*}=\id_{d\oplus d'}$.
\end{ddd}
The  isometries $u$ and $u'$ are called the canonical inclusions.
If an orthogonal sum exists, then it is unique up to unique unitary isomorphism. 

\begin{rem}
In \cite{cank} we discuss the concept of infinite orthogonal sums in $C^{*}$-categories.  
In the present book we only need  the  obvious abstract Definition \ref{erigoowerffrefwefwerf} of finite orthogonal sums.
We will encounter infinite orthogonal sums as well but only in the context of $C^{*}$-categories built from Hilbert spaces where these sums can be understood directly. 
\hB
\end{rem}

An object $d$ in a $C^{*}$-category $\bD$ is called a zero object if it  $\End_{\bD}(d)\cong 0$.
\begin{ddd}
 $\bD$ is called additive\index{additive $C^\ast$-category}\index{$C^\ast$-category!additive} if it admits a zero object and orthogonal sums for all pairs of objects.
\end{ddd}

Note that an additive $C^{*}$-category is automatically unital.  
 
Let $\bC$ be a second $C^{*}$-category, and let  $\phi,\phi'\colon \bC\to \bD$  be two functors.
Then we can define a new functor $\phi\oplus \phi'\colon \bC\to \bD$.  It sends an object $c$ in $\bC$ to a choice of a sum
$\phi(c)\oplus {\phi(c)}$. Furthermore, it sends a morphism  $A\colon c_{0}\to c_{1}$  in $\bC$ to the morphism 
 $u_{1}\phi(A)u_{0}^{*}+u_{1}'\phi'(A)u_{0}^{\prime,*}$, where $u_{i}$ and $u'_{i}$ for $i=0,1$  are the canonical inclusions for the sums $\phi(c_{i})\oplus \phi'(c_{i})$. 
 
 The functor is uniquely determined  up to unique unitary isomorphism by the choices adopted on the level of objects.
 \newcommand{\Homol}{\Kcat}

 The following proposition says that the $K$-theory functor $\Kcat$ preserves sums of functors.

  Let $\bC$ and $\bD$ be in $\Ccat$, and let $\phi,\phi':\bC\to \bD$ be two functors.
  \begin{prop}\label{rguiwtgwreergrgwrgg}
  If $\bD$ is additive, then  we have\index{$K$-theory!preserves sums of functors}
an equivalence $$\Homol(\phi\oplus \phi')\simeq \Homol(\phi)+\Homol(\phi'):\Homol(\bC)\to \Homol(\bD)\, .$$
\end{prop}

 Before we start with the actual proof  of the proposition we show that the $K$-theory functor $\Kcat$ preserves some   products of $C^{*}$-categories.\index{$K$-theory!preserves products of $C^\ast$-categories}

   Let $\bC$ and $\bD$ be in $\Ccat$.
\begin{lem}\label{erighi9wrteogwergrgregwergwreg}Assume:
\begin{enumerate}
\item $\bC$ and $\bD$ are unital.
\item $\bC$ and $\bD$ admit  zero objects.
\end{enumerate}
Then the morphism \begin{equation}\label{fweqdqewdeqwedqwedqwedqwedqwedqwedwedqwedq}
\Homol(\pr_{\bC})\oplus \Homol(\pr_{\bD}):\Homol(\bC\times \bD)\to \Homol(\bC)\oplus \Homol(\bD)
\end{equation} is an equivalence.
 \end{lem}
\begin{proof}
For a set $X$ let $0[X]$ be the category with the set of objects $X$ and with only zero morphisms.
  We consider the  commuting diagram
  \begin{equation}\label{qfhiuqwehfiuqwefefqff}
\xymatrix{0[\Ob(\bC)]\ar[rr]^-{\omega_{\bC}}\ar[d]&&\bC\ar[d]^-{z_{\bC}}\ar[rr]^{=}&& \bC\ar[d]^{z'_{\bC}}\\ 0[\Ob(\bC)]\times \bD\ar[rr]^-{\omega_{\bC}\times \id_{\bD}}&&\bC\times \bD\ar[rr]&&\bC\times 0[\Ob(\bD)]}
\end{equation}     in $\Ccat$. The functor $\omega_{\bC}$  is the canonical inclusion.
 The vertical functors send an object $c$ of $\bC$ to the object $(c,0_{\bD})$ of $\bC\times \bD$, where $0_{\bD}$ is a chosen zero object in $\bD$.
  The action of the left vertical functor on morphisms is clear. Finally, the middle and right vertical functor  send a morphism $A:c\to c'$ in $\bC$ to the morphism $(A,0):(c,0_{\bD})\to (c',0_{\bD})$ in $\bC\times \bD$.
    The horizontal lines are exact sequences of $C^{*}$-categories.
    It is furthermore easy to check that the right vertical functor $z'_{\bC}$ is a unitary equivalence. Indeed, it is obviously fully faithful and essentially surjective. 
    We now apply the functor $\Kcat$ and get the commuting diagram 
\begin{align}\label{qfhiuqwehfiuqwefefqff22}
\xymatrix{
\Kcat(0[\Ob(\bC)])\ar[rr]_-{\Kcat(\omega_{\bC})}\ar[d]&&\Kcat(\bC)\ar[d]^-{\Kcat(z_{\bC})}\ar[rr]&&\Kcat( \bC)\ar[d]^{\Kcat(z_{\bC}')}_{\simeq}\\
\Kcat(0[\Ob(\bC)]\times \bD)\ar[rr]_-{\Kcat(\omega_{\bC}\times \id_{\bD})}&&\Kcat(\bC\times \bD)\ar[rr]&&\Kcat(\bC\times 0[\Ob(\bD)])
}
\end{align} 
in $\Sp^{\la}$.
Using the observations made above we conclude with   Proposition \ref{qerijogergwregwregwerg9} that the horizontal sequences are segments of fibre sequence, and with Corollary \ref{ojewjfowfewfewfewf1} that 
 left vertical morphism is an equivalence. 
 Hence the left square is a push-out.
 
We now use that  there is a unitary equivalence $z'_{\bD}:\bD\to 0[\Ob(\bC)]\times \bD$. Consequently we have an identification $\Kcat(z'_{\bD}):\Kcat(\bD)\stackrel{\simeq}{\to}  \Kcat(0[\Ob(\bC)]\times \bD)$ for the lower left corner.
Finally note that the left upper corner is the $K$-theory of a category consisting only of zero objects and therefore vanishes by Property \ref{wtgwergfervfdsv1} of $\Kast$ and the fact that $A^{f}(0[X])\cong 0$. Using stability in order to identify coproducts with products we conclude that
\begin{equation}\label{dqweodfhqiofnfwfqwefqewfqewfqwf}
\Kcat(z_{\bC})+  \Kcat(z_{\bD}):\Kcat(\bD)\oplus \Kcat(\bD)\to \Kcat(\bC\times \bD)
\end{equation}
is an equivalence.
 We now observe that $\pr_{\bC}\circ z_{\bC}=\id_{\bC}$ and $\pr_{\bD}\circ z_{\bD}=\id_{\bD}$, and that
$\pr_{\bC}\circ z_{\bD}$ and $\pr_{\bC}\circ z_{\bD}$ are zero morphisms.
This immediately implies that  \eqref{fweqdqewdeqwedqwedqwedqwedqwedqwedwedqwedq}
 is an equivalence inverse to \eqref{dqweodfhqiofnfwfqwefqewfqewfqwf}.
 \end{proof}

 \begin{proof} [Proof of Proposition \ref{rguiwtgwreergrgwrgg}] 
Since $\bD$ is additive we can choose a zero object $0_{\bD}$ in $\bD$.
We consider the diagram in $\Sp^{\la}$
\begin{equation}\label{ebwegrfreferwfwefwrefwev}
\xymatrix{
&\Homol(\bD\times \bD)\ar[dr]^-{\quad\Homol(\bigoplus)}\ar[dl]^-{\simeq}_-{\Homol(\pr_{0})\oplus \Homol(\pr_{1})\quad\quad\quad\quad}\\
\Homol(\bD)\oplus \Homol(\bD)\ar[rr]^-{+}
&&\Homol(\bD)}
\end{equation}
 where the 
 left vertical 
 morphism is an equivalence by Lemma  \ref{erighi9wrteogwergrgregwergwreg}. 
 We claim that \eqref{ebwegrfreferwfwefwrefwev} naturally commutes.
 Using the explicit inverse \eqref{dqweodfhqiofnfwfqwefqewfqewfqwf} of $\Homol(\pr_{0})\oplus \Homol(\pr_{1})$ and the universal property of $+$
   it suffices to show that the compositions
\[\mathclap{
\Homol(\bD)\stackrel{\iota_{i}}{\to}\Homol(\bD)\oplus \Homol(\bD)\xrightarrow{\Homol(z_{0})+\Homol(z_{1})}\Homol(\bD\times \bD)\xrightarrow{\Homol(\id_{\bD}\oplus \id_{\bD})} \Homol(\bD)
}\]
 are equivalent to the identity, where $\iota_{i}:\Homol(\bD)\stackrel{\iota_{i}}{\to}\Homol(\bD)\oplus \Homol(\bD)$ denote the canonical inclusions for $i=0,1$. 

 In the case $i=0$ this composition is induced by applying $\Homol$ to the endofunctor
 $s :\bD\to \bD$  which is given as follows:
 \begin{enumerate}
 \item objects: $s $ sends an object $D$ to the representative $D\oplus 0_{\bD} $.
 \item morphisms: $s$ sends a morphism $f:D\to D'$ to the morphism $f\oplus 0:D\oplus 0_{\bD} \to D'\oplus 0_{\bD} $. 
 \end{enumerate}
We have a unitary equivalence $u:\id_{\bD} \to s $ given by the family $(u_{D})_{D\in \Ob(\bD)}$, where
$u_{D}:D\to D\oplus 0_{\bD} $ is the canonical inclusion.
 Hence $\Homol(s )\simeq \Homol(\id_{\bD})$.
 The case $i=1$ is analoguous.


   We have the following  diagram in $\Sp^{\la}$
\begin{equation}\label{svfdviuvhuisdfvsdfvsdfvsdfv}
\xymatrix{
&\Homol(\bC)\ar[dr]^-{\diag_{\Homol(\bC)}}\ar[dl]_-{\Homol(\diag_{\bC})\quad} &\\
\Homol(\bC\times \bC)\ar[rr]_{\simeq}^-{\Homol(\pr_{0})\oplus \Homol(\pr_{1})} \ar[d]^{\Homol(\phi\times \phi')}& &\Homol(\bC)\oplus \Homol(\bC)\ar[d]^{\Homol(\phi)\oplus \Homol(\phi')} \\
\Homol(\bD \times  \bD)\ar[dr]_-{\Homol( \bigoplus)}\ar[rr]_{\simeq}^-{\Homol(\pr_{0})\oplus \Homol(\pr_{1})} &  &\Homol(\bD)\oplus \Homol(\bD) \ar[dl]^-{+}\\&\Homol(\bD)&
}
\end{equation}      
  The lower triangle commutes by \eqref{ebwegrfreferwfwefwrefwev}.
The upper triangle and the middle square obviously commute. The left top-down path is the map
$\Homol(\phi\oplus \phi')$, while the right top-down path is $\Homol(\phi)+\Homol(\phi')$. 
The filler of \eqref{svfdviuvhuisdfvsdfvsdfvsdfv} provides the desired equivalence.
 \end{proof}
 
  \subsection{Coarse \texorpdfstring{$K$}{K}-homology}\label{sec:jkbewrgh}

In this section we define the coarse $K$-homology functors $\KX$ and $\KXql$.  Our main result is Theorem \ref{irgfjiwoowierer23423424243} stating that these functors are coarse homology theories.


Let $X$ be a bornological coarse space.
Note that Hilbert spaces below are assumed to be very small. The Roe categories also defined below will be small $C^{*}$-categories and therefore objects of $\Ccat$. Recall the Definitions \ref{feiojwoeffewfwefewfwef},    \ref{ewoifjweoifiufu89234234324}, \ref{rgioerjog34t4t34t34t} and \ref{defn:iuwn3r2}.

We consider the $\C$-linear $*$-category $\bC^{*}(X)_{0}$  whose objects are the locally finite $X$-controlled Hilbert spaces which are determined on points, and whose morphisms are the bounded operators of controlled propagation.
 Using the properties of the support listed in Subsection \ref{ofjewiofiouioru233244343244}
 and the closure property of the coarse structure of $X$ (Definition \ref{ewgiubhwiergfvsvdfvsfdvs}) we check that $\bC^{*}(X)_{0}$ is indeed a $\C$-linear $*$-category.  
 If $(H,\phi)$ and $(H',\phi')$ are objects of $\bC^{*}(X)_{0}$, then
 the norm of  the Banach space of bounded operators $B(H,H')$ induces a norm on $\Hom_{\bC^{*}(X)_{0}}((H,\phi),(H',\phi'))$. We can form the completion  of the morphism spaces of $\bC^{*}(X)_{0}$ with respect to these norms.

\begin{ddd}\label{foiefoiefoewfewfewiofef}
\index{Roe!category}
We define the Roe category $\bC^{*}(X)$\index{$\bC^{*}(-)$|see{Roe category}} as follows:
\begin{enumerate} \item
The objects of $\bC^{*}(X)$  are the locally finite $X$-controlled Hilbert spaces which are determined on points.
\item The morphism spaces of $\bC^{*}(X)$ are 
  obtained from the morphism spaces of $\bC^{*}(X)_{0}$  by  completing   with respect to the norms described above.  \end{enumerate}
\end{ddd}

It is clear that $\bC^{*}(X)$ is a well-defined $C^{*}$-category in the sense of Definition \ref{qwfijoafvffdfadf}.

For the following recall also Definition \ref{bvcbdfg}.
Let 
 $(H,\phi)$, $(H',\phi')$ and $(H'',\phi'')$  be $X$-controlled Hilbert spaces, and let $A:H\to H'$ and $B:H'\to H''$  be quasi-local operators. 
\begin{lem}\label{webgiojewrgerfewferf} If $(H',\phi')$ is determined on points, then $B\circ A$ is quasi-local.
  \end{lem}
\begin{proof}
If an operator is $U$-quasi-local, then it is also $U'$-quasi-local for every bigger entourage $U'$. We can therefore assume that $A$ and $B$ are both $U$-quasi-local for some entourage $U$ of $X$. 
We will show, that then $BA$ is also $U$-quasi-local.

We consider $\epsilon$ in $(0,\infty)$. Then we can find an integer $n$ in $\nat$ such that 
 \begin{equation}\label{fvpokpovsdvsfsfdv}
\|\phi''(Y')B\phi'(Y)\|\le \frac{\epsilon}{2(1+\|A\|)}\:\:\mbox{and}\:\: \|\phi'(Y')A\phi(Y)\|\le \frac{\epsilon}{2(1+\|B\|)}
\end{equation}
whenever
  $Y'$ and $Y$ are mutually $U^{n}$-separated subsets of $X$.  Assume now that $Z$ and $Z'$ are two mutually $U^{2n}$-separated subsets of $X$. Since  $(H',\phi')$ is determined on points, by Lemma \ref{fghiofweqewfqwed}  we have $$\phi'(U^{n}[Z])+\phi'(X\setminus U^{n}[Z])=\id_{H'}\, .$$
  This gives  the equality 
  $$\phi''(Z')BA\phi(Z)=\phi''(Z')B\phi'(U^{n}[Z])A\phi(Z) +\phi''(Z')B\phi'(X\setminus U^{n}[Z])A\phi(Z)\, .$$
  We now note that $Z'$ and $U^{n}[Z]$ are mutually $U^{n}$-separated, and that also 
  $X\setminus U^{n}[Z]$ and $Z$ are mutually $U1n$-separated. 
  Using the estimates in \eqref{fvpokpovsdvsfsfdv} we conclude that 
\[\|\phi''(Z')BA\phi(Z)\|\le \|\phi''(Z')B\phi'(U[Z])A\phi(Z)\| +\|\phi''(Z')B\phi'(X\setminus U^{n}[Z])A\phi(Z)\|\le  \epsilon\, .\qedhere\]
\end{proof}

We consider the category $\bC^{*}_{\ql}(X)_{0}$ whose objects are all locally finite $X$-controlled Hilbert spaces on $X$ which are determined on points, and whose morphisms are the quasi-local operators.  It is clear that the adjoint  of a quasi-local operator  is quasi-local. A linear combination of a $U$-quasi-local and a $U'$-quasi-local operator is $U\cup U'$-quasi-local.
 By Lemma \ref{qwrifuq9erfewdqewdqewdqewd} the composition of two quasi-local operators is again quasi-local.
 This shows that $\bC^{*}_{\ql}(X)_{0}$ is a $\C$-linear $*$-category. 
  If $(H,\phi)$ and $(H',\phi')$ are objects of $\bC^{*}_{\ql}(X)_{0}$, then
 the norm of  the Banach space of bounded operators $B(H,H')$ induces a norm on $\Hom_{\bC^{*}_{\ql}(X)_{0}}((H,\phi),(H',\phi'))$. We can form the completion  of the morphism spaces of $\bC^{*}_{\ql}(X)_{0}$ with respect  these norms.

\begin{ddd}\label{foiefoiefoewfewfewiofef1}
We define the quasi-local Roe category\index{Roe!category!quasi-local} $\bC_{\ql}^{*}(X)$\index{$\bC_{\ql}^{*}(-)$|see{quasi-local Roe category}}\index{quasi-local!Roe category} as follows:
\begin{enumerate} \item
The objects of $\bC^{*}_{\ql}(X)$  are the locally finite $X$-controlled Hilbert spaces which are determined on points.
\item The morphism spaces of $\bC^{*}_{\ql}(X)$ are 
  obtained from the morphism spaces of $\bC_{\ql}(X)_{0}$  by  completing   with respect to the norms described above.  \end{enumerate}
\end{ddd}

It is again clear that $\bC^{*}_{\ql}(X)$ is a well-defined $C^{*}$-category in the sense of Definition \ref{qwfijoafvffdfadf}.

%
%
%
%

The categories $\bC^{*}(X)$ and $\bC^{*}_{\ql}(X)$ are $C^{*}$-categories which have the same set of objects.
Since operators of controlled propagation are quasi-local   we conclude that 
  morphisms of $\bC^{*}(X)$ form a subset of the morphisms of $\bC^{*}_{\ql}(X)$.

The definition of the Roe categories is closely related to the definitions of the corresponding Roe algebras $\cC^\ast(X,H,\phi)$ and $\cCql^\ast(X,H,\phi)$ from Definition \ref{qekdfjqodqqwdqwdqwdqdwq}.
For every object
$(H,\phi)$ in $\bC^{*}(X)$   {by definition} we have isomorphisms 
$$\End_{\bC^{*}(X)}((H,\phi)) \cong  \cC^{*}(X,H,\phi)\,, \quad  \End_{\bC_{\ql}^{*}(X)}((H,\phi),(H,\phi)) \cong  \cC_{\ql}^{*}(X,H,\phi)\, .$$

\begin{rem}
Our reason for using   the locally finite versions of the Roe algebras to model coarse $K$-homology is that  the Roe categories as defined above are unital.

The obvious alternative to define the Roe categories to consist of all $X$-controlled Hilbert spaces, and to take  locally compact versions of the operators as morphisms, failed, {as we will explain now.}

Recall that Corollary \ref{ojewjfowfewfewfewf1} says that equivalent $C^{*}$-categories have equivalent $K$-theories.
For non-unital $C^{*}$-algebras it would be necessary  to redefine the notion of equivalence using a generalization to $C^{*}$-categories of the notion of a multiplier algebra of a $C^{*}$-algebra. But we were not able to fix all details of an argument which extends Corollary \ref{ojewjfowfewfewfewf1} from the unital to the non-unital case.
 
 This corollary is used in an essential way to show that the coarse $K$-homology functor is coarsely invariant and satisfies excision. 
\hB
\end{rem}

 As a next step we extend the Definitions \ref{foiefoiefoewfewfewiofef} and \ref{foiefoiefoewfewfewiofef1}
to functors
$$\bC^{*},\bC_{\ql}^{*}:\BC\to \Ccat\, .$$
Let $f:X\to X^{\prime}$ be a morphism between bornological coarse spaces.
If $(H,\phi)$ is a locally finite  $X$-controlled Hilbert space, then the $X^{\prime}$-controlled Hilbert space
$f_{*}(H,\phi)$ defined in Definition \ref{ewljfewilewfiouwefwefew} is again locally finite (since $f$ is proper).
Similiarly, if $(H,\phi)$ is determined on points, then $f_{*}(H,\phi)$ is determined on points. 
Here we use the conclusion    $\ref{egkjweogerwgegewg1}\Rightarrow \ref{egkjweogerwgegewg3}$ in  Lemma \ref{fghiofweqewfqwed} for the partition $(f^{-1}(\{x'\}))_{x'\in X'}$ of $X$.
On the level of objects the functors 
$\bC^{*}(f)$ and $\bC^{*}_{\ql}(f)$ send $(H,\phi)$ to $f_{*}(H,\phi)$.

Let $(H,\phi)$ and $(H^{\prime},\phi^{\prime})$ be objects of $\bC^{*}(X)$, and let 
$A$ be in $\Hom_{\bC^{*}(X)}((H,\phi),(H^{\prime},\phi^{\prime}))$ (or  $\Hom_{\bC^{*}_{\ql}(X)}((H,\phi),(H^{\prime},\phi^{\prime}))$, respectively).  Then by definition $A$ can  be approximated by operators  of controlled propagation (or quasi-local  operators, respectively) with respect to the controls $\phi $ and $\phi^{\prime} $.
As in the proof of Lemma  \ref{elkjqdqoiuo3432424324} we see  that 
$A$ can also be approximated by operators  of controlled propagation (or  quasi-local operators, respectively) with respect to the controls $\phi\circ f^{*}$ and $\phi^{\prime}\circ f^{*}$. We can thus define define the functor $\bC^{*}(f)$ (or $\bC^{*}_{\ql}(f)$, respectively) on objects such that it sends $A$ to $A$. 

It is obvious that this description  of $\bC^{*}_{?}(f)$ is compatible with the composition and the $*$-operation.

In the following we use the $K$-theory functor $\Kcat$ for $C^{*}$-categories {of} Definition~\ref{wegjioegergewfewrfewef}.  \begin{ddd}\label{oiwefewiofu9234234234324}
We define the coarse $K$-{homology} functors\index{$K$-homology!coarse}\index{coarse!$K$-homology}
$$\KX, \KXql:\BC\to \Sp$$
\index{$\KX$, $\KXql$}
as compositions
$$\BC\xrightarrow{\bC^{*}, \bC_{\ql}^{*}}\Ccat  \xrightarrow{\Kcat} \Sp^{\la}\, .$$
\end{ddd}

\begin{theorem}\label{irgfjiwoowierer23423424243}
The functors $\KX$  and $\KXql$ are   coarse homology theories.
\end{theorem}

\begin{proof}
In the following we verify the conditions listed in Definition \ref{rgljogreggregrege}. The Propositions \ref{hewifeiiu3984u239824424}, \ref{ioejfiowefjweoifewuf8u3294823442}, \ref{hewifeiiu3984u2398244244} and \ref{hewifeiiu3984u2398244241} together  will imply   the theorem.
\end{proof}

\begin{prop}\label{hewifeiiu3984u239824424}
The functors $\KX$ and $\KXql$ are coarsely invariant.
\end{prop}
\begin{proof}

Let $f,g:X\to X^{\prime}$ be morphisms between bornological coarse spaces.
\begin{lem}\label{iwowiru9823u4923424324}
If $f$ and $g$    are close (Definition \ref{ioewfjweoifoijwfw32424}),
then the functors $\bC^{*}(f)$ and $ \bC^{*}(g)$ (or $  \bC_{\ql}^{*}(f)$ and $ \bC_{\ql}^{*}(g)$)  are unitarily isomorphic.
\end{lem}
\begin{proof}
 We define a unitary isomorphism
$u:\bC^{*}(f)\to \bC^{*}(g)$. The isomorphism $u$ is given by the family of unitaries $(u_{(H,\phi)})_{(H,\phi)\in \Ob( \bC^{*}(X))}$, where 
$$u_{(H,\phi)}:=\id_{H}:(H,\phi\circ f^{*})\to (H,\phi\circ g^{*})\, .$$
Since $f$ and $g$ are close, the identity $\id_{H}$ has indeed controlled propagation. 
The quasi-local case is analoguous.
\end{proof}

We can now finish the proof of the proposition.
Let $X$ be a bornological coarse space. We consider the projection   $\pi:\{0,1\} \otimes X\to X$ and the inclusion  $i_{0}:X\to \{0,1\}\otimes X$, where $\{0,1\}$ has the maximal structures. Then $\pi\circ i_{0}=\id_{X}$ and $i_{0}\circ \pi$ is close to the identity of $\{0,1\}\otimes X$.
By Lemma \ref{iwowiru9823u4923424324} the functor $\bC^{*}(\pi)$ (or $\bC_{\ql}^{*}(\pi)$, respectively) is an equivalence of $C^{*}$-categories.
By Corollary \ref{ojewjfowfewfewfewf1} the morphism $\KX( \pi)$ (or $\KXql(\pi)$, respectively) is an equivalence of spectra. \end{proof}

%

\begin{prop}\label{ioejfiowefjweoifewuf8u3294823442}
The functors $\KX$ and $\KXql$ satisfy excision.
\end{prop}

\begin{proof}
We will discuss the quasi-local case in detail. For $\KX$ the argument is similar.\footnote{A detailed argument in this case can also be found in \cite{Bunke:ad}.}  
Let $(Z,\cY)$ be a complementary pair on a  bornological coarse space
$X$.
We must show that    \begin{equation}\label{veer34534543541}
\xymatrix{{\KX}_{\ql} (Z\cap \cY)\ar[r]\ar[d]&{\KX}_{\ql} (Z) \ar[d]\\{\KX}_{\ql} ( \cY )\ar[r]&{\KX}_{\ql}(X) }
\end{equation}
is a push-out square.

 
 If $f:Z\to X$ is an inclusion of  a subset with induced structures, then we let  
$\bD(Z)$  be the $C^{*}$-category defined as follows:
\begin{enumerate}
\item The objects of $\bD(Z)$ are the objects of $\bC_{\ql}^{*}(X)$.
\item The morphisms $\Hom_{\bD(Z)}((H,\phi),(H^{\prime},\phi^{\prime}))$ are the 
quasi-local bounded operators $(H(Z),\phi_{Z} )\to (H^{\prime}(Z),\phi^{\prime}_{Z} )$.
\end{enumerate}
Here the $Z$-control of $H(Z)$ is defined by  $\phi_{Z}(f):=\phi(\tilde f)_{|H(Z)}$, where $\tilde f$ in $C(X)$ is any extension of $f$ in $C(Z)$ to $X$.  
 
 We have a natural inclusion $ \bD(Z)\to \bC_{\ql}^{*}(X)$ as a closed subcategory given by the identity on objects and the canonical inclusions
 $$\Hom_{\bD(Z)}((H,\phi),(H^{\prime},\phi^{\prime}))\to \Hom_{\bC_{\ql}^{*}(X)}((H,\phi),(H^{\prime},\phi^{\prime}))$$
 on the level of morphisms.
 
 We have  a fully-faithful  functor \begin{equation}\label{uhfweuifhewhweifuwewefwefewf}
\iota_{Z}:\bD(Z)\to \bC^{*}_{\ql}(Z)\, .
\end{equation}
It 
  sends the object $(H,\phi)$ of $\bD(Z)$ to  the object
 $(H(Z),\phi_{Z})$ of  $\bC^{*}_{\ql}(Z)$.  On morphisms it is given by  the  identification
 $$ \Hom_{\bD(Z)}((H,\phi),(H^{\prime},\phi^{\prime}))= \Hom_{\bC_{\ql}^{*}(Z)}((H(Z),\phi_{Z}),(H^{\prime}(Z),\phi^{\prime}_{Z}))\, .$$  This functor is essentially surjective. Indeed, if $(H,\phi)$ is in $\Ob(\bC_{\ql}^{*}(Z))$, then $$f_{*}(H,\phi)=(H,\phi\circ f^{*})\in \Ob(\bC_{\ql}^{*}(X))$$ can be considered as an object of $\bD(Z)$, and we have an unitary isomorphism    $$(H(Z),(\phi\circ f^{*})_{Z})\cong (H,\phi)$$ in $\bC_{\ql}^{*}(Z)$.
 Consequently, $\iota_{Z}:\bD(Z)\to \bC^{*}_{\ql}(Z)$ is a unitary equivalence of $C^{*}$-categories. By     Corollary~\ref{ojewjfowfewfewfewf1} the induced morphism
  \begin{equation}\label{jhewfkwhefehewiuhiu234234}
\Kcat( \bD(Z))\to {\KX}_{\ql}(Z)
\end{equation}  an equivalence of spectra.

 The functor $f_{*}:\bC^{*}_{\ql}(Z)\to \bC^{*}_{\ql}(X)$ identifies $\bC^{*}_{\ql}(Z)$
isomorphically with the full subcategory of all objects $(H,\phi)$ of $\bD(Z)$ with $H(Z)=H$. The inclusion 
$\bC^{*}_{\ql}(Z)\to \bD(Z)$ is a unitary equivalence of $C^{*}$-categories with inverse $\iota_{Z}$.
We therefore have the following factorization 
\begin{equation}\label{wvlkmflvkvadvadsvadvadv}
\xymatrix{{\KX}_{\ql}(Z)\ar[rr]^-{\KXql(f)}\ar[dr]^{\simeq}&&\KXql(X)\\
&\ar@/_-18pt/@{..>}[ul]_-{\simeq}^-{\Kcat(\iota_{Z})}\Kcat(\bD(Z))\ar[ur]&}
\end{equation}



We consider  a big family 
 $\cY=(Y_{i})_{i\in I}$ in $X$. We get an increasing family $(\bD(Y_{i}))_{i\in I}$ (defined as above with $Z$ replaced by $Y_{i}$) of subcategories  of $\bC_{\ql}^{*}(X)$.
 We let
$ \bD(\cY)$   be the closure of the union of this family  in $\bC^{*}_{\ql}(X)$ in the sense of $C^{*}$-categories, see \eqref{defwkjnnjjkwejkf89324234234}.

\begin{lem}\label{fweijweoifoewfewfefewf}
  $\bD( \cY )$ is a  {closed} ideal in $\bC_{\ql}^{*}(X)$.
\end{lem}
 
\begin{proof}
  We  consider a morphism 
 $A$ in $\bD(\cY)$ between two objects.    
  Let moreover  $Q$ be a morphism in $\bC_{\ql}^{*}( X )$ which is left composable with $A$. We must show that $QA\in \bD(\cY)$.
  Since $\bD(\cY)$ is closed it suffices to show this for quasi-local operators $Q$. In the following we will assume that $Q$ is quasi-local.

 The morphism 
 $A$   can be approximated by a family $( A_{i})_{i\in I}$  of morphisms 
 $A_{i}$ in  $\bD(Y_{i}) $ between the same objects.


We consider $\epsilon$ in $(0,\infty)$  and choose $i$ in $I$ such that
\[\| A_{i} -A\|\le \frac{\epsilon}{2(\|Q\|+1)}\, .\]
Here the norms are the Banach norms on the respective morphism spaces.
For an     $X$-controlled Hilbert space $(H,\phi)$ and a subset $B$ of $ X$ in the following we abbreviate the action of $\phi(\chi_{B})$ on $H$ by $B$.
Since $Q$ is weakly quasi-local we can find an entourage $U$ of $X$ such that $$\| B^{\prime}  Q   B   \| \le \frac{\epsilon}{2(\|A_{i}\|+1)}$$ whenever $B$, $B^\prime$ are mutually $U$-separated. Using that the family of subsets $\cY$ is big we now choose  $j$ in $I$ so that $X\setminus Y_{j}$ and $Y_{i}$ are mutually $U$-separated.  Then we   get the inequality
\[\| (X\setminus Y_{j})Q  A_{i} \| \le  \epsilon/2\, .\]
Here we use implicitly the relation $A_{i}=A_{i} Y_{i} $.
Now $ Y_{j} Q A_{i}$ belongs to   $ \bD( \cY   ) $, and $$\| Y_{j} Q A_{i} -QA\|\le \epsilon\, .$$  Since we can take $\epsilon$ arbitrary small and $  \bD( \cY   )$ is closed we get $QA \in \bD( \cY   )$.
   We can analogously conclude that $AQ \in  \bD(\cY)$.
\end{proof}

 Since $\bD(\cY)$ is a closed ideal we can form the quotient
  $C^{*}$-category $\bC_{\ql}^{*}(X )/ \bD( \cY )$, see Lemma \ref{iojfoiwef32904u23424324}.

We now consider a complementary pair $(Z,\cY)$.  The family $(\bD(Z\cap Y_{i}))_{i\in I}$ is an increasing family of subcategories of $\bD(Z)$. We let $\bD(Z\cap \cY)$ be the closure of the union of these subcategories.  The same argument as for Lemma \ref{fweijweoifoewfewfefewf} shows that $\bD(Z\cap \cY)$ is a closed ideal in $\bD(Z)$. Alternatively one could use the unitary equivalence $\iota_{Z}$ in \eqref{uhfweuifhewhweifuwewefwefewf} and its inverse $f_{*}$ in order to reduce the assertion formally to this lemma.

\begin{lem}\label{eriogegwefewrfwe} The inclusion $\bD(Z)\to \bC_{\ql}^{*}(X)$ induces an isomorphism of quotients
\begin{equation}\label{fwjehfhwjkfuifefiuewfwefewfwf1}
 \psi:\bD(Z)/\bD (Z\cap \cY)\to \bC_{\ql}^{*}(X )/\bD ( \cY )\, . 
\end{equation}
 \end{lem}
\begin{proof}
 Note that the inclusion is the identity on objects.
We construct an inverse functor $$\lambda: \bC_{\ql}^{*}(X )/\bD ( \cY )\to \bD(Z)/\bD (Z\cap \cY)\, .$$
It is again the identity on objects. On morphisms it sends  a class
  $[A]$ in  $\bC_{\ql}^{*}(X )/\bD( \cY)$  to $$\lambda([A]):=[ Z A Z]\, .$$
We now argue that
$\lambda$ is well-defined and an inverse of $\psi$.

    We first show that $\lambda$ is well-defined as a map between morphism sets.
  We fix  two objects of $\bC_{\ql}^{*}(X)$ and consider morphisms between these objects.
  We assume that $A$ in $\bD( \cY )$ is such a morphism. 
 Then we can obtain $A$ as a limit of a family $(A_{i})_{i\in I} $ of morphisms   $A_{i}$ in $ \bD(Y_{i})$.   
 Now $ Z A_{i } Z$ belongs to the image of $ \bD(Z\cap \cY_{i})$. %
Hence for the   limit we get $  Z A Z  \in\bD(Z\cap \cY)$.

It is clear that $\lambda$ preserves the involution and is compatible with the $\C$-linear structure on the morphism spaces. It remains to check the compatibility with compositions. 
We  
 consider  two composeable morphisms $A$ and $B$  in $\bC^{*}_{\ql}(X)$.
We must show that $$\lambda([A  B])=\lambda([A])  \lambda ([B])\, .$$ To this end it suffices to show that 
$ Z AB Z - Z A Z BZ\in\bD (Z\cap \cY)$. 
 We can write this difference in the form
 $ZAZ^{c} BZ$, where $Z^{c}:=X\setminus Z$. Let $i$ in $I$ be such that $Z\cup Y_{i}=X$. Then $Z^{c}\subseteq Y_{i}$. 
 Since the ideal is closed we assume that $A$ and $B$ are  quasi-local. 
 
 For every $U$ in the coarse structure $\cC$ of $X$ there exists 
   $j_{U}$ in $I$ with $i \le j_{U}$ such that
 $U[Y_{i}]\subseteq Y_{j_{U}}$. In this case    the sets $Z^{c}$ and $Y^{c}_{j_{U}}$ are mutually  $U$-separated.
  Since $A$ and $B$ are weakly quasi-local we have $$\lim_{U\in \cC}\|ZAZ^{c} BZ-Y_{j_{U}}ZA Z^{c} BY_{j_{U}}Z\|=0\, .$$  
  Note that the projection $Y_{j_{U}}Z$ belongs to $\bD(Z\cap \cY)$. Since this is an ideal we have 
 $Y_{j_{U}}ZA Z^{c} , Z^{c}BY_{j_{U}}Z\in  
 \bD( Z\cap \cY)$.  Since the ideal is closed we finally conclude that
 $ZAZ^{c} BZ\in \bD( Z\cap \cY)$ as required.

 Finally we show that $\lambda$ is an inverse of $\psi$.
  It is easy to see that $\lambda\circ \psi=\id$. We claim  that   $\psi\circ \lambda=\id$.
 We consider the difference $[A]-\psi(\lambda([A]))$ which is represented by $A-  Z A Z $.  Then we study the  term \begin{equation}\label{fwefwefwfew3r3r3ffwef}
 (A- Z A Z)- Y_{i} (A- Z A Z ) Y_{i} \, .
 \end{equation}
Note that the second term belongs to $\bD( \cY)$.
Hence in order to show the claim it suffices to see that \eqref{fwefwefwfew3r3r3ffwef} can be made arbitrary small by choosing $i$ in $I$ appropriately.  
We have the identity \[(A-ZAZ) =Z^{c}A+ZAZ^{c}\,.\]
This gives
\[(A-ZAZ)-Y_{i}(A-ZAZ)Y_{i}=Y_{i}^{c}Z^{c}A+Y_{i}^{c}ZAZ^{c}+Y_{i}Z^{c}AY^{c}+Y_{i}ZAZ^{c}Y_{i}^{c}\,.\]
Using that $X=Y_{i}\cup Z$ for sufficiently large $i$ in $I$ we get $Y_{i}^{c}Z^{c}=0$. Therefore if $i$ in $I$ is large it remains to    consider
\[Y^{c}_{i}ZAZ^{c}+Y_{i}Z^{c}AY_{i}^{c}\,.\]
Since $\bD( \cY)$ is closed we can assume that $A$ is quasi-local. For every $U$ in $\C$   
 we can choose $i_{U}$ in $I$ so large that
$Y_{i_{U}}^{c}$ and $Z^{c}$ are $U$-separated.  Then  
we have
\[\lim_{U\in \cC}\|Y^{c}_{i_{U}}ZAZ^{c}\|=0\ , \quad \lim_{U\in \cC} \|Y_{i_{U}}Z^{c}AY_{i_{U}}^{c}\| =0\, .\qedhere\]
 \end{proof}
After these preparations we can now turn to the actual proof of  
 Proposition~\ref{ioejfiowefjweoifewuf8u3294823442}.  
 By Proposition \ref{qerijogergwregwregwerg9} stating that $\Kcat$ is exact we obtain the following commuting diagram
\begin{equation}\label{cwdnmcbwncwcjkcecw}
\xymatrix{\Kcat( \bD(Z\cap \cY))\ar[r]\ar[d]&\Kcat(\bD(Z) )\ar[d]\ar[r]&\Kcat(\bD(Z) /\bD(Z\cap \cY))\ar[d]_{\Kcat(\psi)}^{\simeq}\\
\Kcat(\bD( \cY ))\ar[r]\ &\Kcat(\bC_{\ql}^{*}( X )) \ar[r]&\Kcat(\bC_{\ql}^{*}(X)/\bD( \cY))}
\end{equation}
where the horizontal sequences are segments of fibre sequences.
The vertical morphisms are induced by the inclusions and the right vertical morphism is an equivalence by Lemma \ref{eriogegwefewrfwe}.
We conclude that the left square is a push-out square in $\Sp^{\la}$. 

We now show that this square is equivalent to the square \eqref{veer34534543541} which is then also a push-out square.
The equivalence is induced by the fillers of   the following big diagram:
$$
\xymatrix{\Kcat( \bD(Z\cap \cY))\ar[ddd]\ar[rrr]&&&\Kcat(\bD(Z))\ar[ddd]\\&\ar[ul]^{\simeq}\KX_{\ql} (Z\cap \cY)\ar[r]\ar[d]&{\KX}_{\ql} (Z) \ar[d]\ar[ur]_{\simeq}^{K(f_{*})}&\\&\KX_{\ql} ( \cY )\ar[r]\ar[dl]_{\simeq}&{\KX}_{\ql}(X) \ar@{=}[dr] &\\ \Kcat(\bD(\cY))\ar[rrr]&&&\KX_{\ql}(X) }$$
We   explain the lower left diagonal equivalence. 
 The inclusions $Y_{i}\to X$ induce unitary equivalences $\bC^{*}_{\ql}(Y_{i})\to \bD(Y_{i})$
which are natural in $i$. Consequently we have an equivalence 
$$\KX_{\ql} ( \cY )\simeq  \colim_{i\in I} \Kcat(\bC^{*}_{\ql}(Y_{i}))\simeq \colim_{i\in I}  \Kcat(\bD(Y_{i}))\stackrel{\text{Prop.}~\ref{erijgowegwergrefewrferfw}}{\simeq} \Kcat(\bD(\cY))\, .$$
The triangle \eqref{wvlkmflvkvadvadsvadvadv} yields the filler of the right square. Furthermore, 
 the family of triangles \eqref{wvlkmflvkvadvadsvadvadv} for $Y_{i}$ in place of $Z$ yields the filler of the lower square. 
 One gets the upper left diagonal equivalence and the upper square in a similar manner. 
 Finally the filler left square is induced from the   family of  commuting squares
 $$\xymatrix{\bD(Z\cap Y_{i}) \ar[d]&\bC^{*}_{\ql}(Z\cap Y_{i})\ar[l]\ar[d]\\\bD(Y_{i}) &\ar[l] \bC^{*}_{\ql}(Y_{i})}$$
 
%
%
%
This completes the proof of Proposition~\ref{ioejfiowefjweoifewuf8u3294823442}.
\end{proof}

Let $X$ be a bornological coarse space.
\begin{prop}\label{hewifeiiu3984u2398244244}
If $X$ is flasque, then
$\KX(X)\simeq 0$ and $\KXql(X)\simeq 0$.
\end{prop}

\begin{proof}
We discuss the quasi-local case. For $\KX$ the argument is similar.\footnote{A detailed argument in this case can also be found in \cite{Bunke:ad}.}  
  
Let $f:X\to X$ be a morphism which implements flasqueness (see Definition \ref{efijewifjewiofwifwfew322423424}).
We define a functor
$$S:\bC^\ast_{\ql}(X)\to \bC^\ast_{\ql}(X)$$ as follows:
\begin{enumerate}
\item On objects $S$ is given by $$S(H,\phi):= \Big( \bigoplus_{n\in \nat} H,\bigoplus_{n\in \nat} \phi\circ f^{n,*} \Big)\, .$$
\item On morphisms we set $$S(A):=\bigoplus_{n\in \nat}A\, .$$
\end{enumerate}
 {C}ondition \ref{fwoiejweoifeowi231} in Definition \ref{efijewifjewiofwifwfew322423424} ensures that $S(H,\phi)$ is locally finite. One furthermore checks in a straightforward manner that $S(H,\phi)$ is determined on points.
 
 If $A$ is $U$-quasi-local, then $S(A)$ is $\bigcup_{n\in \nat} (f^{n}\times f^{n})(U)$-quasi-local.
 By {C}ondition~\ref{sdfn3443243t61} of Definition~\ref{efijewifjewiofwifwfew322423424}  the operator $S(A)$ 
 belongs to $\bC^\ast_{\ql}(X)$.   
 
We now note that the $C^{*}$-category $\bC^{*}_{\ql}(X)$ is additive. The sum of two objects $(H,\phi)$ and $(H',\phi')$ is represented by $(H\oplus H',\phi\oplus \phi')$. As explained in Subsection \ref{ergioheriogergrewferfwef} we can form the direct sum of functors with values in $\bC^{*}_{\ql}(X)$.
 
We note the following 
relation
\begin{equation}\label{kdhewidhiefewfewfewfewf}
 \id_{\bC^{*}_{\ql}(X)}\oplus f_{*}\circ S\cong  S\, .
\end{equation}
of endofunctors of $\bC^{*}_{\ql}(X)$.
We now apply $\Kcat$ and use its additivity shown Proposition  \ref{rguiwtgwreergrgwrgg} in order to get 
\begin{equation*}\label{}
\id_{\KXql(X)}+ \KXql(f_{*})\circ \Kcat\circ S\simeq \Kcat \circ S\, .
\end{equation*}
Since $f_{*}$ is close to $ \id_{X}$ by Condition \ref{asdfgjtzkj561}  in Definition \ref{efijewifjewiofwifwfew322423424}
and $ \KXql$ is coarsely invariant by Proposition \ref{hewifeiiu3984u239824424} we conclude that 
$ \KXql(f_{*})\simeq \id_{\id_{\KXql(X)}}$. The resulting relation
\begin{equation}\label{}
\id_{\KXql(X)}+  \Kcat\circ S\simeq \Kcat \circ S
\end{equation}
in $\End_{\Sp^{\la}}(\KXql(X))$
implies that $\KXql(X)\simeq 0$.
\end{proof}

\begin{prop}\label{hewifeiiu3984u2398244241}
The functors $\KX$ and $\KXql$ are $u$-continuous.
\end{prop}

\begin{proof}
We consider the quasi-local case. For $\KX$ the argument is similar.\footnote{A detailed argument in this case can also be found in \cite{Bunke:ad}.}

Let $X$ be a bornological coarse space with coarse structure $\cC$.
We observe that the    $C^{*}$-categories $\bC^{*}_{\ql}(X)$ and $\bC^{*}_{\ql}(X_{U})$  have the same set of objects, see Remark \ref{weoigjoreogewrg9}.
Let $(H,\phi)$ and $(H^{\prime},\phi^{\prime})$ be $X$-controlled Hilbert spaces.
Then  $\Hom_{\bC_{\ql}^{*}(X_{U})}((H,\phi),(H^{\prime},\phi^{\prime}))$ can be identified with the subspace of 
$\Hom_{\bC_{\ql}^{*}(X)}((H,\phi),(H^{\prime},\phi^{\prime}))$ of all $U$-quasi-local operators. 
The $C^{*}$-category  $\bC_{\ql}^{*}(X)$ is then equal to  the completion of the union of these subcategories for all $U$ in $\cC$. 
 By Proposition \ref{erijgowegwergrefewrferfw} we conclude that 
$$  \colim_{U\in \cC}  \KXql(X_{U}) \simeq \KXql(C) $$ as required. 
\end{proof}
This finishes the proof of Theorem \ref{irgfjiwoowierer23423424243}.

A priori the condition of controlled propagation is stronger than quasi-locality. Hence for every bornological coarse space $X$ we have a 
  natural inclusion $$\bC^{*}(X)\to \bC^{*}_{\ql}(X)$$ of $C^{*}$-categories. It induces a  natural transformation between the coarse $K$-theory functors
\begin{equation}\label{wefh982ruz2kjefhewkfwwef}
\KX\to \KXql\, .
\end{equation}

\begin{lem}\label{iefjweoifoeiwfji23432423432}
If $X$ is discrete, then 
 $\KX(X)\to \KXql(X)$ is an equivalence.
\end{lem}

\begin{proof}
For a discrete bornological coarse space $X$ the natural inclusion $\bC^{*}(X)\to \bC_{\ql}^{*}(X)$ is an equality. 
Indeed, the diagonal is the maximal entourage of such a space. Consequently, the condition of quasi-locality reduces to the condition of controlled (in this case, actually zero) propagation.
\end{proof}

\subsection{Comparison with the classical definition}\label{ijfweoifweoifoi239804234342}

We consider a bornological coarse space $X$ and an $X$-controlled Hilbert space $(H,\phi)$. By $\cC^{*}(X,H,\phi)$  and  $\cC_{\ql}^{*}(X,H,\phi)$ we denote  the Roe algebras (Definition \ref{qekdfjqodqqwdqwdqwdqdwq}) associated to this data.  
Recall that these are the versions generated by locally finite operators of finite propagation (or quasi-local operators, respectively). 

 We let $K_{*}:=\pi_{*}\Kast\colon \Calg^{\la}\to \Ab^{\la,\Z\gr}$ denote the $\Z$-graded abelian group-valued $K$-theory functor for $C^{*}$-algebras (see Subsection \ref{wifjewiof2323443534}), and we let $   \KX_{\ql,*}:=\pi_{*}\KXql$ and $\KX_{*}:=\pi_{*}\KX$ denote the $\Z$-graded abelian group-valued functors derived from the coarse homology functors.
\begin{theorem}[Comparison Theorem]\label{fwefiwjfeiooi234234324434}
If  $(H,\phi)$ is ample, then  we have canonical isomorphisms $$K_{*}(\cC^{*}(X,H,\phi))\cong  \KX_{*}(X) \text{  {and} } K_{*}(\cC_{\ql}^{*}(X,H,\phi))\cong  \KX_{\ql,*}(X)\, .$$
\end{theorem}

\begin{rem}
In Theorem \ref{fwefiwjfeiooi234234324434e} below (and also \cite[Thm.~6.1]{indexclass}  in the equivariant case)  
we refine  the left isomorphism   to a natural equivalence of spectra. \hB
\end{rem}

\begin{proof}
The actual proof of Theorem \ref{fwefiwjfeiooi234234324434} requies a few preparations.
We discuss the case of the controlled-propagation Roe algebra. The quasi-local case is analogous.

We consider a bornological coarse space and an $X$-controlled Hilbert space $(H,\phi)$. Recall Definition~\ref{defn:sfd9823} of the notion of a locally finite subspace. 
Let $H^\prime$ and $H''$ be locally finite subspaces of $H$. Then we can form the closure of their algebraic sum in $H$ for which we use the notation
$$H^{\prime} \mathbin{\bar{+}} H^{\prime\prime}:=\overline{H' + H''}\subseteq H\, .$$
\begin{lem}\label{rojweoifwewffewf}
If 
$(H,\phi)$  is determined on points, then
$H^{\prime} \mathbin{\bar{+}} H^{\prime\prime}$  is a  locally finite subspace of $H$.
\end{lem}

\begin{proof}
By assumption we can find an entourage $V$ of $X$ and  $X$-controls 
 $\phi^{\prime}$ and $\phi^{\prime\prime}$  on $H^{\prime}$ and $H^{\prime\prime}$
such that $(H^{\prime},\phi^{\prime})$ and $(H^{\prime\prime},\phi^{\prime\prime})$ are locally finite $X$-controlled Hilbert spaces and the inclusions
$H^{\prime}\to H$ and $H^{\prime\prime}\to H$ have propagation controlled by   $V^{-1}$.

Fix a symmetric entourage $U$ of $X$. By 
  Lemma~\ref{lem:iohwenrf} we can choose a $U$-separated subset $D$ of   $X$ with $U[D] = X$. 
  
By the   axiom of choice we choose a well-ordering of $D$. We will construct the $X$-control~$\psi$ on $H^{\prime} \mathbin{\bar{+}} H^{\prime\prime}$  by a transfinite induction. In the following argument $f$ denotes a generic element of $C(X)$.

We start with the smallest element $d_0$ of $D$. We set $$B_{d_{0}}:=U[d_{0}]\, .$$
Because of the inclusion
$$\phi(B_{d_0}) H' \subseteq \phi'(V[B_{d_0}]) H^{\prime}$$   the subspace 
 $\phi(B_{d_0}) H'$ is finite-dimensional.  Analogously we conclude that $\phi(B_{d_0}) H''$ is finite-dimensional. Consequently,   $\phi(B_{d_0}) (H' + H'')=\phi(B_{d_0}) (H' \bar{+} H'')$ is finite-dimensional.
We define the subspace $$H_{d_{0}} := \phi(B_{d_0}) (H' + H'') \subseteq H$$ and let $Q_{d_{0}}$ be the orthogonal projection onto $H_{d_{0}}$. We further define the control $\psi_{d_{0}}$ of $H_{d_{0}}$ by $\psi_{d_{0}}(f) := f(d_0) Q_{d_{0}}$.  Using $\psi_{d_{0}}$ we recognize $H_{d_{0}}\subseteq H$  as a locally finite subspace.

Let $\lambda + 1$ in $D$ be a successor ordinal and suppose that the subset $B_{\lambda}$  of  $ U[\lambda]$,   the subspace $H_\lambda$ of $ H$, the projection $Q_{\lambda}$, 
and $\psi_\lambda:C(X)\to B(H_{\lambda})$ have already been constructed such that $\psi_{\lambda}$ recognizes $H_{\lambda}$ as a locally finite subspace of $H$.  Then we define the subset
$$B_{\lambda+1}:=U[\lambda+1]\setminus \bigcup_{\mu\le\lambda} B_{\mu}$$ of $X$.
 As above we observe that $\phi(B_{\lambda + 1})(H'  \mathbin{\bar{+}}  H'') \subseteq H$ is finite-dimensional.  We define the closed subspace $$H_{\lambda + 1} := H_\lambda + \phi(B_{\lambda + 1})(H'  \mathbin{\bar{+}}  H'')$$ of $H$.  We let  $Q_{\lambda + 1}$ be the orthogonal projection onto $H_{\lambda + 1} \ominus H_\lambda$ and define the control $\psi_{\lambda + 1}$ of $H_{\lambda + 1}$ by $$\psi_{\lambda + 1}(f) := \psi_\lambda(f) + f(d_{\lambda + 1})Q_{\lambda + 1}\, .$$ We conclude easily that it recognizes  $H_{\lambda + 1} $ as a locally finite subspace of $H $.

Let $\lambda$ be a limit ordinal.
Then we set $B_{\lambda}:=\emptyset$ and $Q_{\lambda}:=0$.  We  define the closed subspace $$H_\lambda := \overline{\bigcup_{\mu < \lambda} H_\mu}$$ of $H$.  We furthermore  define the control $\psi_\lambda$ of $H_\lambda$ by    $$\psi_\lambda(f) := \sum_{\mu < \lambda} f(d_\mu) Q_\mu\, .$$ We argue that the sum describes a well-defined operator on $H_{\lambda}$. The sum has a well-defined interpretation on $H_{\mu}$ for all $\mu<\lambda$.

Let $(h_{k})_{k\in \nat}$ be a sequence in $\bigcup_{\mu < \lambda} H_\mu $  converging to 
$h$.  For every $k$ in $\nat$  there is  a $\mu$ in $D$ with $\mu <\lambda$ such that
$h_{k}\in H_{\mu}$. We set $\psi_{\lambda}(f)h_{k}:=\psi_{\mu}(f)h_{k}$. This definition does not depend on the choice of $\mu$. 
Given $\epsilon$ in $(0,\infty)$ we can find $ k_{0}$ in $\nat$ such that  
$\|h_{k}-h\|\le \epsilon$ for all $k$ in $\nat$ with $k \ge k_{0}$. Then for integers $k,k'$ in $\nat$ with  $k,k^{\prime}\ge k_{0}$ we have 
 $\|\psi_{\lambda}(f)h_{k}-\psi_{\lambda}(f)h_{k^{\prime}}\|\le \epsilon\|f\|_{\infty}$.
This shows that $(\psi_{\lambda}(f)h_{k})_{k\in \nat}$ is a Cauchy sequence. We define
$$\psi_{\lambda}(f)h:=\lim_{k\to \infty} \psi_{\lambda}(f)h_{k}\, .$$
The estimates $\|\psi_{\lambda}(f)h_{k}\|=\|\psi_{\mu}(f)h_{k}\|\le \|f\|_{\infty}\|h_{k}\|$ for all $k$ in $\nat$ imply the estimate $\|\psi_{\lambda}(f)h\|\le \|f\|_{\infty}\|h\|$ and hence
show that $\psi_{\lambda}(f)$ is continuous. One furthermore checks, using the mutual orthogonality of the $Q_{\mu}$ for $\mu<\lambda$, that $f\mapsto \psi_{\lambda}(f)$ is a $C^{*}$-algebra homomorphism from $C(X)$ to $B(H_{\lambda})$.

   We claim that $ \psi_\lambda$ recognizes $H_\lambda$ as a locally finite subspace of $H $. If $B$    is a bounded subset of $X$, then the inclusion  $\psi_\lambda(B) H_\lambda \subseteq \phi(U[B]) (H'  \mathbin{\bar{+}}  H'')$ shows that  $\psi_\lambda(B) H_\lambda$ is finite-dimensional. Furthermore, the 
  propagation  of the inclusion of $H_{\lambda}$ into $H$ is controlled by   $U$.
  
By construction $(H_{\lambda},\phi_{\lambda})$ is determined on points.
\end{proof}

Let $(H,\phi)$ be an $X$-controlled Hilbert space.
The set of locally finite subspaces of $(H,\phi)$ is partially ordered by inclusion.
Lemma \ref{rojweoifwewffewf} has the following corollary.

\begin{kor}
If  $(H,\phi)$  is determined on points, then the partially ordered set of locally finite subspaces of an $X$-controlled Hilbert space is filtered.
\end{kor}
  
%
%

If $H^{\prime}$   is a locally finite subspace of $H$, then
the Roe algebra
$\cC^{*}(X,H^{\prime},\phi^{\prime})$ (as a subalgebra of $B(H')$) and the 
inclusion $\cC^{*}(X,H^{\prime},\phi^{\prime})\to \cC^{*}(X,H,\phi)$ 
 do not depend on the choice of $\phi^{\prime}$ recognizing $H^{\prime}$ as a locally finite subspace.  
 This provides the connecting maps of the colimit in the statement of the  following lemma. 

 \begin{lem}\label{irfjioweoi32424} We have a canonical isomorphism
$$K_{*}(\cC^{*}(X,H,\phi))\cong \colim_{H^{\prime}} 
K_{*}(\cC^{*}(X,H^{\prime},\phi^{\prime}))$$
where the colimit
runs over the filtered partially ordered
set of locally finite subspaces $H^{\prime}$ of $ H$.
 \end{lem}

\begin{proof}
If $A$ is a bounded linear operator on $H^{\prime}$, then the conditions of having controlled propagation as an operator on $(H,\phi)$ or on $(H^{\prime},\phi^{\prime})$
are equivalent. It follows that
$\cC^{*}(X,H^{\prime},\phi^{\prime})$ is exactly  the subalgebra generated by operators coming from $H^{\prime}$.
So by definition of the Roe algebra
$$\cC^{*}(X,H,\phi)\cong \colim_{H^{\prime}} \cC^{*}(X,H^{\prime},\phi^{\prime})\, ,$$
where the colimit is taken over the filtered partially ordered set of locally finite subspaces of $H$ and interpreted in the category of $C^{*}$-algebras. Since the $K$-theory functor for $C^{*}$-algebras  preserves filtered colimits by Property \ref{wtgwergfervfdsv} of $\Kast$
we conclude the desired isomorphisms.
\end{proof}
 
 We now continue  with  the proof of Theorem \ref{fwefiwjfeiooi234234324434}.
By Proposition \ref{prop:dfs78934} and Definition~\ref{oiwefewiofu9234234234324} 
 we have an isomorphism
\begin{equation}
\label{eq:sdf90823}
\KX_{*}(X)\cong  {K_{*}}(A (\bC^{*}(X))) \, .
\end{equation}

The same argument as for the isomorphism \eqref{rgoijioqjefefqwefqwefqewfqf} shows that we 
have a canonical isomorphism
\begin{equation}
\label{eqi2lkjh43re}
 \colim_{\bD} A(\bD)  \cong  A(\bC^{*}(X))  \, ,
\end{equation}
where $\bD$ runs over the filtered subset of  subcategories of $\bC^{*}(X)$ with finitely many objects. Since the $K_{*}$- commutes with filtered colimits 
we get the canonical isomorphism
\begin{equation*}
\label{eqi2lkkjh6jh43re}
 \colim_{\bD} K_{*}(A(\bD))  \cong K_{*}(A(\bC^{*}(X))) \, .
\end{equation*}

If $H^{\prime}$ is a locally finite subspace of $H$, then we can consider $\cC(X,H^{\prime},\phi^{\prime})$
as a subcategory of $\bC^{*}(X)$ with a single object $(H^{\prime},\phi^{\prime})$. These inclusions induce the two up-pointing homomorphisms in  {D}iagram \eqref{udheuqi32} below.

Let $H'$ and $H''$ be closed locally finite subspaces of $H$ such that
$H^{\prime}\subseteq H^{\prime\prime}$. 
\begin{lem}\label{fhuiewqewwerr}
The following diagram commutes:
\begin{equation}\label{udheuqi32}
 \xymatrix{&K_{*}(A(\bC^{*}(X))) &\\K_{\ast}(\cC^{*}(X,H^{\prime},\phi^{\prime}))\ar[ur]\ar[rr]\ar[dr]&&K_{*}(\cC^{*}(X,H^{\prime\prime},\phi^{\prime\prime}))\ar[ul]\ar[dl]\\&K_{*}(\cC^{*}(X,H ,\phi )) &}
\end{equation}
\end{lem}

\begin{proof}
The {commutativity of the} lower triangle is clear since it is induced from a commutative triangle of Roe algebras.
We now consider the upper triangle. Let $i:\cC^{*}(X,H^{\prime},\phi^{\prime})\to A(\bC^{*}(X))$, $j:\cC^{*}(X,H^{\prime\prime},\phi^{\prime\prime}))\to A(\bC^{*}(X))$, and $k:\cC^{*}(X,H^{\prime},\phi^{\prime})\to  (X,H^{\prime\prime},\phi^{\prime\prime}))$ denote the inclusions. 
We consider the  inclusion  $u:H'\to H''$ in $\Hom_{\bC^{*}(X)}(H',\phi'),(H'',\phi''))$   as a morphism in $A(\bC^{*}(X))$ in the natural way, it lives in single summand of \eqref{efefwefeefefefewfef}.
Then we have equalities of homomorphisms $i=i u^{*}u$ and $i':=uiu^{*}=j\circ k$ from $ \cC^{*}(X,H^{\prime},\phi^{\prime})$ to $A(\bC^{*}(X))
$.
By Property 
  \ref{rgoijwergioregrerfwref} of $\Kast$ the maps $i$ and $i'$ induce the same maps on the level of $K$-theory groups.
Hence the upper triangle commutes as well.
\end{proof}

We continue with the proof of Theorem \ref{fwefiwjfeiooi234234324434}.
By a combination of Lemma \ref{irfjioweoi32424} and Lemma~\ref{fhuiewqewwerr} we get a canonical homomorphism
\begin{equation}\label{delkdjoi3242342344}
\alpha:K_{*}(\cC^{*}(X,H,\phi))\to K_{*}(A(\bC^{*}(X)))\, . 
\end{equation} 

\begin{lem}\label{eifjwefioiofioefewf}
The  homomorphism \eqref{delkdjoi3242342344} is an isomorphism. 
\end{lem}

\begin{proof}
Let $x$ in $K_{*}(\cC^{*}(X,H,\phi))$ be given such that $\alpha(x)=0$. Then  there exists a locally finite subspace $H^{\prime}$ of $H$ such that 
 $x$ is realized by some $x^{\prime}$ in $K_{*}(\cC^{*}(X,H^{\prime},\phi^{\prime}))$. Furthermore there exists a subcategory
$\bD$ of $\bC^{*}(X)$  with finitely many objects  containing $(H^{\prime},\phi^{\prime})$ such that
$x^{\prime}$ maps to zero in $K_{*}(A(\bD))$.
Using ampleness of $(H,\phi)$, by Lemma \ref{rwoijwoifjoeiwewfjewjfo} we can extend the embedding $H^{\prime}\to H$ to an embedding
$H^{\prime\prime}:=\bigoplus_{(H_{1},\phi_{1}) \in \Ob(\bD)}H_{1}\to H$
as a locally finite subspace containing $H^{\prime}$.  Then $x^{\prime}$ maps to zero in
$K_{*}(\cC(X,H^{\prime\prime},\phi^{\prime\prime}))$. Hence $x=0$.

Let now $y$ in $K_{*}(A(\bC^{*}(X)))$ be given. Then there exists a subcategory with finitely many objects
$\bD$ of $\bC^{*}(X)$ and a class $y^{\prime}$ in $K(A(\bD))$ mapping to $y$.
By Lemma \ref{rwoijwoifjoeiwewfjewjfo} we can choose an embedding of $H^{\prime }:=\bigoplus_{(H_{1},\phi_{1}) \in \Ob(\bD)}H_{1}\to H$ as a locally finite subspace. Then $y^{\prime}$ determines an element $x$ in $K_{*}(\cC^{*}(X,H,\phi))$ such that
$\alpha(x)=y$. 
\end{proof}

 {The combination of  Lemma  \ref{eifjwefioiofioefewf} with Equation \eqref{eq:sdf90823} finishes the proof of Theorem \ref{fwefiwjfeiooi234234324434}.}
\end{proof}


\begin{rem}
A version of the Comparison Theorem \ref{fwefiwjfeiooi234234324434} has been discussed in \cite[Sec.~2.2]{hamped}.
\hB
\end{rem}

In combination with Proposition \ref{lem:sdfbi23}  the Theorem \ref{fwefiwjfeiooi234234324434} implies  the following comparison.
Let $X$
 be a bornological coarse space and $(H,\phi)$ an $X$-controlled Hilbert space.
 \begin{kor}\label{iofewoifu9234234234324}
Assume that
\begin{enumerate}
\item $X$ is separable (Definition \ref{weofweoifu9824234343224}),
\item {$X$ is locally finite (Definition~\ref{defn:okmnsf}), and}
\item $(H,\phi)$ is ample (Definition \ref{fjwefewiojoi2jroi23jr23r23r23r}).
\end{enumerate}
Then we have a canonical isomorphism
$$K_{*}(\cC_{\lc}^\ast(X,H,\phi))\simeq \KX_{*}(X)\, .$$
\end{kor}

\begin{ex} Combining the Properties \ref{wergijowergergrgregwefewrfwerf} and \ref{wegoijweroigjergwerg} of $\Kast$ we conclude $\Kast(\IK(\ell^{2}))\simeq  KU$.

We can consider $\ell^{2}$ as an ample $*$-controlled Hilbert space on {the space} $*$.  In this case all four versions of the Roe algebras coincide with $\IK(\ell^{2})$.
We get
\[\KX(*)\simeq \KXql(*)\simeq \Kast(\IK(\ell^{2}))\]
and hence $\KX(*)$ and $\KXql(*)$ are equivalent to $KU$.
\hB
\end{ex}

  Let $X$ be a bornological coarse space.
  \begin{theorem}\label{jfweofjwoeifjewfoewfewfewfewf54tt34t}
 Assume that $X$ has weakly finite asymptotic dimension (Definition \ref{wegoijobgwtwferfrewfer}).
 Then:
\begin{enumerate}
\item \label{fewijwefoew234234} The canonical transformation $\KX(X)\to \KXql(X)$ is an equivalence.
\item\label{fewijwefoew2342341} If $(H,\phi)$ is an ample $X$-controlled Hilbert space, then the inclusion of Roe algebras
$\cC^{*}(X,H,\phi)\hookrightarrow \cC_{\ql}^{*}(X,H,\phi)$ induces an  equivalence
\[K(\cC^{*}(X,H,\phi))\to K(\cC_{\ql}^{*}(X,H,\phi))\]
of $K$-theory spectra.
\end{enumerate}
 \end{theorem}

\begin{proof}
By Theorem \ref{irgfjiwoowierer23423424243} the functors $\KX$ and $\KXql$ are coarse homology theories in the sense of Definition \ref{rgljogreggregrege}.
Assertion~\ref{fewijwefoew234234} follows from Lemma \ref{iefjweoifoeiwfji23432423432} together with Corollary \ref{thm:sdf98245csv}.

Assertion~\ref{fewijwefoew2342341} follows from \ref{fewijwefoew234234}
and Theorem \ref{fwefiwjfeiooi234234324434}.
\end{proof}

For the study of secondary or higher invariants in coarse index theory it is important to refine the isomorphism given in Theorem \ref{fwefiwjfeiooi234234324434} to the spectrum level. {This is the content of Theorem~\ref{fwefiwjfeiooi234234324434e}.\ref{0923erfd} below.}

Let  $X$, $X^{\prime}$ be bornological coarse spaces. Let furthermore  
$(H,\phi)$ be an  ample $X$-controlled Hilbert space and
$(H^{\prime},\phi^{\prime})$ be an ample $X^{\prime}$-controlled Hilbert space. Assume that we are given a morphism $f : X^\prime \to X$ and an isometry $V:H^\prime \to H$. In the following in order to interpret  $\supp(V)$  we view $V$ as an operator from $f_{*}(H,\phi)$ to $(H',\phi')$. We assume that 
 there exists an entourage $U$ of $X$ such that $\supp(V)\subseteq U$.
Then we define a morphism of $C^{*}$-algebras by
$$v:\cC^{*}(X^{\prime},H^{\prime},\phi^{\prime})\to \cC^{*}(X,H,\phi)\, , 
\quad v(A):=VAV^{*}\, .$$
If $A$ has $W$-controlled propagation, then $VAV^{*}$ has $U\circ (f\times f)(W)\circ U^{-1}$-controlled propagation. 
One furthermore checks that the conjugation $V-V^{*}$ preserves local finiteness.
The only relation between $v$ and $f$ is the support condition on $V$. If $V_1$, $V_2$ are two isometries $H^\prime \to H$ satisfying  $\supp(V_{1})\subseteq U$ and $\supp(V_{2})\subseteq U$, then the induced maps $K_\ast(v_1)$ and $K_\ast(v_2)$ on $K$-theory groups are the same. To see this, following  \cite[Lem.~3 in Sec.~4]{higson_roe_yu}, 
we note that $ V_{2}V_{1}^{*}$ is a multiplier on $\cC^{*}(X,H,\phi)$, and that 
$v_{1} (V_{2}V_{1}^{*})^{*} (V_{2}V_{1}^{*})=v_{1}$ and $v_{2}= (V_{2}V_{1}^{*})v_{1} (V_{2}V_{1}^{*})^{*}$.
We now use Property \ref{rgoijwergioregrerfwref} of $\Kast$.
The refinement of  the equality  $K_\ast(v_1)=K_\ast(v_2)$ to the spectrum level is content of Point~\ref{34ertfds} in the next theorem.

\begin{theorem}\label{fwefiwjfeiooi234234324434e}
\mbox{}
\begin{enumerate}
\item\label{0923erfd} There is a canonical (up to equivalence) equivalence of spectra  
$$\kappa_{(X,H,\phi)}:\Kast(\cC^{*}(X,H,\phi))\to \KX(X)\, .$$
\item\label{34ertfds} We have a commuting diagram
\begin{equation}\label{coiehjocwecwecc}
\xymatrix{\Kast(\cC^{*}(X^{\prime},H^{\prime},\phi^{\prime})) \ar[rr]^-{\kappa_{(X^{\prime},H^{\prime},\phi^{\prime})}}\ar[d]^{ \Kast(v)}&&  \KX(X^{\prime})\ar[d]^{f_{*}}\\ \Kast(\cC^{*}(X,H,\phi))   \ar[rr]^-{\kappa_{(X,H,\phi)}}&&\KX(X)}
\end{equation}
\end{enumerate}
\end{theorem}

\begin{proof}
We start with the proof of the first assertion. We consider the category $\bC^{*}(X,H,\phi)$ whose objects are triples 
$(H^{\prime},\phi^{\prime},U)$\footnote{This interpretation of the symbols $H^{\prime}$ and $\phi^{\prime}$ is  local in the proof of the first part of the theorem.}, where   $(H^{\prime},\phi^{\prime})$ is an object of $\bC^{*}(X)$ and $U$ is a  controlled isometric embedding
$U:H^{\prime}\to H$ as a locally finite subspace. 
A morphism $A:(H_{0}^{\prime},\phi_{0}^{\prime},U_{0})\to (H_{1}^{\prime},\phi_{1}^{\prime},U_{1})$
is an operator in $\Hom_{\bC^{*}(X)}((H_{0}',\phi_{0}'),(H_{1}',\phi_{1}'))$.

We consider  the Roe algebra $\cC^{*}(X,H,\phi)$ as a non-unital  $C^{*}$-category  with a single object.
Then we have  functors between non-unital $C^{*}$-categories
$$\bC^{*}(X) \xleftarrow{F}  \bC^{*}(X,H,\phi) \xrightarrow{I} \cC^{*}(X,H,\phi)\, .$$
The functor $F$ forgets the inclusions. The action of $I$ on objects is clear, and it maps a morphism $A:(H^{\prime},\phi^{\prime},U^{\prime})\to  (H^{\prime\prime},\phi^{\prime\prime},U^{\prime\prime})$  in   $\bC^{*}(X,H,\phi)$ to the morphism   $I(A):=U^{\prime\prime}AU^{\prime,*}$  of  $C^{*}(X,H,\phi)$. 

We get an induced diagram
of $K$-theory spectra 
$$\Kcat( \bC^{*}(X))  \xleftarrow{\Kcat(F)} \Kcat (\bC^{*}(X,H,\phi) ) \xrightarrow{\Kcat(I)} \Kast (\cC^{*}(X,H,\phi))\, .$$

\begin{lem}
$K(F)$ is an   equivalence.
\end{lem}
\begin{proof}
We claim that $F$ is a unitary equivalence of $C^{*}$-categories. Then $\Kcat(F)$ is an equivalence of spectra by  Corollary \ref{ojewjfowfewfewfewf1}.

Note that $F$ is fully faithful.
Furthermore, since $(H,\phi)$ is ample, for every object $(H^{\prime},\phi^{\prime})$ of $\bC^{*}(X)$ there exist
a controlled isometric embedding $U:H^{\prime}\to H$ {by Lemma~\ref{rwoijwoifjoeiwewfjewjfo}}. Hence the object $(H^{\prime},\phi^{\prime})$ is in the image of $F$. So $F$ is surjective on objects.
\end{proof}

\begin{lem}
$\Kast(I)$ is an equivalence.
\end{lem}
\begin{proof}
We show that
$K_{*}(I)$  is isomorphism on homotopy groups.

\textbf{Surjectivity:} Let $x$ be in $ K_{*}(\cC^{*}(X,H,\phi))$. By Lemma \ref{irfjioweoi32424} there exists a locally finite subspace
$H^{\prime}$ of $(H,\phi)$ and a class $x^{\prime}$ in $  K_{*}(\cC^{*}(X,H^{\prime},\phi^{\prime}))$
such that $K_{*}(u^{\prime})(x^{\prime})=x$ for 
 the canonical homomorphism $u^{\prime}: \cC^{*}(X,H^{\prime},\phi^{\prime})\to \cC^{*}(X,H,\phi)$ which  sends the operator $A$ in its domain to $U^{\prime}AU^{\prime,*}$. Here $\phi^{\prime}$ is  an $X$-control on $H^{\prime}$ which exhibits it as a locally finite subspace of {$H$}, and $U^{\prime}:H^{\prime}\to H$ is the isometric embedding. We have a commuting diagram of non-unital $C^{*}$-categories 
$$\xymatrix{\cC^{*}(X,H^{\prime},\phi^{\prime})\ar[dr]^{a}\ar[rr]^{u^{\prime}}&&\cC^{*}(X,H,\phi)\\&\bC^{*}(X,H,\phi)\ar[ur]^{I}&}$$
where $a$ sends the unique object of its domain to $(H^{\prime},\phi^{\prime},U)$ and is the obvious map on morphisms.  This factorization implies that $x$ is also in the image of $K_{*}(I)$.

\textbf{Injectivity:} Assume that $ x$ in $ K_{*}(\bC^{*}(X,H,\phi))$ is such that $K_{*}(I)(x)=0$. There is a subcategory $\bD$ of  $\bC^{*}(X,H,\phi)$ with finitely many objects and a class
$y$ in $ K_{*}(\bD)$ with $x=K_{*}(i)(y)$, where $i:\bD\to \bC^{*}(X,H,\phi)$ is the inclusion. {This follows from  Proposition~\ref{wegkoerfrefewrfwerf}.}

We define
$$\tilde H
:=\sum_{(H^{\prime},\phi^{\prime},U^{\prime})\in \bD} H^{\prime}+ H_{1}$$
with the sum taken in $H$, where $H_{1}$ is a locally finite subspace which will be chosen below.
By Lemma \ref{rojweoifwewffewf}  we know that $\tilde H$ is a locally finite subspace of $H$. Hence we can choose a control function
$\tilde \phi$ on $\tilde H$ exhibiting $\tilde H$ as a locally finite subspace and let $\tilde U:\tilde H\to H$ denote the inclusion.
We form the full subcategory category $\tilde \bD$  of  $\bC^{*}(X,H,\phi)$
 with set of objects $\Ob(\bD)\cup \{(\tilde H,\tilde \phi,\tilde U)\}$. 
We have a functor $\tilde I:\bD\to \cC^{*}(X,\tilde H,\tilde   \phi)$ defined similarly as
$I$ and a diagram
$$\xymatrix{
& K_{*}(\bC^{*}(X,H,\phi))\\
 K_{*}(\bD)\ar[r] \ar[d]^{K_{*}(\tilde I)}\ar[ur]^{K_{*}(i)}\ar@/_3cm/[dd]_{K_{*}(I)}& K_{*}(\tilde \bD) \ar[u]\\
 K_{*}(\cC^{*}(X,\tilde H,\tilde \phi))\ar[ur]^{!}\ar[d]^{{K_{*} (\tilde U)}}&\\
 K_{*}(\cC^{*}(X,H,\phi))&}$$
The upper  right triangle is given by the obvious inclusions of $C^{*}$-categories.
Also the left triangle commutes on the level of $C^{*}$-categories.
The middle triangle commutes by an argument which is similar to the argument for the upper triangle in \eqref{udheuqi32},
where the marked morphism is induced by the inclusion $\cC^{*}(X,\tilde H,\tilde \phi)\to \tilde \bD$ of $C^{*}$-categories.

Now by Lemma \ref{irfjioweoi32424}, since $K_{*}(I)(y)=0$, we can choose $H_{1}$ so large such that  $K_{*}(\tilde I)(y)=0$. This implies $x=K_{*}(i)(y)=0$.
\end{proof}

   By definition $\KX
(X) \simeq \Kcat(\bC^{*}(X))$ and so the composition  of an inverse $\Kcat(I)^{-1}$ with $\Kcat(F)$ provides the asserted equivalence
$$\kappa_{(X,H,\phi)}:\Kast(\cC^{*}(X,H,\phi))\to \KX(X)\, .$$
We now show    the second assertion of Theorem~\ref{fwefiwjfeiooi234234324434e}. 
  We have the induced morphism of $C^{*}$-algebras
$$v:\cC^{*}(X^{\prime},H^{\prime},\phi^{\prime})\to \cC^{*}(X,H,\phi)\, , 
\quad v(A):=VAV^{*}\, .$$
The square \eqref{coiehjocwecwecc} 
 is induced from a commuting (in the one-categorical sense) diagram of $C^{*}$-categories 
$$\xymatrix{\cC^{*}(X^{\prime},H^{\prime},\phi^{\prime} )\ar[d]^{ v}&\bC^{*}(X^{\prime},H^{\prime},\phi^{\prime})\ar[l]_-{{I^\prime}}\ar[r]^-{{F^\prime}}\ar[d]&  \bC(X^{\prime})\ar[d]^{f_{*}}\\
\cC^{*}(X,H,\phi)   & \bC^{*}(X,H,\phi)\ar[l]_-{{I}}\ar[r]^-{{F}}&\bC(X)}$$
where all horizontal arrows induce equivalences in $K$-theory, and the middle arrow
is the functor which sends  the object {$(H^{\prime\prime},\phi^{\prime\prime},U^{\prime\prime})$ to $(H^{\prime\prime},\phi^{\prime\prime}\circ f^{*},V\circ U^{\prime\prime})$} and is the identity on morphisms. 
It is then clear that the right square commutes.
One also checks directly that the left square commutes. 

This finishes the proof of  Theorem \ref{fwefiwjfeiooi234234324434e}.
\end{proof}

\subsection{Additivity and coproducts}\label{jnk239s23}

In this section we investigate how coarse $K$-homology interacts with coproducts and free unions. 
We refer to 
Section \ref{jkndsfjh23} for definitions. 
\subsubsection{Additivity}

 \newcommand{\B}{\mathbb{B}}
 \newcommand{\K}{\mathbb{K}}
 
In order to approach additivity  and strong additivity for coarse $K$-homology we need results stating that
the $K$-theory for $C^{*}$-algebras preserves   infinite products at least under certain additional conditions.

Let $(A_i)_{i\in I}$ be a family of   $C^\ast$-algebras.
The  projections $p_{j}:\prod_{i\in I} A_{i}\to A_{j}$ for all $j$ in $I$
give a morphism of   spectra
\begin{equation}\label{vadsvoij3oigferfgfa}
(\Kast(p_{i}))_{i\in I}:\Kast(\prod_{i\in I}A_{i})\to \prod_{i\in I} \Kast(A_{i})\, .
\end{equation}

Then we ask for conditions on the factors ensuring that \eqref{vadsvoij3oigferfgfa} is an equivalence.

\begin{rem}
In order to understand one of the {difficulties} note that
$K$-theory classes are represented by projections or unitaries in matrix algebras. Matrices with coefficients in $\prod_{i\in I}A_{i}$   correspond to families of matrices in the factors of uniformly bounded size. On the other hand,   a class in $\prod_{i\in I} \Kast(A_{i})$ is a  family of $K$-theory classes for the factors. If   the index set $I$ is infinite, the members of this family 
 might  not be representable by a family of matrices of uniformly bounded size, see Remark \ref{erwgopijergwer} and the references give there. 

Therefore conditions which help to avoid matrices are useful. 
$\K$-stability is such a condition. Recall that a $C^{*}$-algebra is called $\K$-stable if $A\cong A\otimes \K$, where $\K:=\K(\ell^{2})$.
But note that an infinite product of $\K$-stable $C^{*}$-algebras might not be $\K$-stable again.
The notion of quasi-$\K$-stability introduced in Definition \ref{wtiogjweogfrefwwerfwreferfsvfvsd}  is better behaved in this respect by Lemma \ref{sojoferfdcadsc}. 
\hB
\end{rem}

Let $A$ be a $C^\ast$-algebra. We let $M_{n}(A)$ denote the $n\times n$-matrix algebra over $A$. Furthermore, 
for a $C^{*}$-algebra $B$ we let $\cM(B)$ denote the multiplier algebra of $B$.

\begin{ddd}[{\cite[Defn.\ 2.7.11]{willett_yu_book}}]\label{wtiogjweogfrefwwerfwreferfsvfvsd}
$A$ is called \emph{quasi-$\K$-stable}\index{quasi-stable} if for every $n$ in $\nat$  there exists an isometry $v$ in $\cM(M_n(A))$  such that $v v^\ast$ is the elementary matrix $e_{11}$.
\end{ddd}

\begin{rem}\label{rem_isometry_v}
Let $v$ be an isometry as in  {Definition \ref{wtiogjweogfrefwwerfwreferfsvfvsd}}.
Because of the relation   $e_{11} v = v v^* v = v$ we know that   $v$ is of the form
\[
v =
\begin{pmatrix}
v_{11} & v_{12} & \cdots & v_{1n}\\
0 & 0 & \cdots & 0\\
\vdots & \vdots & \ddots & \vdots\\
0 & 0 & \cdots & 0
\end{pmatrix}\, ,
\]
where $v_{1i}$ belongs to  $ \cM(A)$ for every $i$  in $\{1,\dots,n\}$. 
Here  we use the compatibility of  multiplier algebras with forming matrices: $\cM(M_n(A)) = M_n(\cM(A))$.
Because  
  $v$ is an isometry, i.e., $v^* v = \id_{\cM(M_n(A))}$, and since   $v v^* =e_{11}$,  we furthermore have the relations
$$v_{1i}^{*}v_{1j}=\delta_{ij}1_{\cM(A)}\, , \quad \sum_{i=1}^{n} v_{{1i}}v_{1i}^{*}=1_{\cM(A)}\, .$$ 
In particular, $v_{1i}$ is an isometry for every  $i$  in $\{1,\dots,n\}$. 

Let 
 $x$ be any element of $M_n(A)$. Then $v x v^*$ is supported only in the upper left corner of $M_n(A)$ and hence may be identified canonically with an element of $A$. This follows from the equality $$e_{11} v x v^*   e_{11}= v v^* v x v^*  v v^*  = v x v^*\, .$$ 
where we 
  use  that $v$ is an isometry, i.e., $v^* v = \id_{\cM(M_n(A))}$, and that   $v v^* =e_{11}$.

 
{Let  $\iota_n\colon A \to M_n(A)$  denote the inclusion as the upper left corner. Let furthermore  
 $\pi_{11}\colon M_n(A) \to A$ be the $\C$-linear projection onto this corner.} 
The $\C$-linear map $A \to A$ given by  $x \mapsto \pi_{11}( v \iota_n(x) v^* )$
is a homomorphism of $C^{*}$-algebras. It  is equal to the conjugation map \begin{equation}
\label{eq_conjugation_v}
x \mapsto v_{11} x v_{11}^*\,,
\end{equation} by the isometry  $v_{11}$  in $\cM(A)$.
%
\hB
\end{rem}

\begin{ex}\label{ex_BH_quasi_stable}
Let $H$ be an infinite-dimensional Hilbert space. The $C^\ast$-algebra $\mathbb{K}(H)$ of all compact  operators on $H$ is quasi-$\K$-stable. In order to see this 
we identify $M_n(\mathbb{K}(H))$ with $\B(H^{\oplus n})$. Since  $H$ is infinite-dimensional we can choose  a unitary $u\colon H^{\oplus n} \to H$. We let   $\iota\colon H \to H^{\oplus n}$ be  the inclusion as the first summand. Then the   isometry $v:=\iota \circ u$ has the required properties.

The same isometries also exhibit $\B(H)$ as a quasi-$\K$-stable $C^{*}$-algebra.
\hB
\end{ex}

\begin{rem}
An elaboration of Example~\ref{ex_BH_quasi_stable} shows that every $\K$-stable $C^\ast$-algebra is quasi-$\K$-stable.

Furthermore, since the $C^\ast$-algebra $\B(H)$ of all bounded linear operators on an infinite-dimensional Hilbert space $H$ is not $\K$-stable, we conclude that the notion of quasi-$\K$-stability is strictly more general than the notion of $\K$-stability.
\hB
\end{rem}

Let $(A_{i})_{i\in I}$ be a family  of $C^{*}$-algebras.
\begin{lem}\label{sojoferfdcadsc}
If $A_{i}$ is quasi-$\K$-stable for every $i$ in $I$, then 
$\prod_{i\in I} A_{i}$ is quasi-$K$-stable.
 \end{lem}
 \begin{proof}
 Let $n$ be  in $\nat$. For every $i$ in $I$ we choose an isometry $v_{i}$ in $\cM(M_{n}(A_{i}))$
 as in the Definition \ref{wtiogjweogfrefwwerfwreferfsvfvsd}. The family $(v_{i})_{i\in I}$ gives rise to an element $v$ in $\cM(M_{n}(\prod_{i\in I} A_{i}))$ satisfying the condition of Definition  \ref{wtiogjweogfrefwwerfwreferfsvfvsd}.
 \end{proof}

 Let $(A_{i})_{i\in I}$ be a family  of $C^{*}$-algebras.

\begin{prop}\label{ejioerigjregregwer9}
If $A_{i}$ is quasi-$\K$-stable for every $i$ in $I$, then
  \begin{equation}\label{vfdvioefjvoijvkdlvfvfvavdscda}
(\Kast(p_{i}))_{i\in I}:\Kast(\prod_{i\in I}A_{i})\to \prod_{i\in I} \Kast(A_{i})\, .
\end{equation}
is an equivalence.
\end{prop}
\begin{proof}
We use the notation  $K_{*}=\pi_{*}\Kast$ for the group-valued $K$-theory functor. 
By Bott periodicity it suffices to show that 
 \begin{equation}\label{vfdvioefjvoijvkdlvfvfvavdscdea}
(K_{*}(p_{i}))_{i\in I}:K_{*}(\prod_{i\in I}A_{i})\to \prod_{i\in I} K_{*}(A_{i})
\end{equation}
is an isomorphism for $*=0,1$.

  In the following we identify a $C^{*}$-algebra $A$ with a subalgebra of $M_{n}(A)$
using the left-upper corner inclusion. 

We let $A^{+}$ denote the unitalization of $A$. We let  further   $e:A^{+}\to \C$ be the canonical projection such that $\ker(e)=A$.
We use the same symbol for the extension of  $e$ to matrix algebras.
For simplicity we will not add a subscript indicating that $e$ is associated with $A$.

In general, an element $x$   in $K_{0}(A)$  can represented by a pair of projections $(p,q)$ in $M_{n}(A^{+})$ for some $n$ in $\nat$
 such that $e(p)=e(q) $. We write $x=[p,q]$. If $A$ is quasi-$\K$-stable and
  $v$ is the isometry as in Definition \ref{wtiogjweogfrefwwerfwreferfsvfvsd}, then using Property  \ref{rgoijwergioregrerfwref}  of $\Kast$ we
 see that  $x=[vpv^{*},vqv^{*}]$.
 Note that $vpv^{*}$ and $vqv^{*}$ are projections in $A^{+}$ such that $e(vpv^{*})=e(vqv^{*}) $. 
 Therefore, if $A$ is quasi-$\K$-stable, then every element in $K_{0}(A)$ can actually  be represented  by a  pair $(p,q)$ of  projections $p,q$ in $A^{+}$ with $e(p)=e(q)$.

Let $p,q$ be projections in $A^{+}$ such that $e(p)=e(q)$.
If $[p,q]=0$, then after increasing $n$ (and filling matrices up by zero) if necessary, there exists a partial isometry $u$ in $M_{n}(A^{+})$ such that $p=uu^{*}$ and $q=u^{*}u$.
 We define the partial isometry $w:=puq$ in $A^{+}$. Then $p=ww^{*}$ and $q=w^{*}w$.

 An element $y$ in $K_{1}(A)$ can be represented by a unitary $u$ in $M_{n}(A^{+})$ for some $n$ in $\nat$ such that $e(u)=1_{M_{n}(\C)}$.   If $A$ is quasi-$\K$-stable and $v$ is the isometry as in Definition \ref{wtiogjweogfrefwwerfwreferfsvfvsd},  then {$u^\prime \coloneqq v(u-1_{M_{n}(A^{+})})v^{*} + 1_{A^+}$} is a unitary in $A^{+}$ (where we use that  the matrix $v(u-1_{M_{n}(A^{+})})v^{*}$ is supported  in the left upper corner and is thus considered as  an element in $A$, see Remark~\ref{rem_isometry_v}) with $e( {u^\prime})=1$. By  Property  \ref{rgoijwergioregrerfwref}  of $\Kast$  we have  $y=[u^\prime]$.
Hence,  if $A$ is quasi-$\K$-stable, then  every element in $K_{1}(A)$ can be represented by a unitary $u$ in $A^{+}$ with $e(u)=1$.

After these preparations  we now turn to the actual proof of Proposition \ref{vfdvioefjvoijvkdlvfvfvavdscda}.

\emph{surjectivity on $K_0$:}

An element  $x=(x_{i})_{i\in I}$ of $\prod_{i \in I} K_0(A_i)$ may be represented by a family  of pairs $((p_{i} ,q_i))_{i\in I}$ of projections $p_i,q_i $ in $A_i^{+}$ such that $e(p_{i})=e(q_{i}) $ for every $i$ in $I$. Since $e(p_{i})$ is a projection we have $e(p_{i})\in \{0,1\}$.
   For $k$ in $\{0,1\}$ we define the subsets
$I_{k}:=\{i\in I\:|\:   e(p_{i})=k\}$ of $I$. Then $(I_{0},I_{1})$ is a partition of $I$.  For $k$ in $\{0,1\}$ we form the familes $(p_{k,i})_{i\in I}$ and $(q_{k,i})_{i\in I}$ 
$$p_{k,i}:=\left\{\begin{array}{cc} p_{i}-k1_{A^{+}}&i\in I_{k}\\0&i\not\in I_{k}
\end{array}\right.\ , \quad q_{k,i}:=\left\{\begin{array}{cc} q_{i}-k1_{A^{+}}&i\in I_{k}\\0&i\not\in I_{k}
\end{array}\right. $$
Note that $p_{k,i},q_{k,i}\in A_{i}$ for all $i$ in $I$.
 Then we define  the elements 
 $$P_{k}:=(p_{k,i})_{i\in I}+ k1_{(\prod_{i\in I} A_{i})^{+}}\ , \quad Q_{k}:= (q_{k,i})_{i\in I}+ k1_{(\prod_{i\in I} A_{i})^{+}}$$
in $(\prod_{i\in I} A_{i})^{+}$. 
One checks that $P_{k}$ and $Q_{k}$ are projections, and that $e(P_{k})=e(Q_{k})$.
Hence we get classes $[P_{k},Q_{k}]$  in $K_{0}( \prod_{i\in I}A_{i})$ for $k$ in $\{0,1\}$.
We claim that $y:=[P_{0},Q_{0}]+[P_{1},Q_{1}]$ in $K_{0}( \prod_{i\in I}A_{i})$  is a preimage of $x$.
Indeed, the projection of $y$ to the $i$'th factor is  given by 
$$[p_{0,i},q_{0,1}]+[p_{1,i}+1_{A_{i}^{+}} ,q_{1,i} +1_{A_{i}^{+}}]\, .$$ If $i\in I_{0}$, then this is equal to 
$$[p_{i},q_{i}]+[1_{A_{i}^{+}},1_{A_{i}^{+}}]=[p_{i},q_{i}]=x_{i}\, .$$ If $i\in I_{1}$, then this class is equal to
$$[0,0]+[p_{i},q_{i}]=[p_{i},q_{i}]=x_{i}$$ again. 
 
\emph{injectivity on $K_0$:}

We consider a class  $[P,Q] $ in $K_{0}(\prod_{i\in I}A_{i})$, where $P,Q$ are projections in   $(\prod_{i\in I}A_{i})^{+}$ such that $e(P)=e(Q) $. Let $k:=e(P)$.  Then we have
$$P=(\tilde p_{i})_{i\in I}+k1_{(\prod_{i\in I }A_{i})^{+}}\qquad \mbox{and}\qquad  Q=(\tilde q_{i})_{i\in I}+k1_{(\prod_{i\in I }A_{i})^{+}}$$  such that $ \tilde p_{i},\tilde q_{i}$ belong to $A_{i}$ and
$p_{i}:=\tilde p_{i}+k1_{A_{i}^{+}}$ and $q_{i}:=\tilde q_{i}+k1_{A_{i}^{+}}$ are projections for every $i$ in $I$. 

Assume that $[P,Q]$ is sent to zero. Then for every $i$ in $I$
 there exists a partial isometry
$u_{i}$ in $A_{i}^{+}$ such that $u_{i}^{*}u_{i}=p_{i}$ and $u_{i}u_{i}^{*}=q_{i}$. If $k=0$, then $e(u_{i})=0$ for all $i$ in $I$. If $k=1$, then 
we can replace $u_{i}$ by  $e(u_{i})^{-1}u_{i}$. 
After this modification we  have  $e(u_{i})=k 1_{A_{i}^{+}}$ for all $i$ in $I$.
Then $U:=(u_{i}-k1_{A_{i}^{+}})+k1_{(\prod_{i\in I }A_{i})^{+}}$ is a partial isometry in
$(\prod_{i\in I }A_{i})^{+}$ such that $U^{*}U=P$ and $UU^{*}=Q$. We conclude that $[P,Q]=0$.

\emph{surjectivity on $K_1$:}

We consider a class  in $\prod_{i \in I} K_{1}(A_i)$ which may be  represented by a family $(u_{i})_{i\in I}$ of unitaries $u_{i}$ in  $A_{i}^{+}$ such that $e(u_{i})=1_{A_{i}^{+}}$. Then $$U:=(u_{i}-1_{A_{i}^{+}})_{i\in I}+1_{(\prod_{i\in I} A_{i})^{+}}$$ is a unitary
in $(\prod_{i\in I}A_{i})^{+}$ with $e(U)=1$. It represents the desired preimage of our class in $K_{1}(\prod_{i\in I}A_{i})$.

\emph{injectivity on $K_1$:}

This is different from the above cases since the relation defining $K_1$ is not algebraic. A unitary represents the trivial element in $K$-theory if it can be connected by a path with the identity. The problem is that the product of an infinite family of paths is not necessarily continuous. This is true only if the family is equi-continuous.

We will solve the problem by showing that by passing to large matrix algebras one can improve the Lipschitz constants of homotopies. This argument is inspired by \cite[Proof of Prop.~12.6.3]{willett_yu_book}.

Let $U$ be a unitary in $(\prod_{i\in I}A_{i})^{+}$ with $e(U)=1$. Then $$U= (u_{i}-1_{A_{i}^{+}})_{i\in I}+ 1_{(\prod_{i\in I}A_{i})^{+}}$$ for unitaries $u_{i}$ in  $A_{i}^{+}$ with $e(u_{i})=1_{A_{i}^{+}}$. Assume that the image of $U$ in $\prod_{i \in I} K_{1}(A_i)$ vanishes. Then  for each $i$ in $I$ there is an $n_i$ in $\IN$ and a homotopy $\tilde{u}_{i,t}$  parametrized by  $[0,1]$ in the unitary group of $M_{n_i}(A_i)^+$ from $\tilde{u}_{i,0} \coloneqq \diag(u_i,1_{A_{i}^{+}},\ldots,1_{A_{i}^{+}})$ to $\tilde{u}_{i,1} = 1_{M_{n_i}(A_{i})^{+}}$ satisfying $e(\tilde{u}_{i,t}) = 1_{M_{n_i}(\C)}$ for all $t$ in $[0,1]$.

We can replace the  homotopies $\tilde{u}_{i,t}$ by Lipschitz paths.
 The argument is as follows. For  any unital $C^*$-algebra $B$ the group $\mathrm{Gl}_n(B)$ of invertible martices with entries in $B$ is an open subgroup of $M_n(B)$. If we have a continuous path   parametrized by $[0,1]$ in  $\mathrm{Gl}_n(B)$, then we get a Lipschitz continuous path  in $\mathrm{Gl}_n(B)$ between the same endpoints    by subdividing the interval $[0,1]$ sufficiently fine   and {using} linear interpolation.  The map $(z,s) \mapsto (1-s)z + s(z z^*)^{-1/2} z$ defines a deformation retraction of $\mathrm{Gl}_n(B)$ onto $U_n(B)$ with the property that it maps Lipschitz continuous path in $\mathrm{Gl}_n(B)$ to Lipschitz continuous path in $U_n(B)$. Therefore, if we have a continuous path in $U_n(B)$, then there exists also a Lipschitz continuous path  in $U_n(B)$ between the same endpoints. Finally, if $A$ is a non-unital $C^*$-algebra and we have a continuous path $u_t$ in $U_n(A^+)$ with $e(u_t) = 1_{M_n(A^+)}$, then the construction above gives a Lipschitz continuous path $u^\prime_t$ in $U_n(A^+)$ from $u_0$ to $u_1$ with $e(u^\prime_t)=1_{M_n(A^+)}$.

For $i$ in $I$ and a natural number $N_i$ we write $\diag(\tilde{u}_{i,0},1,\ldots,1)$ as the product
\begin{equation}\label{erfoijowefwerfweqrfwefwerfwefrwerf}
\begin{pmatrix}
\tilde{u}_{i,0} & 0 & \cdots & 0\\
0 & \tilde{u}_{i,{1/N_i}} & \cdots & 0\\
\vdots & \vdots & \ddots & \vdots\\
0 & 0 & \ldots & \tilde{u}_{i,1}
\end{pmatrix}
\cdot
\begin{pmatrix}
1_{M_{n_i}(A_{i})^{+}} & 0 & \cdots & 0\\
0 & (\tilde{u}_{i,{1/N_i}})^* & \cdots & 0\\
\vdots & \vdots & \ddots & \vdots\\
0 & 0 & \ldots & (\tilde{u}_{i,1})^*
\end{pmatrix}
 \end{equation}
We  now define  a homotopy $\hat{u}_{i,t}$ from $\diag(\tilde{u}_{i,0},1,\ldots,1)$ to $1_{M_{N_i n_i}(A_{i})^{+}}$ inside the unitary group of $M_{N_i n_i}(A_i)^+$ with $e(\hat{u}_{i,t}) = 1_{M_{N_i n_i}(A_i)^+}$.   
The first part of this homotopy is parametrized by the interval $[0,1/2]$ and given by 
 the paths  from $(\tilde{u}_{i,{k/N_{i}}})^*$ to $(\tilde{u}_{i,{(k-1)/N_{i}}})^*$ of the matrix entries  of the second factor for each $k$ in $\{1,\dots,N_{i}\}$.  For the second part, 
 we let   $r_{i,t}$  denote the rotation homotopy between
\begin{align*}
\mathclap{
\begin{pmatrix}
1_{M_{n_i}(A_{i})^{+}} & 0 & 0 & \cdots & 0\\
0 & (\tilde{u}_{i,{0/N_{i}}})^* & 0 & \cdots & 0\\
0 & 0 & (\tilde{u}_{i,{1/N_{i}}})^* & \cdots & 0\\
\vdots & \vdots & \vdots & \ddots & \vdots\\
0 & 0 & 0 & \ldots & (\tilde{u}_{i,{(N_{i}-1)/N_{i}}})^*
\end{pmatrix}}
\end{align*}and
\begin{align*}
\mathclap{ 
 \begin{pmatrix}
(\tilde{u}_{i,{0/N_{i}}})^* & 0 & \cdots & 0 & 0\\
0 & (\tilde{u}_{i,{1/N_{i}}})^* & \cdots & 0 & 0\\
\vdots & \vdots & \ddots & \vdots & \vdots\\
0 & 0 & \cdots & (\tilde{u}_{i,{(N_{i}-1)/N_{i}}})^* & 0\\
0 & 0 & \ldots & 0 & 1_{M_{n_i}(A_{i})^{+}}
\end{pmatrix}}
\end{align*}
parametrized by $[1/2,1]$.
The second part of the homotopy consists of applying the homotopy 
$e(r_{i,t})^* r_{i,t}$ to the second factor of the product in \eqref{erfoijowefwerfweqrfwefwerfwefrwerf}.
We can arrange that the Lipschitz constant   $e(r_{i,t})^* r_{i,t}$ is bounded by $8\pi$ independently of 
$i$.
 
If the Lipschitz constant of $\tilde{u}_{i,t}$ was bounded from above by $L_i$, then the Lipschitz constant of $\hat{u}_{i,t}$ is bounded from above by $\max\{2 L_i / N_{i}, 8\pi\}$.

We choose for each $i$ in $I$ a number $N_i$ such that the resulting family of homotopies $(\hat{u}_{i,t})_{i \in I}$ has a uniform upper bound on their Lipschitz constants. For $i$ in $I$ we let $v_i$ in $\cM(M_{N_i n_i}(A))$ be the isometry from Definition~\ref{wtiogjweogfrefwwerfwreferfsvfvsd}. Then $\bar{u}_{i,t}\coloneqq v_i (\hat{u}_{i,t} - 1_{M_{N_i n_i}(A_{i})^{+}}) v_i^* + 1_{A_{i}^+}$ defines a homotopy in the unitary group of $A_i^+$ connecting $v_{i,11} (u_i-1_{A_i^+}) v_{i,11}^* + 1_{A_i^+}$ to $1_{A_{i}^{+}}$. Here we use \eqref{eq_conjugation_v} and $v_{i,11}$ denotes the upper left corner of $v_i$. The   family $(\bar{u}_{i,t})_{i \in I}$ of Lipschitz paths has a uniform upper bound on their Lipschitz constants. Then $$\bar{U}_t \coloneqq (  \bar{u}_{i,t}-1_{A_{i}^{+}})_{i\in I}+1_{(\prod_{i \in I} A_i)^+}$$ is a homotopy in the unitary group of $(\prod_{i \in I} A_i)^+$ connecting $$(  v_{i,11})_{i\in I} (U-1_{(\prod_{i \in I} A_i)^+}) (  v_{i,11})^{*}_{i\in I} + 1_{(\prod_{i \in I} A_i)^+}$$ to the identity $1_{(\prod_{i \in I} A_i)^+} $. Since $ (  v_{i,11})_{i\in I}$ is an isometry in the $C^*$-algebra $\prod_{i \in I} \cM(A_i)$ which acts via multipliers on $\prod_{i \in I} A_i$
we have the first equality in
\[[U]= [(  v_{i,11})_{i\in I} (U-1_{(\prod_{i \in I} A_i)^+}) (  v_{i,11})^{*}_{i\in I} + 1_{(\prod_{i \in I} A_i)^+}]=0\, .\qedhere\]
\end{proof}

\begin{rem}
In \cite[Thm.\ 2.7.12]{willett_yu_book} it is claimed that \eqref{vfdvioefjvoijvkdlvfvfvavdscdea} is an isomorphism for $*\in \{0,1\}$. 
The argument for surjectivity in the case $*=0$ is sketched in the reference, while rest of the argument  was  left to the reader as a straightforward exercise.  
 So Proposition \ref{ejioerigjregregwer9} {is a solution to} this exercise in  \cite{willett_yu_book}, {in which}    the argument for the injectivity on $K_{1}$  turned out to be more complicated than expected.
%
%
\hB
\end{rem}

For later use let us record the following special case  where $A_{i}=\K$ for every $i$ in $I$.
Let $I$ be any set.

\begin{kor}\label{rigwoergevsfvsfdvfdsv}
The map $(\Kast(p_{i}))_{i\in I}\colon \Kast(\prod_{i\in I}\K)\to \prod_{i\in I} \Kast(\K)$ is an equivalence.
\end{kor}
\begin{proof}
We note that $\K$ is  quasi-$\K$-stable, see Example \ref{ex_BH_quasi_stable}.
\end{proof}

\begin{rem}\label{erwgopijergwer} If one drops the assumption of quasi-$\K$-stability on the factors, then  in general the morphism \eqref{vadsvoij3oigferfgfa} is not an equivalence.
  An explicit   counter-example is the case where $I \coloneqq \IN$ and $A_i \coloneqq \IC$ for each $i$  in $I$ (see \cite[Ex.\ 2.11.15]{willett_yu_book} or \cite[Ex.\ 7.7.3]{higson_roe}).
\hB
\end{rem}

Let $X$ be a bornological coarse space, and let $(H,\phi)$ be an ample $X$-controlled Hilbert space.

\begin{lem}\label{lem_RoeAlg_quasistable}
The Roe algebra $\cC^\ast(X,H,\phi)$ is quasi-$\K$-stable.
\end{lem}
\begin{proof}
By   Lemma~\ref{lem:drf8934} we can assume that $(H,\phi)$ is of the form $(H' \otimes \ell^2, \phi' \otimes \id_{\ell^2})$ 
for some   $X$-controlled Hilbert space $(H',\phi')$.  Then we can identify
$M_{n}(\cC^{*}(X,H' \otimes \ell^2, \phi' \otimes \id_{\ell^2}))$ with 
$\cC^{*}(X,H' \otimes (\ell^2\otimes \C^{n}), \phi' \otimes \id_{\ell^2\otimes \C^{n}})$.
We choose a unitary $u:\ell^{2}\otimes \C^{n}\to \ell^{2}$ and let $\iota:\ell^{2}\to \ell^{2}\otimes \C^{n}$ be the embedding
induced by the first basis vector of $\C^{n}$. Then $v:=\id_{H_{D}}\otimes (\iota\circ u)$
is an isometry which acts as a multiplier on $M_{n}(\cC^{*}(X,H' \otimes \ell^2, \phi' \otimes \id_{\ell^2}))$ 
and satisfies $vv^{*}=e_{1,1}$ as required.
%
%
%
\end{proof}

Recall Proposition \ref{fkwjlkwejlewfewfewfewf3242344} which states that a bornological coarse space  $X$ admits an ample $X$-controlled Hilbert space if and only if it is locally countable.
 Let $(X_i)_{i\in I}$ be a family of locally countable bornological coarse spaces. 
Recall the Definition \ref{foifewieof89u8924r443535} of a free union of such a family. Note that 
  $\bigsqcup_{i\in I}^{\free} X_{i}$ is  again  locally countable and therefore admits an ample $\bigsqcup_{i\in I}^{\free} X_{i}$-controlled Hilbert space $(H,\phi)$. Then 
  for every $i$  in $I$ the $X_i$-controlled Hilbert space $(H(X_i),\phi_{X_i})$ is also ample,
where $H(X_{i})$ is the image of the projection $\phi(X_{i}):=\phi(\chi_{X_{i}})$, and $\phi_{X_{i}}$ is the control obtained from $\phi$ by restriction (see the proof of Proposition \ref{ioejfiowefjweoifewuf8u3294823442} for a similar construction).   
\begin{prop}\label{werijgowregferfwerfwrf}
We have an isomorphism of $C^{*}$-algebras  \begin{equation}
\label{sdfjwef23rfew}
\cC^\ast \Big( \bigsqcup_{i\in I}^{\free} X_{i},H,\phi \Big) \cong \prod_{i \in I} \cC^\ast(X_i,H(X_i),\phi_{X_i})\, .
\end{equation}
\end{prop}
\begin{proof}
Let $e_{i}:H(X_{i})\to H$ denote the inclusion. 
We define a homomorphism 
$$\cC^\ast \Big( \bigsqcup_{i\in I}^{\free} X_{i},H,\phi \Big)\to \prod_{i \in I} \cC^\ast(X_i,H(X_i),\phi_{X_i})$$
such that it sends an operator $A$ in $ \cC^\ast \Big( \bigsqcup_{i\in I}^{\free} X_{i},H,\phi \Big)$ to the family
$(e_{i}^{*}A e_{i})_{i\in I}$.  Since $A$ is bounded, this family is uniformly bounded.
Since the components $(X_{i})_{i\in I}$ are mutually coarsely disjoint in $\bigsqcup_{i\in I}^{\free} X_{i}$ we have 
 $ e_{i}^{*}A e_{i'}=0$ if $i,i'$ are in $I$ with $i\not=i'$, and $e_{i}e_{i}^{*}A e_{i}=Ae_{i}$ for every $i$ in $I$. This implies that 
 $ A\mapsto  (e_{i}^{*}A e_{i})_{i\in I}$      is compatible with the product. 
 It is also immediate from the formula that this  map is compatible with the involution.
   If $A$ is locally finite and 
 has controlled propagation, then the operators $e_{i}^{*}A e_{i}$  on $(H(X_{i}),\phi_{X_{i}})$ are  locally finite and have controlled propagation for all $i$ in $I$. 
 This implies that the map $ A\mapsto  (e_{i}^{*}A e_{i})_{i\in I}$   has values in the product of Roe algebra. 
 
  The inverse sends a family 
 $(A_{i})_{i\in I}$ in $\prod_{i \in I} \cC^\ast(X_i,H(X_i),\phi_{X_i})$ to
 $\sum_{i\in I} e_{i}A_{i} e_{i}^{*}$, where the sum converges strongly.
 Since the family $(A_{i})_{i\in I}$ is uniformly bounded and the family of inclusions $(e_{i})_{i\in I}$ is mutually orthogonal the sum defines  a bounded operator $A$, and the described map $(A_{i})_{i\in I}\mapsto A$ is compatible with the composition and the involution.
 Again, if $A_{i}$ has controlled propagation and is locally finite for every $i$, then $A$ has controlled propagation and is locally finite. 
 So the inverse also has values in the Roe algebra.
 \end{proof}
 
 We will need the following special case of Proposition \ref{werijgowregferfwerfwrf}.
 Let $I$ be a set, and let 
  $(H,\phi)$ be an ample Hilbert space on  $\bigsqcup_{i\in I}^{\free} *$.

 \begin{kor}\label{eriughiwerorferwfwf}
 We have an isomorphism $\cC^{*}\Big( \bigsqcup_{i\in I}^{\free} *,H,\phi \Big)\cong \prod_{i\in I} \K$.
 \end{kor}
\begin{proof}
We let $*_{i}$ denote the point in the component with index $i$.
We note that $H(*_{i})\cong \ell^{2}$  for every $i$ in $I$, and that \begin{equation}\label{aefvkjovvsfvsf}
\cC^{*}(*_{i},H(*_{i} ),\phi_{*_{i}})\cong \K\, .
\end{equation}
\end{proof}

\begin{kor}\label{cor_additivitiy_KX_KXql}
The functors $\KX$ and $\KXql$ are additive.
\end{kor}
\begin{proof}
Since $\KX\to \KXql$ is an  equivalence on discrete bornological coarse spaces by Lemma \ref{iefjweoifoeiwfji23432423432} it suffices to consider the case of $\KX$. 
We choose an ample Hilbert space $(H,\phi)$ on $\bigsqcup_{i\in I}^{\free} *$.
Thhe chain of equivalences (which is a  factorization of the map to be considered)
\begin{eqnarray*}
\KX\Big(\bigsqcup_{i\in I}^{\free} *\Big)&\stackrel{\text{Thm.}~\ref{fwefiwjfeiooi234234324434e}}{\simeq}&
 \Kast\big( \cC^{*}\Big( \bigsqcup_{i\in I}^{\free} *,H,\phi \Big)\big)\\&\stackrel{\text{Cor.}~\ref{eriughiwerorferwfwf}}{\simeq}&
\Kast(\prod_{i\in I} \K)\\&\stackrel{\text{Cor.}~\ref{rigwoergevsfvsfdvfdsv}}{\simeq}&
 \prod_{i\in I} \Kast(\K)\\
 &\stackrel{\eqref{aefvkjovvsfvsf}}{\simeq}&
 \prod_{i\in I} \Kast( \cC^{*}(*_{i},H(*_{i} ),\phi_{*_{i}} ))\\
&\stackrel{\text{Thm.}~\ref{fwefiwjfeiooi234234324434e}}{\simeq}&
 \prod_{i\in I}\KX(*_{i})
\end{eqnarray*}
proves the claim.
\end{proof}

The following is a step towards strong additivity.
Let $(X_i)_{i\in I}$ be a family  of  bornological coarse spaces. 
\begin{kor} If $X_{i}$ is locally countable for every $i$ in $I$, then 
the canonical  morphism  \begin{equation}\label{jfbrefuihdfuivfwdfewfewfqwefewfq}
 \KX\Big(\bigsqcup_{i\in I}^{\free} X_{i}\Big)\to \prod_{i\in I}  \KX(X_{i})
\end{equation}
is an equivalence.  
\end{kor}
\begin{proof}
This follows from Proposition \ref{werijgowregferfwerfwrf}, Theorem \ref{fwefiwjfeiooi234234324434e}, Proposition \ref{ejioerigjregregwer9}, and Lemma \ref{lem_RoeAlg_quasistable}.
\end{proof}

\begin{rem}
In   \cite{Bunke:ad} we will show that the coarse $K$-homology functor $\KX$ is actually strongly additive. Therefore the homomorphism \eqref{jfbrefuihdfuivfwdfewfewfqwefewfq} is an isomorphism for every family $(X_{i})_{i\in I}$
of bornological coarse spaces.  The proof does not use Roe algebras but rather 
works directly with the Roe categories and uses the fact   shown in \cite{cank} that the $K$-theory functor 
for $C^{*}$-categories preserves products of additive $C^{*}$-categories. \hB
\end{rem}


\begin{ex}
In this example we show that for an infinite set $I$ the product $\prod_{i\in I} \K$ is not $\K$-stable.
In view of Corollary \ref{eriughiwerorferwfwf} this implies that 
the Roe algebra $\cC^\ast(X,H,\phi)$  associated to a  bornological coarse space $X$ and an ample $X$-controlled Hilbert space $(H,\phi)$ is in general not $\K$-stable.

%
%
%
%

Let $I$ be an infinite set. In order to show that  $\prod_{i\in I} \K$ is not $\K$-stable we use the following property of $\K$-stable $C^{*}$-algebras.

 If   $A$ is $\K$-stable, then there is a sequence $\{E_n\}_{n \in \IN}$ of mutually orthogonal, mutually equivalent projections in the multiplier algebra $\cM(A)$ of $A$ such that $\sum_{n=0}^\infty E_n = 1$, where the sum converges   in the strict topology.\footnote{If $A$ is $\sigma$-unital, then the converse also holds \cite[Thm.~2.2]{rordam_stable}.}

We now show that 
  $\prod_{i\in I}\K$ does not admit such a sequence.   Assume that $\{E_n\}_{n \in \IN}$ is such a family of projections.
  Then $E_{n}=(E_{i,n})_{i\in I}$ is a family of projections in $\K$ for every $n$ in $\nat$.
 Let $ \kappa:\nat \to I$ be an injective map. It exists since we assume that $I$ is infinite. 
 For every $n$ in $\nat$ we choose a non-zero projection $P_{n}$ in $\K$ which is orthogonal to  $E_{\kappa(n),m}$ for all $m$ in $\nat$ with $m\le n$.  We now define a  projection  $Q:=(Q_{i})_{i\in I}$ in $\prod_{i\in I}\K$,  where $Q_{i}:=P_{n}$ if $i=\kappa(n)$ for some (uniquely determined) $n$ in $\nat$, and $Q_{i}:=0$ else.     
 We have $E_{m}Q=(E_{i,m}Q_{i})_{i\in I}$. If $i=\kappa(n)$ for $n$ in $\nat$, then $E_{i,m}Q_{i}=0$ for all $m$ in $\nat$ with $m\le n$. Hence $\|\sum_{m=0}^{n} E_{m}Q-Q\|=1$ for all $n$ in $\nat$.
 This is a contradiction to the condition $\sum_{n=0}^\infty E_n = 1$ in the strict topology.
 \hB
\end{ex}

\subsubsection{Coproducts}

The goal of this subsection is to show that the functors $\KX$ and $\KXql$ preserve coproducts. 
We start with an explicit description of the coproduct $\bC:=\coprod_{i\in I}\bC_{i}$ of a family  $(\bC_i)_{i \in I}$ 
 of $C^*$-categories. 
 The set of objects  of $\bC$ is  given by $\bigsqcup_{i\in I} \Ob(\bC_{i})$. We write $(c,i)$ for the object given by $c$ in $C_{i}$.  The morphisms of $\bC$   are given by
$$\Hom_{\bC}((c,i),(c^{\prime},i^{\prime})):=\left\{\begin{array}{cc}\Hom_{\bC_{i}}(c,c^{\prime}) &  i=i^{\prime}\\0&i\not=i^{\prime}\end{array}\right.\, .$$
The composition and the involution are given in the obvious way.


By an inspection of the definition of $A$ (see e.g. \eqref{efefwefeefefefewfef}) we observe that  have a canonical isomorphism
\begin{equation}
\label{nwef7823}
\bigoplus_{i \in I} A(\bC_i)\cong 
 A(\bC)    \end{equation}
of $C^{*}$-algebras.

 By Definition \ref{wegjioegergewfewrfewef},  Proposition \ref{prop:dfs78934} and since $K$-theory of $C^{*}$-algebras sends direct sums to coproducts (Property \ref{rgpojweogweggffrw} of $\Kast$)
 we have a canonical equivalence \begin{equation}\label{ergvoiu34oit3r4gregeg}
\bigoplus_{i\in I} \Kcat(\bC_{i})\simeq \Kcat(\bC)\, .
\end{equation}
  
  \begin{rem}
Since $A^{f}$ is a left adjoint it preserves colimits. The coproduct in $C^{*}$-algebras is the free product. Hence $A^{f}(\bC)$ is the free product of the $C^{*}$-algebras $A^{f}(\bC_{i})$ for $i$ in $I$. Since $\Kast$ sends free products to sums this would also lead to the equivalence \eqref{ergvoiu34oit3r4gregeg} avoiding the use of Proposition \ref{prop:dfs78934}.\hB
\end{rem}
  
\begin{prop}\label{wetiugwergwrfrefrwef9}
The functors $\KX$ and $\KXql$ preserve coproducts.
\end{prop}

\begin{proof}
We will only discuss the case of $\KX$ since the quasi-local case $\KXql$ is analogous.

Let $(X_i)_{i \in I}$ be a family of bornological coarse spaces and set $X := \coprod_{i \in I} X_i$. 
We must show that the canonical morphism
$$\bigoplus_{i\in I} \KX(X_{i})\to \KX(X)$$
is an equivalence.
This morphism has the following factorization
\begin{eqnarray*}
\bigoplus_{i\in I} \KX(X_{i})&\simeq&\bigoplus_{i\in I} \Kcat(\bC^{*}(X_{i}))\\
&\stackrel{\eqref{ergvoiu34oit3r4gregeg}}{\simeq} & \Kcat(\coprod_{i\in I}\bC^{*}(X_{i}))\\&\stackrel{(1)}{\simeq}&
 \Kcat((\coprod_{i\in I}\bC^{*}(X_{i})_{\oplus})\\&\stackrel{(2)}{\simeq}&
  \Kcat( \bC^{*}(X))\\&\simeq&
  \KX(X)
\end{eqnarray*}
This first and the last equivalence follow from the definition of $\KX$.
In order to finish the proof we must explain the equivalences marked by $(1)$ and $(2)$.

From the universal property of the coproduct of $C^{*}$-categories we get a functor
\begin{equation}\label{frefoij4o23t34g43g34g3}
\coprod_{i\in I} \bC^{*}(X_{i})\to \bC^{*}(X)
\end{equation}
which sends for every $i$ in $I$ the $X_{i}$-controlled Hilbert space $(H,\phi)$ to the
$X$-controlled Hilbert space $\iota_{i,*}(H, \phi )$, where $\iota_{i}:X_{i}\to X$ is the inclusion.
This functor induces the composition $(2)\circ (1)$.
It is not yet essentially surjective since its image contains only those $X$-controlled Hilbert spaces which are
supported on a single component $X_{i}$. A general $X$-controlled Hilbert space $(H,\phi)$ is locally finite, i.e., $\supp(H,\phi)$ is a locally finite subset of $X$. Note that a subset $B$ of $X$ is bounded if and only if $B\cap X_{i}$ is bounded for every $i$ in $I$.  This implies that 
 the  set  $\{i\in  I\::\: H(X_{i})\not\cong 0\}  $ is finite. In other words, $(H,\phi)$ is supported on finitely many of the components $X_{i}$. 

In order to obtain these objects as well we extend 
the functor  \eqref{frefoij4o23t34g43g34g3}  to the  additive completion of the coproduct  (we use the model introduced by Davis--L{\"u}ck  \cite[Sec.~2]{davis_lueck})
\begin{equation}\label{5z46745z45z45z5}
\Big(\coprod_{i\in I} \bC^{*}(X_{i})\Big)_{\!\oplus}\to \bC^{*}(X)\, .
\end{equation}
This extension exists since 
 $\bC^{*}(X)$ admits finite  sums.
This extension \eqref{5z46745z45z45z5}  sends the finite family
$((H_{1},\phi_{1}),\dots,(H_{n},\phi_{n}))$ in $(\coprod_{i\in I} \bC^{*}(X_{i}))_{\oplus}$ to the sum
$(\bigoplus_{k=1}^{n}H_{k},\oplus_{k=1}^{n}\phi_{k})$ in $\bC^{*}(X)$.
 
In order to define an inverse functor 
we choose an ordering on the index set $I$.
Then we can define a functor
\begin{equation}\label{5z46745z45z45z51}
\bC^{*}(X)\to \Big(\coprod_{i\in I} \bC^{*}(X_{i})\Big)_{\!\oplus}
\end{equation} which sends $(H,\phi)$ to the tuple
$(H(X_{i}),\phi_{X_{i}})^{\prime}$. Here the index ${}^{\prime}$ indicates that  the tuple is obtained from the infinite
family $(H(X_{i}),\phi_{X_{i}})_{i\in I}$  by deleting all zero spaces and listing the remaining (finitely many members) in the order determined by the order on $I$. On morphisms these functors are defined in the obvious way.
One now checks that \eqref{5z46745z45z45z5} and \eqref{5z46745z45z45z51} are inverse to each other equivalences of categories. Indeed, both are fully faithful and essentially surjective.

The  equivalence of $C^{*}$-categories \eqref{5z46745z45z45z5} induces the functor $(2)$ which is then 
an equivalence by Corollary \ref{ojewjfowfewfewfewf1}.
 
 The functor $(1)$ is induced by the canonical inclusion $$\coprod_{i\in I} \bC^{*}(X_{i}) \to  \Big(\coprod_{i\in I} \bC^{*}(X_{i})\Big)_{\!\oplus}\, .$$
 
 Now in general, for a $C^{*}$-category $\bC$
 we have an isomorphism
$A(\bC_{\oplus})\simeq A(\bC)\otimes \mathbb{K}$, and that the inclusion
$\bC\to \bC_{\oplus}$ induces the homomorphism
$$A(\bC)\to A(\bC_{\oplus})\cong A(\bC)\otimes \mathbb{K}$$
given by the  left-upper corner inclusion.  In particular, it induces an equivalence in $K$-theory by the stability Property \ref{wegoijweroigjergwerg} of $\Kast$. 
Using Proposition \ref{prop:dfs78934} we see that \begin{equation}\label{dwcqwoicjodicdcadc}
\Kcat(\bC) \to \Kcat(\bC_{\oplus})
\end{equation}
 is an equivalence. 
In particular, the morphism  $(1)$ is an equivalence.
%
%
%
\end{proof}

\begin{rem}
The functors $\KX$ and $\KXql$ have the stronger property called continuity discussed in \cite[Sec.~5]{equicoarse}, see Remark \ref{wethgrtgffsvfdvsvsdfv}.
In order to verify this property one uses that the objects of $\bC^{*}(X)$ and $\bC_{\ql}^{*}(X)$ are supported on locally finite subsets of $X$. Then the argument is similar to the proof of  \cite[Prop.~8.17]{equicoarse}.
By \cite[Lem.~5.17]{equicoarse} a continuous coarse homology theory preserves coproducts.  This would give an alternative proof of Proposition \ref{wetiugwergwrfrefrwef9}.
\hB
\end{rem}

\begin{rem}
In \cite{cank} we will show that $\Kcat$ preserves Morita equivalences of $C^{*}$-categories. 
Because $\bC\to \bC_{\oplus}$ is a Morita equivalence, this immediately implies that 
$\Kcat(\bC)\to \Kcat(\bC_{\oplus})$ is an equivalence.
\hB
\end{rem}

%

\subsection{Dirac operators}\label{subsec:drsgd}

The goal of this subsection is to explain how the coarse index class  with support of a Dirac operator defined in  \cite{roe_psc_note}
is captured by the coarse $K$-homology theory $\KX$.  It catches a glimpse of the theory developed in \cite{cw,indexclass}.

Let $(M,g)$ be a complete Riemannian manifold. The Riemannian metric induces a metric on the underlying set of $M$ and therefore (Example \ref{welifjwelife89u32or2})    a coarse and a bornological structure $\cC_g$ and $\cB_g$.  
We thus get a bornological coarse space $M_{g}$ in $\BC$. 
Let $S\to M$ be a Dirac bundle\index{Dirac!bundle} ({Gromov--Lawson} \cite{gromov_lawson_psc}) and  denote the associated Dirac operator\index{Dirac!operator} by 
 $\Dirac$.\index{$\Dirac$|see{Dirac operator}}
 In the following we discuss the index class of the Dirac operator in the $K$-theory of the Roe algebra.

 The 
  Weitzenb{\"o}ck formula   
\begin{equation}\label{fwefjhwef8726478234234}
\Dirac^2 = \Delta + \cR
\end{equation}
   expresses  the square of the Dirac operator in terms of the connection Laplacian $\Delta:=\nabla^\ast\nabla$ on $S$ and a bundle endomorphism  $\cR$  in $C^\infty(M,\End(S))$. \index{$\cR$}

For every $m$ in $M$ the value $\cR(m)$ in $\End(S_{m})$ is selfadjoint.   
For  $c$ in $\IR$  we define the subset
\[M_c := \{m \in M :  \cR(m) \ge  c\}\]
of points in $M$ on which $\cR(m)$ is bounded below by $c$.

\begin{ex}
For example, if $S$ is the spinor bundle associated to some spin structure on $M$, then $$\cR=\frac{s}{4}\id_{S}\, ,$$ where $s$ in $C^{\infty}(M)$ is the scalar curvature function, see Lawson--Michelsohn \cite[Equation (8.18)]{lawson_michelsohn}, Lichnerowicz \cite{MR0156292} or Schr{\"o}dinger \cite{schr}. In this case   $M_{c} $  is the subset of $M$ on which the scalar curvature is bounded below by  $ 4c$. 
\hB \end{ex}

Since $M$ is complete the Dirac operator $\Dirac$ and the connection Laplacian $\Delta $ are essentially selfadjoint unbounded operators on the Hilbert space $H:=L^{2}(M,S)$  defined on the dense  domain $C_{c}^{\infty}(M,S)$. 

The Laplacian $\Delta$ is non-negative. So morally the formula \eqref{fwefjhwef8726478234234} shows that  $\Dirac^{2}$  is lower bounded by $c$ on the submanifold $M_{c}$, and the restriction of  $\Dirac$ to this subset is invertible and therefore has vanishing index.   {Roe} \cite[Thm.~2.4]{roe_psc_note}  constructs a large scale index class of the Dirac operator which takes this positivity into account {(the original construction without taking the positivity into account may be found in \cite{roe_index_coarse}).} Our goal in this section is to interpret the construction of Roe in the language of bornological coarse spaces. 

We fix $c$ in $\R$ with $c>0$.
 \begin{ddd} 
We define the big family $\cY_{c}:=\{M\setminus M_{c}\}$ on $M_{g}$.   \end{ddd}

We use the construction explained in Remark~\ref{rem:sdb7834rw} in order to turn $H:=L^{2}(M,S)$ into an $M$-controlled Hilbert space $(H,\phi)$.

We equip every member $Y$ of $\cY_{c }$ with the induced bornological coarse structure. 
 For every member $Y$ of $\cY_{c }$ the inclusion $(H(Y),\phi_{Y})\hookrightarrow  (H,\phi)$ provides an inclusion of Roe algebras (Definition \ref{qekdfjqodqqwdqwdqwdqdwq})
$\cC_{\lc}^{*}(Y,H(Y),\phi_{Y})\hookrightarrow \cC_{\lc}(M,H,\phi)$.
We define the $C^{*}$-algebra
$$\cC_{\lc}^{*}(\cY_{c},H ,\phi) :=\colim_{Y\in \cY_{c}}\cC_{\lc}^{*}(Y,H(Y),\phi_{Y}) \, .$$
We will interpret $\cC_{\lc}^{*}(\cY_{ c},H ,\phi)$ as a closed subalgebra of $\cC_{\lc}(M,H,\phi)$ obtained by forming
  the closure of the union of subalgebras  $\cC_{\lc}^{*}(Y,H(Y),\phi_{Y})$.
\begin{rem}
The $C^{*}$-algebra $\cC_{\lc}^{*}(\cY_{c},H ,\phi)$ has first been considered by Roe  \cite{roe_psc_note}. In his notation it would have the symbol   
$C^{*}((M\setminus M_{c})\subseteq M)$. 
\hB
\end{rem}
We now recall the main steps  of Roe's construction of the large scale index classes
$\ind(\Dirac,c)$ in ${K_{\ast}}(\cC^\ast_{\lc}(\cY_{ c},H, \phi))$.   
  To this end we consider the  $C^\ast$-algebra $\cD^\ast(M,H,\rho)$.
Its definition uses the notion of pseudo-locality and depends on   the representation by multiplication operators  $\rho:C_{0}(M)\to B(H)$  of the $C^{*}$-algebra of continuous functions on $M$ vanishing at $\infty$.

\begin{ddd}
An operator $A$ in $B(H)$ is pseudo-local,\index{pseudo-local operator}\index{operator!pseudo-local} if  $[A,\rho(f)]$ is a compact operator for every function $f $ in $C_0(M)$.
\end{ddd}
Note that in this definition we can not replace $\rho$ by $\phi$. Pseudo-locality is a topological and not a bornological coarse concept. In order to control propagation and to define local compactness we could use $\rho$ instead of $\phi$, {see Remark~\ref{rem43wer45er}}.

\begin{ddd}\label{wefijwifo24425345345345}
The $C^{*}$-algebra $\cD^\ast(M,H,\rho)$ is\index{$\cD^\ast(-)$}
defined as the closed  subalgebra   of $B(H)$ generated by all   pseudo-local operators  of controlled propagation.
\end{ddd}

\begin{lem}
We have an inclusion 
$$\cC^\ast_{\lc}(\cY_{ c},H ,\phi) \subseteq \cD^\ast(M,H,\rho)$$  as a {closed} two-sided {*-}ideal. 
\end{lem}
\begin{proof}
Recall that $\cC^\ast_{\lc}(\cY_{ c},H ,\phi)$ is generated by locally compact,  bounded operators of    controlled propagation which belong to  $\cC_{\lc}^\ast(Y,H(Y),\phi_{Y})$ some $Y$ in $\cY_{c}$.
For short we  say that such an operator is supported on $Y$.

  A locally compact operator $A$ is pseudo-local. Indeed,
  for $f$ in $C_{0}(M)$ we can find a bounded subset $B$ in $\cB_{g}$  such that
  $\rho(f)\phi(B)=\rho(f)=\phi(B)\rho(f)$. But then the commutator
  $$[ A,\rho(f)]=A\phi(B)\rho(f)-\rho(f)\phi(B)A$$ is compact.

In order to show that $\cC^\ast_{\lc}(\cY_{ c},H ,\phi)$ is an ideal we consider generators
 $A $ of $\cC^\ast_{\lc}(\cY_{ c},H ,\phi)$ supported on $Y$ in $ \cY_{c}$ and $Q$  in $\cD^\ast(M,H,\rho)$. In particular, both $A$ and $Q$ have  controlled propagation. 
 
We argue that
  $QA\in \cC^\ast_{\lc}(\cY_{c},H ,\phi)$.  The argument for $AQ$ is analogous. 
  It is clear that $QA$ again has controlled propagation. 
  Let  $U:=\supp(Q)$  be the propagation of~$Q$ (measured with the control $\phi$). Then
  $QA$ is supported on    $U[Y]$ which is also a member of $\cY_{ c}$. 
 
For every bounded subset $B$ of $M_{g}$ we furthermore have the identity $$ \phi(B)QA = \phi(B)Q\phi( U^{-1}[B]) A\, .$$ Since $U^{-1}[B]$ is bounded and $A$ is locally compact we conclude that $ \phi(B)QA$ is compact. Clearly also $QA\phi(B)$ is compact. This shows that $QA$ is locally compact.
\end{proof}

In order to define the index class  we use the boundary operator in $K$-theory 
$$\partial:{K_{\ast+1}}( \cD^\ast(M,H,\rho)/\cC^\ast_{\lc}(\cY_{ c},H ,\phi))\to {K_{\ast}}(\cC^\ast_{\lc}(\cY_{ c},H ,\phi))\, .$$
associated to the short exact sequence of $C^{*}$-algebras
$$0\to  \cC^\ast_{\lc}(\cY_{c},H ,\phi)\to  \cD^\ast(M,H,\rho)\to  \cD^\ast(M,H,\rho)/\cC^\ast_{\lc}(\cY_{c},H ,\phi)\to 0\, .$$
{The value of $\ast $ in $\{0,1\}$ will depend on whether the operator $\Dirac$ is graded or not.}

We choose  a function $\chi_{c}$ in $C^{\infty}(\R)$ with the following properties:
\begin{enumerate}
\item $\supp(\chi_{c}^{2}-1)\subseteq [-c,c]$  
\item $\lim_{t\to \pm\infty}\chi_{c}(t)=\pm 1$.
 \end{enumerate}

If $\psi:\R\to\R$ is a bounded measurable function on $\R$, then we can define the bounded operator $\psi(\Dirac)$ using functional calculus. The construction of the large scale index class of $\Dirac$ depends on the following key result of Roe:
 \begin{lem}[{\cite[Lem.~2.3]{roe_psc_note}}]\label{dijqwodiqwdqwqwdq} If $\psi$ in $C^{\infty}(\R)$ is supported in $[-c,c]$, then 
we have $ \psi(\Dirac)\in  \cC_{\lc}^{*}(\cY_{ c},H ,\phi)$.
\end{lem}

Note that $\chi_{c}^{2}-1$ is supported in $[-c,c]$. The lemma implies therefore that
$$[(1+\chi_{c}(\Dirac))/2]\in K_{0}( \cD^\ast(M,H,\rho)/\cC^\ast_{\lc}(\cY_{ c},H ,\phi))\, .$$
We then define the class 
\begin{equation}\label{weflkjwefeiu23oiuo2332432423424}
\partial [(1+\chi_{c}(\Dirac))/2]\in K_{1}(\cC^\ast_{\lc}(\cY_{ c},H ,\phi))\, .
\end{equation}
This class is independent of the choice of $\chi_{c}$. Indeed, if $\tilde\chi_{c} $ is  a different choice, then
again by the lemma $(1+\chi_{c}(\Dirac))/2$ and $(1+\tilde\chi_{c}(\Dirac))/2$ represent the same class in the quotient $\cD^\ast(M,H,\rho)/\cC^\ast_{\lc}(\cY_{ c},H ,\phi)$.

The class \eqref{weflkjwefeiu23oiuo2332432423424} is the coarse index class in the case the Dirac operator $\Dirac$ is ungraded. In the case that $S$ is graded, i.e., $S = S^+ \oplus S^-$, and the Dirac operator $\Dirac$ is an odd operator, i.e., $\Dirac^{\pm}\colon S^\pm \to S^\mp$, we choose a unitary $U\colon S^- \to S^+$ of controlled propagation. At this point, for simplicity we assume that $M$ has no zero-dimensional components. Then the relevant controlled Hilbert spaces are ample and we can find such an operator $U$ by   Lemma~\ref{lem:drf8934}.  The above lemma implies that
\[[U \chi_c(\Dirac^+)]\in K_{1}( \cD^\ast(M,H^+,\rho^+)/\cC^\ast_{\lc}(\cY_{ c},H^+ ,\phi^+))\, ,\]
where $H^+:=L^{2}(M,S^+)$ and $\rho^+$, $\phi^+$ are the restrictions of $\rho$, $\phi$ to $H^+$.
We then define the class 
\begin{equation}\label{wo233mjnrf423424}
\partial [U \chi_c(\Dirac^+)]\in K_{0}(\cC^\ast_{\lc}(\cY_{ c},H^+ ,\phi^+))\, .
\end{equation}
Again by the lemma, this class is independent of the choice of $\chi_c$. It is also independent of the choice of $U$ since two different choices will result in operators $U \chi_c(\Dirac^+)$ and $\tilde U \chi_c(\Dirac^+)$ such that their difference in the quotient algebra $U \chi_c(\Dirac^+) (\tilde U \chi_c(\Dirac^+))^\ast \sim U \tilde U^\ast$ comes from $\cD^\ast(M,H^+,\rho^+)$. Hence they map to the same element under the boundary operator.

In the following we argue that the coarse index classes \eqref{weflkjwefeiu23oiuo2332432423424} and \eqref{wo233mjnrf423424} can naturally be interpreted as coarse $K$-homology classes in $\KX_{\ast}(\cY_{c})$. We also assume that  $M$ has no zero-dimensional components.
 
By an inspection of the construction  {in Example} \ref{rem:sdb7834rw}  we see that
there is a cofinal subfamily of members $Y$ of $ \cY_{c}$ with the property that $(H(Y),\phi_{Y})$ is ample. We will call such $Y$ good.
An arbitrary member may not be not good, since it may not contain enough of the evaluation points denoted by $d_{\alpha}$
in Example \ref{rem:sdb7834rw}.

Every member of $  \cY_{c}$ is a   separable  bornological coarse space.
  If {the member}  $Y$ is good, then by Proposition \ref{lem:sdfbi23}
the canonical inclusion is an equality $$\cC^\ast (Y,H(Y) ,\phi_{Y})= \cC^\ast_{\lc}(Y ,H(Y) ,\phi_{Y})\, .$$
In view of Theorem \ref{fwefiwjfeiooi234234324434} we have a canonical isomorphism
$${K_{\ast}}(\cC^\ast_{\lc}(Y,H(Y),\phi_{Y}))\cong  {\KX_{\ast}}(Y)\, .$$
Recall that we define
$$\KX(\cY_{c}):=\colim_{Y\in \cY_{c}} \KX(Y)\, .$$
We can restrict the colimit to the cofinal subfamily of good members.
Since {taking} $C^{*}$-algebra $K$-theory and   homotopy groups commutes {with} filtered colimits of $C^{*}$-algebras we get the canonical isomorphism
\begin{equation}
\label{jkbsdf78234wfsd}
{K_{\ast}}(\cC^\ast_{\lc}(\cY_{ c},H,\phi))\cong {\KX_{\ast}}(\cY_{c})\, .
\end{equation}

\begin{ddd}
The large scale index\index{large scale index class}\index{index class}
$\ind_{c}(\Dirac)$ in ${\KX_{\ast}}(\cY_{  c})$ of $\Dirac$ is {defined by \eqref{weflkjwefeiu23oiuo2332432423424} in the ungraded case, resp.\ by \eqref{wo233mjnrf423424} in the graded case,} under the identification \eqref{jkbsdf78234wfsd}.
\end{ddd}

\begin{rem}
A much more detailed construction of this index class (even in the equivariant case) is discussed in 
\cite{indexclass}.  In this paper we furthermore prove a relative index theorem and a compatibility with suspension.
One of the goals of that paper is to connect the analytical constructions  in \cite{MR3551834} with the coarse homotopy theory as developed in the present book.  As the discussion above shows this is not completely trivial since
the analytic representative of  index class lives in the $K$-theory of a Roe algebra associated to an ample $M$-controlled Hilbert space, while the $K$-theory $\KX(M)$ is build from Roe algebras of objects in $\bC^{*}(M)$ which are never ample by local finiteness. The bridge is provided  by the  comparison Theorem \ref{fwefiwjfeiooi234234324434} and its functorial version Theorem \ref{fwefiwjfeiooi234234324434e}.
 \hB
\end{rem}

These index classes are compatible for different choices of $c$. If $c,c'$ are in $\R$ such that  $0<c<c^{\prime}$, then $M_{c^{\prime}}\subseteq M_{c}$. Consequently  every member of $\cY_{c}$ is contained in some member of $\cY_{c^{\prime}}$. 
We thus get a map
$$\iota:\KX(\cY_{ c })\to  \KX(\cY_{c^{\prime} })\, .$$
An inspection of the construction using the independence of the choice of $\chi_{c}$ discussed above
one checks  the relation
$$\iota_{*}\ind_{c }(\Dirac)=\ind_{c^{\prime}}(\Dirac)\, .$$

At the cost of losing some information we can also encode the positivity of $\Dirac$ in a bornology $\cB_{c}$ on $M$.

\begin{ddd}\label{oidqiodqoidu981u98u791313213} 
We define $\cB_{c}$\index{$\cB_{c}$} to be the family of subsets $B$  of $X$ such that $B\cap Y\in \cB_{g}$  for every $Y$ in $\cY_{c}$.
\end{ddd}
It is clear that $\cB_{g}\subseteq \cB_{c}$.

\begin{lem}
 $\cB_{c}$ is a bornology on $M$ which is compatible with the coarse structure $\cC_{g}$.
\end{lem}

\begin{proof}
It is clear that $\cB_{c}$ covers $M$, is closed under subsets and finite unions. 
Let $B$ be in $\cB_{c}$.
For every entourage
$U$ of $M$ and $Y$ a member of $\cY_{ c}$ we have
$$U[B]\cap Y\subseteq U[B\cap U^{-1}[Y]]\, .$$ This subset of $M$ belongs to $ \cB_{g}$, since
$U^{-1}[Y]\in \cY_{ c}$ and hence $B\cap U^{-1}[Y]\in \cB_{g}$.
\end{proof}

We  consider  the bornological coarse space $M_{g,c}:=(M,\cC_{g},\cB_{ c})$ in $ \BC $.  If $Y$ is a member 
of  $\cY_{ c}$ then it is considered as a bornological coarse space with the structures induced from $M_{g}$.
By construction of $\cB_{c}$ the inclusion $Y\to M_{g,c}$ is also proper and hence a morphism.
Taking the colimit we get a morphism of spectra 
 $$\kappa :\KX(\cY_{ c})\to \KX(M_{g,c})\, .$$
 We can consider the class
$$\ind_{M}(\Dirac,c):=\kappa(\ind(\Dirac,c))$$ in ${\KX_{\ast}}(M_{g,c})$.

\begin{ex} The Dirac operator 
 $\Dirac$ is invertible at $\infty$, if there exists $c > 0$ such that $M \setminus M_c$ is compact. In this case 
 $\cY_{c}=\cB_{g}$ is the family of relatively compact subsets of $M$ and $\cB_{c}$ is the maximal bornology.
 Consequently,   
  the map $$p :  M_{g,c} \to \ast$$ is a morphism in $\BC$ and therefore we can consider $p_\ast \ind_{M}(\Dirac,c)$ in ${\KX_{\ast}}(\ast)$.
  If $\Dirac$ is invertible at $\infty$, then it is Fredholm and  $p_\ast\ind_{M}(\Dirac,c)$
 is its Fredholm index.

Note that we  recover  the Fredholm index of $\Dirac$  by applying a morphism. This shows the usefulness of the category $\BC$, where we can change the bornology and coarse structure quite independently from each other.
\hB
\end{ex}

%

\begin{rem} In \cite{cw}
 we discuss more aspects of index theory in realm of coarse homotopy theory including the construction of secondary invariants in \cite{cw} and boundary value problems.  \hB
\end{rem}

\subsection{\texorpdfstring{$K$}{K}-theoretic coarse assembly map}
\label{sec:kjnbsfd981}
 
In this section we discuss the construction of the coarse assembly map following Higson--Roe \cite{hr} and Roe--Siegel \cite{Roe:2012aa}. We will just explain its definition using the language developed in this book. 
We will explain why the comparison result  Corollary \ref{thm:sdf98245csv} is not directly applicable
in order to prove that it is an isomorphism under a finite asymptotic dimension assumption, i.e., the coarse Baum--Connes conjecture. In \cite{ass} we will provide a better construction of the assembly map which indeed allows to apply
 Corollary \ref{thm:sdf98245csv} in order to deduce the coarse Baum--Connes conjecture.

Let $(Y,d)$ be a separable proper metric space,   and let $Y_{d}$ be the associated coarse bornological space. 
We let $Y_{t}$ denote the underlying locally compact topological space of $Y$ which we consider as a topological bornological   space. Let furthermore $(H,\phi)$ be an ample $Y_{d}$-controlled Hilbert space, and let $\rho:C_{0}(Y_{t})\to B(H)$ be a $^\ast$-representation {satisfying the three conditions mentioned in Remark~\ref{rem43wer45er}. The notions of local compactness and controlled propagation for operators on $H$ do not depend on using $\rho$ or $\phi$.} We refer to the Example \ref{rem:sdb7834rw} which explains why we need these requirements and can not just assume that $\phi$ extends $\rho$.

 We get an exact sequence   of
  $C^{*}$-algebras
\begin{equation}\label{vfeklelkjvioerjtert}
0\to \cC^\ast_{\lc}(Y_{d},H,\phi)\to \cD^\ast(Y_{d},H,\rho)\to Q^\ast(Y_{d},H,\rho)\to 0 
\end{equation}
defining the quotient  $Q^\ast(Y_{d},H,\rho)$\index{$Q^\ast(Y_{d},H,\rho)$}, where the $C^{*}$-algebra in the middle is  characterized in Definition \ref{wefijwifo24425345345345}. 

Since $(H,\rho)$ is ample, we have the  Paschke duality isomorphism (marked by $P$)\index{Paschke duality} 
\begin{equation}\label{erkgjhergiu43tui34ut3t3t3t43}K_{*+1}(Q^\ast(Y_{d},H,\rho))\stackrel{P}{\cong} KK_{*}(C_{0}(Y_{t}),\C )\cong K^{\an,\lf}_{*}(Y_{t})\, ,\end{equation}
see  \cite{paschke,higson_roe,MR2220524,higson_paschke_duality}. The analytic locally finite $K$-homology  $K^{\an,\lf}(Y_{t})$ was introduced in Definition \ref{foiehweiofui23ur892342424} (with the notation $K^{\an,\lf}_{\C}$, but here we drop the subscript $\C$), and we use Lemma \ref{ilfjiweofuwei9u9435345345} and Proposition \ref{qerugfihierwgegwerfwefwerfwerfw} (the fact that $K^{\an}$ is locally finite) for the second isomorphism.

The boundary map in $K$-theory associated to the exact sequence of $C^{*}$-algebras \eqref{vfeklelkjvioerjtert} therefore gives rise to a  homomorphism\index{$A$}
\[A:K^{\an,\lf}_{*}(Y_{t})\stackrel{\eqref{erkgjhergiu43tui34ut3t3t3t43}}{\cong} K_{*+1}(Q^\ast(Y_{d},H,\rho)) \xrightarrow{\partial} K_\ast(\cC^\ast_{\lc}(Y_{d},H,\phi)) \cong \KX_{lc,*}(Y_{d}) \cong \KX_{*}(Y_{d})\, ,\]
{where the last two isomorphisms are due to Theorem~\ref{fwefiwjfeiooi234234324434} and Corollary~\ref{iofewoifu9234234234324}.}
It is further known that the homomorphism $A$ is independent of the choice of the $H$, $\phi$ and $\rho$ above {\cite[Sec.~4]{higson_roe_yu}}.
Following \cite[Def.~1.5]{MR2220524} we adopt the following definition.

Let $(Y,d)$ be a separable proper metric space.
\begin{ddd}\label{eifjweoifoiwefwefwefw}
 The homomorphism $$A:K^{\an,\lf}_{*}(Y_{t})\to  \KX_{*}(Y_{d})$$ described above is called the analytic assembly map.\index{analytic!assembly map}\index{assembly map!analytic}
\end{ddd}

We now apply the above construction to the space of {controlled}  probability measures  $P_{U}(X)$ on  a bornological coarse space $X$, see \eqref{gegljio3tuio38tut3t3t34}. 

Let $X$ be a bornological coarse space. 
\begin{ddd}\label{oijioegewr}
$X$ has strongly locally bounded geometry if it has the minimal compatible bornology and for every entourage $U$ of $X$ every $U$-bounded subset of $X$ is finite.\index{bounded geometry!strongly locally}\index{locally!bounded geometry!strongly}\index{strongly!locally bounded geometry}
\end{ddd}

In contrast to the notion of  ``strongly bounded geometry''    (Definition~\ref{ejfwjefklwejlkewfwefwefwefewf}) we drop   the condition of uniformity of the bound  on the cardinalities of $U$-bounded subsets.

Assume that $X$ is  a bornological coarse space  which has strongly  locally bounded geometry.
For every entourage $U$ of $X$
  containing the diagonal  we consider the space $P_{U}(X)$.  Since $X$ has strongly locally bounded geometry $P_{U}(X)$  is a  locally finite simplicial complex. We equip the simplices of  $P_{U}(X)$ with the spherical 
metric and $P_{U}(X)$ itself with the induced path-metric $d$ (recall that points in different components have infinite distance). With this metric $(P_{U}(X),d)$ is a proper metric space and we denote by $P_{U}(X)_{d}$   the underlying       bornological coarse space, and by $P_{U}(X)_{t}$ the underlying topological bornological space.
The Dirac measures provide an embedding $X\to P_{U}(X)$. This map is an equivalence $X_{U}\to P_{U}(X)_{d}$ in $\BC$, see Definition \ref{ioioieorwerr3245345345432}.
We therefore get a map
$$A_U:K^{\an,\lf}_{*}(P_{U}(X)_{t})\xrightarrow{A} \KX_{*}(P_{U}(X)_{d})\stackrel{\ref{hewifeiiu3984u239824424}}{\cong} \KX_{*}(X_{U})\, .$$
These maps are compatible with the comparison maps on domain and target associated to inclusions $U\subseteq U^{\prime}$ of entourages.
We now form the colimit over the entourages $  \cC$ of $X$  and get the homomorphism 
$$\mu:  QK_\ast^{\an,\lf}(X) \cong \colim_{U\in\cC}K^{\an,\lf}_{*}(P_{U}(X)_{t})\to \colim_{U\in \cC}\KX_{*}(X_{U}) \cong \KX_{*}(X)\, ,$$
{where $  QK^{\an,\lf}$ is defined in Definition  \ref{wefiwiofuwe987u982523453453}, the first isomorphism explained in  Remark \ref{eifewfioeu9ewu982742345345} and the second isomorphism follows from  Proposition \ref{hewifeiiu3984u2398244241}.}
%

In the construction above we assumed that $X$ is a bornological  coarse space with strongly locally bounded geometry. 
 The condition of strongly locally bounded geometry is not invariant under equivalences of bornological coarse spaces. For example, the inclusion $\Z\to \R$ is an equivalence, but $\Z$  has strongly locally bounded geometry, while $\R$ has not.  

Let $X$ be a bornological coarse space.
\begin{ddd}\label{fwejfoiwejfewojfoefjewfewfewf}
$X$ has locally bounded geometry\index{bounded geometry!locally}\index{locally!bounded geometry} if  it is equivalent to a
bornological coarse space of strongly locally bounded geometry. 
\end{ddd}
Note that bounded geometry (Definition \ref{dsf892323234}) implies locally bounded geometry.

We now observe that the domain (see Proposition \ref{cegiojiojergergeg}) and target of the homomorphism $\mu$ are coarsely invariant. By naturality we can therefore  define the homomorphism $\mu$ for all $X$ of  locally bounded geometry using
a locally finite approximation $X^{\prime}\to X$
$$\mu:  Q K^{\an,\lf}_{*}(X)\cong  Q K^{\an,\lf}_{*}(X^{\prime})\to \KX_{*}(X^{\prime})\cong \KX_{*}(X)\, .$$

Let  $X$ be a bornological coarse space of locally bounded geometry. 
\begin{ddd}\label{wefijweiofoiwe245435}
 The homomorphism $$\mu:  Q K^{\an,\lf}_{*}(X)\to \KX_{*}(X)$$\index{$Ass$|see{coarse assembly map}}described above is called the $K$-theoretic coarse assembly map.\index{assembly map!coarse}\index{coarse!assembly map}
\end{ddd}

\begin{rem}
The notation
$  Q K^{\an,\lf}_{*}(X)$ for the domain of the assembly map looks complicated and this is not an accident.
This coarse homology group is defined using a complicated procedure starting with functional analytic data and using the homotopy theoretic machine of locally finite homology theories in order to produce the functor $K^{\an,\lf}$ which 
we then feed into the coarsification machine   introduced in Definition \ref{wefiwiofuwe987u982523453453}. The construction of the assmbly map heavily depends on the analytic picture, in particular on Paschke duality and the boundary operator in $K$-theory.

The spaces $P_{U}(X)$ which occcur during the construction of the assembly map are all locally finite-dimensional simplicial complexes. So by Corollary \ref{lijfoieijwofoewefewfewwf} we could replace
$QK^{\an,\lf}(X)$ by $Q(KU \wedge \Sigma^{\infty}_{+})^{\lf}$ if we wished.
\hB
\end{rem}

\begin{rem}\label{fejweiofewoiwefwf}
One could ask wether the comparison result Corollary \ref{thm:sdf98245csv} implies the coarse Baum--Connes conjecture \index{coarse!Baum--Connes problem}\index{Baum--Connes problem}  stating that the coarse assembly map $\mu$ is an isomorphism, if $X$ is a bornological coarse space of bounded geometry of weakly finite   asymptotic dimension.

Unfortunately, Corollary \ref{thm:sdf98245csv} 
does not directly apply to the 
coarse assembly map $\mu$. 

First of all $\mu$ is not defined for all bornological coarse spaces.
Furthermore it is not a transformation between  spectrum-valued functors. 

The first problem can easily be circumvented by replacing   the category $\BC$ by its full subcategory of 
 bornological coarse spaces of locally bounded geometry everywhere and applying a corresponding version of
 Corollary \ref{thm:sdf98245csv}.

The  second problem is more serious. One would need to refine the constructions leading to the map $A$ (Definition \ref{eifjweoifoiwefwefwefw}) to the spectrum level.  

In \cite{ass} we will study the construction of spectrum-valued assembly maps in a systematic manner.
It will turn out that in general we must modify the construction of the domain of the assembly map.
\hB
\end{rem}

\begin{rem}\label{rem243erds2343e}
There are examples of spaces $X$ with bounded geometry such that the $K$-theoretic coarse assembly map $\mu:  Q K^{\an,\lf}_{*}(X)\to \KX_{*}(X)$ is not surjective, i.e., these spaces are counter-examples to the coarse Baum--Connes conjecture. For the construction of these examples see Higson \cite{counterex_coarse_BC} and Higson--Lafforgue--Skandalis \cite{counterex_BC}.

Composing the $K$-theoretic coarse assembly map with the natural transformation $\KX \to \KXql$ we get the quasi-local version
\[\mu_{\ql}:  Q K^{\an,\lf}_{*}(X)\to \KX_{ql,*}(X)\]
of the coarse assembly map. The natural question is now whether the above surjectivity counter-examples still persist in the quasi-local case.

We were not able to adapt Higson's arguments to the quasi-local case, i.e., the question whether we do have surjectivity counter-examples to the quasi-local version of the coarse Baum--Connes conjecture is currently open.
\hB
\end{rem}

\bibliographystyle{alpha}
\bibliography{born}

\printindex

\end{document}